\documentclass[onefignum,onetabnum]{siamart190516}

\usepackage[utf8]{inputenc}
\usepackage{graphicx}
\usepackage{amsmath,amssymb}
\usepackage{stmaryrd}
\usepackage{tikz}
\usepackage{tikz-cd}
\usepackage{diagbox}

\usepackage[inline]{enumitem} 

\usetikzlibrary{arrows,matrix}

\usetikzlibrary{calc,patterns,shapes.geometric}
\usetikzlibrary{arrows,shapes,positioning}
\usetikzlibrary{decorations.markings,decorations.pathmorphing,decorations.pathreplacing}
\usetikzlibrary{decorations.text,decorations.shapes}

\usepackage[bbgreekl]{mathbbol}

\usepackage{bbm}  
\usepackage{comment}
\usepackage{subfig}

\newif\ifshowfig

\showfigfalse  

\ifshowfig
    \providecommand{\plotfig}[1]{#1}
\else
  \providecommand{\plotfig}[1]{%
    $$\text{(Plot hidden for now, define \texttt{showfigtrue} in header to show)}$$
  }
\fi

\usepackage{hyperref}

\input{macros_ppc}

\headers{Broken FEEC on mapped multipatch domains}{Y. Güçlü, S. Hadjout and M. Campos Pinto}

\title{A broken FEEC framework for electromagnetic problems on mapped multipatch domains%
\thanks{Preprint version of October 6, 2022
\funding{This work was partially supported by the European Council under the Horizon 2020 Project Energy oriented Centre of Excellence for computing applications - EoCoE, Project ID 676629.}
}}

\author{Yaman Güçlü%
  \thanks{Max-Planck-Institut für Plasmaphysik, Garching, Germany (%
  \email{yaman.guclu@ipp.mpg.de}, \email{said.hadjout@ipp.mpg.de}, \email{martin.campos-pinto@ipp.mpg.de}
  ).}
  \and Said Hadjout\footnotemark[2]
  \and Martin Campos Pinto\footnotemark[2]
}

\begin{document}

\maketitle

\begin{abstract}
We present a framework for the structure-preserving approximation of partial differential 
equations on mapped multipatch domains, extending the classical theory of finite element exterior calculus (FEEC) 
to discrete de Rham sequences which are broken, i.e., fully discontinuous across the patch interfaces.
Following the Conforming/Nonconforming Galerkin (CONGA) schemes developed in \cite{Campos-Pinto.Sonnendrucker.2016.mcomp,conga_hodge}, our approach is based on:
\begin{enumerate*}
  \item[(i)] the identification of a conforming discrete de Rham sequence with stable commuting projection operators,
  \item[(ii)] the relaxation of the continuity constraints between patches, and
  \item[(iii)] the construction of conforming projections mapping back to the conforming subspaces,
  allowing to define discrete differentials on the broken sequence.
\end{enumerate*}

This framework combines the advantages of conforming FEEC discretizations 
(e.g. commuting projections, discrete duality and Hodge-Helmholtz decompositions) 
with the data locality and implementation simplicity of interior penalty methods 
for discontinuous Galerkin discretizations.
We apply it to several initial- and boundary-value problems, 
as well as eigenvalue problems arising in electromagnetics.
In each case our formulations are shown to be well posed 
thanks to an appropriate stabilization of the jumps across the interfaces,
and the solutions are 
extremely robust with respect to the stabilization parameter.

Finally we describe a construction using tensor-product splines on mapped cartesian patches,
and we detail the associated matrix operators. Our numerical experiments confirm
the accuracy and stability of this discrete framework, 
and they allow us to verify that expected structure-preserving properties 
such as divergence or harmonic constraints are respected to floating-point accuracy.
 
\end{abstract}



\tableofcontents

\section{Introduction}
\label{sec:intro}

Thanks to enlightening research conducted over the last few decades 
\cite{Bossavit.1998.ap,Hyman_Shashkov_1999_jcp,Hiptmair.2002.anum,Gross_Kotiuga_2004_cup,%
Arnold.Falk.Winther.2006.anum,Boffi.2010.anum,Buffa_2011},
it is now well understood that preserving the geometrical de Rham structure 
of physical problems 
is a key tool in the design of good finite element methods. 
A significant field of success is electromagnetics,
where this principle has produced stable and accurate discretization methods,
from simplicial Whitney forms \cite{Whitney.1957.pup,Bossavit_1988_whitney} 
to high order curved elements in isogeometric analysis
\cite{buffa2010isogeometric,daVeiga_2014_actanum},  
via edge N\'ed\'elec elements \cite{Monk.1993.jcam,Bonazzoli_Rapetti_2016_na}.
The latter, in particular, have been proven to yield Maxwell solvers that are free 
of spurious eigenmodes 
in a series of works dedicated to this issue 
\cite{Bossavit.1990.IEEE-tm,
Boffi.Fernandes.Gastaldi.Perugia.1999.sinum,%
Caorsi.Fernandes.Raffetto.2000.sinum,Monk.Demkowicz.2001.mcomp}.
Here the existence of stable commuting projection operators plays a central role, as highlighted
in the unifying analysis of finite element exterior calculus (FEEC)
\cite{Arnold.Falk.Winther.2006.anum,Arnold.Falk.Winther.2010.bams}. 

A central asset of structure-preserving finite elements is their ability to reproduce discrete Hodge-Helmholtz
decompositions which, in combination with proper commuting projection operators, allow them to preserve
important physical invariants such as the divergence constraints in Maxwell's equations
\cite{CPJSS_2014_crm,MTO_2015_cpc,CPMS_2016_amc}, or
the Hamiltonian structure of the MHD and Vlasov-Maxwell equations
\cite{gempic_2016_jpp,HPRW_2020_jcp,CPKS_variational_2020}.
In the latter application the aforementioned commuting projection operators couple structure-preserving finite element fields with numerical particles.


As they primarily involve strong differential operators, structure-preserving finite elements have been 
essentially developed within the scope of {\em conforming} methods, where the discrete spaces
form a sub-complex of the continuous de Rham sequence.
In practice this imposes continuity conditions at the cell interfaces
which strongly degrade the locality of key operations such as
$L^2$~projections, as sparse finite element mass matrices have no sparse inverses in general.
In the framework of dual complexes this leads to global discrete Hodge operators mapping 
the dual de Rham sequence to the primal one, as well as to global commuting projection operators 
on the dual sequence, as the latter also rely on $L^2$~projection operators on the finite element spaces.



In this article we follow the {\em broken} FEEC approach 
\cite{conga_hodge} first developed for Conforming/Nonconforming Galerkin (CONGA) schemes
in \cite{Campos-Pinto.Sonnendrucker.2016.mcomp,Campos-Pinto.Sonnendrucker.2017a.jcm}.
The principle is to consider local de Rham sequences on subdomains 
and a global finite element space that is broken, i.e. fully discontinuous at the interfaces.
The strong differential operators are then applied by ways of conforming projection 
operators that enforce the proper continuity conditions at the interfaces.
%
This approach yields block-diagonal matrices for $L^2$ projections (i.e.~mass matrices),
Hodge operators, and dual commuting projections,
while it satisfies the key properties of discrete structure-preservation,
such as primal/dual commuting diagrams and discrete Hodge-Helmholtz decompositions \cite{conga_hodge}.
The coupling between subdomains is encoded by the conforming projection operators, which can be highly local:
in practice these may only involve the averaging of degrees of freedom across interfaces,
yielding low-rank sparse matrices.
Ultimately this approach allows us to construct structure-preserving finite element solvers
on complex domains, which are easy to implement and efficient.

Specifically, we extend the theory to the discretization of several boundary value problems 
arising in electromagnetics, and we describe its application to multipatch mapped spline discretizations 
\cite{Buffa_2011,daVeiga_2014_actanum} 
on general non-contractible domains. As our results show, this approach allows us to preserve most properties
of the conforming FEEC approximations, such as stability and accuracy of the solutions, topological invariants, 
as well as harmonic and divergence constraints of the discrete fields.

The outline is as follows:
We first recall in Section~\ref{sec:prince} the main lines of FEEC discretizations
using conforming spaces and we describe their extension to broken spaces,
with a detailed description of the fully discrete diagrams involving the
primal (strong) and dual (weak) de Rham sequences and their respective broken commuting 
projection operators. 
In Section \ref{sec:pbms} we apply this discretization framework to a series
of classical electromagnetic problems, namely Poisson's and harmonic Maxwell's equations,
curl-curl eigenvalue problems, and magnetostatic problems; we also recall the approximation of
the time-dependent Maxwell equations from \cite{Campos-Pinto.Sonnendrucker.2016.mcomp}.
For each problem we state a priori results about the solutions, assuming
the well-posedness of the corresponding conforming FEEC discretization.
In Section \ref{sec:geo_bfeec} we detail the construction of
a geometric broken-FEEC spline discretization on mapped multipatch domains.
Here we consider a 2D setting for simplicity, but the same method applies to 3D domains.
In Section~\ref{sec:num} we conduct extensive numerical experiments for the electromagnetic problems
described in the article, which verify the robustness of our approach.
Finally, in Section~\ref{sec:con} we summarize the main results of our work
and provide an outlook on future research.

%

\section{Principle of FEEC and Broken-FEEC discretizations}
\label{sec:prince}

\subsection{De Rham sequences and Hodge-Helmholtz decompositions}
\label{sec:dR}

In this work we consider 
discretizations of Hilbert de Rham sequences. At the continuous level these are of the form
\begin{equation} \label{dR}
  V^0 \xrightarrow{ \mbox{$~ \grad ~$}}
    V^1 \xrightarrow{ \mbox{$~ \curl ~$}}
      V^2 \xrightarrow{ \mbox{$~ \Div ~$}}
        V^3
\end{equation}
with infinite-dimensional spaces such as
\begin{equation} \label{dR_spaces_hom}
  V^0 = H^1_0(\Omega), \quad
    V^1 = H_0(\curl;\Omega), \quad
      V^2 = H_0(\Div;\Omega), \quad
        V^3 = L^2(\Omega).
\end{equation}
Following the well established analysis of Hilbert complexes by Arnold, Falk and Winther 
\cite{Arnold.Falk.Winther.2006.anum,Arnold.Falk.Winther.2010.bams} we also consider the dual sequence
\begin{equation} \label{dR*}
  V^*_0 \xleftarrow{ \mbox{$~ \Div ~$}}
    V^*_1 \xleftarrow{ \mbox{$~ \curl ~$}}
      V^*_2 \xleftarrow{ \mbox{$~ \grad ~$}}
        V^*_3
\end{equation}
involving the adjoint differential operators (denoted with their usual name)
and their corresponding domains, namely 
\begin{equation} \label{dR_spaces_inhom}
  V^*_3 = H^1(\Omega), \quad
    V^*_2 = H(\curl;\Omega), \quad
      V^*_1 = H(\Div;\Omega), \quad
        V^*_0 = L^2(\Omega).
\end{equation}
Here the construction is symmetric, in the sense that the inhomogeneous sequence \eqref{dR_spaces_inhom} 
could have been chosen for the primal one and the homogeneous \eqref{dR_spaces_hom} for the dual one.
Since the symmetry is broken in the finite element discretization, 
to fix the ideas in this article we mostly consider the choice \eqref{dR}--\eqref{dR_spaces_inhom}, 
except for a few places where we adopt a specific notation.
A key property of these sequences is that each operator maps into the kernel of the next one,
i.e., we always have $\curl \grad = 0$ and $\Div \curl = 0$.
This allows us to write an orthogonal Hodge-Helmholtz decomposition for $L^2(\Omega)^3$, of the form
\begin{equation} \label{HH}
  L^2(\Omega)^3 = \grad V^0 \poplus \cH^1 \poplus \curl V^*_2
\end{equation}
where $\cH^1 := \{ v \in V^1 \cap V^*_1 : \curl v = \Div v = 0\}$, see e.g. \cite[Eq.~(15)]{Arnold.Falk.Winther.2010.bams}.
This space corresponds to harmonic 1-forms, as it coincides with the kernel of the 1-form Hodge-Laplace operator
\begin{equation} \label{L1}
  \cL^{1} := -\grad \Div + \curl \curl
\end{equation} 
seen as an operator $\cL^1: D(\cL^1) \to L^2(\Omega)$ with domain space
\begin{equation} \label{DL1}
  D(\cL^1) := \{v \in V^1 \cap V^*_1 : \curl v \in V^*_2 \text{ and } \Div v \in V^0 \}.
\end{equation} 

Another orthogonal decomposition for $L^2(\Omega)^3$ is 
\begin{equation} \label{HH2}
  L^2(\Omega)^3 = \curl V^1 \poplus \cH^2 \poplus \grad V^*_3,
\end{equation}
where $\cH^2 := \{w \in V^2 \cap V^*_2 : \curl w = \Div w = 0\}$ is the space of harmonic 2-forms, 
which coincides with the kernel of the 2-form Hodge-Laplace operator 
$$\cL^{2} := -\grad \Div + \curl \curl
$$ with domain space 
$$D(\cL^2) := \{w \in V^2 \cap V^*_2 : \curl w \in V^*_1 \text{ and } \Div w \in V^3\}.$$

We point out that, while the 1-form and 2-form Hodge-Laplace operators are formally identical, their domain spaces differ in the boundary conditions and hence $\cH^1 \neq \cH^2$: in the case considered here where the primal sequence has homogeneous boundary conditions, 
the harmonic 1-forms have vanishing tangential trace on~$\partial\Omega$, while the harmonic 2-forms have vanishing normal trace on~$\partial\Omega$.
In the symmetric case where the non-homogeoneous de Rham sequence is considered as the primal one, 
one obtains the same decompositions \eqref{HH} and \eqref{HH2} but with opposite order:
$(\cH^\ell)_\text{non-hom.} = \cH^{3-\ell}$.
This isomorphism, known as Poincaré duality, is provided by the Hodge star operator; see \cite[Sec.~5.6 and 6.2]{Arnold.Falk.Winther.2010.bams}.

On contractible domains the above sequences are exact in the sense that 
the image of each operator coincides exactly with the kernel of the next one, 
and the harmonic space is trivial, $\cH^1 = \{0\}$.
However if the domain $\Omega$ is non-contractible, this is no longer the case and
there exist non trivial harmonic forms.
Indeed, the dimensions of these two harmonic spaces depend on the domain topology: 
$\dim(\cH^2) = \dim(\cH^1)_\text{non-hom.} = b_1(\Omega)$ is the first Betti number, which counts the number of ``tunnels'' through the domain,
while $\dim(\cH^1) = \dim(\cH^2)_\text{non-hom.} = b_2(\Omega)$ is the second Betti number, which counts the number of ``voids'' enclosed by the domain.

\subsection{Conforming FEEC discretizations}
\label{sec:feec}

Finite Element Exterior Calculus (FEEC) discretizations consist 
of Finite Element spaces that form discrete de Rham sequences,
\begin{equation} \label{dR_hc}
  V^{0,c}_h \xrightarrow{ \mbox{$~ \grad ~$}}
    V^{1,c}_h \xrightarrow{ \mbox{$~ \curl ~$}}
      V^{2,c}_h \xrightarrow{ \mbox{$~ \Div ~$}}
        V^{3,c}_h~.
\end{equation}
Here, the superscript~$c$ indicates that the spaces are assumed {\em conforming}
in the sense that $V^{\ell,c}_h \subset V^{\ell}$, and the subscript~$h$ loosely
represents some discretization parameters, such as the resolution of an underlying mesh.
A key tool in the analysis of FEEC discretizations is the existence of projection operators 
$\Pi^{\ell,c}_h : V^\ell \to V^{\ell,c}_h$ that commute with the differential operators, 
in the sense that the relation
\begin{equation} \label{cd}
d^\ell \Pi^{\ell,c}_h v = \Pi^{\ell+1,c}_h d^\ell v
\qquad %
v \in V^\ell
\end{equation}
holds for the different operators in the sequence \eqref{dR_hc}, 
\begin{equation} \label{dell}
  d^0 = \grad,
  \qquad 
  d^1 = \curl,
  \qquad %
  d^2 = \Div.
\end{equation}
In particular, the stability and the accuracy of several discrete problems posed in the sequence \eqref{dR_hc},
relative to usual discretization parameters such as the mesh resolution $h$ 
or the order of the finite element spaces, follow from the stability 
of the commuting projections in $V^\ell$ or $L^2$ norms,
see e.g. \cite[Th.~3.9 and 3.19]{Arnold.Falk.Winther.2010.bams}.

Denoting explicitely by
\begin{equation} \label{d_c}
\grad^c_h := \grad|_{V^{0,c}_h}, 
\qquad 
\curl^c_h := \curl|_{V^{1,c}_h}, 
\qquad
\Div^c_h := \Div|_{V^{2,c}_h}
\end{equation}
the differential operators restricted to the discrete spaces, 
we define their discrete adjoints
\begin{equation} \label{wtops_c}
\left\{\begin{aligned}
  \wt \Div^c_h &:= (-\grad^c_h)^* : V^{1,c}_h \to V^{0,c}_h 
  \\ 
  \wt \curl^c_h &:= (\curl^c_h)^*: V^{2,c}_h \to V^{1,c}_h 
  \\ 
  \wt \grad^c_h &:= (-\Div^c_h)^* : V^{3,c}_h \to V^{2,c}_h
\end{aligned}\right.
\end{equation}
by $L^2$ duality, i.e., 
\begin{equation} \label{d*_c}
\left\{\begin{aligned}
\sprod{\wt \Div^c_h v}{\vp} &= -\sprod{v}{\grad \vp} \qquad &&\forall  v \in V^{1,c}_h, ~  \vp \in V^{0,c}_h
\\
\sprod{\wt \curl^c_h \bw}{v} &= \sprod{w}{\curl v} \qquad &&\forall w \in V^{2,c}_h, ~ v \in V^{1,c}_h
\\
\sprod{\wt \grad^c_h \rho}{w} &= -\sprod{\rho}{\Div w} \qquad &&\forall \rho \in V^{3,c}_h, ~ w \in V^{2,c}_h
\end{aligned}\right.
\end{equation}
where $\sprod{\cdot}{\cdot}$ denotes the $L^2(\Omega)$ scalar product.
This yields a compatible discretization of both the primal and dual sequences \eqref{dR}, \eqref{dR*} 
in strong and weak form, respectively, using the same spaces \eqref{dR_hc}.
This framework can be summarized in the following diagram
\begin{equation} \label{CD_hc}
  \begin{tikzpicture}[ampersand replacement=\&, baseline] 
  \matrix (m) [matrix of math nodes,row sep=3em,column sep=5em,minimum width=2em] {
        ~~ V^{0} ~ \bbb
            \& ~~ V^{1} ~ \bbb
                \& ~~ V^{2} ~ \bbb
                      \& ~~ V^{3} ~ \bbb
    \\
    ~~ V^{0,c}_h ~ \bbb
        \& ~~~ V^{1,c}_h \bbb   
            \& ~~ V^{2,c}_h ~ \bbb
              \& ~~ V^{3,c}_h ~ \bbb
    \\
    ~~ V^*_{0} ~ \bbb
        \& ~~ V^*_{1} ~ \bbb
            \& ~~ V^*_{2} ~ \bbb
            \& ~~ V^*_{3} ~ \bbb
    \\
  };
  \path[-stealth]
  (m-1-1) edge node [above] {$\grad$} (m-1-2)
          edge node [right] {$\Pi^{0,c}_h$} (m-2-1)
  (m-1-2) edge node [above] {$\curl$} (m-1-3)
          edge node [right] {$\Pi^{1,c}_h$} (m-2-2)
  (m-1-3) edge node [above] {$\Div$} (m-1-4)
          edge node [right] {$\Pi^{2,c}_h$} (m-2-3)
  (m-1-4) edge node [right] {$\Pi^{3,c}_h$} (m-2-4)
  (m-2-1.10) edge node [above] {$\grad^c_h$} (m-2-2.170)
  (m-2-2.10) edge node [above] {$\curl^c_h$} (m-2-3.170)
  (m-2-3.10) edge node [above] {$\Div^c_h$} (m-2-4.170)
  (m-2-2.190) edge node [below] {$\wt \Div^c_h$} (m-2-1.350)
  (m-2-3.190) edge node [below] {$\wt \curl^c_h$} (m-2-2.350)
  (m-2-4.190) edge node [below] {$\wt \grad^c_h$} (m-2-3.350)
  (m-3-2) edge node [above] {$\Div$} (m-3-1)
  (m-3-3) edge node [above] {$\curl$} (m-3-2)
  (m-3-4) edge node [above] {$\grad$} (m-3-3)
  (m-3-1) edge node [right] {$Q_{V^{0,c}_h}$} (m-2-1)
  (m-3-2) edge node [right] {$Q_{V^{1,c}_h}$} (m-2-2)
  (m-3-3) edge node [right] {$Q_{V^{2,c}_h}$} (m-2-3)
  (m-3-4) edge node [right] {$Q_{V^{3,c}_h}$} (m-2-4)
  ;
  \end{tikzpicture}
\end{equation}
where the operators $Q_{V^{\ell,c}_h} : L^2(\Omega) \to V^{\ell,c}_h$ represent $L^2$~projections to the conforming discrete spaces.
We observe that these projections commute with the dual differential operators, as a result of~\eqref{d*_c}.
\newline
\newline
{\em Assumption 1}~
  Throughout the article we assume that the primal projection operators 
  $\Pi^{\ell,c}_h$ are $L^2$ stable and satisfy the commuting property \eqref{cd}.

\begin{remark} \label{rem:Pi}
  In practice commuting projection operators also play an important role as they permit
  to approximate coupling or source terms in a structure-preserving way, see e.g.~\cite{CPKS_variational_2020}.
  For this purpose $V^\ell$ or $L^2$-stable projections may not be the best choices 
  as they can be difficult to apply, and simpler commuting projections,
  defined through proper degrees of freedom, are often preferred. 
  These projections are then usually defined on sequences 
  \begin{equation} \label{dR_U}
    U^{0} \xrightarrow{ \mbox{$~ \grad ~$}}
      U^{1} \xrightarrow{ \mbox{$~ \curl ~$}}
        U^{2} \xrightarrow{ \mbox{$~ \Div ~$}}
          U^{3}
  \end{equation}
  involving spaces $U^\ell \subset V^\ell$ that require more smoothness (or integrability) as the ones in
  \eqref{dR_spaces_hom}.
  We refer to e.g. \cite{Nedelec.1980.numa,kreeft2011mimetic,Boffi.Brezzi.Fortin.2013.scm,CPKS_variational_2020}
  for some examples, and to Section~\ref{sec:proj_bfeec} below for more details on such 
  constructions.  
\end{remark}

Another asset of FEEC discretizations is to provide structure-preserving Hodge-Helmholtz decompositions
for the different spaces. For 1-forms, the discrete analog of the continuous decomposition \eqref{HH}
reads
\begin{equation} \label{HH_c}
  V^{1,c}_h = \grad^c_h V^{0,c}_h \poplus \cH^1_h \poplus \wt \curl^c_h V^{2,c}_h 
\end{equation}
where $\cH^1_h := \{v \in V^{1,c}_h : \curl^c_h v = \wt \Div^c_h v = 0\}$ 
is the kernel of the discrete Hodge-Laplace operator $\cL^{1,c}_h: V^{1,c}_h \to V^{1,c}_h$
defined as
\begin{equation} \label{L1hc}
  \cL^{1,c}_h := -\grad^c_h \wt \Div^c_h + \wt \curl^c_h \curl^c_h~.
\end{equation}
The space $\cH^1_h$ may thus be seen as discrete harmonic 1-forms
and under suitable approximation properties of the discrete spaces, 
its dimension coincides with that of the continuous harmonic forms $\cH^1$,
which corresponds to a Betti number of $\Omega$ depending on the boundary conditions, 
see \cite[Sec.~5.6 and 6.2]{Arnold.Falk.Winther.2010.bams}.

\subsection{Broken FEEC discretizations}
\label{sec:bfeec}

In the case where the domain is decomposed
in a partition of open subdomains $\Omega_k$, $k = 1, \dots, K$,
we now consider local FEEC sequences 
\begin{equation} \label{dR_h_loc}
  V^{0}_h(\Omega_k) \xrightarrow{ \mbox{$~ \grad_{\Omega_k}  ~$}}
    V^{1}_h(\Omega_k) \xrightarrow{ \mbox{$~ \curl_{\Omega_k} ~$}}
      V^{2}_h(\Omega_k) \xrightarrow{ \mbox{$~ \Div_{\Omega_k} ~$}}
        V^{3}_h(\Omega_k)
\end{equation}
and global spaces obtained by a simple juxtaposition of the local ones
\begin{equation} \label{Vh_broken}
  V^\ell_h := \{ v \in L^2(\Omega) : v|_{\Omega_k} \in V^{\ell}_h(\Omega_k) \}.
\end{equation}

One attractive feature of the broken spaces \eqref{Vh_broken} is that they are naturally equipped with local
basis functions $\Lambda^\ell_i$ that are supported each on a single patch $\Omega_k$ with $k = k(i)$.
The corresponding mass matrices are then patch-diagonal, i.e., block-diagonal with blocks 
corresponding to the different patches, so that their inversion
-- and hence the $L^2$ projection on $V^\ell_h$ -- can be performed in each
patch independently of the others. 

In general the spaces \eqref{Vh_broken} are not subspaces of their infinite-dimensional 
counterparts, indeed it is well-known that piecewise smooth fields must satisfy some interface 
constraints in order to be globally smooth \cite{Boffi.Brezzi.Fortin.2013.scm}:
for $H^1$ smoothness the fields must be continuous on the interfaces, while
$H(\curl)$ and $H(\Div)$ smoothness of vector-valued fields require the continuity 
of the tangential and normal components, respectively, on the interfaces.
As these constraints are obviously not satisfied by the broken spaces \eqref{Vh_broken},
we have in general
$$
V^{\ell}_h \not\subset V^\ell.
$$
The approach developed for the CONGA schemes in \cite{Campos-Pinto.Sonnendrucker.2016.mcomp,conga_hodge} 
extends the construction of Section~\ref{sec:feec} to this broken FEEC setting, 
by associating to each discontinuous space a projection operator
on its conforming subspace,
\begin{equation} \label{conf_proj}
  P^\ell_h: V^\ell_h \to  V^\ell_h \cap V^\ell =: V^{\ell,c}_h.
\end{equation}
Provided that these conforming spaces form a de Rham sequence \eqref{dR_hc},
this allows us to define new primal differential operators on the broken spaces
\begin{equation} \label{dh}
\left\{\begin{alignedat}{3}
  \grad_h &:= \grad P^0_h &: ~V^{0}_h \to V^{1,c}_h \subset V^{1}_h \\
  \curl_h &:= \curl P^1_h &: ~V^{1}_h \to V^{2,c}_h \subset V^{2}_h \\
  \Div_h  &:= \Div  P^2_h &: ~V^{2}_h \to V^{3,c}_h \subset V^{3}_h
\end{alignedat}\right.
\end{equation}
and new dual ones 
$\wt \Div_h:V^{1}_h \to V^{0}_h$, $\wt \curl_h: V^{2}_h \to V^{1}_h$ and $\wt \grad_h: V^{3}_h \to V^{2}_h$ 
as $L^2$ adjoints, characterized by the relations
\begin{equation} \label{d*}
\left\{\begin{alignedat}{4}
\sprod{\wt \Div_h v}{\vp}   &= -\sprod{v}{\grad P^0_h \vp}         \qquad && \forall &    v &\in V^{1}_h, &~ \vp &\in V^{0}_h \\
\sprod{\wt \curl_h w}{v}    &= \phantom{-}\sprod{w}{\curl P^1_h v} \qquad && \forall &    w &\in V^{2}_h, &~   v &\in V^{1}_h \\
\sprod{\wt \grad_h \rho}{w} &= -\sprod{\rho}{\Div P^2_h w}         \qquad && \forall & \rho &\in V^{3}_h, &~   w &\in V^{2}_h~.
\end{alignedat}\right.
\end{equation}
We represent this broken-FEEC discretization with the following diagram 
\begin{equation} \label{CD_h}
  \begin{tikzpicture}[ampersand replacement=\&, baseline] 
  \matrix (m) [matrix of math nodes,row sep=3em,column sep=5em,minimum width=2em] {
        ~~ V^{0} ~ \bbb
            \& ~~ V^{1} ~ \bbb
                \& ~~ V^{2} ~ \bbb
                      \& ~~ V^{3} ~ \bbb
    \\
    ~~ V^{0}_h ~ \bbb
        \& ~~~ V^{1}_h \bbb   
            \& ~~ V^{2}_h ~ \bbb
              \& ~~ V^{3}_h ~ \bbb
    \\
    ~~ V^*_{0} ~ \bbb
        \& ~~ V^*_{1} ~ \bbb
            \& ~~ V^*_{2} ~ \bbb
            \& ~~ V^*_{3} ~ \bbb
    \\
  };
  \path[-stealth]
  (m-1-1) edge node [above] {$\grad$} (m-1-2)
          edge node [right] {$\Pi^{0}_h$} (m-2-1)
  (m-1-2) edge node [above] {$\curl$} (m-1-3)
          edge node [right] {$\Pi^{1}_h$} (m-2-2)
  (m-1-3) edge node [above] {$\Div$} (m-1-4)
          edge node [right] {$\Pi^{2}_h$} (m-2-3)
  (m-1-4) edge node [right] {$\Pi^{3}_h$} (m-2-4)
  (m-2-1.10) edge node [above] {$\grad_h$} (m-2-2.170)
  (m-2-2.10) edge node [above] {$\curl_h$} (m-2-3.170)
  (m-2-3.10) edge node [above] {$\Div_h$} (m-2-4.170)
  (m-2-2.190) edge node [below] {$\wt \Div_h$} (m-2-1.350)
  (m-2-3.190) edge node [below] {$\wt \curl_h$} (m-2-2.350)
  (m-2-4.190) edge node [below] {$\wt \grad_h$} (m-2-3.350)
  (m-3-2) edge node [above] {$\Div$} (m-3-1)
  (m-3-3) edge node [above] {$\curl$} (m-3-2)
  (m-3-4) edge node [above] {$\grad$} (m-3-3)
  (m-3-1) edge node [right] {$\tilde \Pi^0_h$} (m-2-1)
  (m-3-2) edge node [right] {$\tilde \Pi^1_h$} (m-2-2)
  (m-3-3) edge node [right] {$\tilde \Pi^2_h$} (m-2-3)
  (m-3-4) edge node [right] {$\tilde \Pi^3_h$} (m-2-4)
  ;
  \end{tikzpicture}
\end{equation}
where $\Pi^\ell_h$ and $\tilde \Pi^\ell_h$ are projection operators
that commute with the new primal and dual sequences. 
We postpone their description to the next section.

We note that the presence of conforming projections $P^\ell_h$
in the differential operators \eqref{dh} leads to larger kernels. 
Specifically, we have
$$
\left\{
\begin{alignedat}{2}
\ker \grad_h &= \big(V^{0,c}_h \cap \ker \grad &&\big) \oplus \ker P^0_h \\
\ker \curl_h &= \big(V^{1,c}_h \cap \ker \curl &&\big) \oplus \ker P^1_h \\
\ker \Div_h  &= \big(V^{2,c}_h \cap \ker \Div  &&\big) \oplus \ker P^2_h
\end{alignedat}
\right.  
$$
where the projection kernels
\begin{equation} \label{kerP}
  \ker P^\ell_h = (I-P^\ell_h) V^\ell_h
\end{equation}
correspond intuitively to ``jump spaces'' associated with the 
conforming projections.
In \cite{Campos-Pinto.2016.cras,Campos-Pinto.Sonnendrucker.2017a.jcm}
these extended kernels motivated a modification of the discrete 
differential operators in order to retain exact sequences on contractible domains. 
Here we follow the approach of \cite{conga_hodge}
where the jump spaces \eqref{kerP} are handled by stabilisation and filtering 
operators in the equations, through extended Hodge-Helmholtz decompositions.


\subsection{Commuting projections with broken degrees of freedom}
\label{sec:proj_bfeec}

A practical approach for designing commuting projection operators 
$\Pi^{\ell,c}_h$ and $\Pi^\ell_h$ for the above diagrams 
is to define them via commuting degrees of freedom on the 
discrete Finite Element spaces. In the standard case of conforming 
spaces these are linear forms
$$
\sigma^{\ell,c}_i : U^\ell \to \RR, \qquad i = 1, \dots, N^{\ell,c} := \dim(V^{\ell,c}_h)
$$
defined on infinite-dimensional spaces $U^\ell \subset V^\ell$ satisfying $d^\ell U^\ell \subset U^{\ell+1}$,
that are unisolvent in $V^{\ell,c}_h$ and commute with 
the differential operators
in the sense that there exist coefficients $D^\ell_{i,j}$ such that 
\begin{equation} \label{dof-comp}
  \sigma^{\ell+1,c}_i( d^\ell v) = \sum_{j=1}^{N^{\ell,c}} D^\ell_{i,j} \sigma^{\ell,c}_j(v) 
    \qquad \text{for } ~ i = 1, \dots N^{\ell+1,c}, \quad v \in U^\ell
\end{equation}
where we have denoted again $d^0 = \grad$, $d^1 = \curl$ and $d^2=\Div$.
Letting $\Lambda^{\ell,c}_i$ be the basis of $V^{\ell,c}_h$ characterized
by the relations $\sigma^{\ell,c}_i(\Lambda^{\ell,c}_j) = \delta_{i,j}$ for 
$i,j = 1, \dots N^{\ell,c}$, one verifies indeed that the operators 
$$
\Pi^{\ell,c}_h: U^\ell \to V^{\ell,c}_h, \quad v \mapsto \sum_{i=1}^{N^{\ell,c}} \sigma^{\ell,c}_i(v) \Lambda^{\ell,c}_i
$$
are projections satisfying the commuting property \eqref{cd}, see e.g. \cite[Lemma~2]{CPKS_variational_2020}
and in passing we note that 
$D^\ell$ is the matrix of $d^\ell$ in the bases just defined.

\bigskip
Here, we extend this approach as follows:
\begin{itemize}
  \item {\em Broken degrees of freedom.} 
  Each local space $V^{\ell}_h(\Omega_k)$, $k = 1, \ldots, K$, is
  equip{\-}ped with unisolvent degrees of freedom 
  \begin{equation} \label{brok-dof}
  \sigma^{\ell}_{k,\mu} : U^\ell(\Omega_k) \to \RR, \qquad  \mu \in \cM^\ell_h(\Omega_k)
  \end{equation}
  defined on local spaces $U^\ell(\Omega_k) \subset V^\ell(\Omega_k)$ 
  satisfying $d^\ell U^\ell(\Omega_k) \subset U^{\ell+1}(\Omega_k)$, 
  with multi-index sets with cardinality $\#\cM^\ell_h(\Omega_k) = \dim(V^{\ell}_h(\Omega_k))$. 
  On the full domain we simply set
  \begin{equation} \label{brok-dof-2}
    \sigma^{\ell}_{k,\mu}(v) := \sigma^{\ell}_{k,\mu}(v|_{\Omega_k}).
  \end{equation}
  
  \item {\em Broken basis functions.} 
  To the above degrees of freedom we associate 
  local basis functions $\Lambda^\ell_{k,\mu} \in V^\ell_h(\Omega_k)$
  (extended by zero outside of their patch),
  characterized by the relations
  $$
  \sigma^{\ell}_{k,\mu}(\Lambda^{\ell}_{k,\nu}) = \delta_{\mu,\nu}
  \qquad \text{ for } \quad \mu, \nu \in \cM^\ell_h(\Omega_k).
  $$
  
  \item {\em Local commutation property.} 
    A relation similar to \eqref{dof-comp} must hold on each 
    patch $\Omega_k$, i.e., there exist coefficients $D^\ell_{k,\mu,\nu}$ such that 
    \begin{equation} \label{loc-dof-comp}
      \sigma^{\ell+1}_{k,\mu}(d^\ell v) = \sum_{\nu \in \cM^\ell_h(\Omega_k)} D^\ell_{k,\mu,\nu} \sigma^{\ell}_{k,\nu}(v) 
    \end{equation}
    holds for all $\mu \in \cM^{\ell+1}_h(\Omega_k)$ and all $v \in U^\ell(\Omega_k)$.
    
  \item {\em Inter-patch conformity}. 
    The global projection on $V^{\ell}_h$, 
    \begin{equation} \label{brok-proj}
      \Pi^{\ell}_h: v \mapsto \sum_{k=1}^K \Pi^{\ell}_{h,k} v 
      \quad \text{ with } \quad 
      \Pi^{\ell}_{h,k}: v \mapsto \sum_{\mu \in \cM^{\ell}_h(\Omega_k)} \sigma^{\ell}_{k,\mu}(v) \Lambda^{\ell}_{k,\mu}      
    \end{equation}
    must map smooth functions to conforming finite element fields, namely
    \begin{equation} \label{conf-dof}
      \Pi^\ell_h U^\ell \subset V^{\ell,c}_h \quad \text{ where } \quad U^\ell := \{v \in V^\ell : v|_{\Omega_k} \in U^\ell(\Omega_k), ~ \forall k = 1, \dots K\}.
    \end{equation}
\end{itemize}
This setting guarantees that the commutation properties of the local projection operators extend to the global ones, since
one has then
$$
d^\ell P^\ell_h \Pi^\ell_h v = d^\ell \Pi^\ell_h v = \Pi^{\ell+1}_h d^\ell v, \qquad v \in U^\ell,
$$
and it yields simple expressions for the latter in the broken bases. 
We refer to Section~\ref{sec:geo_bfeec} for a detailed construction of broken degrees of freedom that satisfy the above properties.

The dual projection operators can then be defined following the canonical 
approach of \cite{conga_hodge}, as ``filtered'' $L^2$ projections
\begin{equation}  \label{tPi_def1}
  \tilde \Pi^\ell_h = (P^\ell_h)^* Q_{V^\ell_h} : L^2 \to V^\ell_h,
\end{equation}
where $Q_{V^\ell_h}$ is now the $L^2$ projection on the broken space $V^\ell_h$,
\begin{equation}  \label{tPi_def2}
  \sprod{\tilde \Pi^\ell_h v}{w} = \sprod{v}{P^\ell_h w}, \qquad \forall v \in L^2, ~ w \in V^\ell_h.
\end{equation}
By definition of the dual differentials \eqref{d*},
these operators indeed commute with the dual part of the diagram \eqref{CD_h}, 
in the sense that
\begin{equation}  \label{tCD}
 \t \Pi^0_h \Div = \wt \Div_h \t\Pi^1_h, 
 \qquad  
 \t \Pi^1_h \curl = \wt \curl_h \t\Pi^2_h, 
 \qquad 
 \t \Pi^2_h \grad = \wt \grad_h \t\Pi^3_h
\end{equation}
hold on $V^*_1$, $V^*_2$ and $V^*_3$ respectively. 
We further observe that \eqref{tPi_def1}--\eqref{tPi_def2} defines projection operators
on the subspaces of $V^\ell_h$ corresponding to the orthogonal complements of the ``jump spaces'' \eqref{kerP}.
Indeed, $(\t \Pi^\ell_h)^2 = \t \Pi^\ell_h$ holds with 
$$
\Ima(\t \Pi^\ell_h) = \Ima((P^\ell_h)^*) = (\ker P^\ell_h)^\perp = \{ v \in V^\ell_h : \sprod{v}{(I-P^\ell_h)w} = 0, ~ \forall w \in V^\ell_h\}.
$$

\subsection{Differential operator matrices in primal and dual bases}

Using broken degrees of freedom and basis functions as described in the previous section,
we now derive practical representations for the operators involved in the diagram \eqref{CD_h}.
To do so we first observe that the non-conforming differential operators $d^\ell_h := d^\ell P^\ell_h$
can be reformulated as $d^\ell_h = d^\ell_{\pw} P^\ell_h$ where 
$d^\ell_{\pw}: v \to \sum_{k=1}^K \one_{\Omega_k} d^\ell|_{\Omega_k} v$
is the patch-wise differential operator which coincides with $d^\ell$ on all $v$ such that
$d^\ell v \in L^2$.
Using some implicit flattening $(k,\mu) \mapsto i \in \{1, \dots, N^\ell\}$
for the multi-indices, we represent these 
operators as matrices 
\begin{equation} \label{GCD_matrix}
  (\matG)_{i,j} = \sigma^1_i(\grad_{\pw} \Lambda^0_j),
  \quad 
  (\matC)_{i,j} = \sigma^2_i(\curl_{\pw} \Lambda^1_j),
  \quad 
  (\matD)_{i,j} = \sigma^3_i(\Div_{\pw} \Lambda^2_j)
\end{equation}
of respective sizes $N^1 \times N^0$, $N^2 \times N^1$ and $N^3 \times N^2$.
These matrices have a ``patch-diagonal'' structure, in the sense that 
they are block-diagonal, with blocks corresponding to
the differential operators in each independent patch 
(namely, the matrices $D^\ell_k$ in \eqref{loc-dof-comp}).
The matrices of the conforming projections, seen as
endomorphisms in $V^\ell_h$, are
\begin{equation}
  \label{matP}
  (\matP^\ell)_{i,j} = \sigma^\ell_i(P^\ell_h \Lambda^\ell_j),
  \quad i,j = 1, \dots, N^\ell.
\end{equation}
In general $\matP^\ell$ is not patch-diagonal, as it maps in the coefficient space
of the conforming spaces $V^{\ell,c}_h$, where the global smoothness corresponds to 
some matching of the degrees of freedom across the interfaces.

With these elementary matrices we can build the operator matrices 
of the non-conforming differential operators $d^\ell_h$: they read 
%
\begin{equation*} \label{matsD-general}
  \begin{split}
    \left(\matO(d^\ell_h)\right)_{i,j}
   :\!\!&= \sigma^{\ell+1}_i\! \left(d^\ell_h \Lambda^\ell_j\right)
    = \sigma^{\ell+1}_i\! \left(d^\ell_{\pw} P^\ell_h \Lambda^\ell_j\right)
    = \sigma^{\ell+1}_i \Big(d^\ell_{\pw} \sum_{k=1}^{N^\ell} \Lambda^\ell_k \sigma^\ell_k\! \left(P^\ell_h \Lambda^\ell_j\right)\Big) \\
   &= \sum_{k=1}^{N^\ell} \sigma^{\ell+1}_i\! \left(d^\ell_{\pw} \Lambda^\ell_k\right)\,
      \sigma^\ell_k\! \left(P^\ell_h \Lambda^\ell_j\right)
    = \sum_{k=1}^{N^\ell} \left(\matD^\ell\right)_{i,k} \left(\matP^\ell\right)_{k,j}
    = \left(\matD^\ell \matP^\ell\right)_{i,j} ,
  \end{split}
\end{equation*}
where we have denoted $(\matD^\ell)_\ell = (\matG, \matC, \matD)$ for brevity. 
Therefore,
\begin{equation} \label{matsD}
  \left\{\begin{alignedat}{3}
  &\matO(\grad_h) 
    &&= \Big(\sigma^1_i\!\left(\grad P^0_h \Lambda^0_j\right)&\Big)_{\substack{1 \le i \le N^1 \\ 1 \le j \le N^0}}
     &= \matG\matP^0,
  \\
  &\matO(\curl_h) 
    &&= \Big(\sigma^2_i\!\left(\curl P^1_h \Lambda^1_j\right)&\Big)_{\substack{1 \le i \le N^2 \\ 1 \le j \le N^1}}
     &= \matC\matP^1,
  \\
  &\matO(\Div_h) 
    &&= \Big(\sigma^3_i\!\left(\Div P^2_h \Lambda^2_j\right)&\Big)_{\substack{1 \le i \le N^3 \\ 1 \le j \le N^2}}
     &= \matD\matP^2.
\end{alignedat}\right.
\end{equation}
A simple representation of the dual discrete operators \eqref{d*} 
is obtained in the dual bases $\{\t\Lambda^\ell_i: i = 1, \dots, N^\ell\}$ of the broken spaces $V^\ell_h$.
These are characterized by the relations
\begin{equation} \label{tLambda}
\t\sigma^\ell_i(\t\Lambda^\ell_j) = \delta_{i,j}, \qquad i, j = 1, \dots N^\ell
\end{equation}
where the dual degrees of freedom are defined as
\begin{equation} \label{tsigma}
  \t\sigma^\ell_i(v) := \sprod{v}{\Lambda^\ell_i}. 
\end{equation}
It follows from~\eqref{tLambda} and~\eqref{tsigma} that the primal and dual bases are in \(L^2\) duality, i.e.~$\sprod{\Lambda^\ell_i}{\t\Lambda^\ell_j} = \delta_{i,j}$.
The matrices of the dual differential operators in the dual bases read then
\begin{equation*} \label{matsDt-general}
\begin{split}
  \big(\t\matO(\t d^{3-\ell}_h)\big)_{i,j}
  :\!\!&= \t\sigma^\ell_i \big(\t d^{3-\ell}_h \t\Lambda^{\ell+1}_j\big)
   = \big\langle \t d^{3-\ell}_h \t\Lambda^{\ell+1}_j, \Lambda^\ell_i \big\rangle
   = (-1)^{\ell+1} \big\langle \t\Lambda^{\ell+1}_j, d^\ell_h \Lambda^\ell_i \big\rangle \\
  &= (-1)^{\ell+1} \Big\langle \t\Lambda^{\ell+1}_j, \sum_{k=1}^{N^{\ell+1}} \Lambda^{\ell+1}_k \sigma^{\ell+1}_k \big(d^\ell_h\Lambda^\ell_i\big) \Big\rangle \\
  &= (-1)^{\ell+1} \sum_{k=1}^{N^{\ell+1}} \big\langle \t\Lambda^{\ell+1}_j, \Lambda^{\ell+1}_k \big\rangle \sigma^{\ell+1}_k \big(d^\ell_h\Lambda^\ell_i\big) \\
  &= (-1)^{\ell+1} \sum_{k=1}^{N^{\ell+1}} \delta_{j,k} \left(\matD^\ell \matP^\ell\right)_{k,i}
   = (-1)^{\ell+1} \left(\matD^\ell \matP^\ell\right)_{j,i} ,
\end{split}
\end{equation*}
where again $(\matD^\ell)_\ell = (\matG, \matC, \matD)$ for brevity. Therefore,
\begin{equation} \label{matsDt}
  \left\{\begin{alignedat}{3}
  &\t\matO(\wt \Div_h)
    &&= \Big(\t\sigma^0_i\big(\wt \Div_h \t\Lambda^1_j\big)&\Big)_{\substack{1 \le i \le N^0 \\ 1 \le j \le N^1}}
     &= -(\matG\matP^0)^T,
  \\
  &\t\matO(\wt \curl_h)
    &&= \Big(\t\sigma^1_i\big(\wt \curl_h \t\Lambda^2_j\big)&\Big)_{\substack{1 \le i \le N^1 \\ 1 \le j \le N^2}}
     &= \phantom{-}(\matC\matP^1)^T,
  \\
  &\t\matO(\wt \grad_h)
    &&= \Big(\t\sigma^2_i\big(\wt \grad_h \t\Lambda^3_j\big)&\Big)_{\substack{1 \le i \le N^2 \\ 1 \le j \le N^3}}
     &= -(\matD\matP^2)^T.
  \end{alignedat}\right.
\end{equation}
The change-of-basis matrices are described in the next section.


\subsection{Primal-dual diagram in matrix form} 
\label{sec:CD}

Using the matrix form of the differential operators just described,
we extend the broken-FEEC diagram \eqref{CD_h} as follows:

\begin{equation} \label{CD}
\begin{tikzpicture}[ampersand replacement=\&, baseline] 
\matrix (m) [matrix of math nodes,row sep=3em,column sep=5em,minimum width=2em] {
   ~~ V^0 ~ \bbb
   \& ~~ V^1 ~ \bbb
    \& ~~ V^2 ~ \bbb
      \& ~~ V^3 ~ \bbb
  \\
  ~~ V_h^0 ~ \bbb
    \& ~~ V_h^1 ~ \bbb
    \& ~~ V_h^2 ~ \bbb
    \& ~~ V_h^3 ~ \bbb
\\
~~ \cC^0 ~ \bbb
  \& ~~ \cC^1 ~ \bbb
    \& ~~ \cC^2 ~ \bbb
    \& ~~ \cC^3 ~ \bbb
\\
~~ \t \cC^0 ~ \bbb
\& ~~ \t \cC^1 ~ \bbb
\& ~~ \t \cC^2 ~ \bbb
\& ~~ \t \cC^3 ~ \bbb
\\
~~ V_h^0 ~ \bbb
  \& ~~ V_h^1 ~ \bbb
  \& ~~ V_h^2 ~ \bbb
  \& ~~ V_h^3 ~ \bbb
\\
~~ V^*_0 ~ \bbb
\& ~~ V^*_1 ~ \bbb
\& ~~ V^*_2 ~ \bbb
\& ~~ V^*_3 ~ \bbb
\\
};
\path[-stealth]
(m-1-1) edge node [above] {$\grad$} (m-1-2)
(m-1-2) edge node [above] {$\curl$} (m-1-3)
(m-1-3) edge node [above] {$\Div$} (m-1-4)
(m-1-1) edge node [right] {$\Pi^0_h$} (m-2-1)
(m-1-2) edge node [right] {$\Pi^1_h$} (m-2-2)
(m-1-3) edge node [right] {$\Pi^2_h$} (m-2-3)
(m-1-4) edge node [right] {$\Pi^3_h$} (m-2-4)
(m-1-1) edge [bend right=40] node [pos=0.2, left] {$\arrsigma^0$} (m-3-1)
(m-1-2) edge [bend right=40] node [pos=0.2, left] {$\arrsigma^1$} (m-3-2)
(m-1-3) edge [bend right=40] node [pos=0.2, left] {$\arrsigma^2$} (m-3-3)
(m-1-4) edge [bend right=40] node [pos=0.2, left] {$\arrsigma^3$} (m-3-4)
(m-2-1) edge node [above] {$\grad_h$} (m-2-2)
(m-2-2) edge node [above] {$\curl_h$} (m-2-3)
(m-2-3) edge node [above] {$\Div_h$} (m-2-4)
(m-3-1.75) edge node [right] {$\cI^0$} (m-2-1.285)
(m-3-2.75) edge node [right] {$\cI^1$} (m-2-2.285)
(m-3-3.75) edge node [right] {$\cI^2$} (m-2-3.285)
(m-3-4.75) edge node [right] {$\cI^3$} (m-2-4.285)
(m-2-1.255) edge node [pos=0.2, left] {$\arrsigma^0$} (m-3-1.105)
(m-2-2.255) edge node [pos=0.2, left] {$\arrsigma^1$} (m-3-2.105)
(m-2-3.255) edge node [pos=0.2, left] {$\arrsigma^2$} (m-3-3.105)
(m-2-4.255) edge node [pos=0.2, left] {$\arrsigma^3$} (m-3-4.105)
(m-3-1) edge node[auto] {$ \matG \matP^0$} (m-3-2)
(m-3-2) edge node[auto] {$ \matC \matP^1$} (m-3-3)
(m-3-3) edge node[auto] {$ \matD \matP^2$} (m-3-4)
%
%
(m-4-1.105) edge [dashed] node [left] {$\t \matH^0$} (m-3-1.255)
(m-4-2.105) edge [dashed] node [left] {$\t \matH^1$} (m-3-2.255)
(m-4-3.105) edge [dashed] node [left] {$\t \matH^2$} (m-3-3.255)
(m-4-4.105) edge [dashed] node [left] {$\t \matH^3$} (m-3-4.255)
(m-3-1.285) edge [dashed] node [right] {$\matH^0$} (m-4-1.75)
(m-3-2.285) edge [dashed] node [right] {$\matH^1$} (m-4-2.75)
(m-3-3.285) edge [dashed] node [right] {$\matH^2$} (m-4-3.75)
(m-3-4.285) edge [dashed] node [right] {$\matH^3$} (m-4-4.75)
%
%
(m-5-2) edge node [above] {$\wt \Div_h$} (m-5-1)
(m-5-3) edge node [above] {$\wt \curl_h$} (m-5-2)
(m-5-4) edge node [above] {$\wt \grad_h$} (m-5-3)
(m-5-1.105) edge node [pos=0.2, left] {$\t \arrsigma^0$} (m-4-1.255)
(m-5-2.105) edge node [pos=0.2, left] {$\t \arrsigma^1$} (m-4-2.255)
(m-5-3.105) edge node [pos=0.2, left] {$\t \arrsigma^2$} (m-4-3.255)
(m-5-4.105) edge node [pos=0.2, left] {$\t \arrsigma^3$} (m-4-4.255)
(m-4-1.285) edge node [right] {$\t \cI^0$} (m-5-1.75)
(m-4-2.285) edge node [right] {$\t \cI^1$} (m-5-2.75)
(m-4-3.285) edge node [right] {$\t \cI^2$} (m-5-3.75)
(m-4-4.285) edge node [right] {$\t \cI^3$} (m-5-4.75)
(m-4-2) edge node [above] {$-(\matG \matP^0)^T $} (m-4-1)
(m-4-3) edge node [above] {$ (\matC \matP^1)^T $} (m-4-2)
(m-4-4) edge node [above] {$-(\matD \matP^2)^T $} (m-4-3)
(m-6-1) edge [bend left=40] node [pos=0.2, left] {$(\matP^0)^T\t \arrsigma^0$} (m-4-1)
(m-6-2) edge [bend left=40] node [pos=0.2, left] {$(\matP^1)^T\t \arrsigma^1$} (m-4-2)
(m-6-3) edge [bend left=40] node [pos=0.2, left] {$(\matP^2)^T\t \arrsigma^2$} (m-4-3)
(m-6-4) edge [bend left=40] node [pos=0.2, left] {$(\matP^3)^T\t \arrsigma^3$} (m-4-4)
(m-6-1) edge node [right] {$\t \Pi^0_h$} (m-5-1)
(m-6-2) edge node [right] {$\t \Pi^1_h$} (m-5-2)
(m-6-3) edge node [right] {$\t \Pi^2_h$} (m-5-3)
(m-6-4) edge node [right] {$\t \Pi^3_h$} (m-5-4)
(m-6-2) edge node [above] {$\Div$} (m-6-1)
(m-6-3) edge node [above] {$\curl$} (m-6-2)
(m-6-4) edge node [above] {$\grad$} (m-6-3)
;
\end{tikzpicture}
\end{equation}
Here $\cC^\ell$ and $\tilde \cC^\ell$ are the spaces of scalar coefficients associated with 
the primal and dual basis functions described in the previous Section.
Both are of the form $\RR^{N^\ell}$ with $N^\ell = \dim(V^\ell_h)$, but we use a different
notation to emphasize the different roles played by the coefficient vectors.
The interpolation operators $\cI^\ell: \cC^\ell \to V^\ell_h$
and $\t \cI^\ell: \t \cC^\ell \to V^\ell_h$, which read
\begin{equation} \label{cI}
  \cI^\ell: \arr{c} \mapsto \sum_{i=1}^{N^\ell} c_i \Lambda^\ell_i
  \qquad \text{ and }\qquad
  \t \cI^\ell: \t {\arr{c}} \mapsto \sum_{i=1}^{N^\ell} \t c_i \t \Lambda^\ell_i~,  
\end{equation}
are the right-inverses of the respective degrees of freedom $\arrsigma^\ell$ and
$\t\arrsigma^\ell$, and also their left-inverses on $V^\ell_h$. 

The matrices $\matH^\ell$ and $\t \matH^\ell$ are change-of-basis matrices which 
allow us to go from one sequence to the other, in the sense that
$\cI^\ell =\t \cI^\ell \matH^\ell$ and $\t \cI^\ell = \cI^\ell \t \matH^\ell$:
they correspond to discrete Hodge operators~\cite{Hiptmair_hodge_2001}.
Here it follows from the duality construction \eqref{tLambda}--\eqref{tsigma} 
that they are given by the mass matrices and their inverses,
\begin{equation} \label{HM_matrix}
  \matH^\ell = (\t \matH^\ell)^{-1} = \matM^\ell
  \quad  \text{ where } \quad 
  (\matM^\ell)_{i,j} = \sprod{\Lambda^\ell_i}{\Lambda^\ell_j}.
\end{equation}
In our framework, the Hodge matrices have a patch-diagonal structure due to the local support
of the broken basis functions. In particular, the dual basis functions are also 
supported on a single patch.

The remaining operators are the primal and dual commuting projections,
respectively defined by \eqref{brok-proj} and \eqref{tPi_def1}--\eqref{tPi_def2},
expressed in terms of primal and dual coefficients.
Using the interpolation operators \eqref{cI} we may write them as
\begin{equation} \label{tPi_op}
  \Pi^\ell = \cI^\ell \arrsigma^\ell \quad \text{ on } ~ U^\ell \subset V^\ell
  \qquad \text{ and } \qquad
  \t \Pi^\ell = \t \cI^\ell (\matP^\ell)^T \t \arrsigma^\ell \quad \text{ on } ~ V^*_\ell~.
\end{equation}
Finally we remind that here the commutation of the primal (upper) diagram follows from the
assumed properties of the primal degrees of freedom \eqref{brok-dof}--\eqref{conf-dof},
while the commutation of the dual (lower) diagram follows from the 
weak definition of the dual differential operators.

\begin{remark}[conforming case]
  The conforming FEEC diagram \eqref{CD_hc} 
  (corresponding, e.g., to a single patch discretization)  
  may be extended in the same way as the broken-FEEC one \eqref{CD_h}. 
  Apart from the differences between the discrete differential operators and commuting projection 
  operators already visible in \eqref{CD_hc} and \eqref{CD_h}, the resulting
  6-rows diagram would be formally the same as \eqref{CD}, but no conforming projection matrices
  would be involved any longer.  
\end{remark}

\begin{remark}[change of basis]
  \label{rem:cob}
  It may happen that the basis functions $\Lambda^\ell_i$ associated with the 
  commuting degrees of freedom $\sigma^\ell_i$ are not the most convenient ones 
  when it comes to actually implement the discrete operators. In such cases 
  a few changes are to be done to the diagram \eqref{CD}, such as the
  ones described in Section~\ref{sec:cob} for local spline spaces.
\end{remark}

\subsection{Discrete Hodge-Laplace and jump stabilization operators}
\label{sec:HL_S}

In addition to the first order differential operators, 
an interesting feature of the diagram \eqref{CD} is to provide us with natural
discretizations for the Hodge-Laplace operators. 
On the space $V^0$ this is the standard Laplace operator, $\cL^0 = -\Delta$,
which is discretized as
\begin{equation} \label{L0h}
\cL^0_h := -\wt \Div_h \grad_h : V^0_h \to V^0_h
\end{equation}
with an operator matrix in the primal basis that reads
\begin{equation} \label{matL0}
\matL^0 = \t \matH^0 (\matG\matP^0)^T \matH^1 \matG\matP^0~.
\end{equation}
As for the Hodge-Laplace operator for 1-forms, 
$\cL^1 = \curl \curl - \grad \Div$, its discretization is
\begin{equation} \label{L1h}
  \cL^1_h := \wt \curl_h \curl_h -\grad_h \wt \Div_h 
    : V^1_h \to V^1_h  
\end{equation} 
with an operator matrix (again in the primal basis) 
\begin{equation} \label{matL1}
\matL^1 = \t\matH^1 (\matC\matP^1)^T \matH^2 \matC\matP^1 + \matG\matP^0 \t\matH^0 (\matG\matP^0)^T \matH^1.
\end{equation} 
The Hodge-Laplace operators for 2 and 3-forms are discretized similarly.

As discussed in Section~\ref{sec:feec}, a key asset of FEEC discretizations 
is their ability to preserve the exact dimension of the harmonic forms,
defined here as the kernel of the Hodge-Laplace operators.
In our broken-FEEC framework this property is a priori not preserved
by the non-conforming operators due to the extended kernels of the conforming
projection operators, but it is for the {\em stabilized} operators
  $\cL^\ell_{h,\alpha} = \cL^\ell_h + \alpha S^\ell_h$ 
where 
\begin{equation}\label{S_ell}
  S^\ell_h := (I-P^\ell_h)^* (I-P^\ell_h)
\end{equation}
is the symmetrized projection operator on the jump space \eqref{kerP},
and $\alpha$ is a stabilization parameter. 
Indeed, it was shown in \cite{conga_hodge} that
the kernel of $\cL^\ell_{h,\alpha}$ coincides with that of the conforming
discrete Hodge-Laplacian, for any positive value of $\alpha > 0$.
Finally we note that the corresponding stabilization matrix is readily derived from the operators 
in the diagram. It reads
\begin{equation} \label{matS}
  \matS^\ell 
    = \t \matH^\ell (\matI-\matP^\ell)^T \matH^\ell (\matI-\matP^\ell).
\end{equation}

\section{Application to electromagnetic problems}
\label{sec:pbms}

In this section we propose broken-FEEC approximations for several problems arising 
in electromagnetics, and we state a priori results for the solutions. 
For the sake of completeness we provide several formulations (all equivalent) for the 
discrete problems: one in operator form, one in weak form and one in matrix form,
using the primal and dual bases introduced in Section~\ref{sec:proj_bfeec}. 
As pointed out in Remark~\ref{rem:cob}, practical implementation may be more efficient with
other basis functions. In this case the matrices need to be changed, as will be described
in Section~\ref{sec:cob} below.

Throughout this section we will mostly work with 
the homogeneous spaces $V^\ell$ corresponding to \eqref{dR_spaces_hom}. 
In the few cases where we will need the inhomogeneous spaces, we 
shall use a specific notation $\bar V^\ell$ for the spaces in \eqref{dR_spaces_inhom}, 
for the sake of clarity.

\subsection{Poisson's equation}
\label{sec:poisson}

For the Poisson problem with homogeneous Dirichlet boundary conditions:
given $f \in L^2(\Omega)$, find $\phi \in V^0 = H^1_0(\Omega)$ such that
\begin{equation} \label{poisson_f}
  - \Delta \phi = f,
\end{equation}
we combine the stabilized broken-FEEC discretization described in Section~\ref{sec:HL_S}
for the Laplace operator, and a dual commuting projection \eqref{tPi_def1}--\eqref{tPi_def2} for the source.
The resulting problem reads: Find $\phi_h \in V^0_h$ such that
\begin{equation} \label{poisson_hLS}
  (\cL^0_h + \alpha S^0_h) \phi_h = \t\Pi^0_h f
\end{equation}
and it enjoys the following property.

\begin{proposition} \label{prop:poisson}
  For all $\alpha \neq 0$, equation \eqref{poisson_hLS} admits a unique solution $\phi_h$
  which belongs to the conforming space $V^{0,c}_h = V^0_h \cap H^1_0(\Omega)$
  and solves the conforming Poisson problem
  \begin{equation} \label{poisson_hc}
    \sprod{\grad \vp^c}{\grad \phi_h} = \sprod{\vp^c}{f} \qquad \forall \vp^c \in V^{0,c}_h.
  \end{equation}
  In particular, $\phi_h$ is independent of both $\alpha$ and the specific projection $P^0_h$.
\end{proposition}

\begin{proof}
  By definition of the different operators, \eqref{poisson_hLS} reads
  \begin{equation} \label{poisson_h}
    \big(-\wt \Div_h \grad_h + \alpha (I-P^0_h)^* (I-P^0_h)\big) \phi_h = \t\Pi^0_h f
  \end{equation}
  which may be reformulated using test functions $\vp \in V^{0}_h$, as
  \begin{equation} \label{poisson_hw}
    \sprod{\grad P^0_h\vp}{\grad P^0_h \phi_h}
    + \alpha \sprod{(I-P^0_h)\vp}{(I-P^0_h)\phi_h}
     = \sprod{P^0_h\vp}{f} .
  \end{equation}
  Taking $\vp = (I-P^0_h)\phi_h$ yields $\phi_h = P^0_h \phi_h \in V^{0,c}_h$ as long as $\alpha \neq 0$, 
  and \eqref{poisson_hc} follows by considering test functions $\vp$ in the conforming subspace $V^{0,c}_h$.
  \qed
\end{proof}

The matrix formulation of \eqref{poisson_hLS} is easily derived
by using the diagram~\eqref{CD}, writing the equation in the dual basis
in order to obtain symmetric matrices as is usual with the Poisson problem.
Denoting by $\arr{\phi} = \arrsigma^0(\phi_h)$ the coefficient vector of $\phi_h$
in the primal basis, we thus find
\begin{equation} \label{poisson_m}
  \matA^0 \arr{\phi} = (\matP^0)^T \t \arrsigma^0(f)
\end{equation}
with a stabilized stiffness matrix
\begin{equation} \label{A0}
\matA^0
  = \matH^0 (\matL^0 + \alpha \matS^0) 
  = (\matG \matP^0)^T \matH^1 \matG \matP^0 
    + \alpha (\matI-\matP^0)^T \matH^0 (\matI-\matP^0).
\end{equation}

\subsection{Time-harmonic Maxwell's equation}
\label{sec:maxwell}

Maxwell's equation with homogeneous boundary conditions reads:
given $\omega \in \RR$ and $J \in L^2(\Omega)$, find $u \in H_0(\curl;\Omega)$ such that
\begin{equation} \label{max_f}
  - \omega^2 u + \curl\curl u = J.
\end{equation}
Here the (complex) electric field corresponds to $E(t,x) = i\omega\, u(x) e^{-i \omega t}$.
When $\omega^2$ is not an eigenvalue of the $\curl \curl$ operator, equation~\eqref{max_f}
is well-posed: see e.g.~\cite[Th.~8.3.3]{Assous.Ciarlet.Labrunie.2018.sp}.
For this problem we propose a stabilized broken-FEEC discretization where $u_h \in V^1_h$ solves
\begin{equation} \label{maxwell_hfilt}
  (-\omega^2(P^1_h)^*P^1_h + \wt \curl_h \curl_h + \alpha S^1_h)u_h = \t \Pi^1_h J,
\end{equation}
with a 
parameter $\alpha \in \RR$.
We remind that $S^1_h := (1-P^1_h)^* (1-P^1_h)$ is the 
jump stabilization 
operator, 
according to~\eqref{S_ell} and~\eqref{kerP} for $\ell=1$.
With test functions $v \in V^{1}_h$, 
the discrete problem~\eqref{maxwell_hfilt} writes
\begin{equation} \label{maxwell_hwfilt}
     - \omega^2 \sprod{P^1_hv}{P^1_hu_h}
    + \sprod{\curl P^1_h v}{\curl P^1_h u_h}
      +\alpha \sprod{(I-P^1_h)v}{(I-P^1_h)u_h} = \sprod{P^1_h v}{J}.
\end{equation}
Note that the zeroth-order term is filtered by the symmetric operator
$(P^1_h)^* P^1_h$: 
this allows us to obtain a conforming solution in the broken space.

\begin{proposition} \label{prop:maxwell}
  Let $\alpha \neq 0$, and $\omega$ such that $\omega^2$ is not an eigenvalue of the
  {\em conforming} discrete $\wt\curl^c_h \curl^c_h$ operator defined in \eqref{d_c}--\eqref{d*_c}.
  Then equation~\eqref{maxwell_hfilt} admits a unique solution $u_h$
  which belongs to the conforming space $V^{1,c}_h = V^{1}_h \cap H_0(\curl;\Omega)$,
  and solves the conforming FEEC Maxwell problem
  \begin{equation} \label{maxwell_hc}
      - \omega^2 \sprod{v^c}{u_h} + \sprod{\curl v^c}{\curl u_h} = \sprod{v^c}{J},
      \qquad \forall v^c \in V^{1,c}_h.
  \end{equation}
  In particular, $u_h$ is independent of both $\alpha$ and the specific projection $P^1_h$.
\end{proposition}

\begin{proof}
Taking $v = (I-P^1_h)u_h$ in~\eqref{maxwell_hwfilt} gives $u_h = P^1_h u_h \in V^{1,c}_h$ as long as $\alpha \neq 0$, and \eqref{maxwell_hc} follows by considering test functions $v = v^c$ in $V^{1,c}_h$.
Existence and uniqueness follow from the assumption that $\omega$ is not an eigenvalue of the
conforming curl curl operator.
\qed
\end{proof}

The matrix form of \eqref{maxwell_hfilt} is easily derived from diagram \eqref{CD}.
It reads 
\begin{equation} \label{maxwell_m}
  \matA^1 \arr{u} = (\matP^1)^T \t \arrsigma^1(J)
\end{equation}
where $\arr{u} := \arrsigma^1(u_h)$ is the (column) vector containing the coefficients of $u_h$
in the primal basis of $V^1_h$,
and $\matA^1$ is the resulting stabilized stiffness matrix,
\begin{equation} \label{A1}
  \matA^1 = (\matP^1)^T (-\omega^2 \matH^1 + \matC^T \matH^2 \matC) \matP^1
    + \alpha (\matI-\matP^1)^T \matH^1 (\matI-\matP^1)    .
\end{equation}

\subsection{Lifting of boundary conditions}
\label{sec:lifting}

In the case of inhomogeneous boundary conditions a standard approach is to introduce a lifted solution
and solve the modified problem in the homogeneous spaces.
For a Poisson equation of the form 
\begin{equation} \label{poisson_fg}
\left\{\begin{aligned}
    - \Delta \phi &= f  \qquad \text{ in ~  $\Omega$}
    \\
    \phi &= g  \qquad \text{ on ~  $\partial \Omega$}
\end{aligned}
\right.
\end{equation}
this corresponds to introducing $\phi_g \in H^1(\Omega)$
such that $\phi_g = g$ on $\partial \Omega$, and characterizing the solution
as $\phi = \phi_g + \phi_0$, where $\phi_0 \in V^0 = H^1_0(\Omega)$ solves \eqref{poisson_f} with a modified source,
namely
$$
- \Delta \phi_0 = f + \Delta \phi_g.
$$
For an inhomogeneous Maxwell equation of the form 
\begin{equation} \label{max_fg}
  \left\{\begin{aligned}
      - \omega^2 u + \curl\curl u &= J  \qquad \text{ in ~ $\Omega$}
      \\
      n \times u &= g  \qquad \text{ on ~ $\partial \Omega$}
  \end{aligned}
  \right.
\end{equation}
we introduce $u_g \in H(\curl;\Omega)$
such that $n  \times u_g = g$ on $\partial \Omega$, and characterize the solution
as $u = u_g + u_0$, where $u_0 \in H_0(\curl;\Omega)$ solves \eqref{max_f} with
a modified source, namely
\begin{equation} \label{max_fbc}
  - \omega^2 u_0 + \curl\curl u_0 = J + (\omega^2 - \curl\curl) u_g.
\end{equation}


The lifting approach can be applied in broken-FEEC methods by combining
the discretizations of the homogeneous and inhomogeneous sequences \eqref{dR_spaces_hom}
and \eqref{dR_spaces_inhom}, which amounts to combining different conforming projections.
For clarity 
we denote in this section 
the respective homogeneous and inhomogeneous spaces by $V^\ell$ and $\bar V^\ell$.
At the discrete level the broken spaces $V^\ell_h$ have no boundary conditions, 
so that the distinction only appears in the conforming projections to the
spaces $V^{\ell,c}_h = V^\ell_h \cap V^\ell$ and $\bar V^{\ell,c}_h = V^\ell_h \cap \bar V^\ell$,
which we naturally denote by $P^\ell_h$ and $\bar P^\ell_h$ respectively.
In practice the latter projects on the inhomogeneous conforming spaces,
while the former further sets the boundary degrees of freedom to zero.

For the Poisson equation where $V^0 = H^1_0(\Omega)$ and $\bar V^0 = H^1(\Omega)$, 
we first compute $\phi_{g,h} \in V^{0}_h$ 
that approximates $g$ on $\partial \Omega$,
and then we compute the homogeneous part of the solution,
$
\phi_{0,h} := \phi_h - \phi_{g,h} \in V^{0}_h,
$
by solving the homogeneous CONGA problem with modified source
\begin{multline} \label{poisson_hbc}
  \sprod{\grad P^0_h\vp}{\grad P^0_h \phi_{0,h}}
  + \alpha \sprod{(I-P^0_h)\vp}{(I-P^0_h)\phi_{0,h}}
   = \sprod{P^0_h\vp}{f} \\
   - \sprod{\grad P^0_h\vp}{\grad \bar P^0_h \phi_{g,h}}
\end{multline}
for all $\vp \in V^{0}_h$,
which also reads in matrix form
\begin{equation} \label{poisson_mbc}
  \matA^0 \arr{\phi}_0 = (\matP^0)^T \Big(\t \arrsigma^0(f) - \matG^T \matH^1 \matG \bar \matP^0 \arr{\phi}_g \Big)
\end{equation}
where $\matA^0$ is the stabilized stiffness matrix \eqref{A0}, $\bar \matP^0$ is the matrix of the 
inhomogeneous projection $\bar P^0_h$, and the coefficient vectors involve the broken degrees of freedom
$\arr{\phi}_0 = \arrsigma^0(\phi_{0,h})$, $\arr{\phi}_g = \arrsigma^0(\phi_{g,h})$ in the full broken space $V^0_h$. 

For the Maxwell equation where $V^1 = H_0(\curl;\Omega)$ and $\bar V^1 = H(\curl;\Omega)$, 
the method consists of first computing $u_{g,h} \in V^{1}_h$ 
such that $n \times u_{g,h}$ approximates $g$ on $\partial \Omega$,
and then characterizing the homogeneous part of the solution,
$
u_{0,h} := u_h - u_{g,h} \in V^{1}_h,
$
by 
\begin{multline} \label{maxwell_hbc}
  - \omega^2 \sprod{P^1_h v}{P^1_h u_{0,h}}
    + \sprod{\curl P^1_h  v}{\curl P^1_h  u_{0,h}}
    + \alpha \sprod{(I-P^1_h) v}{(I-P^1_h) u_{0,h}} =\\
  = \sprod{P^1_h  v}{ J} + \omega^2 \sprod{P^1_h v}{ \bar P^1_h u_{g,h}}
    -\sprod{\curl P^1_h  v}{\curl \bar P^1_h u_{g,h}}
\end{multline}
for all $v \in V^{1}_h$.
In matrix terms, this reads
\begin{equation} \label{maxwell_mbc}
  \matA^1 \arr{u}_0 = (\matP^1)^T \Big(\t \arrsigma^1(J) + (\omega^2 \matH^1 - \matC^T \matH^2 \matC)\bar \matP^1 \arr{u}_g \Big)
\end{equation}
where $\matA^1$ is the stabilized stiffness matrix \eqref{A1}, $\bar \matP^1$ is the matrix of the 
inhomogeneous projection $\bar P^1_h$, and the coefficient vectors involve the broken degrees of freedom
$\arr{u}_0 = \arrsigma^1(u_{0,h})$, $\arr{u}_g = \arrsigma^1(u_{g,h})$ in the full broken space $V^1_h$. 
\begin{proposition} \label{prop:maxwell_inhom}
  Let $\alpha$ and $\omega$ be as in Prop.~\ref{prop:maxwell}.
  Then \eqref{maxwell_hbc} admits a unique solution
  $u_{0,h}$ which belongs to the conforming space 
  $V^{1,c}_h = V^{1}_h \cap H_0(\curl;\Omega)$.
  The projection of the full solution in $\bar V^{1,c}_h = V^{1}_h \cap H(\curl;\Omega)$,
  $u^c_h := \bar P^1_h u_h$, 
  approximates the boundary condition,
  $n \times u^c_h \approx g$ on $\partial \Omega$,
  and it solves the discrete conforming Maxwell equation inside $\Omega$, in the sense that
  \begin{equation} \label{maxwell_hbc_c}
      - \omega^2 \sprod{v^c}{u^c_h} + \sprod{\curl v^c}{\curl u^c_h} = \sprod{v^c}{J},
      \qquad \forall v^c \in V^{1,c}_h.
  \end{equation}
  In particular, $u^c_h$ is independent of both $\alpha$ and $P^1_h$.
  Moreover if the lifted boundary condition $u_{g,h}$ is chosen in $\bar V^{1,c}_h$,
  then $u_h = u^c_h$.
\end{proposition}
\begin{proof}
  The conformity and unicity of $u_{0,h}$ can be shown with the same arguments as in Prop.~\ref{prop:maxwell}.
  This shows that $u^c_h = u_{0,h} + \bar P^1_h u_{g,h}$ and that the tangential trace of $u^c_h$ 
  coincides with that of $\bar P^1_h u_{g,h}$ on the boundary, which itself should coincide with that of 
  $u_{g,h}$ as there are no conformity constraints there for the inhomogeneous space
  $\bar V^1 = H(\curl;\Omega)$. Equation \eqref{maxwell_hbc_c} and the additional observations follow easily.
\qed
\end{proof}

%
%

\subsection{Eigenvalue problems}

Broken-FEEC approximations to Hodge-Laplace eigenvalue problems of the form
\begin{equation} \label{L_ev}
  \cL^\ell u = \lambda u
\end{equation}
are readily derived using the discrete operators presented in Section~\ref{sec:HL_S}.
For a general stabilization parameter $\alpha \ge 0$, they take the form:
\begin{equation} \label{Lh_ev}
      (\cL^\ell_h + \alpha S^\ell_h) u_h = \lambda_h u_h
\end{equation}
with $u_h \in V^\ell_h \setminus \{0\}$.
Their convergence can be established under the
assumption of a strong stabilization regime and of moment-preserving
properties for the conforming projections, see \cite{conga_hodge}.
There it has also been shown that for arbitrary positive penalizations $\alpha > 0$,
the zero eigenmodes coincide with their conforming counterparts, namely the conforming harmonic forms,
$$
\ker (\cL^\ell_h + \alpha S^\ell_h) = 
\ker \cL^{\ell,c}_h = \cH^\ell_h
$$
see \eqref{HH_c}--\eqref{L1hc}.

For the first space $V^0 = H^1_0(\Omega)$ the corresponding operator is the standard Laplace operator $\cL^0 = -\Delta$ 
and its broken-FEEC discretization reads:
Find $\lambda_h \in \RR$ and $\phi_h \in V^0_h \setminus \{0\}$ such that
\begin{equation} \label{L0_evh}
  (\cL^0_h + \alpha S^0_h) \phi_h = \lambda_h \phi_h
\end{equation}
where $\cL^0_h$ and $S^0_h$ are 
given by \eqref{L0h} and \eqref{S_ell}, with stabilization parameter $\alpha$.
For $V^1 = H_0(\curl;\Omega)$ the Hodge-Laplace operator is
$\cL^1 = \curl \curl -\grad \Div$ 
and its discretization reads:
Find $\lambda_h \in \RR$ and $u_h \in V^1_h \setminus \{0\}$ such that
\begin{equation} \label{L1_evh}
  (\cL^1_h + \alpha S^1_h) u_h = \lambda_h u_h
\end{equation}
where $\cL^1_h = \wt \curl_h \curl_h - \grad_h \wt \Div_h$, see \eqref{L1h},
and $S^1_h$ is given again by \eqref{S_ell}. 
Eigenvalue problems in $V^2_h$ and $V^3_h$ are discretized in a similar fashion.

Another important eigenvalue problem is that of the curl-curl operator, for which one could simply adapt the broken-FEEC discretization~\eqref{L1_evh} used for the Hodge-Laplace operator, by dropping the term $(-\grad_h \wt\Div_h)$.
By testing such an equality with $v \in \grad P^0_h V^0_h$ we see that
the non-zero eigenmodes satisfy $\wt \Div_h u_h = 0$,
hence they are also eigenmodes of \eqref{L1_evh} and their convergence follows 
from the results mentionned above in the case of strong penalization regimes. 
In the unpenalized case ($\alpha = 0$) their convergence was also established, as
shown in \cite{Campos-Pinto.Sonnendrucker.2016.mcomp}.

The main drawback of the approach just described is that the computed eigenvalues do not coincide with those of the 
conforming operator $\wt \curl^c_h \curl^c_h$, but only approximate them.
(We recall that the knowledge of the eigenvalues of $\wt \curl^c_h \curl^c_h$ is needed in order to verify the hypotheses of Proposition~\ref{prop:maxwell} for the source problem: $\omega^2$ should not concide with one of them.)
To overcome this difficulty we propose a novel broken-FEEC discretization, which consists of the generalized eigenvalue problem
\begin{equation} \label{CC_evh}
  \wt \curl_h \curl_h u_h = \lambda_h \left[ (P^1_h)^\ast P^1_h + (I-P^1_h)^\ast (I-P^1_h) \right] u_h.
\end{equation}
With test functions $v \in V^1_h$, the discrete problem~\eqref{CC_evh} writes
\begin{equation} \label{CC_evh_wf}
    \langle \curl P^1_h v, \curl P^1_h u_h \rangle =
    \lambda_h \left[ \langle P^1_h v, P^1_h u_n \rangle +
                     \langle (1 - P^1_h) v, (1 - P^1_h) u_h \rangle
              \right]
\end{equation}
and its matrix formulation is provided 
below.
\begin{proposition} \label{prop:CC_evh}
  For $\lambda_h \neq 0$, the CONGA eigenvalue problem~\eqref{CC_evh}
  has the same solutions as the conforming eigenproblem
  \begin{equation} \label{CC_evhc}
    \wt \curl^c_h \curl^c_h u_h = \lambda_h u_h, \quad \text{with ~ $u_h \in V^{1,c}_h$}.    
  \end{equation}
  For $\lambda_h = 0$, the corresponding eigenspace is
  \begin{equation} \label{CC_evhc_ker}
    \ker \big(\wt \curl_h \curl_h \big) 
      = \big(\grad V^{0,c}_h \poplus \cH^1_h\big) \oplus (I-P^1_h) V^1_h.
  \end{equation}
\end{proposition}
\begin{proof}
  For $\lambda_h \neq 0$ we take $v = (I - P^1_h) u_h$ in~\eqref{CC_evh_wf}, and find that $u_h = P^1_h u_h \in V^{1,c}_h$.
  Hence we can restrict the test space to functions $v = P^1_h v \in V^{1,c}_h$, and obtain the conforming eigenproblem $\langle \curl v, \curl u_h \rangle = \lambda_h \langle v, u_h \rangle$. Conversely, we verify that 
  if $u_h \in V^{1,c}_h$ solves \eqref{CC_evhc} then it also solves \eqref{CC_evh_wf} for all $v \in V^1_h$.
  The relationship~\eqref{CC_evhc_ker} between the null spaces follows from the equalities
    $\ker \big(\wt \curl_h \curl_h \big) = \ker (\curl_h)$, 
    $\ker \big(\wt \curl^c_h \curl^c_h \big) = \ker (\curl^c_h) = \grad V^{0,c}_h \poplus \cH^1_h $
    and 
    \[
    \ker (\curl_h)
      = \ker (\curl^c_h)  \oplus (I-P^1_h) V^1_h
    \]
  which was observed in, e.g., \cite{Campos-Pinto.Sonnendrucker.2016.mcomp}.
  \qed
\end{proof}

These eigenvalue problems 
may be expressed in symmetric matrix form by 
expressing the eigenmodes in the primal bases and the equations in the dual bases,
as we did for the Poisson and Maxwell equations in Section~\ref{sec:poisson} and \ref{sec:maxwell}.
For the Poisson eigenvalue problem~\eqref{L0_evh} we obtain
\begin{equation} \label{L0_evmat}
      \matA^0 \arr{\phi} = \lambda_h \matH^0 \arr{\phi}
\end{equation}
with $\matA^0 = \matH^0(\matL^0 + \alpha \matS^0) = 
(\matG \matP^0)^T \matH^1 \matG \matP^0 
  + \alpha (\matI-\matP^0)^T \matH^0 (\matI-\matP^0)$ as in \eqref{A0}.
Similarly we can write the Hodge-Laplace eigenvalue problem~\eqref{L1_evh} as
\begin{equation} \label{L1_evmat}
      \matA^1 \arr{u} = \lambda_h \matH^1 \arr{u}
\end{equation}
with matrix
$$
\matA^1 = 
(\matC\matP^1)^T \matH^2 \matC\matP^1 + \matH^1 \matG\matP^0 \t\matH^0 (\matG\matP^0)^T \matH^1
  + \alpha (\matI-\matP^1)^T \matH^1 (\matI-\matP^1),
$$
and the curl-curl eigenvalue problem~\eqref{CC_evh} as
\begin{equation} \label{CC_evmat}
  (\matC\matP^1)^T \matH^2 \matC\matP^1 \arr{u}
  = \lambda_h \left[ (\matP^1)^T \matH^1 \matP^1 + (\matI - \matP^1)^T \matH^1 (\matI - \matP^1) \right] \arr{u}.
\end{equation}

\subsection{Magnetostatic problems with harmonic constraints}
\label{sec:MS}

We next consider a magnetostatic problem of the form
\begin{equation}\label{ms}
  \left\{
  \begin{aligned}
    \Div B &= 0
    \\
    \curl B &= J
  \end{aligned}
  \right.  
\end{equation}
posed for a current $J \in L^2(\Omega)$. 
Unlike the Poisson and Maxwell source problems above, 
this problem requires an additional constraint in non-contractible domains.
Indeed it can only be well-posed if 
$\ker \curl \cap \ker \Div = \{0\}$, i.e., if the harmonic forms are trivial.
To formulate this more precisely we need to take into account particular
boundary conditions.

\subsubsection{Magnetostatic problem with pseudo-vacuum boundary condition}
\label{sec:MS_vbc}

We first consider a boundary condition of the form $n \times B = 0$,
which is sometimes used to model pseudo-vacuum boundaries \cite{Jackson_al_magneto_pvbc_2014}.
A mixed formulation is then: Find $B \in H_0(\curl;\Omega)$ and 
$p \in H^1_0(\Omega)$ such that
\begin{equation} \label{ms_w}
  \left\{
  \begin{aligned}
    \sprod{\grad q}{B}& = 0  \qquad && \forall q \in H^1_0(\Omega)
    \\
    \sprod{v}{\grad p} + \sprod{\curl v}{\curl B} &= \sprod{\curl v}{J} \qquad && \forall v \in H_0(\curl;\Omega)
  \end{aligned}
  \right.  
\end{equation}
see e.g. \cite{Beirao_da_Veiga_Brezzi_Dassi_Marini_Russo_2017}.
Here the auxiliary unknown $p$ is a Lagrange multiplier for the divergence constraint written in weak form.

The well-posedness of \eqref{ms_w} is related to the space of harmonic 1-forms
$\cH^1$, defined as the kernel of the Hodge-Laplace
operator \eqref{L1} in the homogeneous sequence, i.e., 
\begin{equation} \label{H_gc_hom_remind}
  \cH^1 = \ker \cL^1 = \{ v \in H_0(\curl;\Omega): ~  \curl v = \Div v = 0\}.
\end{equation}
As discussed in \cite{Arnold.Falk.Winther.2010.bams}, $\cH^1 = \cH^1(\Omega)$ is isomorphic to a de Rham cohomology group and its
dimension corresponds to a Betti number of the domain $\Omega$. With homogeneous, resp. inhomogeneous boundary conditions 
its dimension is $b_2(\Omega)$ the number of cavities, 
resp. $b_1(\Omega)$ the number of handles in $\Omega$.
Thus in a contractible domain we have $\cH^1 = \{0\}$ and the problem is well posed, but in general 
$\cH^1$ may not be trivial. Additional constraints must then be imposed on the solution,
such as
$
B \in (\cH^1)^\perp,
$
which can be associated with an additional Lagrange multiplier $z \in \cH^1$. The constrained problem 
then consists of finding $B \in H_0(\curl;\Omega), p \in H^1_0(\Omega)$ and $z \in \cH^1$, such that
\begin{equation} \label{magneto_w}
\left\{
\begin{aligned}
  \sprod{\grad q}{B} &= 0  \qquad && \forall q \in H^1_0(\Omega)
  \\
  \sprod{v}{\grad p} + \sprod{\curl v}{\curl B}  + \sprod{v}{z} &= \sprod{\curl v}{J} \qquad && \forall v \in H_0(\curl;\Omega)
  \\
  \sprod{w}{B} &= 0  \qquad && \forall w \in \cH^1.
\end{aligned}
\right.
\end{equation}

In our broken-FEEC framework this problem is discretized as: 
find $p_h \in V^0_h$, $B_h \in V^1_h$ and $z_h \in \cH^1_h$, such that
\begin{equation} \label{magneto_h}
  \left\{
  \begin{aligned}
    \alpha^0 \sprod{(I-P^0_h)q}{(I-P^0_h)p_{h}} + \sprod{\grad P^0_h q}{B_h} &= 0  
    \\
    \sprod{v}{\grad P^0_h p_h}  + \sprod{\curl P^1_h v}{\curl P^1_h B_h} \mspace{100mu} & 
    \\
      + \alpha^1 \sprod{(I-P^1_h)v}{(I-P^1_h)B_{h}} + \sprod{v}{z_h}  &= \sprod{\curl P^1_h v}{J} 
    \\
    \sprod{w}{B_h} &= 0  
  \end{aligned}
  \right.
\end{equation}
for all $q \in V^0_h$, $v \in V^1_h$ and $w \in \cH^1_h$.
Here, $\alpha^\ell \in \RR$ are stabilization parameters for the jump terms in $V^0_h$ and $V^1_h$, 
and $\cH^1_h = \ker (\cL^1_h + \alpha^1 S^1_h)$ 
may be computed as the kernel of the stabilized discrete Hodge-Laplace operator \eqref{L1h}.
As discussed above it coincides with that of the conforming Hodge-Laplace operator, hence 
\begin{equation} \label{H1_h_hom}
  \cH^1_h = \{v \in V^{1,c}_h : \curl v = 0 ~ \text{ and } ~ \sprod{\grad q}{v} = 0, ~ \forall q \in V^{0,c}_h\}.
\end{equation}
In operator form, this amounts to 
finding $p_h \in V^0_h$, $B_h \in V^1_h \cap (\cH^1_h)^\perp$ and $z_h \in \cH^1_h$ such that
\begin{equation} \label{magneto_h_op}
  \left\{
  \begin{aligned}
    \alpha^0 S^0_h p_{h} + \wt \Div_h B_h &= 0  
    \\
    \grad_h p_h + \wt \curl_h \curl_h B_h + \alpha^1 S^1_h B_{h} + z_h  &= \wt \curl_h \t \Pi^2_h J 
  \end{aligned}
  \right.
\end{equation}

\begin{proposition} \label{prop:MS}
  System~\eqref{magneto_h} (or equivalently \eqref{magneto_h_op}) is well-posed for arbitrary non-zero stabilization parameters 
  $\alpha^0$, $\alpha^1 \neq 0$, and its solution satisfies
  \begin{equation}
    \begin{aligned}
      &B_h \in V^{1,c}_h \cap (\cH^1_h \oplus \grad V^{0,c}_h)^\perp 
      \\
      &p_h \in V^{0,c}_h \cap \ker \grad = \{0\}
      \\
      &z_h = 0.
    \end{aligned}
  \end{equation}
Moreover the discrete magnetic field can be decomposed according to \eqref{HH_c}, 
as 
$$
B_h = B^{\rm g}_h + B^{\rm h}_h + B^{\rm c}_h 
  \in \grad^c_h V^{0,c}_h \poplus \cH^1_h \poplus \wt \curl^c_h V^{2,c}_h
$$ 
with discrete grad, harmonic and curl components characterized by
\begin{equation}
  \left\{\begin{aligned}
  & B^{\rm g}_h = 0
  \\
  & B^{\rm h}_h = 0
  \\
  & \sprod{w}{\curl B^{\rm c}_h} = \sprod{w}{J} \qquad && \forall w \in \curl V^{1,c}_h.
  \end{aligned} \right.
\end{equation}
In particular, the solution is independent of the non-zero stabilization parameters and 
the conforming projection operators.
\end{proposition}
\begin{proof} 
  The result follows by using appropriate test functions and the orthogonality of the 
  conforming discrete Hodge-Helmholtz decomposition \eqref{HH_c}.
  In particular, taking $v = z_h$ (in $\cH^1_h \subset V^{1,c}_h$) yields
  $z_h=0$, taking $q = (I-P^0_h) p_h$ and $v = \grad \bar P^0_h p_h$
  shows that $p_h = 0$, and the characterization of $B_h$ follows by taking first $v = (I-P^1_h) B_h$ and 
  then $v \in V^{1,c}_h$, an arbitrary conforming test function.
  \qed
\end{proof}

\begin{remark}
  Prop.~\ref{prop:MS} states that the auxiliary unknowns $p_h$ and $z_h$ 
  may be discarded a posteriori from the system. However they are needed 
  in \eqref{magneto_h} to have a square matrix.  
\end{remark}
In practice the discrete harmonic fields can be represented by a basis 
of $\cH^1_h$, of the form $\Lambda^{1,H}_i$, $i = 1, \ldots, \dim(\cH^1_h) = b_2(\Omega)$,
computed as the zero eigenmodes of a penalized $\cL^1_h + \alpha S^1_h$,
where $\alpha > 0$ is arbitrary and $b_2(\Omega)$ is the Betti number of order 2, i.e. the number of cavities in $\Omega$.
Using these harmonic fields, one assembles the rectangular mass matrix
\begin{equation} \label{mat_MH}
  (\matM^{1,H})_{i,j} = \sprod{\Lambda^{1}_i}{\Lambda^{1,H}_j}, \qquad i = 1, \ldots, N^1, ~ j = 1, \ldots,  b_2(\Omega).
\end{equation}
which allows us to rewrite \eqref{magneto_h} as
\begin{equation} \label{magneto_mat}
  \left\{
  \begin{aligned}
    \alpha^0 \matS^0 \arr{p} + (\matG \matP^0)^T \matH^1\arr{B} &=0
    \\
    \matH^1 \matG \matP^0 \arr{p}
    + \big(
    (\matC\matP^1)^T \matH^2 \matC\matP^1 + \alpha^1 \matS^1 \big) \arr{B}  + \matM^{1,H} \arr{z}  &= (\matC\matP^1)^T \t \arrsigma^2(J)
    \\
    (\matM^{1,H})^T \arr{B} &= 0
  \end{aligned}
  \right.
\end{equation}
where $\arr{p}$, $\arr{B}$ and $\arr{z}$ contain the coefficients of $p_h$, $B_h$ and $z_h$
in the (primal) bases of $V^0_h$, $V^1_h$ and $\cH^1_h$ respectively.

\subsubsection{Magnetostatic problem with metallic boundary conditions}
\label{sec:MS_mbc}

We next consider a boundary condition of the form 
$$
n \cdot B = 0
$$
which corresponds to a metallic (perfectly conducting) boundary. 
In our framework it can be approximated by using the inhomogeneous sequence 
\eqref{dR_spaces_inhom} which we denote by $\bar V^\ell$ as in Section~\ref{sec:lifting},
and by modelling the boundary condition in a weak form.
The mixed formulation takes a form similar to \eqref{ms_w} with a few changes. Namely,
we now look for $p \in \bar V^0 = H^1(\Omega)$ and $B \in \bar V^1 = H(\curl;\Omega)$, solution to
\begin{equation} \label{ms_w_pc}
\left\{
\begin{aligned}
  \sprod{\grad q}{B} &= 0  \qquad && \forall q \in H^1(\Omega)
  \\
  \sprod{v}{\grad p} + \sprod{\curl v}{\curl B} &= \sprod{\curl v}{ J} \qquad && \forall v \in H(\curl;\Omega).
\end{aligned}
\right.
\end{equation}
Compared to \eqref{ms_w}, the only change is in the spaces chosen for the test functions.
In particular we may formally decompose the second equation into a first set of test functions 
with $q|_{\partial \Omega} = 0$ which enforces $\Div B = 0$, and a second set with 
$q|_{\partial \Omega} \neq 0$ which enforces $n \cdot B = 0$ on the boundary.

Again, this problem is not well-posed in general: on domains with holes (i.e., handles) 
a constraint must be added corresponding to harmonic forms,
which leads to an augmented system with one additional Lagrange multiplier $z_h$ 
as we did in the previous section. 
Here the proper space of harmonic forms is the one associated with the inhomogeneous 
sequence \eqref{dR_spaces_inhom} involving $\bar V^1 = H(\curl;\Omega)$ and $\bar V^*_1 = H_0(\Div;\Omega)$.
Thus, it reads
\begin{equation} \label{H1_inhom}
  \bar \cH^1(\Omega) = \{ v \in H_0(\Div;\Omega): ~  \curl v = \Div v = 0\}.
\end{equation}
We also observe that now $p$ is only defined up to constants, so that an additional constraint needs to be added, such as 
$p \in (\ker \grad)^\perp$. In practice this constraint may either be enforced in a similar way as the harmonic constraint,
or by a so-called regularization technique where a term of the form $\ve \sprod{q}{p}$, with $\ve \neq 0$,
is added to the first equation of \eqref{ms_w_pc}. The constrained problem reads then:
Find $B \in H(\curl;\Omega), p \in H^1(\Omega)$ and $z \in \bar \cH^1$, such that
\begin{equation} \label{magneto_mw}
\left\{
\begin{aligned}
  \ve \sprod{q}{p} + \sprod{\grad q}{B} &= 0  \qquad && \forall q \in H^1(\Omega)
  \\
  \sprod{v}{\grad p} + \sprod{\curl v}{\curl B}  + \sprod{v}{z} &= \sprod{\curl v}{J} \qquad && \forall v \in H(\curl;\Omega)
  \\
  \sprod{w}{B} &= 0  \qquad && \forall w \in \bar \cH^1.
\end{aligned}
\right.
\end{equation}
We note that here the result does not depend on $\ve$ (indeed by testing with $v = \grad p$ we find that $p$ is a constant,
and with $q = p$ we find that $p = 0$), hence we may set $\ve = 1$.

At the discrete level we write a system similar to \eqref{magneto_h} or \eqref{magneto_h_op}, but we must replace 
the space \eqref{H1_h_hom} with 
\begin{equation} \label{H1_h_inhom}  
  \bar \cH^1_h = \{v \in \bar V^{1,c}_h : \curl v = 0 ~ \text{ and } ~ \sprod{\grad q}{v} = 0, ~ \forall q \in \bar V^{0,c}_h\}
\end{equation}
where we have denoted $\bar V^{\ell,c}_h = V^\ell_h \cap \bar V^\ell$.
Observe that the for the inhomogeneous sequence, the decomposition \eqref{HH_c} becomes
\begin{equation} \label{HH_c_inhom}
  \bar V^{1,c}_h = \grad^c_h \bar V^{0,c}_h \poplus \bar \cH^1_h \poplus \wt \curl^c_h \bar V^{2,c}_h 
\end{equation}
where $\wt \curl^c_h: \bar V^{2,c}_h \to \bar V^{1,c}_h$ is now the discrete adjoint of
$\curl^c_h: \bar V^{1,c}_h \to \bar V^{2,c}_h$.

We compute the space \eqref{H1_h_inhom} as the kernel of the stabilized Hodge-Laplace operator 
associated with the inhomogeneous sequence, 
$\bar \cH^1_h = \ker (\bar \cL^1_h + \alpha^1 \bar S^1_h)$, with an arbitrary positive $\alpha^1 > 0$.
Note that the matrix of $\bar \cL^1_h$ takes the same form as in \eqref{matL1}, 
\begin{equation} \label{matL1_inhom}
  \bar \matL^1 = \t\matH^1 (\matC\bar \matP^1)^T \matH^2 \matC\bar \matP^1 + \matG\bar \matP^0 \t\matH^0 (\matG\bar \matP^0)^T \matH^1
\end{equation} 
where $\bar \matP^0$ and $\bar \matP^1$ are the conforming projection matrices on the {inhomogeneous} spaces, which leave the boundary degrees 
of freedom untouched instead of setting them to zero. The Hodge (mass) and differential matrices are those of the broken spaces
which have no boundary conditions, hence they are the same as in Section~\ref{sec:MS_vbc}.
The resulting system reads:
find $p_h \in \bar V^0_h$, $B_h \in \bar V^1_h$ and $z_h \in \bar \cH^1_h$, such that
\begin{equation} \label{magneto_h_metal}
  \left\{
  \begin{aligned}
    \sprod{q}{p_h} + \alpha^0 \sprod{(I-\bar P^0_h)q}{(I-\bar P^0_h)p_{h}} + \sprod{\grad \bar P^0_h q}{B_h} &= 0  
    \\
    \sprod{v}{\grad \bar P^0_h p_h}  + \sprod{\curl \bar P^1_h v}{\curl \bar P^1_h B_h} \mspace{100mu} & 
    \\
      + \alpha^1 \sprod{(I-\bar P^1_h)v}{(I-\bar P^1_h)B_{h}} + \sprod{v}{z_h}  &= \sprod{\curl \bar P^1_h v}{J} 
    \\
    \sprod{w}{B_h} &= 0  
  \end{aligned}
  \right.
\end{equation}
for all $q \in \bar V^0_h$, $v \in \bar V^1_h$ and $w \in \bar \cH^1_h$.
An operator form similar to \eqref{magneto_h_op} can also be written.

Finally the matrix form of our broken-FEEC magnetostatic equation with metallic boundary conditions 
is similar to \eqref{magneto_mat}, with the following differences:
\begin{itemize}

  \item the conforming projection operators map in the full conforming spaces $\bar V^\ell$ (boundary dofs are not set to zero),

  \item the discrete harmonic space is now $\bar \cH^1_h$ the kernel of the $\bar \cL^1_h$ operator associated with the full sequence,
  which in practice is also implemented by using conforming projection operators on the inhomogeneous spaces,
  
  \item a regularization term $\matM^0 \arr{p}$ involving the mass matrix in $V^0_h$ is added to the first equation to fully determine $p$.
\end{itemize}

\begin{proposition} \label{prop:MS_metal}
  System~\eqref{magneto_h_metal} is well-posed for arbitrary stabilization parameters 
  $\alpha^0 > 0$, $\alpha^1 \neq 0$, and its solution satisfies
  \begin{equation}
    \begin{aligned}
      &B_h \in \bar V^{1,c}_h \cap (\bar \cH^1_h \oplus \grad \bar V^{0,c}_h)^\perp 
      \\
      &p_h = 0
      \\
      &z_h = 0.
    \end{aligned}
  \end{equation}
Moreover the discrete magnetic field can be decomposed according to \eqref{HH_c_inhom},
as 
$$
B_h = B^{\rm g}_h + B^{\rm h}_h + B^{\rm c}_h 
  \in \grad^c_h \bar V^{0,c}_h \poplus \bar \cH^1_h \poplus \wt \curl^c_h \bar V^{2,c}_h
$$ 
with discrete grad, harmonic and curl components characterized by
\begin{equation}
  \left\{\begin{aligned}
  & B^{\rm g}_h = 0
  \\
  & B^{\rm h}_h = 0
  \\
  & \sprod{w}{\curl B^{\rm c}_h} = \sprod{w}{J} \qquad && \forall w \in \curl \bar V^{1,c}_h.
  \end{aligned} \right.
\end{equation}
In particular, the solution is independent of the parameters $\alpha^0 > 0$, $\alpha^1 \neq 0$,
and the conforming projection operators.
\end{proposition}
\begin{proof}
  Taking $v = z_h$ (in $\bar \cH^1_h \subset \bar V^{1,c}_h$) and using the orthogonality \eqref{HH_c_inhom} shows that $z_h = 0$.
  Taking next $v = \grad \bar P^0_h p_h$ and using again \eqref{HH_c_inhom} 
  shows that $\bar P^0_h p_h$ is a constant, and with $q = 1$ we see that $\sprod{1}{p_h} = 0$. 
  It follows that with $q = (I-\bar P^0_h) p_h$ the first equation becomes $\norm{p_h}^2 + \alpha^0 \norm{(I-\bar P^0_h)p_{h}}^2 = 0$,
  hence $p_h = 0$. Taking next $v = (I-\bar P^1_h) B_h$ yields $\norm{(I-\bar P^1_h) B_h}^2 = 0$, hence 
  $B_h \in \bar V^{1,c}_h$. Since it is orthogonal to both $\bar \cH^1_h$ and $\grad\bar V^{0,c}_h$, it belongs
  to $\wt \curl^c_h V^{2,c}_h$, and it is characterized by $\sprod{\curl v}{\curl B_h} = \sprod{\curl v}{J}$
  for all $v \in \bar V^{1,c}_h$, which completes the proof.
  \qed
\end{proof}

\subsection{Time-dependent Maxwell's equations}
\label{sec:td_max}

We conclude this section by recalling the CONGA discretization of Maxwell's time-dependent equations,
\begin{equation} \label{AF}
    \left\{
    \begin{aligned}
    &\dt E - \curl B = - J
    \\
    &\dt B + \curl E = 0
    \end{aligned}
    \right.
\end{equation}
which consists of computing $E_h \in V^1_h$ and $B_h \in V^2_h$ solution to
\begin{equation} \label{AF_h}
    \left\{
    \begin{aligned}
    &\dt E_h - \wt \curl_h B_h = - J_h
    \\
    &\dt B_h + \curl_h E = 0
    \end{aligned}
    \right.
\end{equation}
with a discrete source approximated with the dual commuting projection
\begin{equation} \label{Jh_tPi}
  J_h = \t \Pi^1_h J.
\end{equation}
This discretization has been proposed and studied 
in \cite{Campos-Pinto.Sonnendrucker.2016.mcomp,Campos-Pinto.Sonnendrucker.2017a.jcm,Campos-Pinto.Sonnendrucker.2017b.jcm},
where it has been shown to be structure-preserving and have long-time stability properties: 
in addition to conserve energy in the absence of sources,
it preserves exactly the discrete Gauss laws
\begin{equation} \label{Gauss_h}
    \left\{
    \begin{aligned}
    &\wt \Div_h E_h = \t \Pi^0_h \rho
    \\
    &\Div_h B_h = 0
    \end{aligned}
    \right.
\end{equation}
thanks to the commuting diagram property $\wt \Div_h \t \Pi^1_h J = \t\Pi^0_h \Div J$ 
satisfied by the dual projection operators, see \eqref{tCD}. Moreover,
the filtering by $(P^1_h)^*$ involved in the source approximation operator
\eqref{tPi_def1} avoids the accumulation of approximation errors 
in the non-conforming part of the kernel of the CONGA $\curl \curl$ operator
\cite{Campos-Pinto.Sonnendrucker.2016.mcomp}.

\begin{remark} \label{rem:Gauss_h}
  The commuting diagram property \eqref{tCD} is actually not needed for 
  the preservation of the discrete Gauss law. One can indeed verify that 
  \eqref{Gauss_h} still holds for a current source defined by a simple $L^2$ 
  projection, 
  \begin{equation} \label{Jh_L2}
  J_h = P_{V^1_h} J
  \end{equation}
  since one has in this case 
  $$\sprod{\wt \Div_h J_h}{\vp} = -\sprod{J_h}{\grad P^0_h \vp} 
    = -\sprod{J}{\grad P^0_h \vp} = \sprod{\Div J}{P^0_h\vp} = \sprod{\tilde \Pi^0_h \Div J}{ \vp}
  $$
  so that \eqref{Gauss_h} follows from the 
  compatibility relation $\Div J + \dt \rho = 0$ satisfied by the exact sources.
  The additional filtering of the source by $(P^1_h)^*$ involved in the dual commuting 
  projection \eqref{Jh_tPi} is nevertheless important for 
  long term stability, as analyzed in \cite{Campos-Pinto.Sonnendrucker.2016.mcomp} 
  and verified numerically in Section~\ref{sec:td_max_num} below.
\end{remark} 

An attractive feature of the CONGA discretization is that for an explicit time-stepping method such as the standard leap-frog scheme
\begin{equation} \label{AF_lf}
    \left\{
    \begin{aligned}
      &B^{n+\frac 12}_h = B^{n}_h - \frac{\Dt}{2}\curl_h E^{n}_h
      \\
      &E^{n+1}_h = E^n_h + \Dt \big(\wt \curl_h B^{n+\frac 12}_h - \t \Pi^1_h J^{n+\frac 12}\big)
      \\
      &B^{n+1}_h = B^{n+\frac 12}_h - \frac{\Dt}{2}\curl_h E^{n+1}_h
    \end{aligned}
    \right.
\end{equation}
each time step is purely local in space. This is easily seen in the matrix form of \eqref{AF_lf}
\begin{equation} \label{AF_mat}
    \left\{
    \begin{aligned}
      &\arrB^{n+\frac 12} = \arrB^{n} - \frac{\Dt}{2}\matC\matP^1 \arrE^{n}
      \\
      &\arrE^{n+1} = \arrE^n + \Dt \t\matH^1  \big((\matC\matP^1)^T \matH^2 \arrB^{n+\frac 12} - (\matP^1)^T\t \arrsigma^1(J^{n+\frac 12})\big)
      \\
      &\arrB^{n+1} = \arrB^{n+\frac 12} - \frac{\Dt}{2}\matC\matP^1 \arrE^{n+1}
    \end{aligned}
    \right.
\end{equation}
where all the matrices except $\matP^1$ are patch-diagonal, and $\matP^1$ only couples 
degrees of freedom of neighboring patches.

Another key property is that for a time-averaged current source,
\begin{equation} \label{tav_J}
  J^{n+\frac 12} := \Dt^{-1}\int_{t^n}^{t^{n+1}} J(t') \rmd t',  
\end{equation}
the discrete Gauss law \eqref{Gauss_h} is also preserved by the fully discrete solution,
namely 
\begin{equation} \label{Gauss_hn}
    \left\{
    \begin{aligned}
    &\wt \Div_h E^n_h = \t \Pi^0_h \rho(t^n)
    \\
    &\Div_h B^n_h = 0
    \end{aligned}
    \right.
\end{equation}
holds for all $n > 0$, provided it holds for $n=0$. 
We also remind that in the absence of source ($J=0$), the leap-frog time scheme 
requires a standard CFL condition for numerical stability, of the form 
\begin{equation} \label{Dt_cond}
  \Dt\tnorm{\curl_h} < 2  
\end{equation}
where
$\tnorm{\curl_h} = \max_{\bF \in V^1_h} \frac{\norm{\curl_h \bF}}{\norm{\bF}}$.
Indeed the discrete pseudo-energy
$$
\cH^{n,*}_h := \frac 12 \big(\norm{\bE^n_h}^2 + \norm{B^{n+\frac 12}_h}^2\big) 
  + \frac {\Dt}{2} \sprod{\curl_h \bE^n_h}{B^{n+\frac 12}_h}
$$
is a constant, $\cH^{n,*}_h = \cH^{*}_h$, so that \eqref{Dt_cond} yields 
a long-time bound 
\begin{equation*} 
  \Big(1+\frac{\Dt\tnorm{\curl_h}}{2}\Big)^{-1} \cH^{*}_h 
  \le \frac 12 \big(\norm{\bE^n_h}^2 + \norm{B^{n+\frac 12}_h}^2\big) 
    \le \Big(1-\frac{\Dt\tnorm{\curl_h}}{2}\Big)^{-1} \cH^{*}_h  
\end{equation*}
valid for all $n \ge 0$. In practice we set $\Dt = C_{\text{cfl}} (2/\tnorm{\curl_h})$
with $C_{\text{cfl}} \approx 0.8$, and we evaluate the operator norm of $\curl_h$
with an iterative power method, noting that it amounts to computing the spectral 
radius of a diagonalizable matrix with non-negative eigenvalues, namely
$$
\tnorm{\curl_h}^2 
= \max_{\bF \in V^1_h}\frac{\sprod{\wt \bcurl_h \curl_h \bF}{\bF}}{\norm{\bF}^2} 
= \rho\big(\t \matH^1 (\matC \matP^1)^T \matH^2 \matC \matP^1\big).
$$



\section{Geometric broken-FEEC discretizations with mapped patches}
\label{sec:geo_bfeec}

In this section we detail the construction of a broken FEEC discretization 
on a 2D multipatch domain where each patch is the image of the reference square,
\begin{equation} \label{Ok}
  \Omega_k = F_k(\hat \Omega), \qquad \hat \Omega = \left]0,1\right[^2
\end{equation}
with smooth diffeomorphisms $F_k$, $k = 1, \dots, K$. 
For simplicity we consider the case 
of planar domains $\Omega_k \subset \RR^2$ and orientation-preserving mappings,
in the sense that the Jacobian matrices
$DF_k(\hat \bx)$ 
have positive Jacobian determinants $J_{F_k}(\hat \bx) = \det (DF_k(\hat \bx)) > 0$ 
for all $\hat \bx \in \hat \Omega$. We also assume that the multipatch decomposition
is geometrically conforming, in a sense that is specified in Section~\ref{sec:mp_conf}

Examples of such domains are represented on Figure~\ref{fig:mp} where the left one 
is a standard curved L-shaped domain with circular patch boundaries that is used in some reference academic test-cases, 
see e.g. \cite{dauge_benchmarks}, and the right one is a non contractible domain 
that also involves analytical mappings, shaped as a simplified pretzel for cultural reasons.
In the latter domain we remind that the presence of holes leads a priori to non exact de Rham sequences, 
and hence non trivial harmonic forms. 
\begin{figure} [!ht]
    \subfloat{  
      \includegraphics[width=0.35\textwidth]{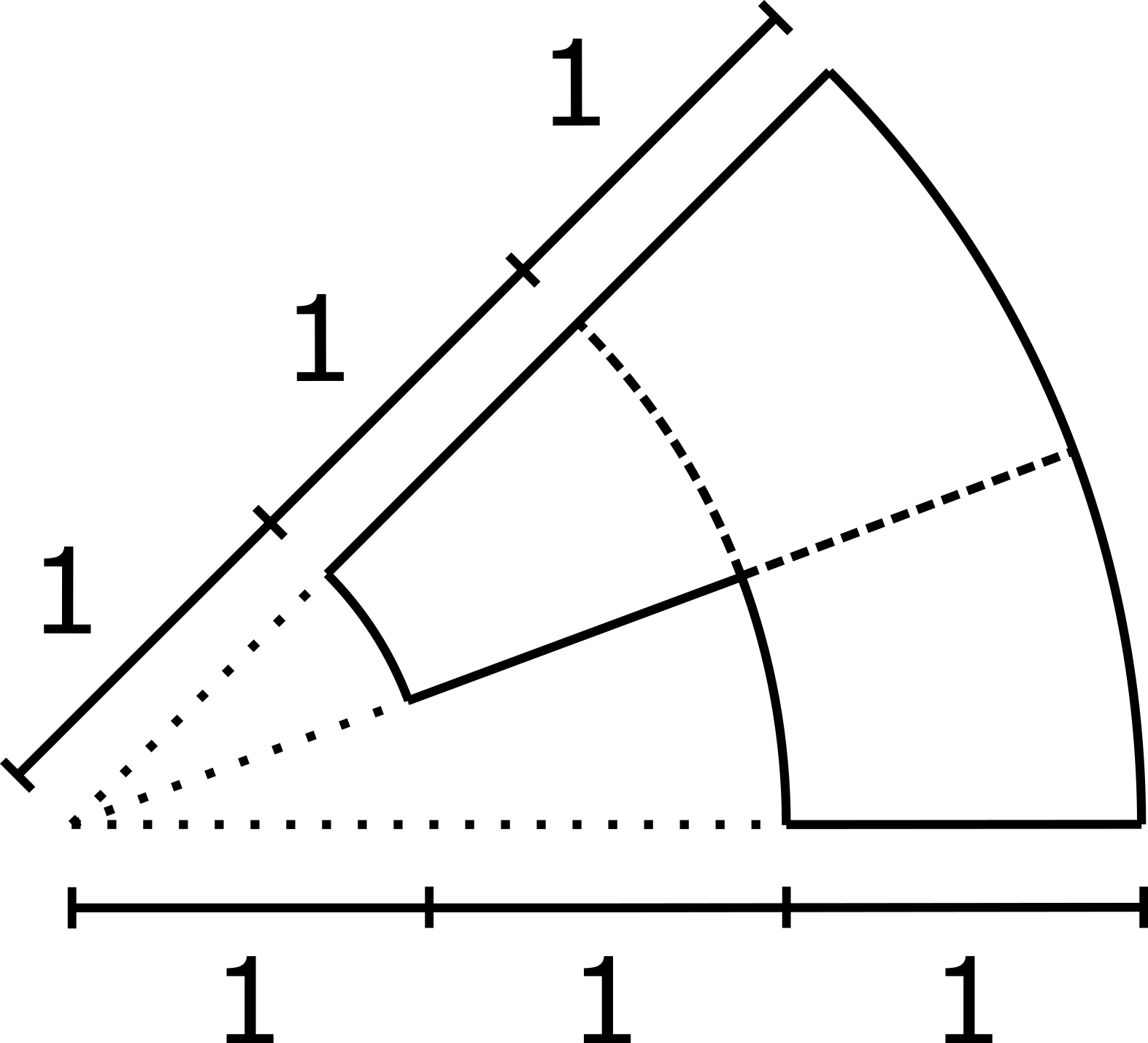}
    }
    \hfill
    \subfloat{ 
      \includegraphics[width=0.45\textwidth]{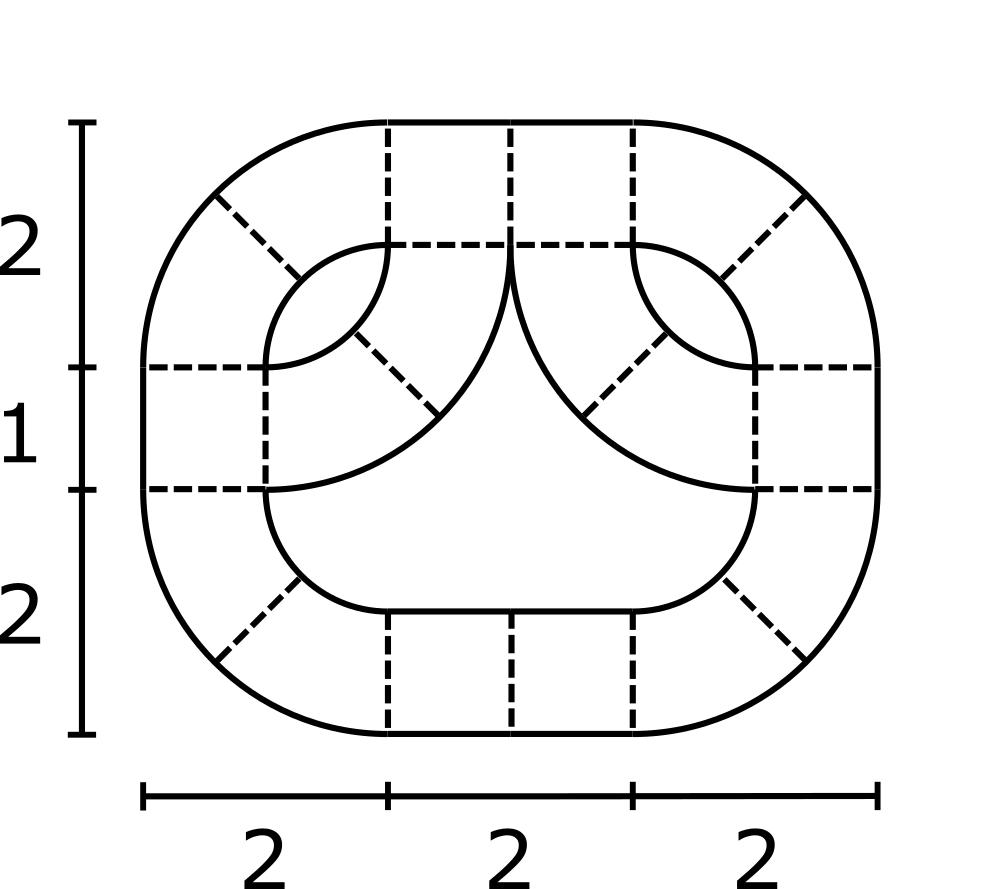}
    }
  \caption{
  Two multipatch domains: a curved L-shaped (left) and a non-contractible domain (right).
  Solid and dashed lines are used to represent the domain boundaries and the patch interfaces,
  respectively.
  }
  \label{fig:mp}
\end{figure}

\subsection{De Rham sequences in 2D}

In 2D the de Rham sequence \eqref{dR} may be reduced in two different ways, namely
\begin{equation} \label{dR_gc}
    V^0 \xrightarrow{ \mbox{$~ \bgrad ~$}}
      V^1 \xrightarrow{ \mbox{$~ \curl ~$}}
        V^2
\end{equation}
and 
\begin{equation} \label{dR_cd}
  V^0 \xrightarrow{ \mbox{$~ \bcurl ~$}}
    V^1 \xrightarrow{ \mbox{$~ \Div ~$}}
      V^2
\end{equation}
where $\bcurl $ and $\curl$ denote the vector and scalar-valued curl operators respectively
(for clarity we use bold fonts to denote vector-valued fields and operators in the subsequent sections).
Here the operators in the second sequence are dual to those in the first one,
and in each case we may consider spaces with homogeneous boundary conditions,
namely
\begin{equation} \label{dR_gc_hom}
  V^0 = H^1_0(\Omega), \quad
    V^1 = \bH_0(\curl;\Omega), \quad
      V^2 = L^2(\Omega)
\end{equation}
for the first sequence and
\begin{equation} \label{dR_cd_hom}
  V^0 = H_0(\bcurl;\Omega), \quad
      V^1 = \bH_0(\Div;\Omega), \quad
        V^2 = L^2(\Omega)
\end{equation}
for the second one, or the full inhomogeneous spaces
\begin{equation} \label{dR_gc_inhom}
  V^0 = H^1(\Omega), \quad
    V^1 = \bH(\curl;\Omega), \quad
      V^2 = L^2(\Omega),
\end{equation}
and
\begin{equation} \label{dR_cd_inhom}
  V^0 = H(\bcurl;\Omega), \quad
      V^1 = \bH(\Div;\Omega), \quad
        V^2 = L^2(\Omega).
\end{equation}

In the numerical examples presented in this article we will consider the first sequence \eqref{dR_gc}
as the primal one, both with homogeneous spaces \eqref{dR_gc_hom} or inhomogeneous ones
\eqref{dR_gc_inhom}.
As mentionned in Section~\ref{sec:lifting} this distinction will only affect the conforming 
projection operators, indeed the broken spaces consist of local spaces which have no boundary 
conditions. This point will be further discussed in Remark~\ref{rem:bc} below.


\subsection{Reference spline complexes}
\label{sec:ref}

On each mapped patch the local sequence \eqref{dR_h_loc} will be built as the push-forward 
of tensor-product spline spaces defined on the reference patch $\hat \Omega = \left]0,1\right[^2$
following \cite{buffa2010isogeometric,Buffa_2011}, and we shall equip it with geometric degrees of freedom.
In order to simplify the construction of the conforming spaces~$\{V^{\ell,c}_h\}_\ell$ over domains with 
non-trivial inter-patch connectivities, we will define a discretization which has reflectional and rotational 
symmetries over~$\hat\Omega$. This amounts to using the same knot sequence along every dimension of~$\hat\Omega$, 
and to require such a knot sequence to be left-right symmetric as explained below.

We now recall the construction 
of the reference tensor-product spline complex with maximal coordinate degree $p \ge 1$
on a grid with $N$ cells per dimension.
For this we equip the interval $[0,1]$ with an open knot sequence
of the form
\begin{equation} \label{knots}
  0 = \xi_0 = \cdots = \xi_p < \xi_{p+1} 
  < \cdots < 
  \xi_{n-1} < \xi_n = \cdots = \xi_{n+p} = 1
\end{equation}
where $n = N+p$, and assumed symmetric for simplicity, in the sense
that $\xi_i = 1 - \xi_{n+p-i}$ for all $i$.
For $i = 0, \dots, n-1$ and $q \in \{p-1,p\}$ we then let $\cN^{q}_{i}$ be the normalized B-spline of
degree $q$ associated with the knots
$(\xi_i, \dots, \xi_{i+q+1})$, see \cite[Def.~4.19]{Schumaker_2007}.
The spline space $\SS_q = \SS_q({\bs \xi}):= \Span\{\cN^{q}_{i} : i = 0, \dots, n-1\}$
then corresponds to splines of maximal regularity on the given subdivision, namely
$
\SS_q = \{ v \in \cC^{q-1}([0,1]) : v|_{(\xi_{q+j},\xi_{q+j+1})} \in \PP_{q}, ~ \forall j = 0, \dots, N-1\}.
$
Notice that here $\cN^{p-1}_{0}$ vanishes identically; it is only kept in the formulas for notational simplicity.
Tensor-product spline spaces of degree $\bq \in \{p-1,p\}^2$
are then defined on the reference domain $\hat \Omega$ as
\begin{equation}\label{SSbq}
  \SS_{\bq}
    := \Span \big\{\cN^{\bq}_{\bi}:
    \bi \in \range{0}{n-1}^2 \big\}
    \quad \text{ with } \quad
  \cN^{\bq}_\bi(\hat \bx) := \prod_{d=1}^2 \cN^{q_d}_{i_d}(\hat x_d).
\end{equation}
The reference spline complex \cite{buffa2010isogeometric} associated with the $\bgrad-\curl$
sequence \eqref{dR_gc} reads then
\begin{equation} \label{refdR_gc}
  \hat \VV^0 := \SS_{(p,p)}  ~ \xrightarrow{ \mbox{$ \bgrad $} }~
  \hat \VV^1 := \begin{pmatrix}
    \SS_{(p-1,p)} \\ \SS_{(p,p-1)}
    \end{pmatrix}
  ~ \xrightarrow{ \mbox{$ \curl $} }~
  \hat \VV^2 := \SS_{(p-1,p-1)}
\end{equation}
while the one associated to the $\bcurl-\Div$ sequence \eqref{dR_cd} reads
\begin{equation} \label{refdR_cd}
  \hat \VV^0 := \SS_{(p,p)}  ~ \xrightarrow{ \mbox{$ \bcurl $} }~
  \hat \VV^1 := \begin{pmatrix}
    \SS_{(p,p-1)} \\ \SS_{(p-1,p)}
    \end{pmatrix}
  ~ \xrightarrow{ \mbox{$ \Div $} }~
  \hat \VV^2 := \SS_{(p-1,p-1)}.
\end{equation}

\subsection{Geometric degrees of freedom on the reference patch}
\label{sec:ref-dofs}

Following~\cite{Bochev_Hyman_2006_csd,Robidoux.2008.histo,Gerritsma.2011.spec,CPKS_variational_2020}
we next equip the reference complexes \eqref{refdR_gc} and \eqref{refdR_cd} with geometric degrees of freedom
which are known to commute with the differential operators. They are based on an interpolation grid 
$$
0 = \zeta_{0} < \cdots < \zeta_{n-1} = 1
$$
for the univariate spline space $\SS_p$ in the sense of
\cite[Th.~4.61]{Schumaker_2007}, that is also symmetric, namely
$\zeta_i = 1-\zeta_{n-1-i}$ for all $i = 0, \dots, n-1$. 
Here the usual choice is to consider Greville points associated with the knots \eqref{knots}.
On the reference domain $\hat \Omega$
we then consider the interpolatory nodes, edges and cells
\begin{equation} \label{subgrid}
  \left\{ \begin{aligned}
    &\hat \ttn_{\bi} := (\zeta_{i_1}, \zeta_{i_2})
    \quad &&\text{ for }
    \bi \in \hat \cM^0 
    \\
    &\hat \tte_{d,\bi} := [\hat \ttn_{\bi}, \hat \ttn_{\bi+\uvec_d}] 
    \quad &&\text{ for }
    (d, \bi) \in \hat \cM^1
    \\
    &\hat \ttc_{\bi} := [\hat \tte_{1,\bi}, \hat \tte_{1,\bi+\uvec_1}] = \prod_{1\le d\le 2}[\zeta_{i_d}, \zeta_{i_d+1}]
    &&\text{ for }
    \bi \in \hat \cM^2.
  \end{aligned} \right.
\end{equation}
Here the square brackets $[\cdot]$ denote convex hulls, $\uvec_d$ is the canonical
basis vector of $\RR^2$ along dimension $d$ and 
the multi-index sets read 
\begin{equation*}
  \left\{ \begin{aligned}
  & \hat \cM^0 :=
    \{\bi : \bi \in \range{0}{n-1}^2 \}
  \\
  & \hat \cM^1 :=
    \{(d, \bi) : d \in \range{1}{2}, \bi \in \range{0}{n-1}^2, i_d+1 < n\}
  \\
  & \hat \cM^2 :=
    \{\bi : \bi \in \range{0}{n-2}^2 \}.
  \end{aligned} \right.
\end{equation*}
Geometric degrees of freedom associated with the $\bgrad-\curl$
sequence \eqref{dR_gc} are then defined 
as the linear forms
\begin{equation} \label{geodofs_gc}
    \left\{ \begin{aligned}
        &\hat \sigma^{0}_{\bi}(v) :=
            v(\hat \ttn_{\bi})
        \quad &&\text{ for }
        \bi \in \hat \cM^0
        \\
        &\hat \sigma^{1}_{d,\bi}(\bv) :=
            \int_{\hat \tte_{d,\bi}} \bv \cdot \uvec_d
        \quad &&\text{ for }
        (d, \bi) \in \hat \cM^1
        \\
        &\hat \sigma^{2}_{\bi}(v) := \int_{\hat \ttc_{\bi}} v
        \quad &&\text{ for }
        \bi \in \hat \cM^2
    \end{aligned} \right.
\end{equation}
where we observe that the reference edges are oriented according to their
natural parametrization,
\begin{equation} \label{params_e}
  \hat \bx^{\tte}_{d,\bi}(s) := \hat \ttn_\bi + (s - \zeta_{i_d}) \uvec_d
  \quad \text{ for } s \in [\zeta_{i_d},\zeta_{i_d+1}].
\end{equation}
For the $\bcurl-\Div$ sequence \eqref{dR_cd} we define $\hat \sigma^{0}_{\bi}$ and
$\hat \sigma^{2}_{\bi}$ similarly, while the intermediate degrees of freedom are modified such that
\begin{equation} \label{geodofs_hdiv}
        \hat \sigma^{1}_{d,\bi}(\bv) :=
            \int_{\hat \tte_{d,\bi}} \bv \cdot \uvec_{d}^\perp
        \quad \text{ for }
        (d, \bi) \in \hat \cM^1,
\end{equation}
where~$\uvec_{d}^\perp := (\uvec_{2}, -\uvec_{1})^T$ corresponds to a $\pi/2$ rotation (counterclockwise) 
of the basis vector~$\uvec_{d}$. 
As these degrees of freedom are unisolvent 
we let $\hat \Lambda^\ell_\mu$, $\mu \in \hat \cM^\ell$, be the {\em geometric basis} for
the spline space $\hat \VV^\ell$, characterized by the duality relations
\begin{equation} \label{geosplines}
  \hat \sigma^\ell_\mu(\hat \Lambda^\ell_\nu) = \delta_{\mu, \nu} \quad
  \text{ for } ~~ \mu, \nu \in \hat \cM^\ell.
\end{equation}

Because the degrees of freedom $\hat \sigma^0_\mu$ 
are based on point values they cannot be applied to general functions in $H^1(\hat \Omega)$,
indeed this space contains discontinuous and even unbounded functions. Similarly the 
degrees of freedom $\hat \sigma^1_\mu$ are not well-defined on $H(\curl;\hat \Omega)$
since functions in this space may not have tangential traces on local edges.
Thus we are in the situation described in Remark~\ref{rem:Pi} where the degrees 
of freedom are not well-defined on the sequence 
\begin{equation} \label{dR_hat}
  \hat V^0 := H^1(\hat \Omega) \xrightarrow{ \mbox{$~ \bgrad ~$}}
    \hat V^1 := H(\curl;\hat \Omega) \xrightarrow{ \mbox{$~ \curl ~$}}
      \hat V^2 := L^2(\hat \Omega).
\end{equation}
Hence, we need to specify some proper domain spaces.

\begin{proposition} \label{prop:ref_dofs}
The reference degrees of freedom $\hat \sigma^\ell_\mu$
are well defined on the domain spaces $\hat U^\ell := \hat V^\ell \cap \hat U^\ell_{L^1}$,
where we have set 
\begin{equation} \label{U_refdofs_gc}
  \hat U^0_{L^1} := W^1_{1,2}(\hat \Omega)  ~ \xrightarrow{ \mbox{$ \bgrad $} }~
  \hat U^1_{L^1} := \begin{pmatrix}
    W^1_{2}(\hat \Omega) \\ W^1_{1}(\hat \Omega)
    \end{pmatrix}
  ~ \xrightarrow{ \mbox{$ \curl $} }~
  \hat U^2_{L^1} := L^1(\hat \Omega)
\end{equation}
for the $\bgrad-\curl$ sequence \eqref{dR_gc}, 
and 
\begin{equation} \label{U_refdofs_cd}
  \hat U^0_{L^1} := W^1_{1,2}(\hat \Omega)  ~ \xrightarrow{ \mbox{$ \bcurl $} }~
  \hat U^1_{L^1} := \begin{pmatrix}
    W^1_{1}(\hat \Omega) \\ W^1_{2}(\hat \Omega)
    \end{pmatrix}
  ~ \xrightarrow{ \mbox{$ \Div $} }~
  \hat U^2_{L^1} := L^1(\hat \Omega)
\end{equation}
for the $\bcurl-\Div$ sequence \eqref{dR_cd}, with anisotropic Sobolev spaces defined as
$$
\begin{cases}
    W^1_{1,2}(\hat \Omega) := \{v \in L^1(\hat \Omega): \partial_1\partial_2 v \in L^1(\hat \Omega)\},
    & \text{ and }
    \\
    W^1_{d}(\hat \Omega) := \{v \in L^1(\hat \Omega): \partial_d v \in L^1(\hat \Omega)\}
    & \text{ for } d \in \{1,2\}.
\end{cases}
$$
Moreover they commute with the local differential operators, namely we have
\begin{equation} \label{ref_cd}
  \hat \sigma^{\ell+1}_{\mu}(d^\ell v) =
  \sum_{\nu \in \hat \cM^\ell} \hat D^\ell_{\mu,\nu} \hat \sigma^{\ell}_{\nu}(v)
  \qquad \forall ~ v \in \hat U^\ell,
  \quad \forall ~ \mu \in \hat \cM^\ell,
\end{equation}
with graph-incidence Kronecker-product matrices $\hat D^\ell$.
\end{proposition}
\begin{proof}
  Nodal degrees of freedom $\hat \sigma^0_\mu$ are well defined 
  on $\hat U^0_{L^1} = W^1_{1,2}(\hat \Omega)$ because this space is continuously imbedded in $C^0(\hat \Omega)$:
  this can be verified using Remarks~9 and~13 from \cite[Sec.~9]{Brezis.2010.fa},
  and a density argument.
  Next we observe that edge degrees of freedom \eqref{geodofs_gc} along horizontal edges $\tte_{1,\bi}$ are of the form 
  $$
  \hat \sigma_{1,\bi}(\bv) = \int_{\zeta_{i_1+1}}^{\zeta_{i_1}} v_1(x',\zeta_{i_2})\rmd x' = \phi(\zeta_{i_1+1},\zeta_{i_2})-\phi(\zeta_{i_1},\zeta_{i_2})
  $$
  where $\phi(x, y) := \int_0^x v_1(x',y)\rmd x'$.
  If $v_1 \in W^1_2(\hat \Omega)$, then we see that both $\phi$ and $\partial_1\partial_2\phi = \partial_2 v_1$ are in $L^1(\hat \Omega)$, 
  hence $\phi \in W^1_{1,2}(\hat \Omega) \subset C^0(\hat \Omega)$, in particular $\hat \sigma_{1,\bi}$ is indeed well defined on $\hat U^1_{L^1}$.
  Similarly we see that $\hat \sigma_{2,\bi}$ is also well defined on $\hat U^1_{L^1}$, and the other degrees of freedom are analyzed 
  with the same reasoning. We also verify easily that in both cases, namely \eqref{U_refdofs_gc} and \eqref{U_refdofs_cd}, 
  the spaces $\hat U^\ell_{L^1}$ form a de Rham sequence. Taking the intersection with \eqref{dR_hat} 
  thus yields a subsequence $\hat U^\ell \subset \hat V^\ell$ where the degrees of freedom are indeed 
  well-defined.
  As for the commuting property, it is well known and follows from the Stokes formula,
  see e.g. \cite{Bochev_Hyman_2006_csd}. We refer to \cite{CPKS_variational_2020}
  for a description of the graph-incidence and Kronecker-product structure of the matrices $\hat D^\ell$.
  \qed
\end{proof}

\subsection{Broken FEEC spaces on the mapped patches}
\label{sec:mapped_spaces}

Following \cite{Hiptmair.2002.anum,Buffa_2011,Perse.Kormann.Sonnendrucker_2020} 
we define the local spaces \eqref{dR_h_loc} on the mapped patches 
$\Omega_k = F_k(\hat \Omega)$ as push-forwards of the reference spline spaces, namely
\begin{equation} \label{Vh_loc}
  V^\ell_h(\Omega_k) := \cF^\ell_k(\hat \VV^\ell) 
\end{equation}
where the push-forward transforms associated with a mapping $F_k$ are obtained as 
the 2D reduction of the usual ones in 3D.
For the $\bgrad-\curl$ sequence \eqref{dR_gc} they read
\begin{equation} \label{pf_gc}
  \left\{
  \begin{aligned}
  &\cF^0_k : \hat v \mapsto v := \hat v \circ F_k^{-1}
  \\
  &\cF^1_k : \hat \bv \mapsto \bv :=  \big(DF_k^{-T} \hat \bv \big)\circ F_k^{-1}
  \\
  &\cF^2_k : \hat v \mapsto v :=  \big(J_{F_k}^{-1} \hat v \big)\circ F_k^{-1}
  \end{aligned}
  \right.
\end{equation}
and for the $\bcurl-\Div$ sequence \eqref{dR_cd} they read
\begin{equation} \label{pf_cd}
  \left\{
  \begin{aligned}
    &\cF^0_k : \hat v \mapsto v := \hat v \circ F_k^{-1}
    \\
    &\cF^1_k : \hat \bv \mapsto \bv :=  \big(J_{F_k}^{-1} DF_k \hat \bv \big)\circ F_k^{-1}
    \\
    &\cF^2_k : \hat v \mapsto v :=  \big(J_{F_k}^{-1} \hat v \big)\circ F_k^{-1}
  \end{aligned}
  \right.
\end{equation}
We remind that $DF_k = \big(\partial_b (F_k)_a(\hat \bx)\big)_{1 \le a,b \le 2}$
denotes the Jacobian matrix of $F_k$, and $J_{F_k}$ 
its (positive) metric determinant corresponding to the surface measure for two-dimensional manifolds in $\RR^3$,
see e.g. \cite{Hiptmair.2002.anum,kreeft2011mimetic,Buffa_Doelz_Kurz_Schoeps_Vazquez_Wolf:2019,%
Holderied_Possanner_Wang_2021}.
The inverse transforms $(\cF^\ell_k)^{-1}$ are called pull-backs, and a fundamental property of these transforms
is that they commute with the differential operators \cite{Hiptmair.2002.anum,Gerritsma.2011.spec}, namely
\begin{equation}\label{cd_pb}
  (\cF^{\ell+1}_k)^{-1}(d^\ell) = \hat d^\ell (\cF^{\ell}_k)^{-1}
\end{equation}
holds with $(d^0, d^1)$ the differential operators in the sequences \eqref{dR_gc} or \eqref{dR_cd},
and $(\hat d^0, \hat d^1)$ their counterparts on the logical (reference) domain.
As these transforms are linear operators, the local spaces are spanned by the push-forwarded basis functions.
In particular, setting 
\begin{equation} \label{bf}
  \Lambda^\ell_{k,\mu} := \begin{cases}
    \cF^\ell_k(\hat \Lambda^\ell_\mu) &\text{ on } \Omega_k
    \\
    0  &\text{ on } \Omega \setminus \Omega_k
\end{cases}
\quad \text{ for } (k, \mu) \in \range{1}{K} \times \hat \cM^\ell
\end{equation}
provides us with local bases for the global broken spaces \eqref{Vh_broken}, i.e., 
\begin{equation} \label{Vh_span}
  V^\ell_h = \Span\{ \Lambda^\ell_{k,\mu} : (k, \mu) \in \cM^\ell_h \}
  \quad \text{ where } \quad 
  \cM^\ell_h := \range{1}{K} \times \hat \cM^\ell.
\end{equation}

\subsection{Geometric degrees of freedom on the mapped patches}
\label{sec:geo_dofs}

In connection with the multipatch structure of $V^\ell_h$
we define broken degrees of freedom using pull-back transforms, which are the inverse of the push-forward operators \eqref{pf_gc} or \eqref{pf_cd}, to obtain
\begin{equation} \label{sell}
  \sigma^\ell_{k,\mu}(v) := \hat \sigma^\ell_{\mu}\big((\cF^{\ell}_k)^{-1}(v|_{\Omega_k})\big)
\quad \text{ for } (k,\mu) \in \cM^\ell_h.
\end{equation}
By construction these degrees of freedom are in duality with the local basis functions \eqref{bf}, i.e. 
\begin{equation} \label{geodual}
  \sigma^\ell_{k,\mu}(\Lambda^\ell_{l,\nu}) = \delta_{(k,\mu), (l,\nu)} \quad
  \text{ for } ~~ (k,\mu), (l,\nu) \in \cM^\ell_h.
\end{equation}
A key feature of the pull-back operators is to carry the geometric nature of the reference degrees of freedom to the mapped elements.
For $\ell = 0$ the degrees of freedom simply consist of pointwise evaluations on the mapped nodes, i.e.
\begin{equation} \label{s0}
\sigma^0_{k,\bi}(v) = v|_{\Omega_k}(\ttn_{k,\bi}) \quad \text{ with } \quad \ttn_{k,\bi} := F_k(\hat \ttn_\bi).
\end{equation}
For $\ell = 1$ with the $\bgrad-\curl$ sequence \eqref{dR_gc}, they correspond to line integrals along 
the mapped edges $\tte_{k,\mu} := F_k(\hat \tte_\mu)$ with $\mu = (d,\bi) \in \hat \cM^1$. 
Specifically, using the pull-back $\hat \bv_k := (\cF^{1}_k)^{-1}(\bv|_{\Omega_k}) = DF_k^T (\bv|_{\Omega_k} \circ F_k)$ 
from \eqref{pf_gc} and the parametrization \eqref{params_e}, we have
\begin{equation} \label{s1_gc}
  \sigma^1_{k,\mu}(\bv) 
  = \int_{\zeta_{i_d}}^{\zeta_{i_d+1}} \big(\bv|_{\Omega_k}(F_k(\hat \bx^\tte_{\mu}(s))\big) \cdot \big(DF_k(\hat \bx^\tte_\mu(s)) \uvec_d\big) \rmd s
  = \int_{\tte_{k,\mu}} \!\! \bv|_{\Omega_k} \cdot \btau 
\end{equation}
where $\btau$ is the unit vector tangent to $\tte_{k,\mu}$ with the orientation inherited from that of $\hat \tte_\mu$,
as mapped by $F_k$: at
$\bx = \bx^\tte_{k, \mu}(s) = F_k(\hat \bx^\tte_\mu(s))$
it reads
\begin{equation} \label{btau}
  \btau(\bx) = \frac{\t \btau(s)}{\norm{\t \btau(s)}}
  \quad \text{ with } \quad
  \t \btau(s) := \frac{\partial \bx^\tte_{k,\mu}}{\partial s}(s) = DF_k(\hat \bx^\tte_\mu(s)) \uvec_d.
\end{equation}
For $\ell = 1$ in the $\bcurl-\Div$ sequence \eqref{dR_cd}, the pull-back corresponding to \eqref{pf_cd} reads
$\hat \bv_k := (\cF^{1}_k)^{-1}(\bv|_{\Omega_k}) = J_{F_k} DF_k^{-1} (\bv|_{\Omega_k} \circ F_k)$, so that 
\eqref{geodofs_hdiv} gives
\begin{equation} \label{s1_cd}
  \sigma^1_{k,\mu}(\bv) 
  = \int_{\zeta_{i_d}}^{\zeta_{i_d+1}} \big(\bv|_{\Omega_k}(F_k(\hat \bx^\tte_{\mu}(s))\big) 
      \cdot \big(J_{F_k} DF_k^{-T}(\hat \bx^\tte_\mu(s)) \uvec^\perp_d\big) \rmd s
  = \int_{\tte_{k,\mu}} \!\! \bv|_{\Omega_k} \cdot \btau^\perp 
\end{equation}
where $\btau^\perp$ is the result of a $\pi/2$ rotation (counterclockwise) of the unit tangent vector~\eqref{btau}, and reads
\begin{equation} \label{btau_perp}
\btau^\perp = \frac{\t \btau^\perp(s)}{\norm{\t \btau(s)}}
\quad \text{ with } \quad
\t \btau^\perp(s) = J_{F_k} DF_k^{-T}(\hat \bx^\tte_\mu(s)) \uvec_d^\perp.
\end{equation}
Finally for $\ell=2$ the degrees of freedom are the integrals on the mapped cells,
\begin{equation} \label{s2}
\sigma^2_{k,\bi}(v) 
= \int_{\hat \ttc_{\bi}} \big(v|_{\Omega_k}(F_k(\hat \bx)\big) J_{F_k}(\hat \bx) \rmd \hat \bx = \int_{\ttc_{k,\bi}} \! v|_{\Omega_k}
\quad \text{ with } \quad \ttc_{k,\bi} := F_k(\hat \ttc_\bi).
\end{equation}

\begin{proposition} \label{prop:dofs}
The broken degrees of freedom $\sigma^\ell_{k,\mu}$
are well defined on the local domain spaces
\begin{equation} \label{U_dofs}
  U^\ell(\Omega_k) := \cF^{\ell}_k(\hat U^\ell) = \{v \in V^\ell(\Omega_k) : (\cF^{\ell}_k)^{-1}(v) \in \hat U^\ell\}
\end{equation}
which involve the spaces $\hat U^\ell$ from \eqref{U_refdofs_gc} and \eqref{U_refdofs_cd},
and form local de Rham sequences. Moreover these degrees of freedom commute with the local 
differential operators. Namely, \eqref{loc-dof-comp} 
holds on $U^\ell(\Omega_k)$ with the coefficients $D^\ell_{k,\mu,\nu} = \hat D^\ell_{\mu,\nu}$ 
from Prop.~\ref{prop:ref_dofs} which are independent of the mapping $F_k$.
\end{proposition}
\begin{proof}
The first statement is a direct consequence of the definition \eqref{sell} and Prop.~\ref{prop:ref_dofs}.
In particular the fact that the spaces \eqref{U_dofs} form a sequence follows from the similar property of the 
reference spaces and the commutation \eqref{cd_pb} of the pull-backs.
As for the local commuting property, it also follows from \eqref{cd_pb} and the similar property \eqref{ref_cd} 
of the reference degrees of freedom,
indeed
\begin{multline*} 
    \sigma^{\ell+1}_{k,\mu}(d^\ell v) 
    = \hat \sigma^{\ell+1}_{\mu}((\cF^{\ell+1}_k)^{-1}(d^\ell v))
    = \hat \sigma^{\ell+1}_{\mu}(\hat d^\ell (\cF^{\ell}_k)^{-1} v)
    \\
    = \sum_{\nu \in \hat \cM^\ell} \hat D^\ell_{\mu,\nu} \hat \sigma^{\ell}_{\nu}((\cF^{\ell}_k)^{-1} v)
    = \sum_{\nu \in \hat \cM^\ell} \hat D^\ell_{\mu,\nu} \sigma^{\ell}_{k,\nu}(v).   
\end{multline*}
\qed
\end{proof}

\subsection{Conformity of the multipatch geometry}
\label{sec:mp_conf}

We next assume that the patches are {\em fully conforming} in the sense that
any interface $\Gamma_{k,l} = \partial \Omega_k \cap \partial \Omega_l$ between
any two distinct patches
\begin{itemize}
  \item[(i)] is either a vertex, or a full edge of both patches,
  \item[(ii)] admits the same parametrization from both patches, up to the orientation.
\end{itemize}
In the case where the interface is a vertex, condition (ii) is empty.
If it is a full patch edge of the form
\begin{equation} \label{patch_conf_1}
  \Gamma_{k,l} = F_k([\hat \bx_0,\hat \bx_0 + \uvec_d]) = F_l([\hat \by_0,\hat \by_0 + \uvec_b])  
\end{equation} 
where $\hat \bx_0$ and $\hat \by_0$ are vertices of the reference domain $\hat \Omega$,
condition~(ii) means that
\begin{equation} \label{patch_conf_2}
F_k(\hat \bx_0 + s\uvec_d) = F_l(\hat \by_0 + \theta_{k,l}(s)\uvec_b), \qquad s \in [0,1]
\end{equation} 
holds with $\theta_{k,l}$ an affine bijection on $[0,1]$, that is $\theta_{k,l}(s) =s$ or $1-s$.

Since the reference knot sequence was supposed symmetric in Section~\ref{sec:ref}, 
these conditions imply that the patches are fully matching in the sense of Assumption 3.3 
from \cite{Buffa_isogeometric_mp_2015} (applied to the scalar spline spaces in the sequence).

\subsection{Conforming projection operators}

An important property of the geometric degrees of freedom is that they provide us with a simple characterization of 
the discrete fields $v_h$ in the broken spaces $V^\ell_h$ which actually belong to the conforming spaces
$V^{\ell,c}_h := V^{\ell}_h \cap V^\ell$.
In order to belong to the space~$V^\ell$ appearing in the continuous de Rham sequence~\eqref{dR}, a field must satisfy regularity conditions which are well known for piecewise-smooth functions.
For $\ell = 0$, the condition for $H^1(\Omega)$ regularity amounts to continuity across the patch interfaces, which is equivalent to requiring that the broken degrees of freedom associated to the same interpolation node coincide:
\begin{equation}\label{conf_s0}
  \ttn_{k,\bi} = \ttn_{l,\bj} \quad \implies \quad
  \sigma^0_{k,\bi}(v_h) = \sigma^0_{l,\bj}(v_h).
\end{equation}
For $\ell = 1$ in the $\bgrad-\curl$ sequence \eqref{dR_gc}, the interface constraint for $\bH(\curl;\Omega)$ regularity
consists in the continuity of the tangential traces, which amounts to requiring
that edge degrees of freedom \eqref{s1_gc} associated to the same edge coincide, up to the orientation:
\begin{equation}\label{conf_s1}
  \tte_{k,\mu} = \tte_{l,\nu} \quad \implies \quad
  \sigma^1_{k,\mu}(v_h) = \ve^\tte(k,\mu;l,\nu) \sigma^1_{l,\nu}(v_h).
\end{equation}
Here $\ve^\tte(k,\mu;l,\nu) = \pm 1$ is the relative orientation of the mapped edges $\tte_{k,\mu}$ and $\tte_{l,\nu}$,
characterized by the relation
$$
\frac{\bx^\tte_{k,\mu}}{\partial s}(s) = \ve^\tte(k,\mu;l,\nu)\frac{\bx^\tte_{l,\nu}}{\partial s}(\theta_{k,l}(s)),
$$
where $\bx^\tte_{k,\mu}(s) = \partial F_k(\hat \bx^\tte_\mu(s))$ and similarly for $\bx^\tte_{l,\nu}$, see \eqref{btau}, and $\theta_{k,l}$ is the affine bijection involved in the interpatch conformity assumption
\eqref{patch_conf_1}--\eqref{patch_conf_2}.

For $\ell = 1$ in the $\bcurl-\Div$ sequence \eqref{dR_cd}, the interface constraint for $\bH(\Div;\Omega)$ regularity
consists in the continuity of the normal traces, which again amounts to requiring
that edge degrees of freedom associated to the same edge coincide, up to the orientation: this constraint takes the same form as
\eqref{conf_s1} with the corresponding definition \eqref{s1_cd}.
Finally as $V^2 = L^2(\Omega)$ there are no constraints for $\ell=2$, i.e., the spaces $V^{2}_h$ and $V^{2,c}_h$ coincide. 
We gather these constraints in a single formula
\begin{equation} \label{conf_cond}
  \ttg^\ell_{k,\mu} = \ttg^\ell_{l,\nu}
  \quad \implies \quad
  \sigma^\ell_{k,\mu}(v_h) = \ve^\ell(k,\mu;l,\nu) \sigma^\ell_{l,\nu}(v_h)  
\end{equation}
where $\ttg^\ell_{k,\mu}$ denotes the geometric element (node, edge or cell) of dimension $\ell$ associated with the multi-index
$(k,\mu) \in \cM^\ell_h$, $\ve^1$ denotes the relative orientation of two edges,
and $\ve^0 = \ve^2 := 1$.

Thanks to the conformity assumption of the multipatch geometry and to the symmetry of the interpolatory grid,
each broken degree of freedom on an interface can be matched to those of the adjacent patches in such a way that the constraints above are satisfied.
Accordingly, the resulting broken discrete field will belong to the conforming space~$V^{\ell,c}_h$.
In particular we may 
define simple conforming projection operators $P^\ell$ by averaging 
the broken degrees of freedom associated with interface elements.
Using the broken basis functions \eqref{bf} this yields an expression similar
to the one given in \cite{conga_hodge} for tensor-product polynomial elements,
with additional relative orientation factors due to the general mapping configurations,
\begin{equation} \label{conf_P}
  P^\ell_h \Lambda^\ell_{l,\nu} :=
  \frac{1}{\#\cM^\ell_h(\ttg^\ell_{l,\nu})} \sum_{(k,\mu) \in \cM^\ell_h(\ttg^\ell_{l,\nu})}
  \ve^\ell(k,\mu;l,\nu)\Lambda^{\ell}_{k,\mu}
\end{equation}
where $\cM^\ell_h(\ttg^\ell_{l,\nu})$ contains the broken indices corresponding to the same 
geometric element.
The entries of the corresponding operator matrix \eqref{matP} read then
\begin{equation} \label{conf_P_mat}
\matP_{(k,\mu),(l,\nu)} = \begin{cases}
  \frac{\ve^\ell(k,\mu;l,\nu)}{\#\cM^\ell_h(\ttg^\ell_{l,\nu})} \quad &\text{ if } \ttg^\ell_{k,\mu} = \ttg^\ell_{l,\nu}
  \\
  0 \quad &\text{ otherwise }
\end{cases}
\end{equation}
for all $(k,\mu), (l,\nu) \in \cM^\ell_h$.
We may summarize our construction with the following result.
\begin{proposition} \label{prop:summary}
The broken degrees of freedom defined in Section~\ref{sec:geo_dofs} 
satisfy the properties listed in Section~\ref{sec:proj_bfeec},
with local domain spaces $U^\ell(\Omega_k)$ defined in \eqref{U_dofs} and 
local differential matrices independent of the mappings $F_k$.
The operators $P^\ell_h : V^\ell_h \to V^{\ell}_h$ defined by \eqref{conf_P}
are projections on the conforming subspaces $V^{\ell,c}_h = V^{\ell}_h \cap V^{\ell}$.
\end{proposition}
\begin{proof}
The boundedness of $\sigma^\ell_{k,\mu}$ on the local spaces $U^\ell(\Omega_k)$, as well as the
commuting properties, have been verified in Prop.~\ref{prop:dofs}.
The inter-patch conformity property \eqref{conf-dof} follows from the fact that piecewise smooth functions that 
belong to $H^1(\Omega)$, resp. $H(\curl;\Omega)$ and $H(\Div;\Omega)$, admit a unique trace,
resp. tangential and normal trace on patch interfaces \cite{Boffi.Brezzi.Fortin.2013.scm}.
The same property, together with the interpolation nature of the geometric basis functions, allows verifying that the
operators $P^\ell_h : V^\ell_h \to V^\ell_h$ are characterized by the relations 
\begin{equation}
  \sigma^\ell_{k,\mu}(P^\ell_h v_h) = 
  \frac{1}{\#\cM^\ell_h(\ttg^\ell_{k,\mu})} \sum_{(l,\nu) \in \cM^\ell_h(\ttg^\ell_{k,\mu})}
  \ve^\ell(k,\mu;l,\nu) \sigma^\ell_{l,\nu}(v_h)
\end{equation}
for all $(k,\mu) \in \cM^\ell_h$.
Using the geometric condition \eqref{conf_cond}, this allows verifying that $P^\ell_h$
is indeed  a projection on the conforming subspace $V^{\ell,c}_h$. 
\qed
\end{proof}

\begin{remark}[boundary conditions] \label{rem:bc}
  As stated, the above construction actually corresponds to the inhomogeneous sequences
  \eqref{dR_gc_inhom} and \eqref{dR_cd_inhom}.
  If the conforming spaces are homogeneous, as in
  \eqref{dR_gc_hom} or \eqref{dR_cd_hom}, then $P^\ell_h$ should further set 
  the boundary degree of freedom to 0, which amounts to restricting the non-zeros entries 
  in \eqref{conf_P_mat} to the geometrical elements $\ttg^\ell_{k,\mu}$ that are inside $\Omega$.
  We note that this does not affect the broken spaces $V^\ell_h$, as already observed 
  in Section~\ref{sec:lifting}, since they are not required to have boundary conditions.
\end{remark}

\begin{remark}[extension to 3D] \label{rem:3D}
  The extension to the 3D setting of the above construction poses no particular difficulty,
  using the same tensor-product and mapped spline spaces as in \cite{Buffa_2011,daVeiga_2014_actanum},
  and the same geometric degrees of freedom as in
  \cite[Sec.~6.1]{CPKS_variational_2020} for the primal sequence. 
\end{remark}

\subsection{Primal-dual matrix diagram with B-splines}
\label{sec:cob}

In practice, a natural approach is to work with B-splines. Indeed the
geometric (interpolatory) splines $\Lambda^\ell_{k,\mu}$ introduced 
in Sections~\ref{sec:ref-dofs} and \ref{sec:mapped_spaces}
are not known explicitely: they depend on the degrees of freedom, i.e. on the
interpolation grids \eqref{subgrid}. They are also less local than the B-splines,
as they are in general supported on their full patch $\Omega_k$: for patches with 
many cells, this would lead to an expensive computation of the fully populated mass matrices.

In our numerical experiments we have followed \cite{daVeiga_2014_actanum,CPKS_variational_2020} 
and used tensor-product splines $\hat B^\ell_{\mu}$ on the reference domain $\hat \Omega$,
composed of normalized B-splines $\cN^p_i$ in the dimensions of degree $p$
and of Curry-Schoenberg B-splines $\cD^{p-1}_{i} = \Big(\frac{p}{\xi_{i+p+1}-\xi_{i+1}}\Big) \cN^{p-1}_{i+1}$ 
(also called M-splines) in the dimensions of degree $(p-1)$.
On the mapped patches $\Omega_k = F_k(\hat \Omega)$ the basis functions are defined again 
as push-forwards $B^\ell_{k,\mu} := \cF^\ell_k(\hat B^\ell_\mu)$ for $(k,\mu) \in \cM^\ell_h$.

The discrete elements in the matrix diagram \eqref{CD} must then be adapted for B-splines:
one first observe that the change of basis is provided by the patch-wise collocation matrices 
whose diagonal blocks read
\begin{equation} \label{matK}
  \matK^\ell_{(k,\mu), (k,\nu)} = \sigma^\ell_{k,\mu}(B^\ell_{k,\nu}) = \hat \sigma^\ell_{\mu}(\hat B^\ell_{\nu}),
  \qquad \mu, \nu \in \hat \cM^\ell
\end{equation}
where the second equality uses \eqref{sell}.
From $B^\ell_{k,\nu} = \sum_\mu \matK^\ell_{(k,\mu), (k,\nu)} \Lambda^\ell_{k,\mu}$ 
it follows that the B-spline coefficients of the geometric projection \eqref{brok-proj}, namely
\begin{equation} \label{proj-B}
 \Pi^\ell_h v = \sum_{(k,\mu) \in \cM^\ell_h} \sigma^\ell_{k,\mu}(v) \Lambda^\ell_{k,\mu} 
    = \sum_{(k,\mu) \in \cM^\ell_h}\beta^\ell_{k,\mu}(v) B^\ell_{k,\mu}
\end{equation}
read  (using again some implicit flattening $\cM^\ell_h \ni (k,\mu) \mapsto i \in \{1, \dots, N^\ell\}$)
\begin{equation} \label{cob}
  \arrbeta^\ell(v) = (\matK^\ell)^{-1} \arrsigma^\ell(v).
\end{equation}
Accordingly, we obtain the matrices $\matP^\ell_B = (\beta^\ell_i(P^\ell_h B^\ell_j))_{1 \le i,j \le N^\ell}$ 
of the conforming projection operator \eqref{conf_P} in the B-spline bases
by combining the matrices \eqref{conf_P_mat} with the above change of basis: this gives
\begin{equation} \label{conf_P_matB} 
  \matP^\ell_B = (\matK^\ell)^{-1} \matP^\ell \matK^\ell.
\end{equation}
Here we note that the collocation matrices $\matK^\ell$ are Kronecker products of univariate banded matrices,
see \cite[Sec.~6.2]{CPKS_variational_2020}, so that \eqref{cob} and \eqref{conf_P_matB} may be implemented 
in a very efficient way.
The primal sequence is completed by computing the patch-wise differential matrices in the B-spline bases,
$$
(\matD^\ell_B)_{(k,\mu),(k,\nu)} := \beta^{\ell+1}_{k,\mu}(d^\ell B^\ell_{k,\nu}).
$$
Thanks to the univariate relation $\frac{\dd}{\dd x}\cN^p_i = \cD^{p-1}_{i-1} - \cD^{p-1}_{i}$, 
these are patch-wise incidence matrices, just as the ones in the geometric bases
\cite{daVeiga_2014_actanum,CPKS_variational_2020}.
For the dual sequence we use bi-orthogonal splines characterized by the relations
$$
\t B_{k,\mu} \in V^\ell_h(\Omega_k), 
\quad \t \beta^\ell_{k,\nu}(\t B^\ell_{k,\mu}) = \delta_{\mu,\nu} \quad \forall \mu, \nu \in \hat \cM^\ell
$$
with dual degrees of freedom
\begin{equation} \label{dual_dofs_B}
  \t \beta^\ell_{k,\mu}(v) := \sprod{v}{B^\ell_{k,\mu}}.  
\end{equation}
This leads to defining again the primal Hodge matrices as patch-diagonal mass matrices,
$\matH^\ell_B = \matM^\ell_B$ and the dual Hodge ones as their patch-diagonal inverses
$\t \matH^\ell_B = (\matM^\ell_B)^{-1}$.
These matrices can be computed on the reference domain according to
\begin{equation} \label{M_matrix_map}
  (\matM^\ell_B)_{(k,\mu),(k,\nu)} = \sprod{B^\ell_{k,\mu}}{B^\ell_{k,\nu}} = 
  \sprod{\cF^\ell_k(\hat B^\ell_\mu)}{\cF^\ell_k(\hat B^\ell_\nu)}.
\end{equation}
Specifically, for the 2D $\bgrad-\curl$ sequence \eqref{dR_gc}, the transforms \eqref{pf_gc} yield
\begin{equation} \label{M_matrix_map_gc}
  (\matM^\ell_B)_{(k,\mu),(k,\nu)} =
  \begin{cases} 
  \int_{\hat \Omega} \hat B^0_\mu \hat B^0_\nu J_{F_k} \rmd \hat \bx,
    & \text{ for $\ell = 0$}
    \\
    \int_{\hat \Omega} (\hat B^1_\mu)^T (DF_k^T DF_k)^{-1} \hat B^1_\nu J_{F_k} \rmd \hat \bx,
    & \text{ for $\ell = 1$}
    \\
    \int_{\hat \Omega} (\hat B^2_\mu)^T \hat B^2_\nu J_{F_k}^{-1} \rmd \hat \bx,
    & \text{ for $\ell = 2$}
  \end{cases}
\end{equation}
and for the 2D $\bcurl-\Div$ sequence \eqref{dR_cd} we find from \eqref{pf_cd} that
\begin{equation} \label{M_matrix_map_cd}
  (\matM^1_B)_{(k,\mu),(k,\nu)} =
    \int_{\hat \Omega} (\hat B^1_\mu)^T DF_k^T DF_k \hat B^1_\nu J_{F_k}^{-1} \rmd \hat \bx,
\end{equation}
while $\matM^0_B$ and $\matM^2_B$ take the same values as in \eqref{M_matrix_map_gc}.
In 3D the formulas can be derived from the pull-backs given in \cite{daVeiga_2014_actanum}.

Finally the coefficients of the dual commuting projections \eqref{tPi_def1}--\eqref{tPi_def2} 
in the dual B-spline bases read
$$
\tilde \arrbeta^\ell(\t \Pi^\ell_h v) = (\matP_B^\ell)^T \tilde \arrbeta^\ell(v)
$$ 
using the dual degrees of freedom \eqref{dual_dofs_B} and the same implicit flattening 
as in \eqref{cob}, \eqref{conf_P_matB}.

Gathering the above elements we obtain a new version of diagram \eqref{CD},
where the coefficient spaces (still defined as $\RR^{N^\ell}$)
are now denoted $\cC^\ell_B$ and $\t \cC^\ell_B$ to indicate that they contain 
coefficients in the B-spline and dual B-spline basis, respectively.
\begin{equation} \label{CD_B}
\begin{tikzpicture}[ampersand replacement=\&, baseline] 
\matrix (m) [matrix of math nodes,row sep=3em,column sep=5em,minimum width=2em] {
   ~~ V^0 ~ \bbb
   \& ~~ V^1 ~ \bbb
    \& ~~ V^2 ~ \bbb
      \& ~~ V^3 ~ \bbb
  \\
  ~~ V_h^0 ~ \bbb
    \& ~~ V_h^1 ~ \bbb
    \& ~~ V_h^2 ~ \bbb
    \& ~~ V_h^3 ~ \bbb
\\
~~ \cC^0_B ~ \bbb
  \& ~~ \cC^1_B ~ \bbb
    \& ~~ \cC^2_B ~ \bbb
    \& ~~ \cC^3_B ~ \bbb
\\
~~ \t \cC^0_B ~ \bbb
\& ~~ \t \cC^1_B ~ \bbb
\& ~~ \t \cC^2_B ~ \bbb
\& ~~ \t \cC^3_B ~ \bbb
\\
~~ V_h^0 ~ \bbb
  \& ~~ V_h^1 ~ \bbb
  \& ~~ V_h^2 ~ \bbb
  \& ~~ V_h^3 ~ \bbb
\\
~~ V^*_0 ~ \bbb
\& ~~ V^*_1 ~ \bbb
\& ~~ V^*_2 ~ \bbb
\& ~~ V^*_3 ~ \bbb
\\
};
\path[-stealth]
(m-1-1) edge node [above] {$\grad$} (m-1-2)
(m-1-2) edge node [above] {$\curl$} (m-1-3)
(m-1-3) edge node [above] {$\Div$} (m-1-4)
(m-1-1) edge node [pos=0.6, right] {$\Pi^0_h$} (m-2-1)
(m-1-2) edge node [pos=0.6, right] {$\Pi^1_h$} (m-2-2)
(m-1-3) edge node [pos=0.6, right] {$\Pi^2_h$} (m-2-3)
(m-1-4) edge node [pos=0.6, right] {$\Pi^3_h$} (m-2-4)
(m-1-1) edge [bend right=40] node [pos=0.1, left] {$(\matK^0)^{-1} \arrsigma^0$} (m-3-1)
(m-1-2) edge [bend right=40] node [pos=0.1, left] {$(\matK^1)^{-1} \arrsigma^1$} (m-3-2)
(m-1-3) edge [bend right=40] node [pos=0.1, left] {$(\matK^2)^{-1} \arrsigma^2$} (m-3-3)
(m-1-4) edge [bend right=40] node [pos=0.1, left] {$(\matK^3)^{-1} \arrsigma^3$} (m-3-4)
(m-2-1) edge node [above] {$\grad_h$} (m-2-2)
(m-2-2) edge node [above] {$\curl_h$} (m-2-3)
(m-2-3) edge node [above] {$\Div_h$} (m-2-4)
(m-3-1.75) edge node [right] {$\cI^0_B$} (m-2-1.285)
(m-3-2.75) edge node [right] {$\cI^1_B$} (m-2-2.285)
(m-3-3.75) edge node [right] {$\cI^2_B$} (m-2-3.285)
(m-3-4.75) edge node [right] {$\cI^3_B$} (m-2-4.285)
(m-2-1.255) edge node [pos=0.2, left] {$\arrbeta^0$} (m-3-1.105)
(m-2-2.255) edge node [pos=0.2, left] {$\arrbeta^1$} (m-3-2.105)
(m-2-3.255) edge node [pos=0.2, left] {$\arrbeta^2$} (m-3-3.105)
(m-2-4.255) edge node [pos=0.2, left] {$\arrbeta^3$} (m-3-4.105)
(m-3-1) edge node[auto] {$ \matG_B \matP^0_B$} (m-3-2)
(m-3-2) edge node[auto] {$ \matC_B \matP^1_B$} (m-3-3)
(m-3-3) edge node[auto] {$ \matD_B \matP^2_B$} (m-3-4)
%
%
(m-4-1.105) edge [dashed] node [left] {$\t \matH^0_B$} (m-3-1.255)
(m-4-2.105) edge [dashed] node [left] {$\t \matH^1_B$} (m-3-2.255)
(m-4-3.105) edge [dashed] node [left] {$\t \matH^2_B$} (m-3-3.255)
(m-4-4.105) edge [dashed] node [left] {$\t \matH^3_B$} (m-3-4.255)
(m-3-1.285) edge [dashed] node [right] {$\matH^0_B$} (m-4-1.75)
(m-3-2.285) edge [dashed] node [right] {$\matH^1_B$} (m-4-2.75)
(m-3-3.285) edge [dashed] node [right] {$\matH^2_B$} (m-4-3.75)
(m-3-4.285) edge [dashed] node [right] {$\matH^3_B$} (m-4-4.75)
%
%
(m-5-2) edge node [above] {$\wt \Div_h$} (m-5-1)
(m-5-3) edge node [above] {$\wt \curl_h$} (m-5-2)
(m-5-4) edge node [above] {$\wt \grad_h$} (m-5-3)
(m-5-1.105) edge node [pos=0.2, left] {$\t \arrbeta^0$} (m-4-1.255)
(m-5-2.105) edge node [pos=0.2, left] {$\t \arrbeta^1$} (m-4-2.255)
(m-5-3.105) edge node [pos=0.2, left] {$\t \arrbeta^2$} (m-4-3.255)
(m-5-4.105) edge node [pos=0.2, left] {$\t \arrbeta^3$} (m-4-4.255)
(m-4-1.285) edge node [right] {$\t \cI^0_B$} (m-5-1.75)
(m-4-2.285) edge node [right] {$\t \cI^1_B$} (m-5-2.75)
(m-4-3.285) edge node [right] {$\t \cI^2_B$} (m-5-3.75)
(m-4-4.285) edge node [right] {$\t \cI^3_B$} (m-5-4.75)
(m-4-2) edge node [above] {$-(\matG_B \matP^0_B)^T $} (m-4-1)
(m-4-3) edge node [above] {$ (\matC_B \matP^1_B)^T $} (m-4-2)
(m-4-4) edge node [above] {$-(\matD_B \matP^2_B)^T $} (m-4-3)
(m-6-1) edge [bend left=40] node [pos=0.2, left] {$(\matP^0_B)^T\t \arrbeta^0$} (m-4-1)
(m-6-2) edge [bend left=40] node [pos=0.2, left] {$(\matP^1_B)^T\t \arrbeta^1$} (m-4-2)
(m-6-3) edge [bend left=40] node [pos=0.2, left] {$(\matP^2_B)^T\t \arrbeta^2$} (m-4-3)
(m-6-4) edge [bend left=40] node [pos=0.2, left] {$(\matP^3_B)^T\t \arrbeta^3$} (m-4-4)
(m-6-1) edge node [right] {$\t \Pi^0_h$} (m-5-1)
(m-6-2) edge node [right] {$\t \Pi^1_h$} (m-5-2)
(m-6-3) edge node [right] {$\t \Pi^2_h$} (m-5-3)
(m-6-4) edge node [right] {$\t \Pi^3_h$} (m-5-4)
(m-6-2) edge node [above] {$\Div$} (m-6-1)
(m-6-3) edge node [above] {$\curl$} (m-6-2)
(m-6-4) edge node [above] {$\grad$} (m-6-3)
;
\end{tikzpicture}
\end{equation}

Note that here the primal and dual interpolation operators are simply
\begin{equation} \label{cIB}
    \cI^\ell_B: \arr{b} \mapsto \sum_{i=1}^{N^\ell} b_i B^\ell_i
    \qquad \text{ and }\qquad
    \t \cI^\ell_B: \t {\arr{b}} \mapsto \sum_{i=1}^{N^\ell} \t b_i \t B^\ell_i~.  
\end{equation}


\section{Numerical results}
\label{sec:num}

In this section we illustrate several of the numerical schemes described in 
Section~\ref{sec:pbms}, using the 2D multipatch spline spaces presented above.
The results have been obtained with the Psydac library \cite{psydac}.

\subsection{Poisson problem with homogeneous boundary conditions}
\label{sec:num_poisson_hom}

We first test our CONGA scheme for the homogeneous Poisson problem
presented in Section \ref{sec:poisson}.
For this we consider an analytical solution on the pretzel-shaped domain 
shown in Figure~\ref{fig:mp}, given by 
\begin{equation} \label{phi_manu_poisson_elliptic}
  \phi(\bx) = \exp\Big(-\frac{\tau^2(\bx)}{2 \sigma^2} \Big)  
\qquad \text{ with }\qquad 
\tau(\bx) = as^2(\bx) + bt^2(\bx) - 1.
\end{equation}
Here, $s(\bx) = \t x -\t y$, $t(\bx) = \t x +\t y$
with $\t x = x- x_0$, $\t y= y-y_0$
and we take   
$x_0 = y_0 = 1.5$, 
$a = (1/1.7)^2$, 
$b = (1/1.1)^2$ 
and $\sigma = 0.11$
in order to satisfy the homogeneous boundary condition with accuracy $\approx 1e-10$.
The associated manufactured source is
\begin{equation} \label{f_manu_poisson_elliptic}
  f = -\Delta \phi = -\left( \frac{\tau^2 \norm{\nabla \tau}^2}{\sigma^4} - \frac{\tau \Delta \tau + \norm{\nabla\tau}^2}{\sigma^2}\right)\phi  ~.
\end{equation}

In Table~\ref{tab:Poisson_L2_error} 
we show the relative $L^2$ errors corresponding to different grids and spline degrees, 
and in Figure~\ref{fig:poisson} we plot the numerical solutions corresponding 
to spline elements of degree $3 \times 3$ on each patch.
These results show that the numerical solutions converge towards the exact one
as the grids are refined. 
We do not observe significant improvements however when higher order polynomials are used.
As the next results will show, this is most likely due to the steep nature of the solution.

\begin{table}[!htbp]
\centering
\begin{tabular}{|l|l|l|l|l|}
\hline
\diagbox[width=10em]{cells p.p.}{degree} & \multicolumn{1}{p{12mm}|}{$2\times 2$} & \multicolumn{1}{p{12mm}|}{$3 \times 3$} & \multicolumn{1}{p{12mm}|}{$4 \times 4$} & \multicolumn{1}{p{12mm}|}{$5 \times 5$}
\\ 
\hline
 $4  \times 4$  & 0.99420 & 0.73962 & 0.35919 & 0.41536 \\ \hline
 $8  \times 8$  & 0.11599 & 0.12811 & 0.13921 & 0.16493 \\ \hline
 $16 \times 16$ & 0.00692 & 0.00892 & 0.01013 & 0.00684 \\ \hline
\end{tabular}
\caption{Relative $L^2$ errors for the Poisson problem with homogeneous boundary 
conditions discretized with \eqref{poisson_hLS}, 
using various numbers of cells per patch and spline degrees.
In each case the error is computed with respect to the finite element projection $\Pi^0_h \phi$ of the analytic solution \eqref{phi_manu_poisson_elliptic} on the space $V^0_h$.}
\label{tab:Poisson_L2_error}
\end{table}

\begin{figure}[!htbp]
\def \plotdircoarse {poisson_hom/pretzel_f_manu_poisson_elliptic_nc=4_deg=3}
\def \plotdirmed {poisson_hom/pretzel_f_manu_poisson_elliptic_nc=8_deg=3}
\def \plotdirfine {poisson_hom/pretzel_f_manu_poisson_elliptic_nc=16_deg=3}
\def \plotfn {eta=0_mu=1_gamma_h=10.0_phi_h}
\begin{center}
  \subfloat[$4 \times 4$ cells p.p.]{%
    \includegraphics[width=0.3\textwidth]{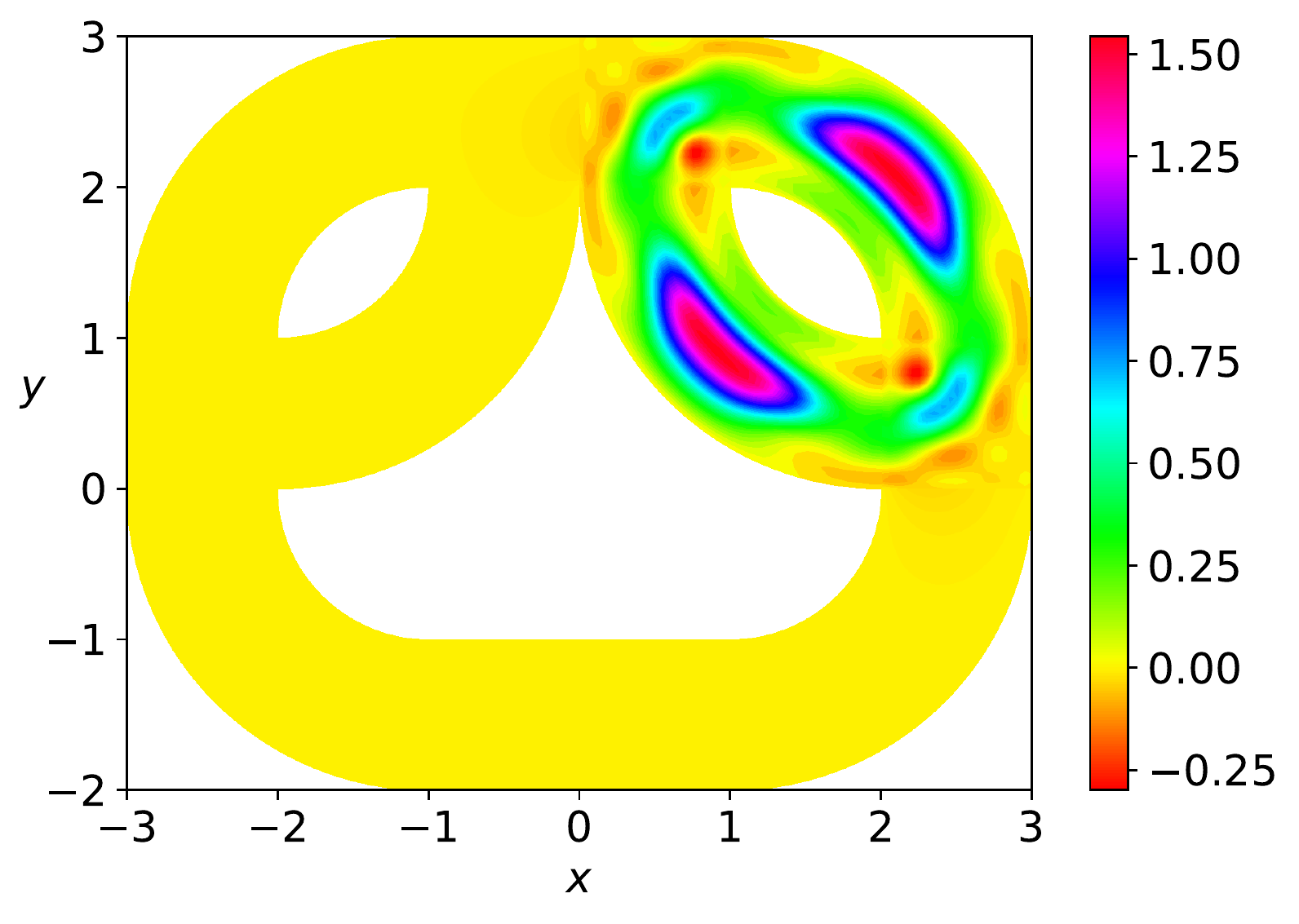}
  }
  \hspace{10pt}
  \subfloat[$8 \times 8$ cells p.p.]{%
    \includegraphics[width=0.3\textwidth]{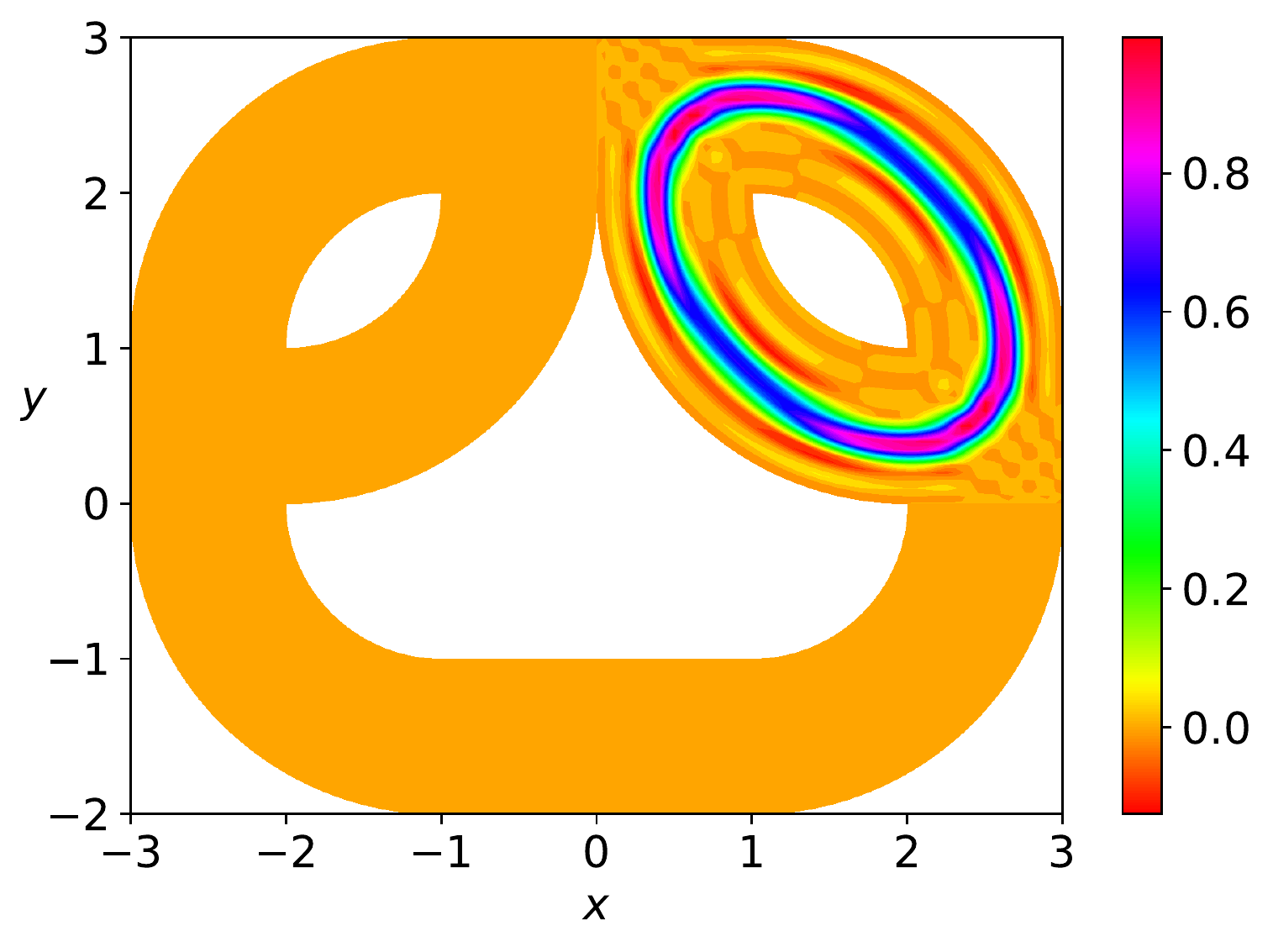}
  }
  \hspace{10pt}
  \subfloat[$16 \times 16$ cells p.p.]{
    \includegraphics[width=0.3\textwidth]{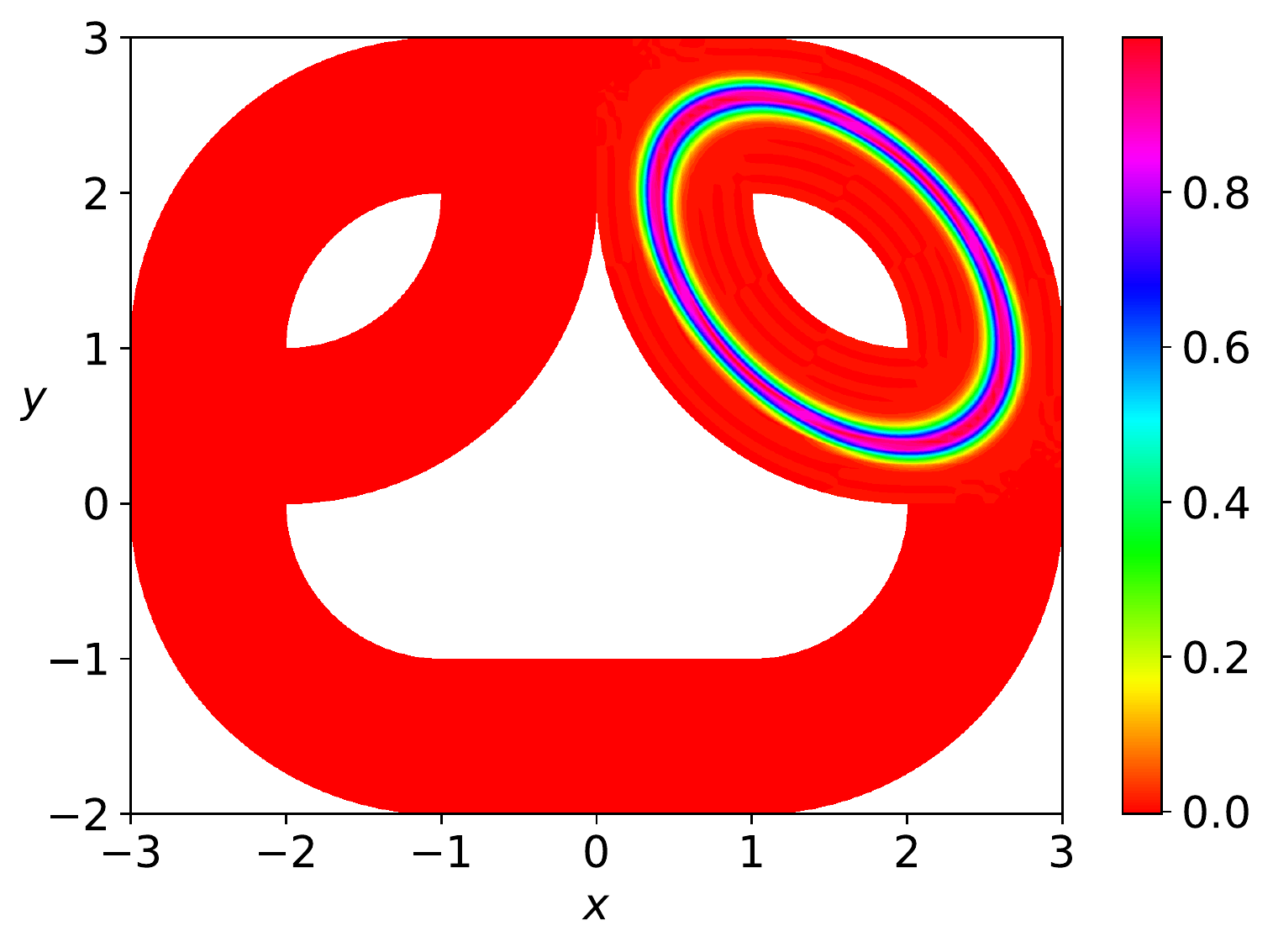}
  }
\end{center}
\caption{
  Solutions $\phi_h \in V^0_h$ computed by the scheme \eqref{poisson_hLS}, 
  with spline elements of degree $3 \times 3$ and
  $N \times N$ of cells per patches as indicated,
  corresponding to some of the errors shown in Table~\ref{tab:Poisson_L2_error}.
}
\label{fig:poisson}
\end{figure}

\subsection{Poisson problem with inhomogeneous boundary conditions}
\label{sec:num_poisson_inhom}

Our second test is with an inhomogeneous Poisson-Dirichlet problem
\eqref{poisson_fg}, using the lifting method described 
in Section~\ref{sec:lifting} for the boundary condition
and the solver tested just above for the homogeneous part of the solution. 
Specifically, we define $\phi_{g,h} \in V^0_h$ by computing its 
boundary degrees of freedom from the data $g$ on $\partial \Omega$
(this is straightforward with our geometric boundary degrees of freedom
\eqref{s0}), 
and we compute $\phi_{0,h} = \phi_{h}-\phi_{g,h}$ by solving \eqref{poisson_hbc}.
As we are not constrained by a specific condition on the domain boundary 
we consider a smooth solution to assess whether high order convergence 
rates can be observed despite the singularities in the domain boundaries. 
Specifically, we use again the pretzel-shaped domain and set the source 
and boundary condition as
\begin{equation} \label{fg_sincos}
  f(\bx) = -2\pi^2 \sin(\pi x)\cos(\pi y) ~ \text{ in } \Omega,
  \qquad 
  g(\bx) = \phi(\bx) ~ \text{ on } \partial\Omega
\end{equation}
where $\phi(\bx) = \sin(\pi x)\cos(\pi y)$ is the exact solution.

In Figure~\ref{fig:poisson_cc} we plot the convergence curves corresponding to various
$N \times N$ grids for the 18 patches of the domain, and various degrees 
$p = 2, \dots 5$. They show that the solutions converge with optimal rate $p+1$
(and even $p+2$ for $p=2$) as the patch grids are refined, which confirms the numerical
accuracy of our stabilized CONGA formulation for the Poisson problem with smooth solutions.

\begin{figure}[!htbp]
\def \plotdir {poisson_inhom/pretzel_f_manu_poisson_sincos_nc=16_deg=3}
\def \plotfn {eta=0_mu=1_gamma_h=10.0_phi_h}
\def \plotdircc {convergence_curves}
\def \plotfncc {Poiss_inhom}
\begin{center}
  \subfloat[convergence curves]{%
    \includegraphics[height=0.35\textwidth]{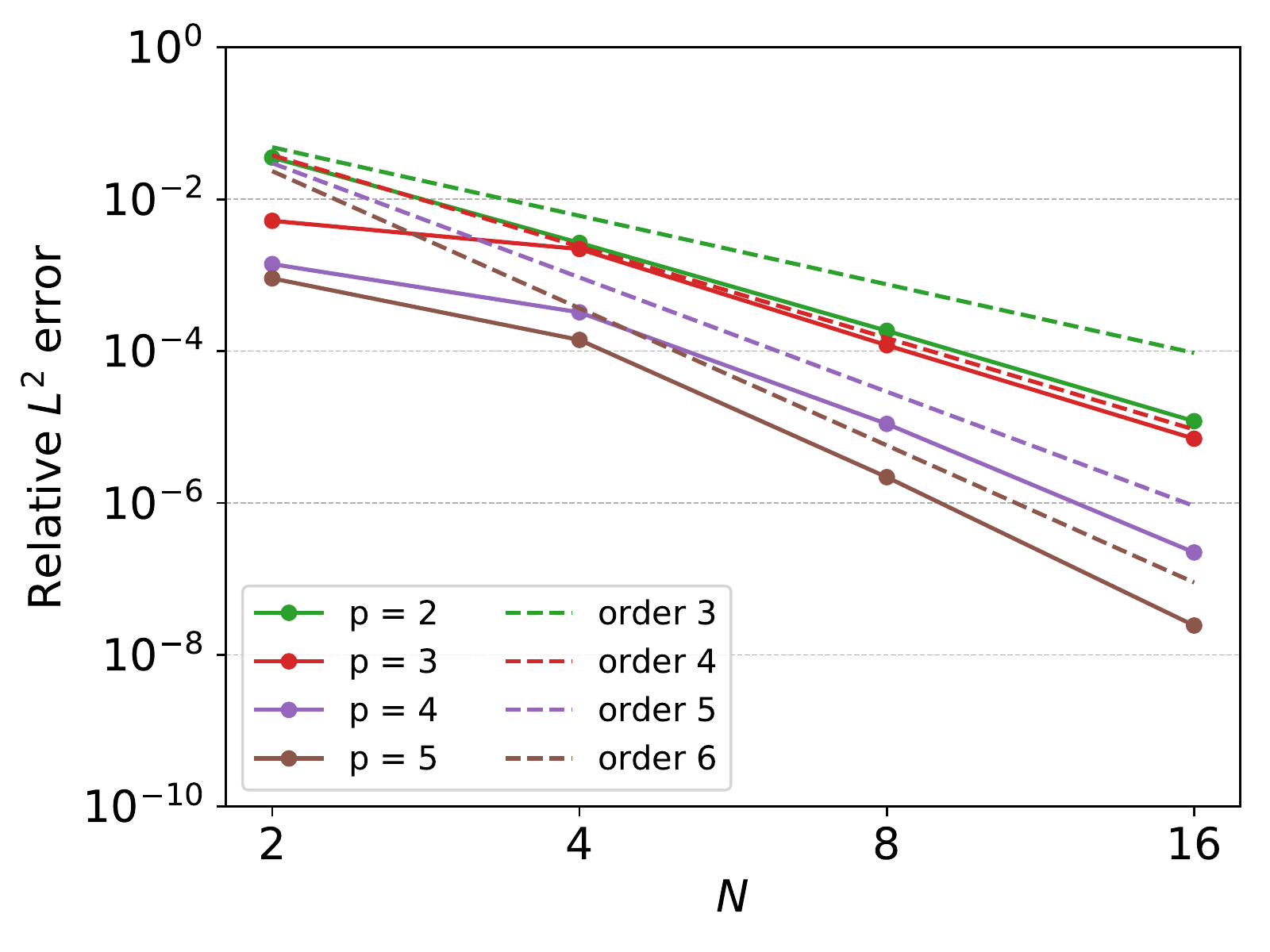}
  }
  \hspace{5pt}
  \subfloat[solution $\phi_h$ for $N=16$, $p=3$]{%
    \includegraphics[height=0.35\textwidth]{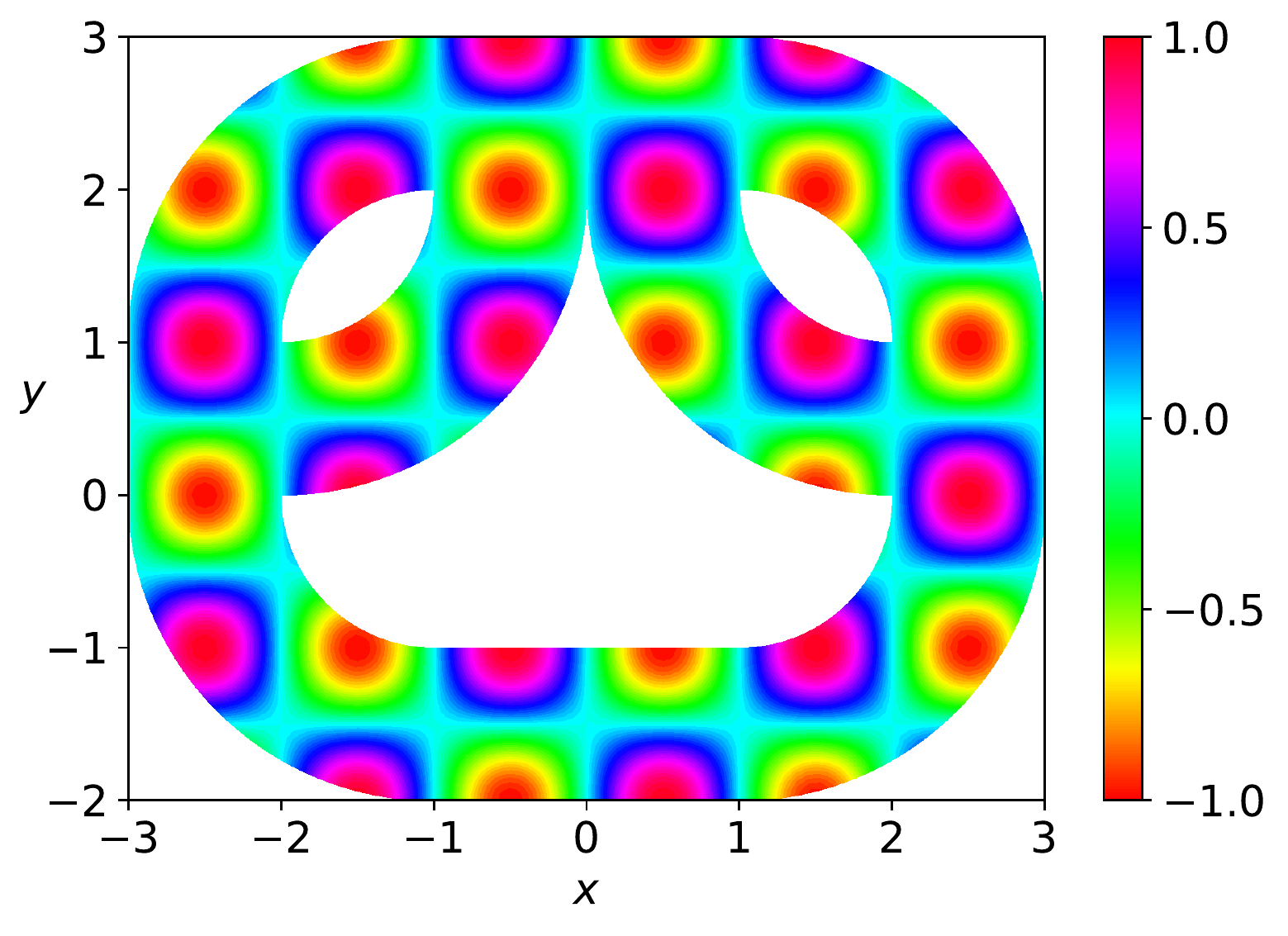}
  }
\end{center}
\caption{
Convergence study for the inhomogeneous Poisson solver with source and boundary conditions
given by \eqref{fg_sincos}.
The left plot shows the relative $L^2$ errors corresponding to various grids of 
$N \times N$ cells per patch (i.e., $18 N^2$ cells in total) and spline degrees $p \times p$ 
as indicated. The right plot shows one numerical solution $\phi_h$ of good accuracy.
}
\label{fig:poisson_cc}
\end{figure}

%
%
%

\subsection{Harmonic Maxwell problem with homogeneous boundary conditions}

We next turn to a numerical assessment of our CONGA solvers for the time-harmonic Maxwell equation,
and as we did for the Poisson equation we begin with the homogeneous case presented 
in Section \ref{sec:maxwell}.
Since now the solution depends a priori on the time pulsation $\omega$ we opt for 
a physically relevant current source localized around the upper right hole of the 
pretzel-shaped domain,
\begin{equation} \label{source_max_elliptic}
  \bJ =  \phi \bcurl \tau    
\end{equation}
where $\phi$ and $\tau$ are as in \eqref{phi_manu_poisson_elliptic}. We plot this source
in Figure~\ref{fig:maxwell_source}.

\begin{figure}[!htbp]
\def \plotdir {maxwell_hom_eta=50/pretzel_f_elliptic_J_nc=20_deg=6}
\def \plotfnvf {fh_tilde_Pi_vf}
\def \plotfnamp {fh_tilde_Pi}
\begin{center}
  \subfloat[vector field $\bJ$]{%
    \includegraphics[height=0.35\textwidth]{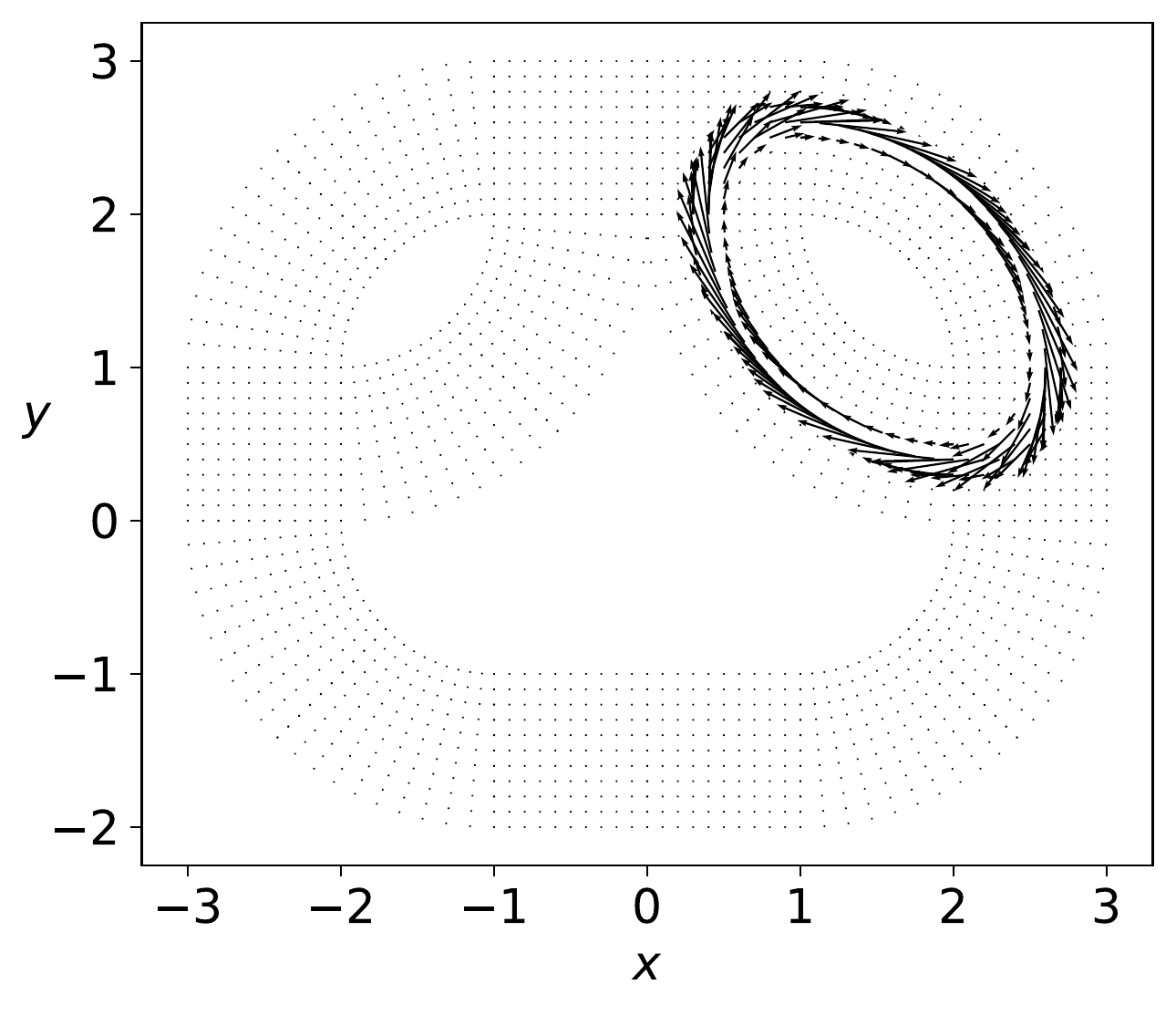}
  }
  \hspace{5pt}
  \subfloat[amplitude $\abs{\bJ}$]{%
    \includegraphics[height=0.35\textwidth]{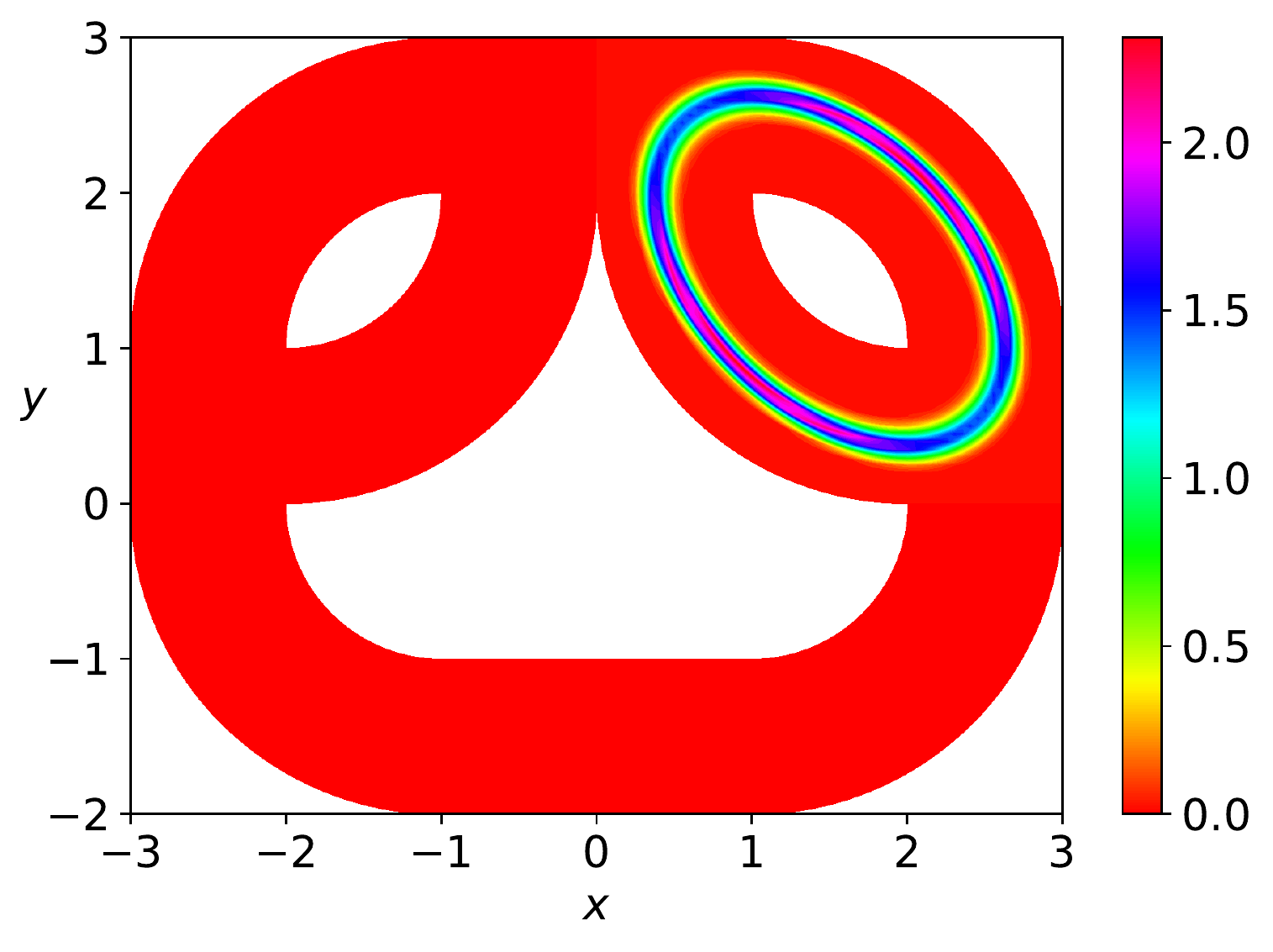}
  }
\end{center}
\caption{
Source \eqref{source_max_elliptic} (vector field and amplitude) for the 
homogeneous Maxwell solutions plotted in Figure~\ref{fig:max_ref_sols} below.
}
\label{fig:maxwell_source}
\end{figure}

\begin{figure}[!htbp]
\begin{center}
  \def \plotdir {maxwell_hom_eta=50/pretzel_f_elliptic_J_nc=20_deg=6}
  \def \plotfnvf {eta=-50.0_mu=1_nu=0_gamma_h=10_Pf=tilde_Pi_uh_vf}
  \def \plotfn   {eta=-50.0_mu=1_nu=0_gamma_h=10_Pf=tilde_Pi_uh}
  \subfloat[vector field ($\omega = \sqrt{50}$)]{%
    \includegraphics[height=0.35\textwidth]{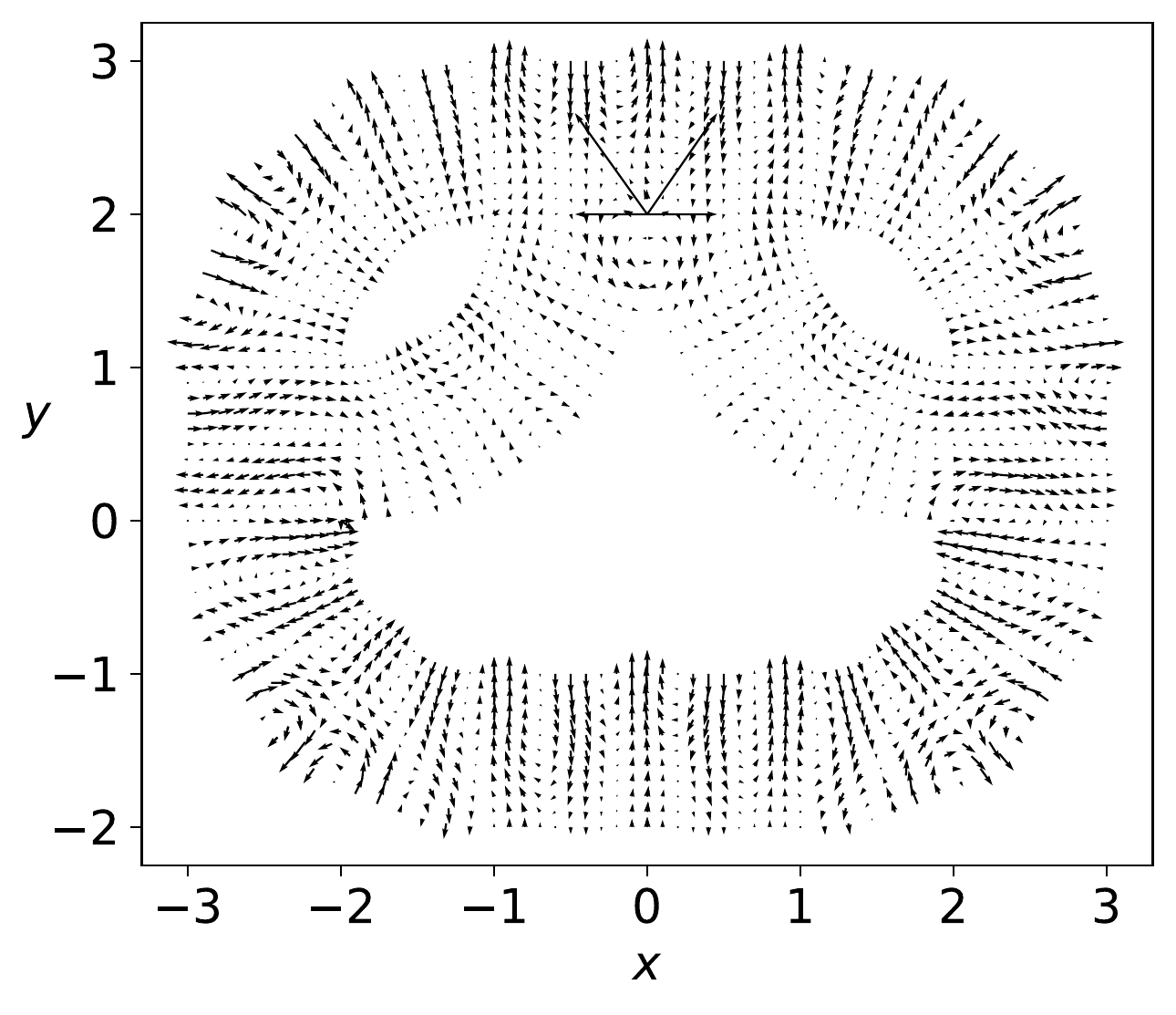}
  }
  \hspace{5pt}
  \subfloat[amplitude ($\omega = \sqrt{50}$)]{%
    \includegraphics[height=0.35\textwidth]{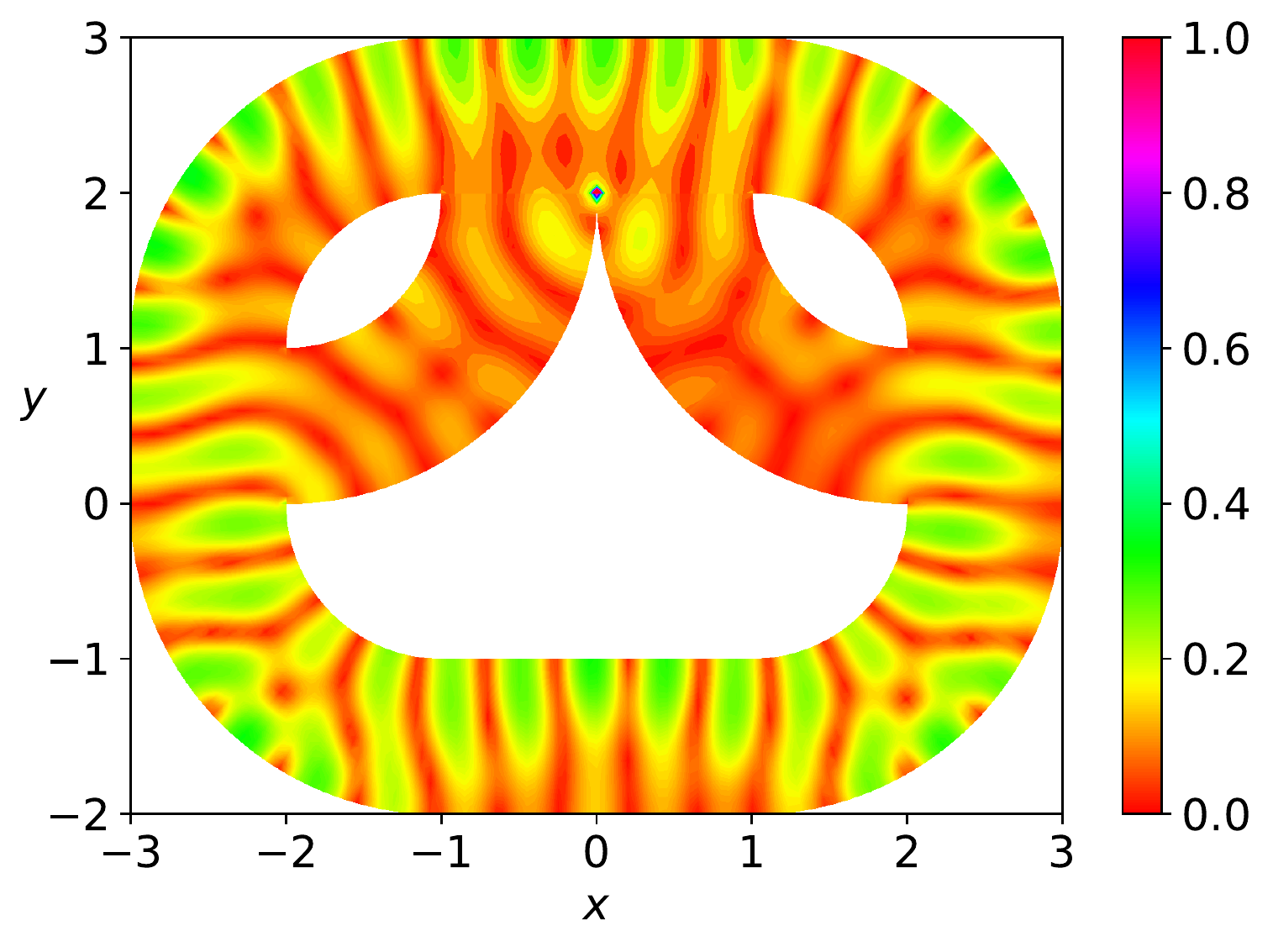}
  }
  \\
  \def \plotdir {maxwell_hom_eta=170/pretzel_f_elliptic_J_nc=20_deg=6}
  \def \plotfnvf {eta=-170.0_mu=1_nu=0_gamma_h=10_Pf=tilde_Pi_uh_vf}
  \def \plotfn   {eta=-170.0_mu=1_nu=0_gamma_h=10_Pf=tilde_Pi_uh}
  \subfloat[vector field ($\omega = \sqrt{170}$)]{%
    \includegraphics[height=0.35\textwidth]{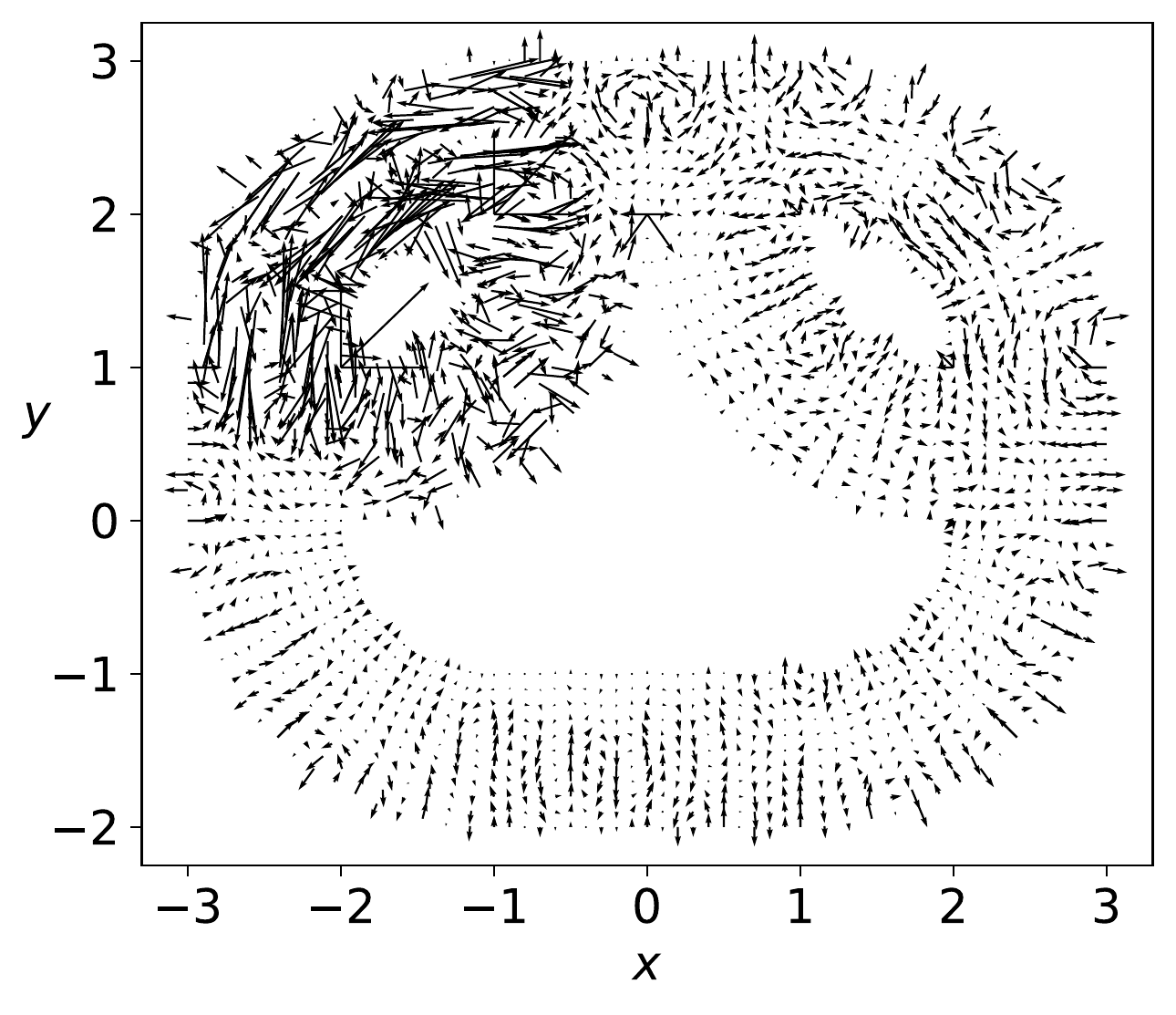}
  }
  \hspace{5pt}
  \subfloat[amplitude ($\omega = \sqrt{170}$)]{%
    \includegraphics[height=0.35\textwidth]{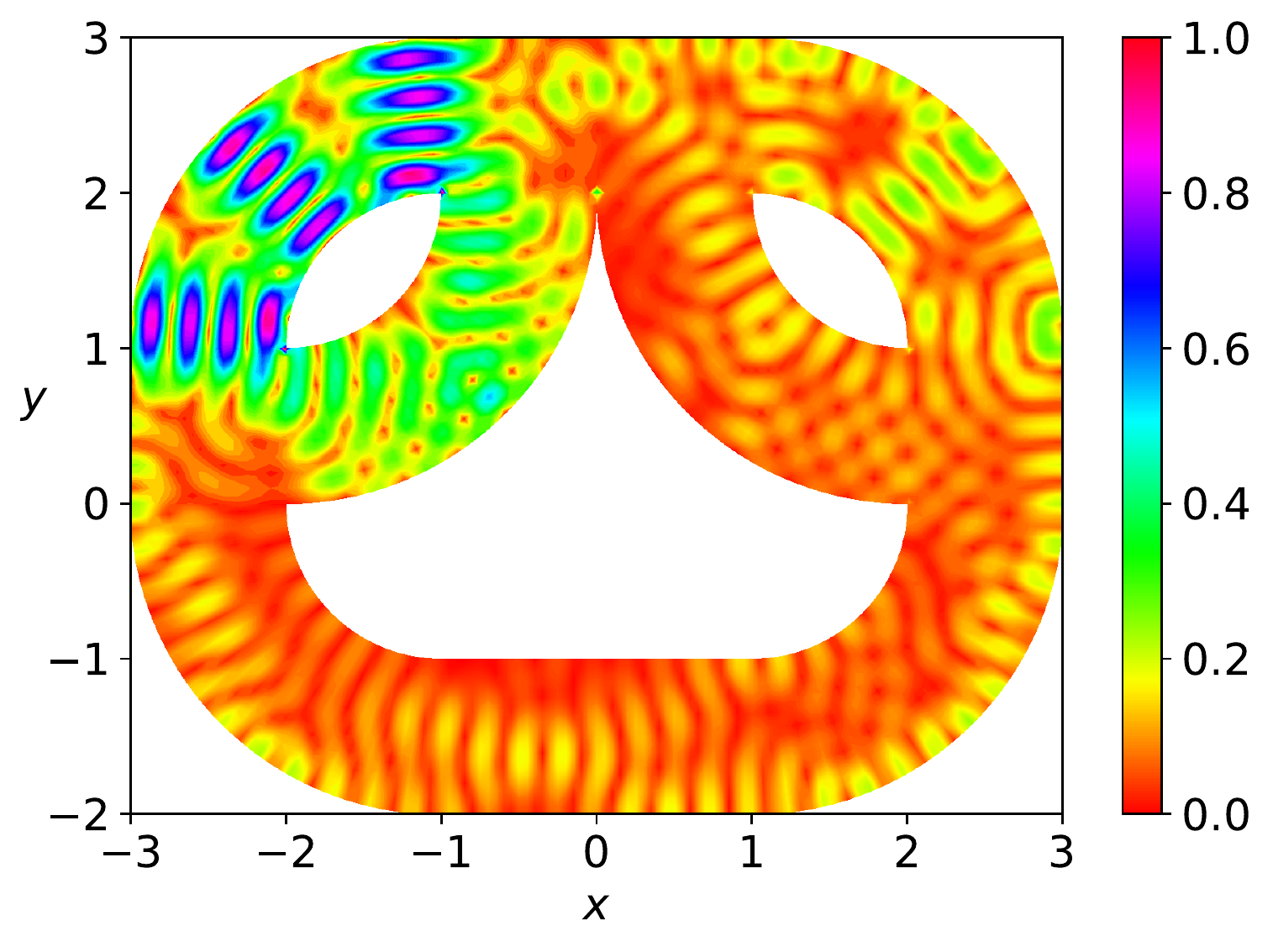}
  }
\end{center}
\caption{
Reference numerical solutions for the homogeneous Maxwell problem 
with source \eqref{source_max_elliptic} and
time pulsation $\omega = \sqrt{50}$ (top) and $\omega = \sqrt{170}$ (bottom).
These solutions have been obtained using $20 \times 20$ cells per patch and spline elements 
of degree $6 \times 6$. The vector fields are shown on the left while the amplitudes are shown on the right.
}
\label{fig:max_ref_sols}
\end{figure}


\begin{table}[!htbp]
\centering
\begin{tabular}{c}
\hline
$\phantom{\Big(} \mspace{160mu}$ 
Errors for $\omega = \sqrt{50}$ 
$\phantom{\Big(} \mspace{160mu}$  \\
\hline   
\end{tabular}
\begin{tabular}{|l|l|l|l|l|}
\hline
\diagbox[width=10em]{cells p.p.}{degree} & \multicolumn{1}{p{12mm}|}{$2\times 2$} & \multicolumn{1}{p{12mm}|}{$3 \times 3$} & \multicolumn{1}{p{12mm}|}{$4 \times 4$} & \multicolumn{1}{p{12mm}|}{$5 \times 5$}
\\ 
\hline   
 $4  \times 4$   & 1.32277 & 1.53875 & 0.42168 & 0.02399 \\ \hline
 $8  \times 8$   & 0.45461 & 0.03990 & 0.03640 & 0.02547 \\ \hline
 $16 \times 16$  & 0.02158 & 0.02354 & 0.01333 & 0.00738 \\ \hline
\end{tabular}
\begin{tabular}{c}
\hline
$\phantom{\Big(} \mspace{160mu}$ 
Errors for $\omega = \sqrt{170}$ 
$\phantom{\Big(} \mspace{160mu}$  \\
\hline   
\end{tabular}
\begin{tabular}{|l|l|l|l|l|}
\hline
\diagbox[width=10em]{cells p.p.}{degree} & \multicolumn{1}{p{12mm}|}{$2\times 2$} & \multicolumn{1}{p{12mm}|}{$3 \times 3$} & \multicolumn{1}{p{12mm}|}{$4 \times 4$} & \multicolumn{1}{p{12mm}|}{$5 \times 5$}
\\ 
\hline   
 $4  \times 4$   & 0.99472 & 1.00119 & 1.00864 & 1.00200 \\ \hline
 $8  \times 8$   & 1.03397 & 1.39890 & 0.29048 & 0.11714 \\ \hline
 $16 \times 16$  & 0.68206 & 0.00828 & 0.00880 & 0.00484 \\ \hline
\end{tabular}
\caption{$L^2$ errors for the solution of the homogeneous Maxwell problem
with source \eqref{source_max_elliptic} and time pulsation $\omega$ as indicated.
Here the errors are computed using the numerical reference solutions
shown in Figure~\ref{fig:max_ref_sols}.
}
\label{tab:max_L2_error}
\end{table}

\begin{figure}[!htbp]
\begin{center}
  \def \plotdircoarse {maxwell_hom_eta=50/pretzel_f_elliptic_J_nc=4_deg=3}
  \def \plotdirmed {maxwell_hom_eta=50/pretzel_f_elliptic_J_nc=8_deg=3}
  \def \plotdirfine {maxwell_hom_eta=50/pretzel_f_elliptic_J_nc=16_deg=3}
  \def \plotfn {eta=-50.0_mu=1_nu=0_gamma_h=10_Pf=tilde_Pi_uh}
  \subfloat[$4 \times 4$ cells p.p.]{%
    \includegraphics[width=0.3\textwidth]{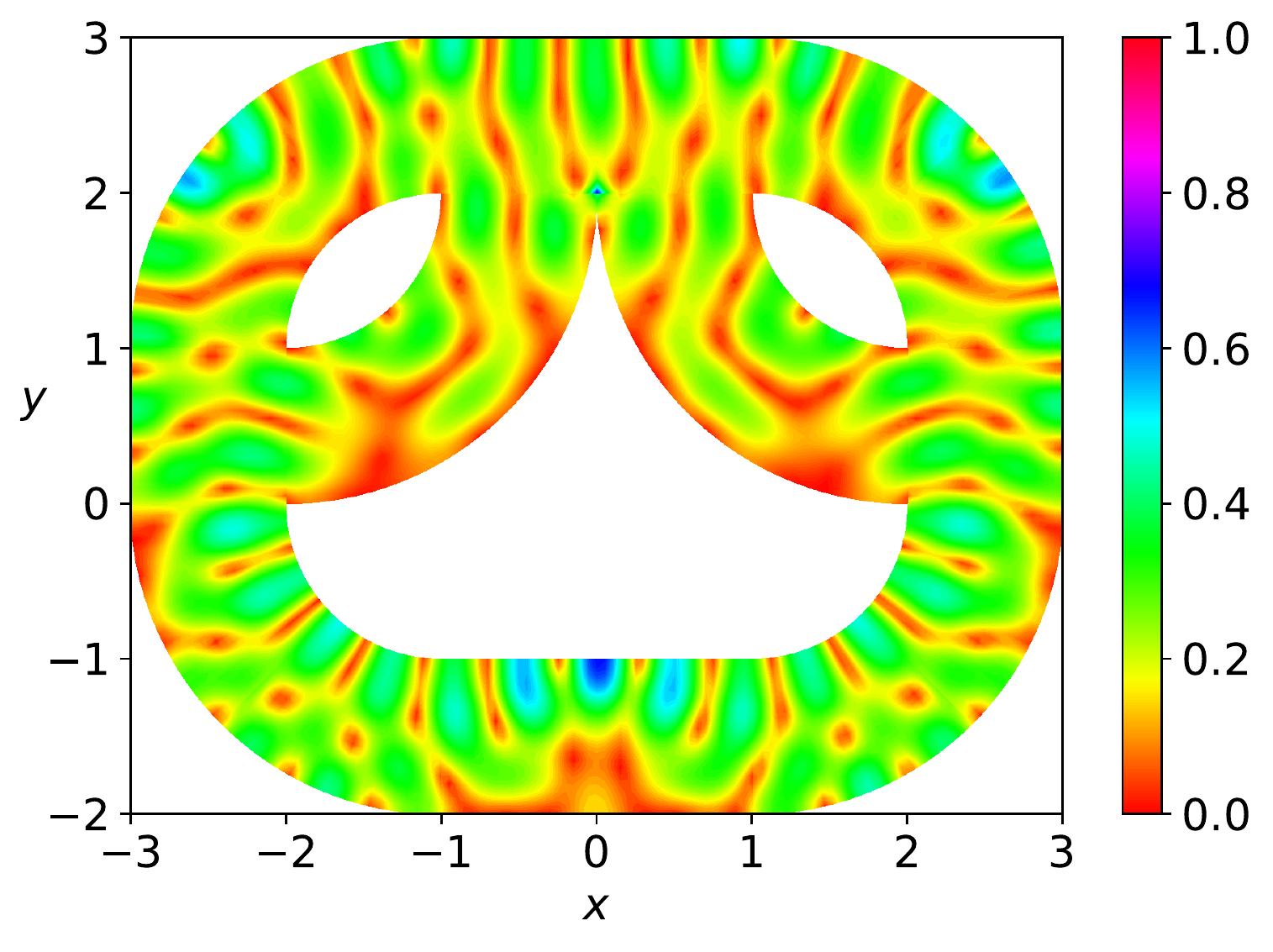}
  }
  \hspace{10pt}
  \subfloat[$8 \times 8$ cells p.p.]{%
    \includegraphics[width=0.3\textwidth]{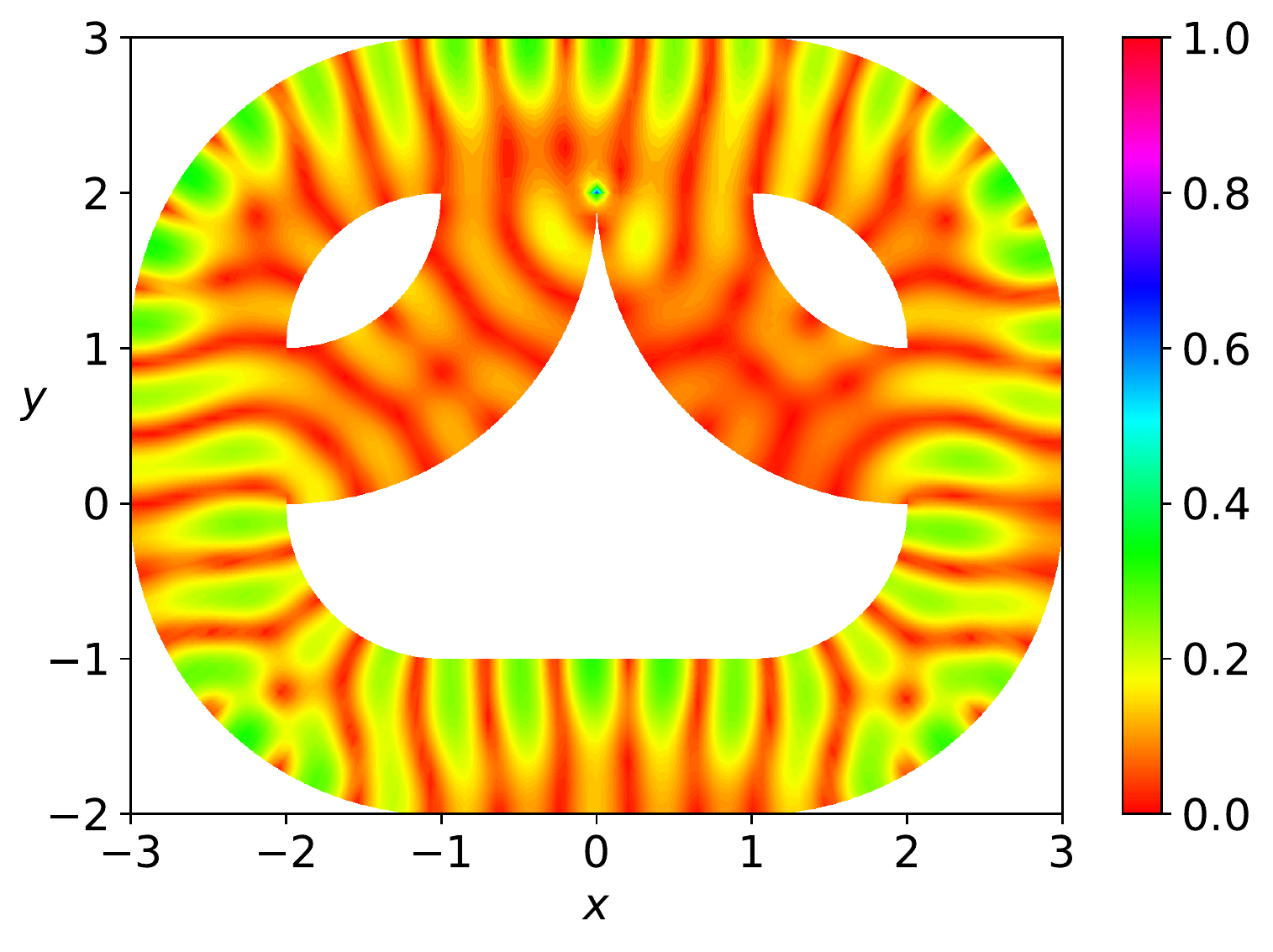}
  }
  \hspace{10pt}
  \subfloat[$16 \times 16$ cells p.p.]{
    \includegraphics[width=0.3\textwidth]{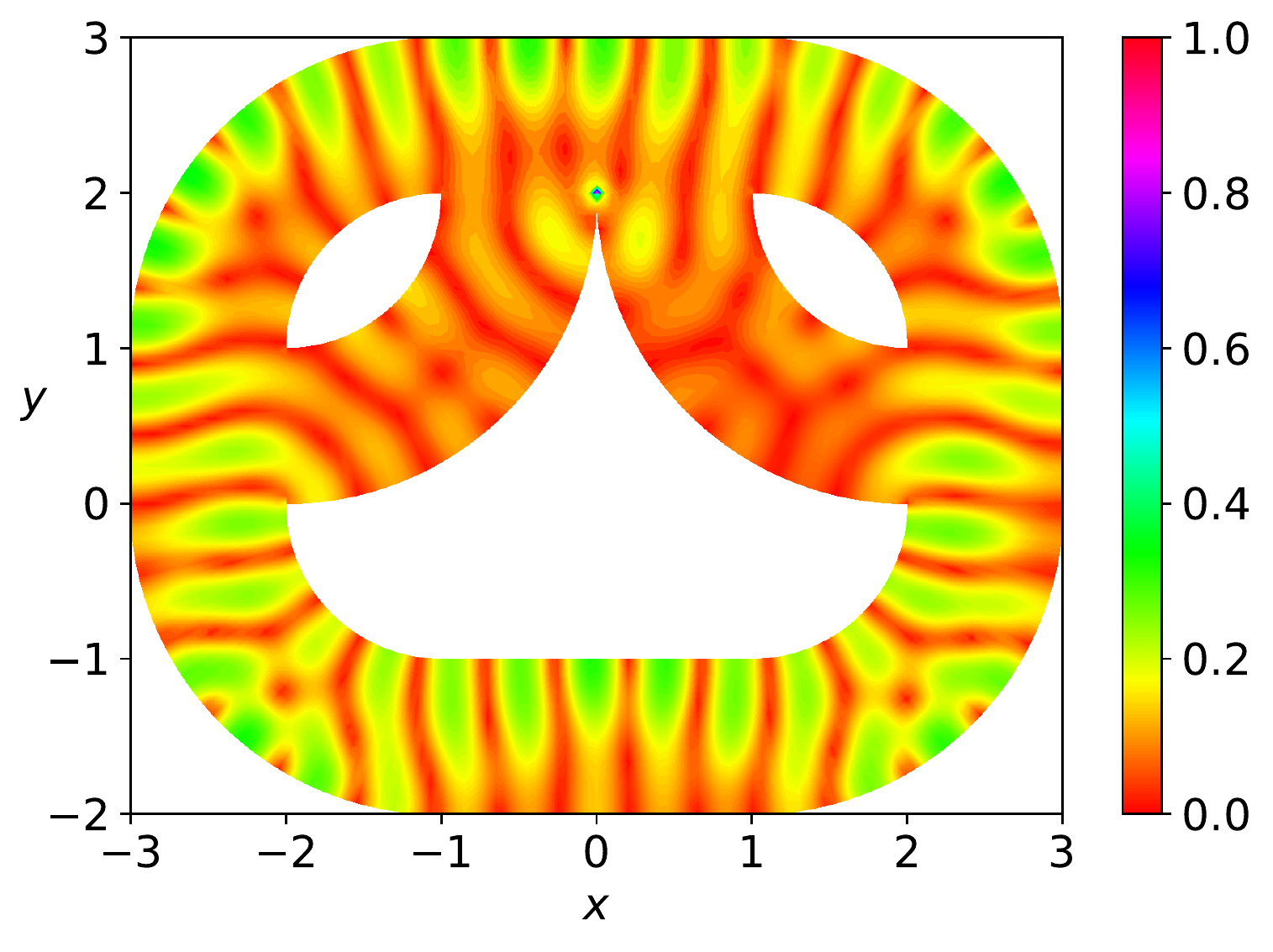}
  }
  \\
  \def \plotdircoarse {maxwell_hom_eta=170/pretzel_f_elliptic_J_nc=4_deg=3}
  \def \plotdirmed {maxwell_hom_eta=170/pretzel_f_elliptic_J_nc=8_deg=3}
  \def \plotdirfine {maxwell_hom_eta=170/pretzel_f_elliptic_J_nc=16_deg=3}
  \def \plotfn {eta=-170.0_mu=1_nu=0_gamma_h=10_Pf=tilde_Pi_uh}
  \subfloat[$4 \times 4$ cells p.p.]{%
    \includegraphics[width=0.3\textwidth]{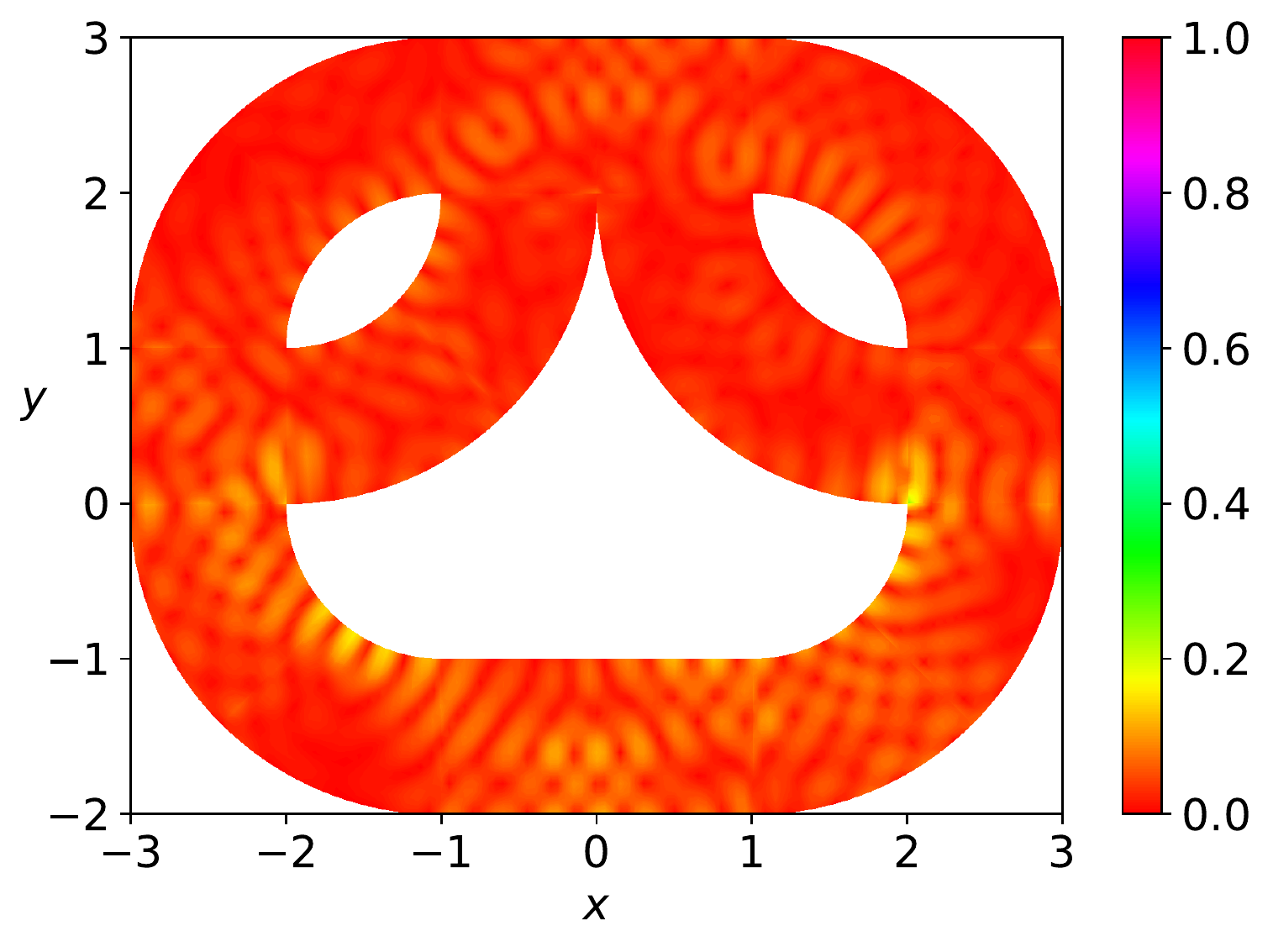}
  }
  \hspace{10pt}
  \subfloat[$8 \times 8$ cells p.p.]{%
    \includegraphics[width=0.3\textwidth]{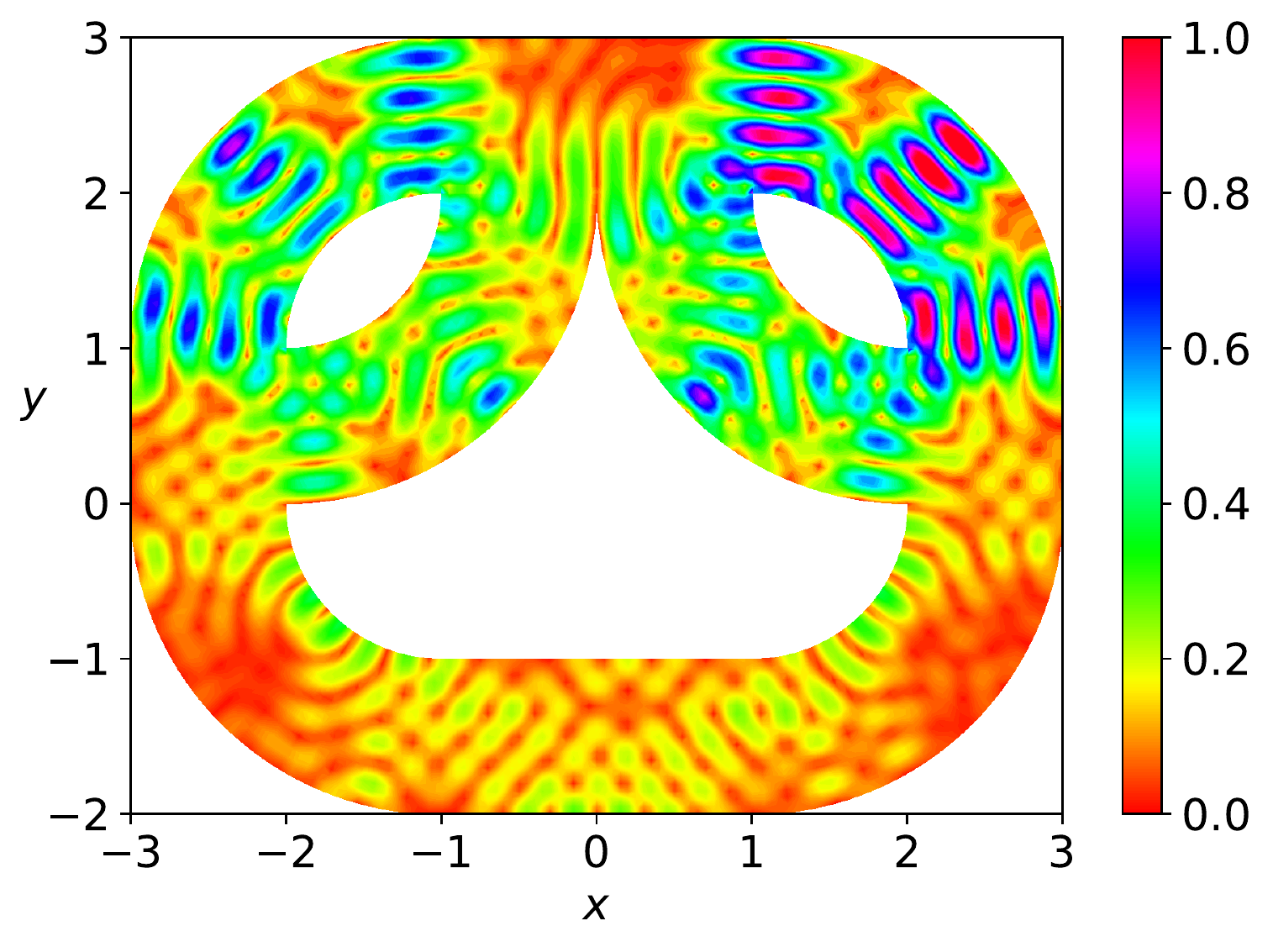}
  }
  \hspace{10pt}
  \subfloat[$16 \times 16$ cells p.p.]{
    \includegraphics[width=0.3\textwidth]{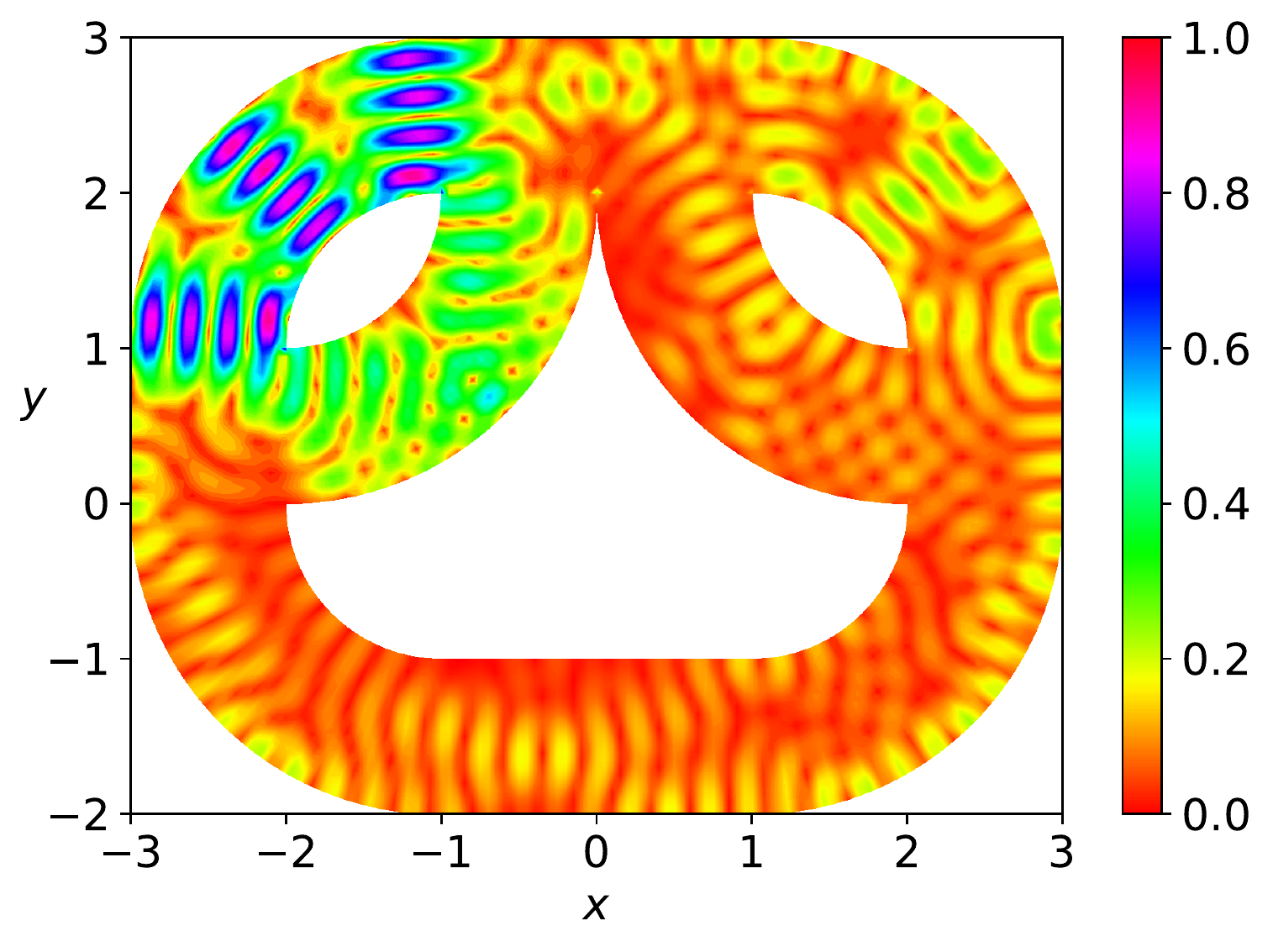}
  }
\end{center}
\caption{
Numerical solutions obtained with spline elements of degree $3 \times 3$ and
different grids as indicated, corresponding to the errors shown in 
Table~\ref{tab:max_L2_error}.
}
\label{fig:maxwell_hom}
\end{figure}

We then consider two values for the time pulsation, namely 
$\omega=\sqrt{50}$ and $\omega=\sqrt{170}$, and for each of these values
we use as reference solutions the numerical solutions computed using our 
method on a mesh with $20 \times 20$ cells per patch (i.e. 7200 cells in total)
and elements of degree $6 \times 6$ in each patch. These reference solutions
are shown in Figure~\ref{fig:max_ref_sols}.
Interestingly, we observe that for the higher pulsation $\omega=\sqrt{170}$
the source triggers a time-harmonic field localized around the upper left hole,
opposite to where the source is.

In Table~\ref{tab:max_L2_error} we show the $L^2$ errors corresponding to different grids 
and spline degrees for the two values of the pulsation $\omega$, and we also plot in 
Figure~\ref{fig:maxwell_hom} the different solutions corresponding 
to spline elements of degree $3 \times 3$.
These results show that the numerical solutions converge towards the reference ones
as the grids are refined, and with a faster convergence in the case of the 
lower pulsation $\omega = \sqrt{50}$, due to the higher smoothness of the 
corresponding solution.

\subsection{Time-harmonic Maxwell problem with inhomogeneous boundary conditions}
\label{sec:num_maxwell_inhom}

As we did for the Poisson problem, we next test our Maxwell solver with 
a smooth solution and handle the inhomogeneous boundary conditions 
with the lifting method described in Section~\ref{sec:lifting}.
Specifically, we define $\bu_{g,h} \in V^1_h$ by computing its 
boundary degrees of freedom from the data $g = \bn \times \bu$ on $\partial \Omega$
(this is again straightforward with our geometric boundary degrees of freedom
\eqref{s1_gc}), 
and we compute $\bu_{0,h} = \bu_{h}-\bu_{g,h}$ by solving \eqref{maxwell_hbc}.
Here we take $\omega = \pi$ and consider \eqref{max_fg}
with the source-solution pair
\begin{equation} \label{J_sincos}
  \bJ = \begin{pmatrix}
  - \pi^2 \sin(\pi y)\cos(\pi x) \\ 0
  \end{pmatrix},
  \qquad 
  \bu = \begin{pmatrix} \sin(\pi y) \\ \sin(\pi x)\cos(\pi y) \end{pmatrix}
\end{equation}
and a boundary condition given by $g := \bn \times \bu$ on $\partial\Omega$.

In Figure~\ref{fig:maxwell_cc} we plot the convergence curves corresponding to 
spline elements of various degrees $p\times p$ and $N \times N$ cells per patch. 
The results are similar to what was observed for the Poisson problem: the solutions 
converge with optimal rate $p+1$ as the grids are refined
(again a rate of $p+2$ is observed for $=2$) which confirms the numerical accuracy of our 
approach for the Maxwell problem combined with a geometric lifting technique for the Dirichlet boundary condition.

\begin{figure}[!htbp]
\def \plotdir {maxwell_inhom/pretzel_f_manu_maxwell_inhom_nc=16_deg=3}
\def \plotfn {eta=-9.8696_mu=1_nu=0_gamma_h=10_Pf=tilde_Pi_uh}
\def \plotdircc {convergence_curves}
\def \plotfncc {Max_inhom}
\begin{center}
  \subfloat[convergence curves]{%
    \includegraphics[height=0.35\textwidth]{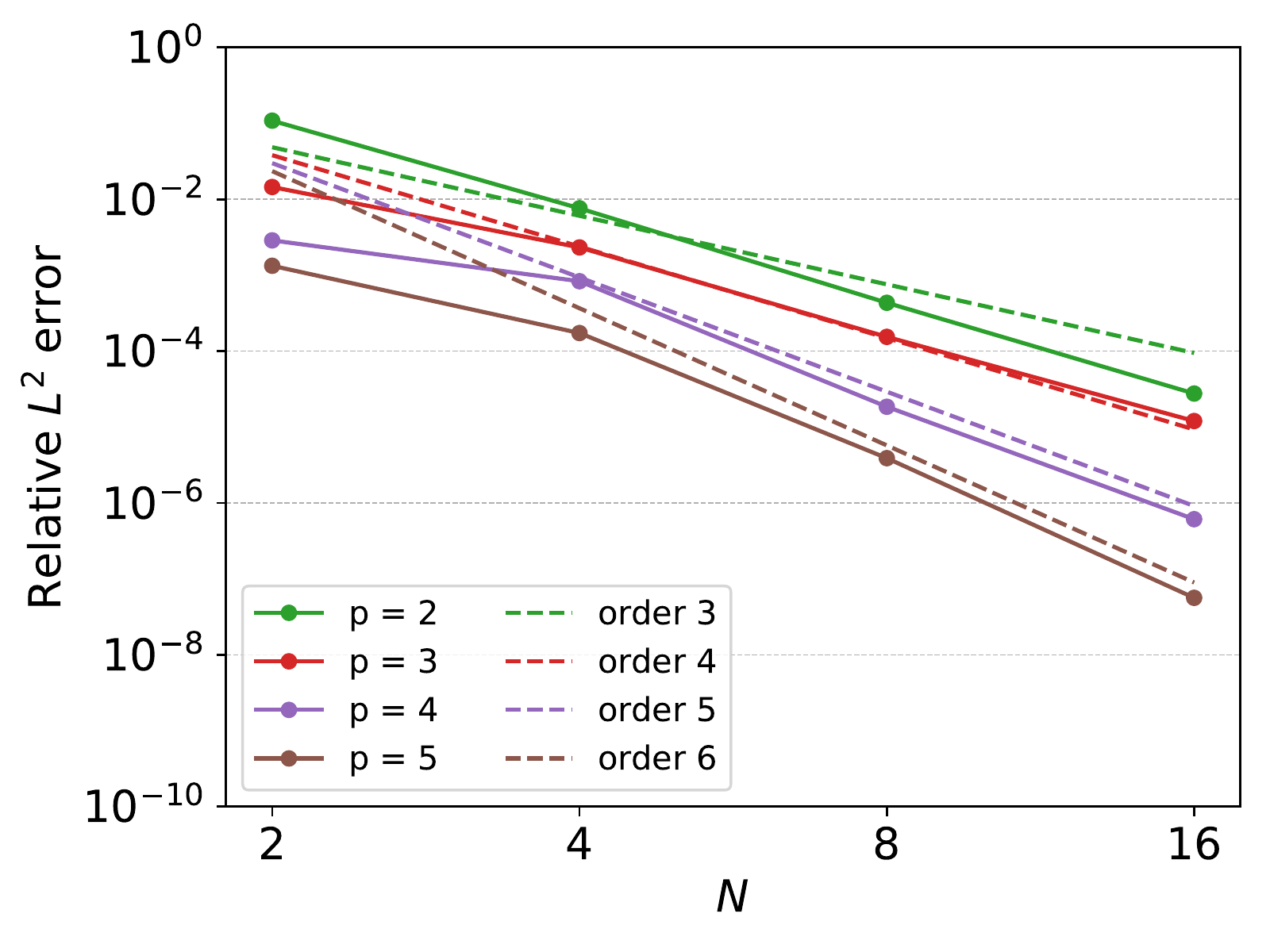}
  }
  \hspace{5pt}
  \subfloat[solution $\abs{\bu_h}$ for $N=16$, $p=3$]{%
    \includegraphics[height=0.35\textwidth]{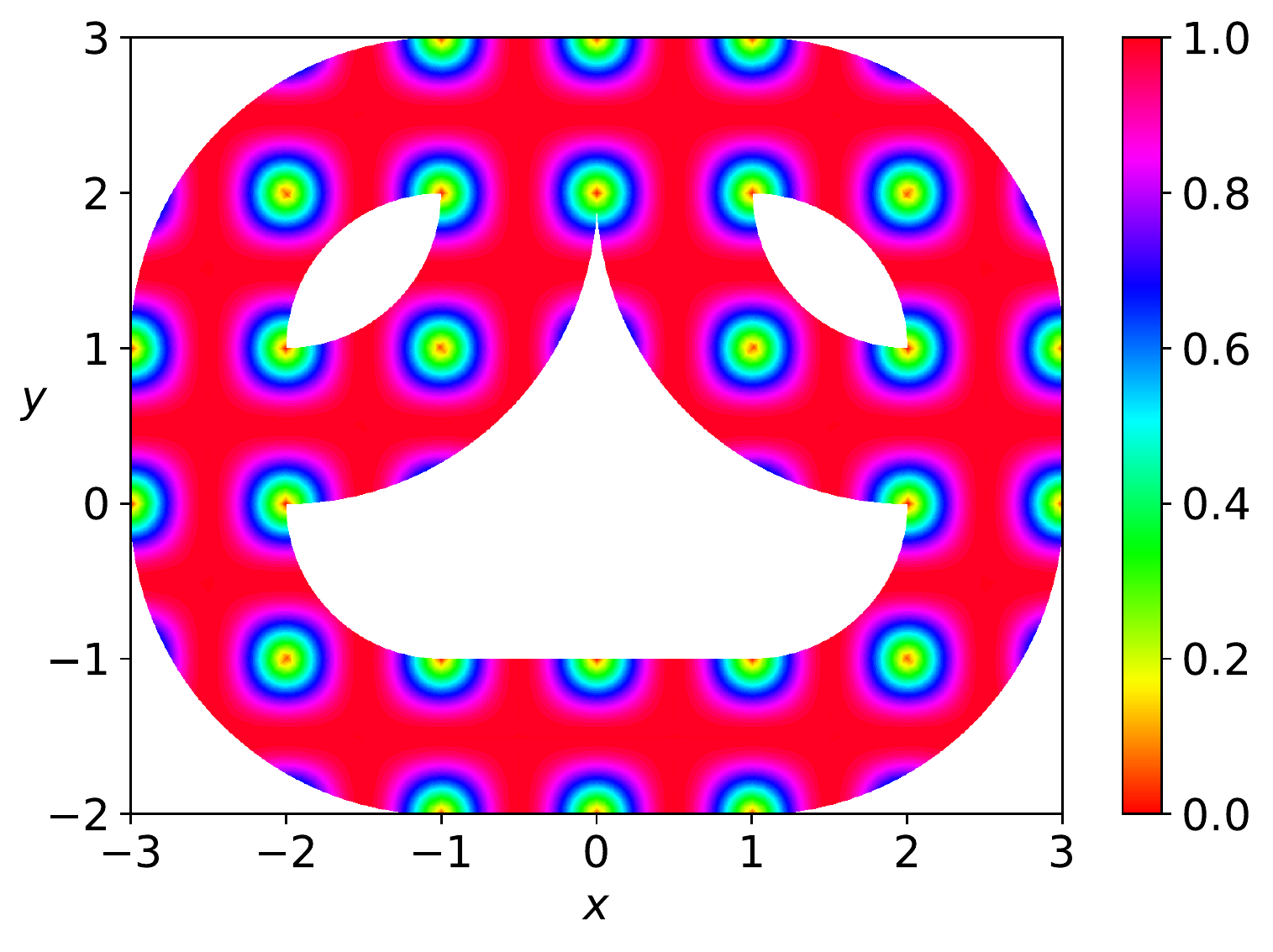}
  }
\end{center}
\caption{
Convergence study for the time-harmonic Maxwell solver with 
source given by \eqref{J_sincos} and inhomogeneous boundary conditions 
$\bn \times \bu = g$.
Relative $L^2$ errors are plotted on the left for various grids of 
$N \times N$ cells per patch (corresponding to a total of $18 N^2$ cells) 
and spline degrees $p \times p$ as indicated. 
The right plot shows the amplitude of a numerical solution 
$\bu_h$ of good accuracy.
}
\label{fig:maxwell_cc}
\end{figure}

\subsection{Eigenvalue problems}

We next assess the accuracy of our CONGA approximation \eqref{CC_evh}
for the curl-curl eigenvalue problem.

We test our discretization on the two domains shown in Figure~\ref{fig:mp},
and plot in Figure~\ref{fig:emodes} the amplitude of the first five eigenmodes,
together with their positive eigenvalues.
Here the eigenmodes are computed using spline elements of degree $6 \times 6$ and 
$56 \times 56$, resp. $20 \times 20$ cells per patch in the case of the curved L-shaped, 
resp. pretzel-shaped domain composed of 3, resp. 18 patches.
On the former domain this corresponds to 9408 cells in total and 22692 degrees of freedom
for the broken space $V^1_h$, while on the latter domain it corresponds to 7200 cells 
and 23400 degrees of freedom for $V^1_h$ (a higher value than for the L-shaped domain
despite less cells, because of the duplication of boundary dofs at the patch interfaces).

In Figure~\ref{fig:ev_pbm_cc} we then plot the relative eigenvalue errors 
\begin{equation} \label{err_lambda}
  e_{h,i} = \frac{\abs{\lambda_i-\lambda_{h,i}}}{\max(\lambda_i,\lambda_{h,i})}  
\end{equation}
as a function of the eigenvalue index $i$, 
for degrees $p = 3$ and $5$, and $N \times N$ cells per patch 
with $N = 2, 4, 8$ and $16$.
For the curved L-shaped domain we use as reference the eigenvalues 
provided as benchmark in \cite{dauge_benchmarks,Durufle.2006.phd}, and 
for the pretzel-shaped domain we use the eigenvalues computed using our CONGA scheme,
with as many reliable digits as we could find using uniform patches with degree 
$6 \times 6$ and $N \times N$ cells per patch with $N \le 20$ (this limit being imposed by the fact that we compute the
matrix eigenmodes with Scipy's \texttt{eigsh} solver with a sparse LU decomposition).

This allows us to verify numerically that the discrete eigenvalues converge towards the exact ones,
with smaller errors corresponding to the smoother eigenmodes visible in Figure~\ref{fig:emodes}.

\begin{figure}[!htbp]
\def \plotdirL {cc_eigenpbm_L_shape/curved_L_shape_nc=56_deg=6}
\def \plotdirP {cc_eigenpbm_pretzel/pretzel_f_nc=20_deg=6}

\begin{center}
  \subfloat[$\abs{\bu_1}$, $\lambda_1 = 1.81857115231$]{%
    \includegraphics[width=0.3\textwidth]{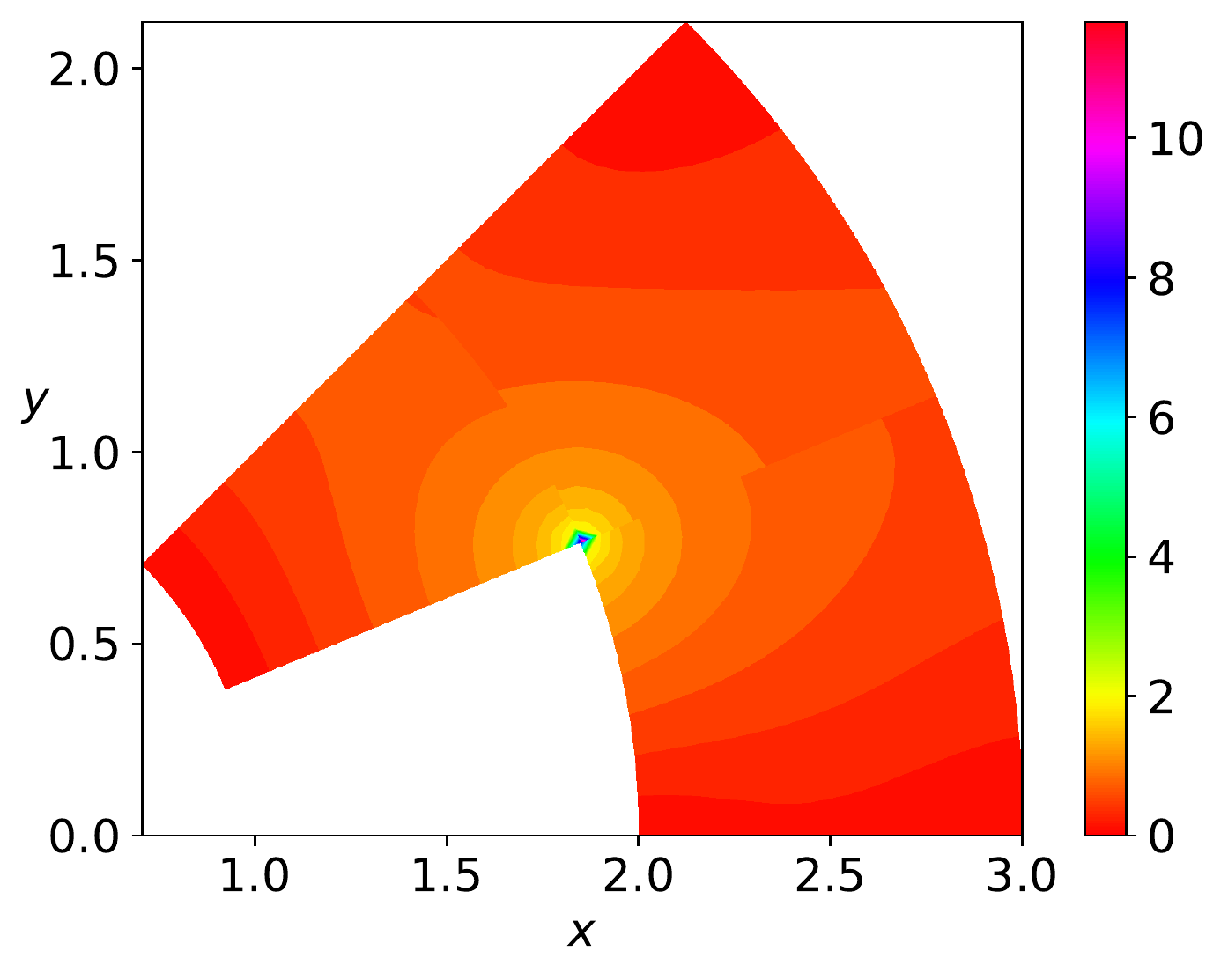}
  }
  \hspace{15pt}
  \subfloat[$\abs{\bu_1}$, $\lambda_1 = 0.1795$]{%
    \includegraphics[width=0.3\textwidth]{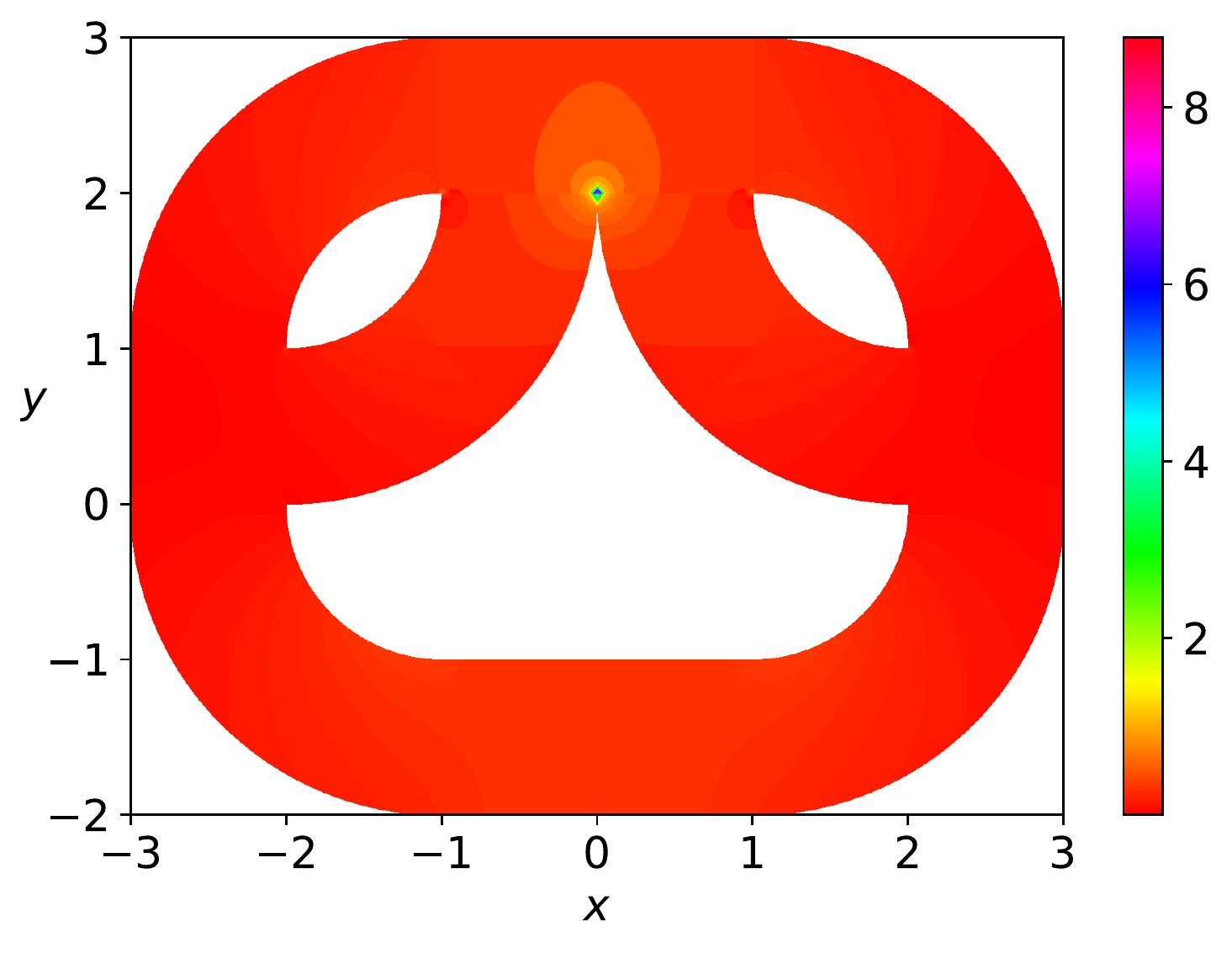}
  }
  \\
  \subfloat[$\abs{\bu_2}$, $\lambda_2 = 3.49057623279$]{%
    \includegraphics[width=0.3\textwidth]{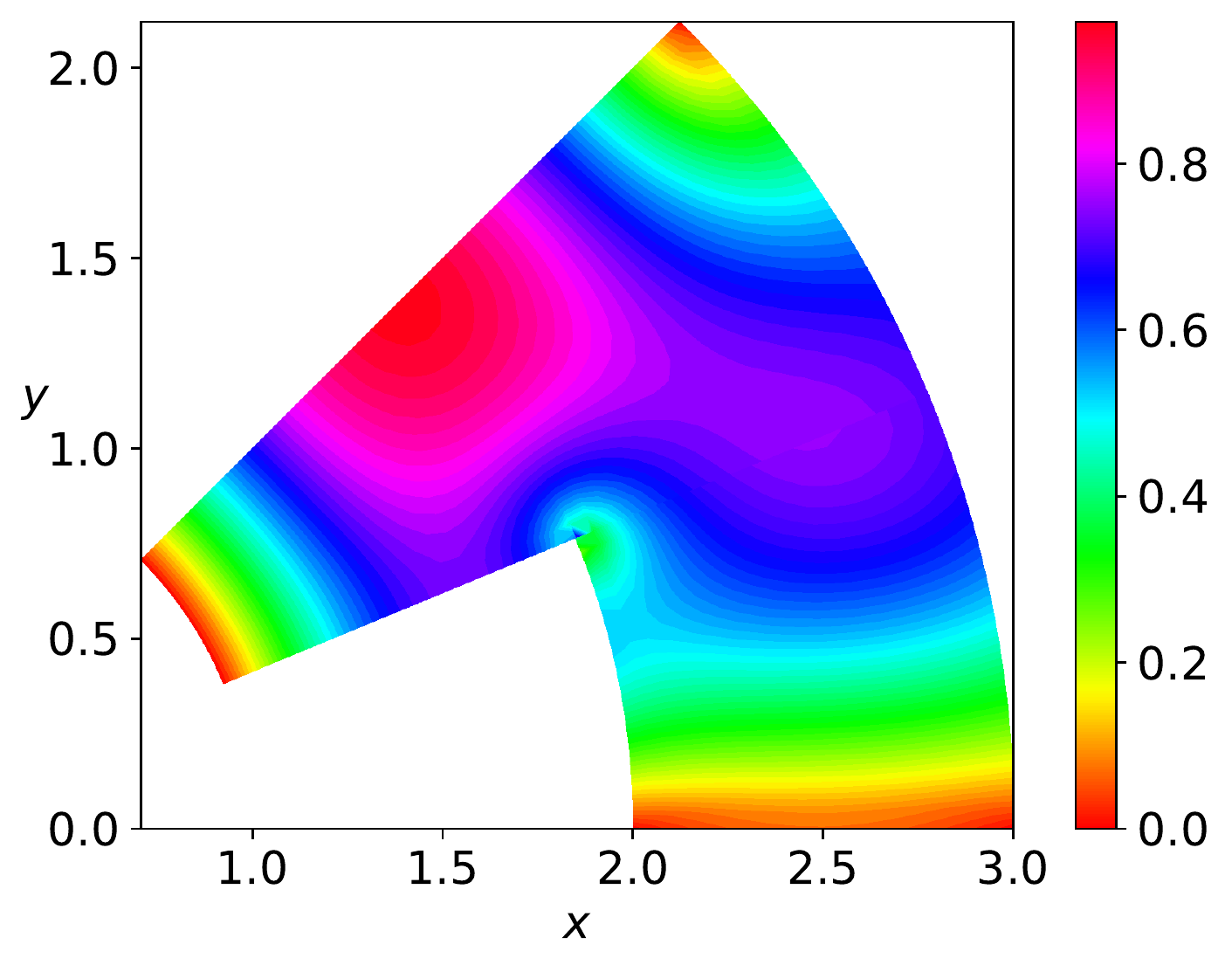}
  }
  \hspace{15pt}
  \subfloat[$\abs{\bu_2}$, $\lambda_2 = 0.19922$]{%
    \includegraphics[width=0.3\textwidth]{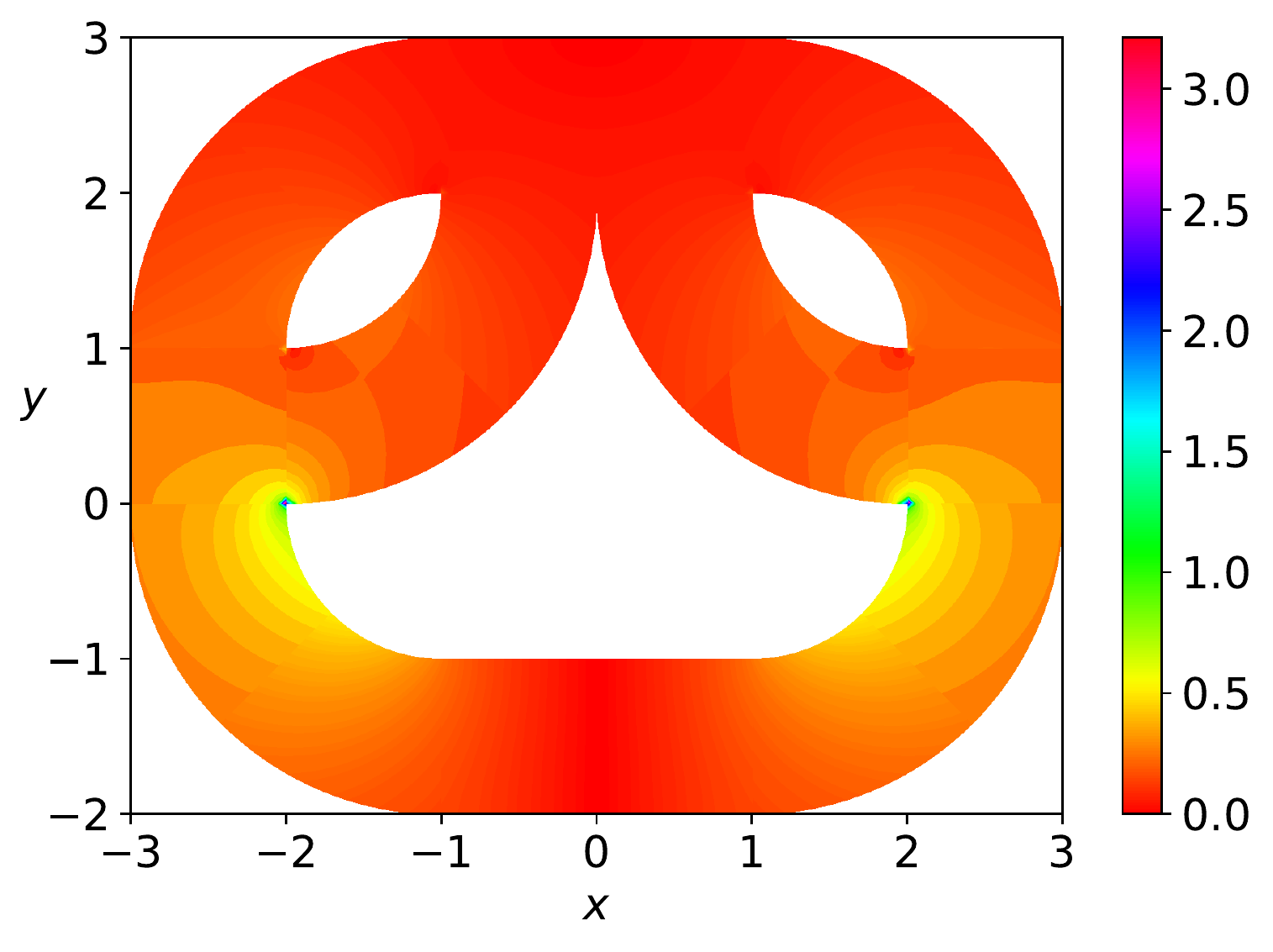}
  }
  \\
  \subfloat[$\abs{\bu_3}$, $\lambda_3 = 10.0656015004$]{%
    \includegraphics[width=0.3\textwidth]{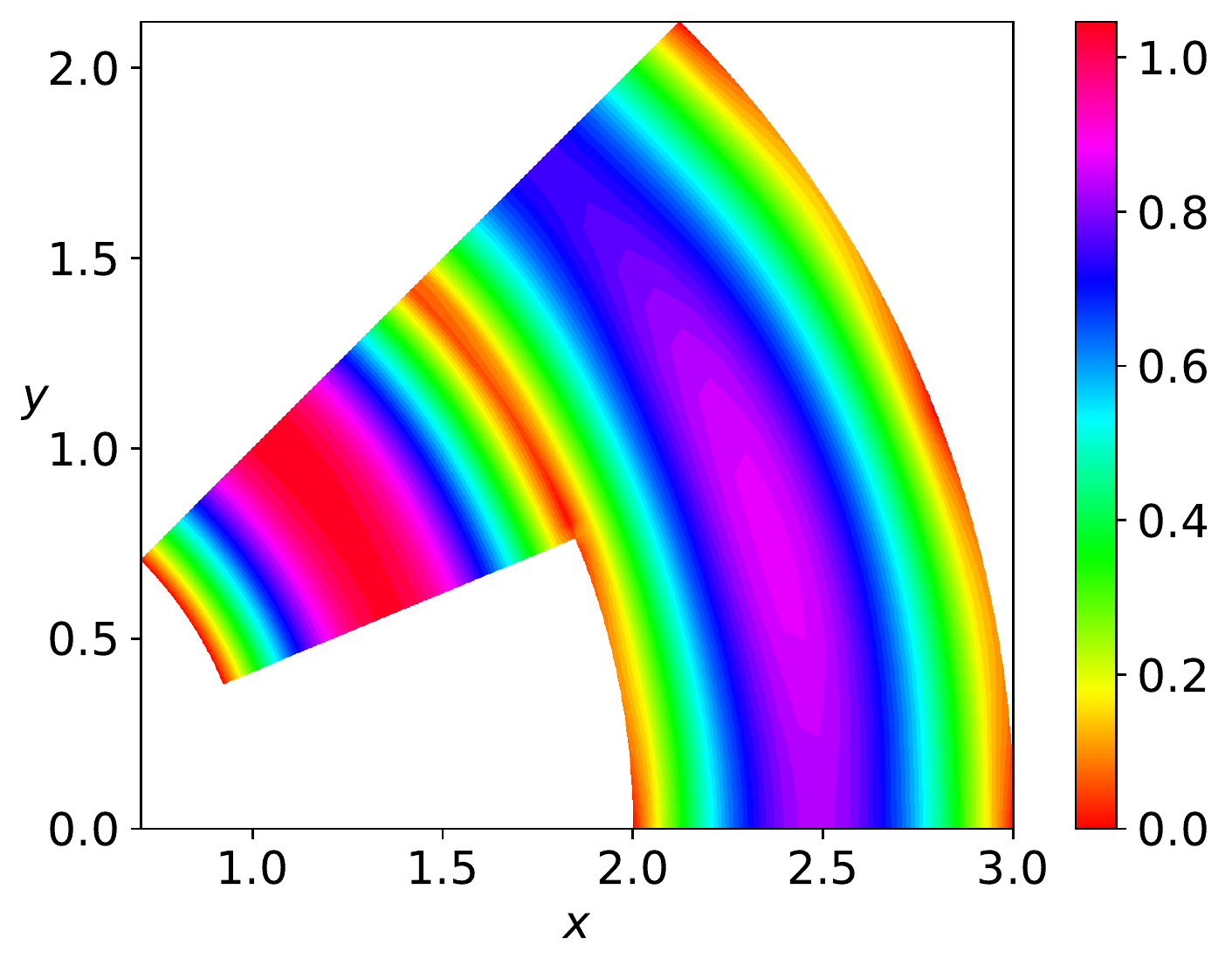}
  }
  \hspace{15pt}
  \subfloat[$\abs{\bu_3}$, $\lambda_3 = 0.6992$]{%
    \includegraphics[width=0.3\textwidth]{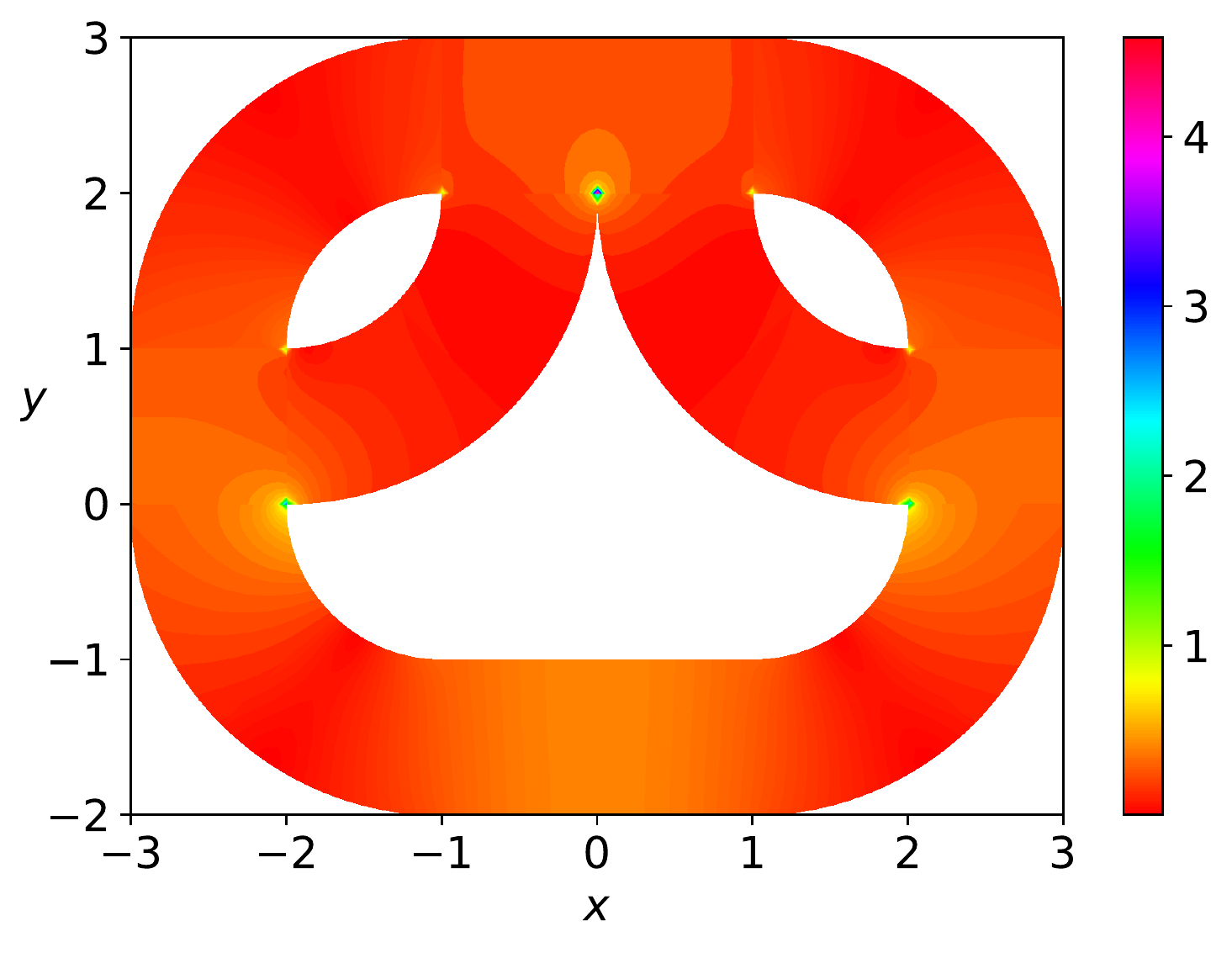}
  }
  \\
  \subfloat[$\abs{\bu_4}$, $\lambda_4 = 10.1118862307$]{%
    \includegraphics[width=0.3\textwidth]{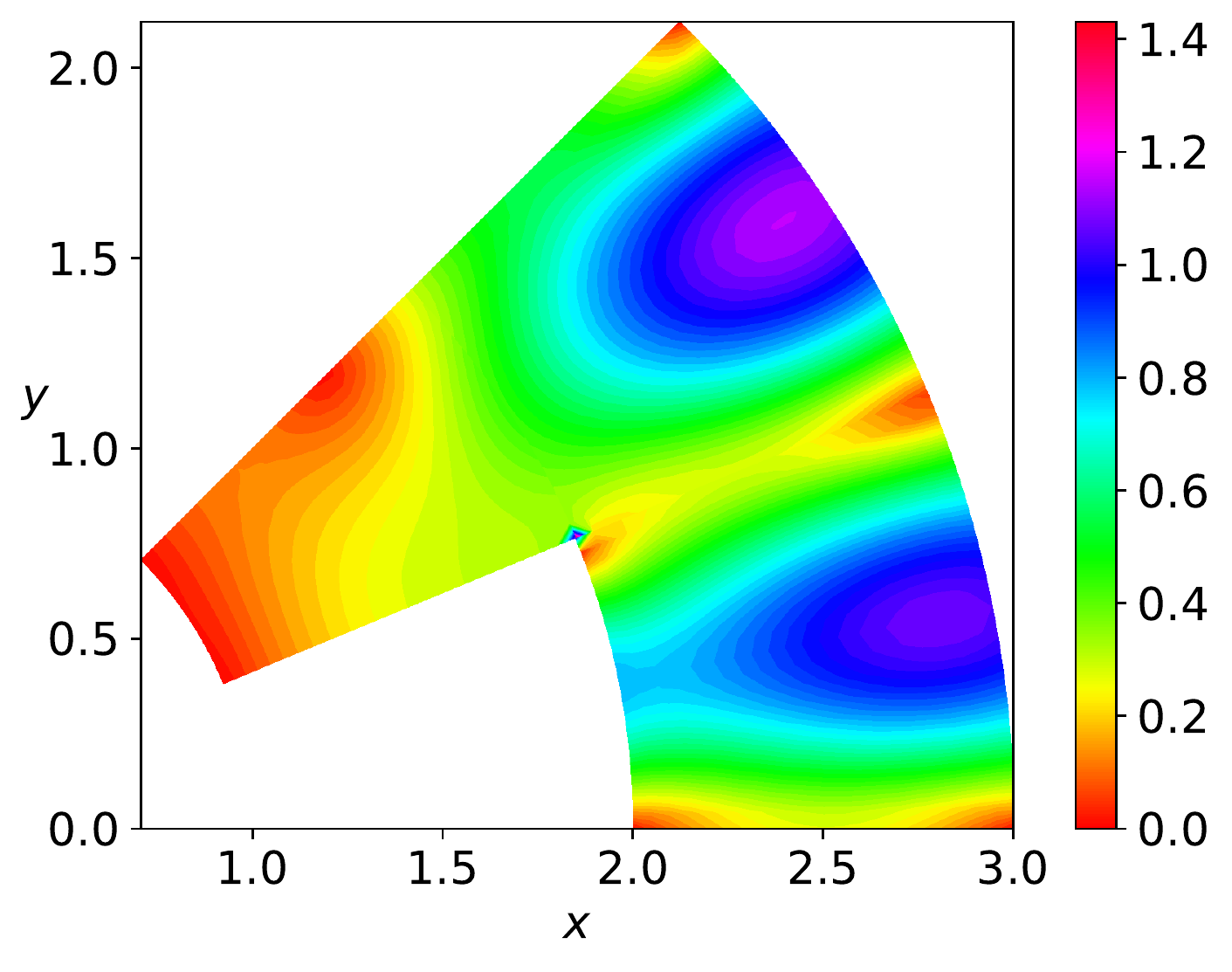}
  }
  \hspace{15pt}
  \subfloat[$\abs{\bu_4}$, $\lambda_4 = 0.870941$]{%
    \includegraphics[width=0.3\textwidth]{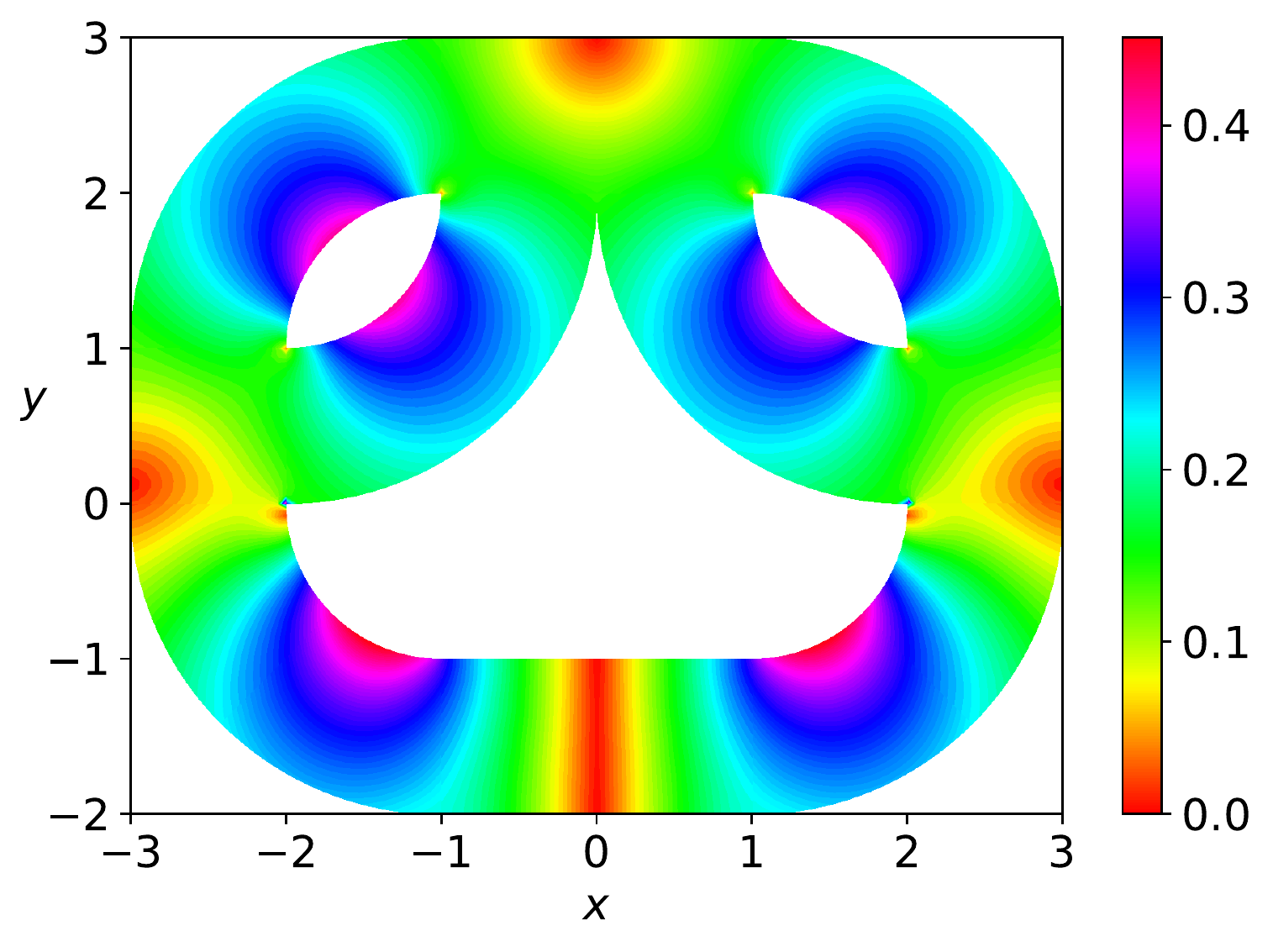}
  }
  \\
  \subfloat[$\abs{\bu_5}$, $\lambda_5 = 12.4355372484$]{%
    \includegraphics[width=0.3\textwidth]{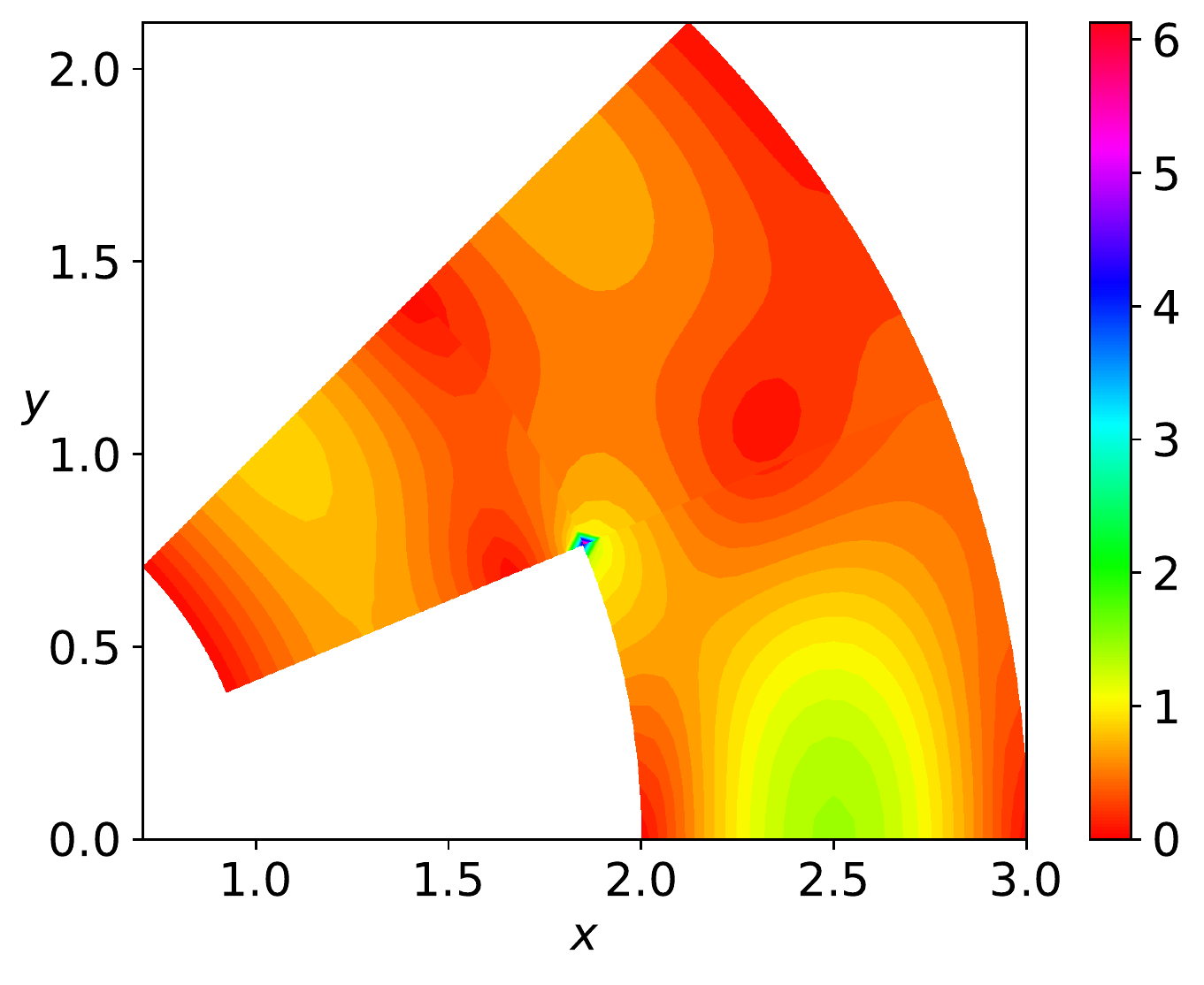}
  }
  \hspace{15pt}
  \subfloat[$\abs{\bu_5}$, $\lambda_5 = 1.1945$]{%
    \includegraphics[width=0.3\textwidth]{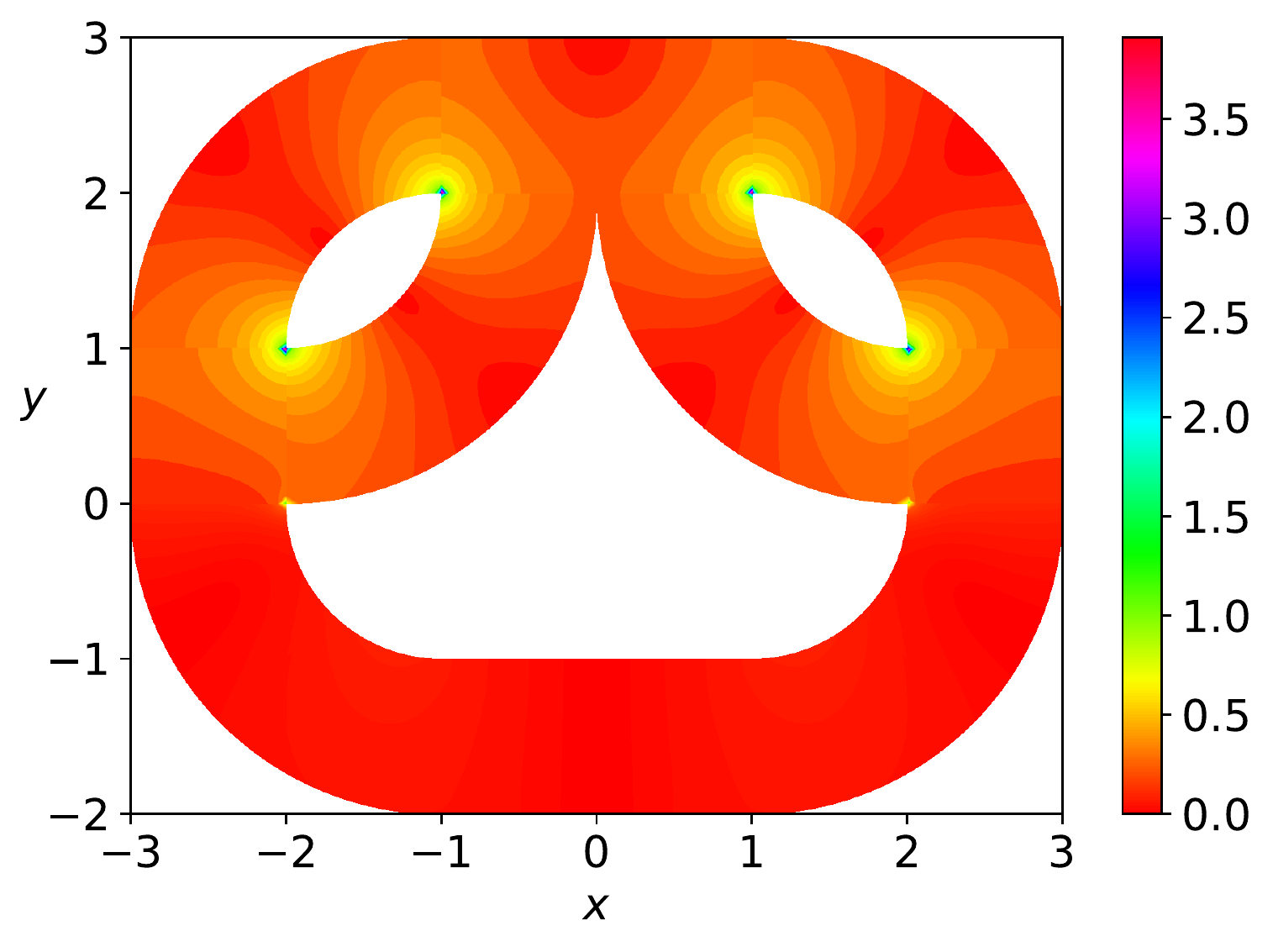}
  }
\end{center}
\caption{
Eigenmodes of the curl-curl problem on the curved-L-shaped domain (left) 
and on the pretzel domain (right), obtained on fine grids with $20 \times 20$
cells per patch and degree $6 \times 6$.
}
\label{fig:emodes}
\end{figure}

\begin{figure}[!htbp]
\def \plotdircc {convergence_curves}
\begin{center}
  \subfloat[curved L-shaped domain, $p=3$]{%
    \includegraphics[width=0.4\textwidth]{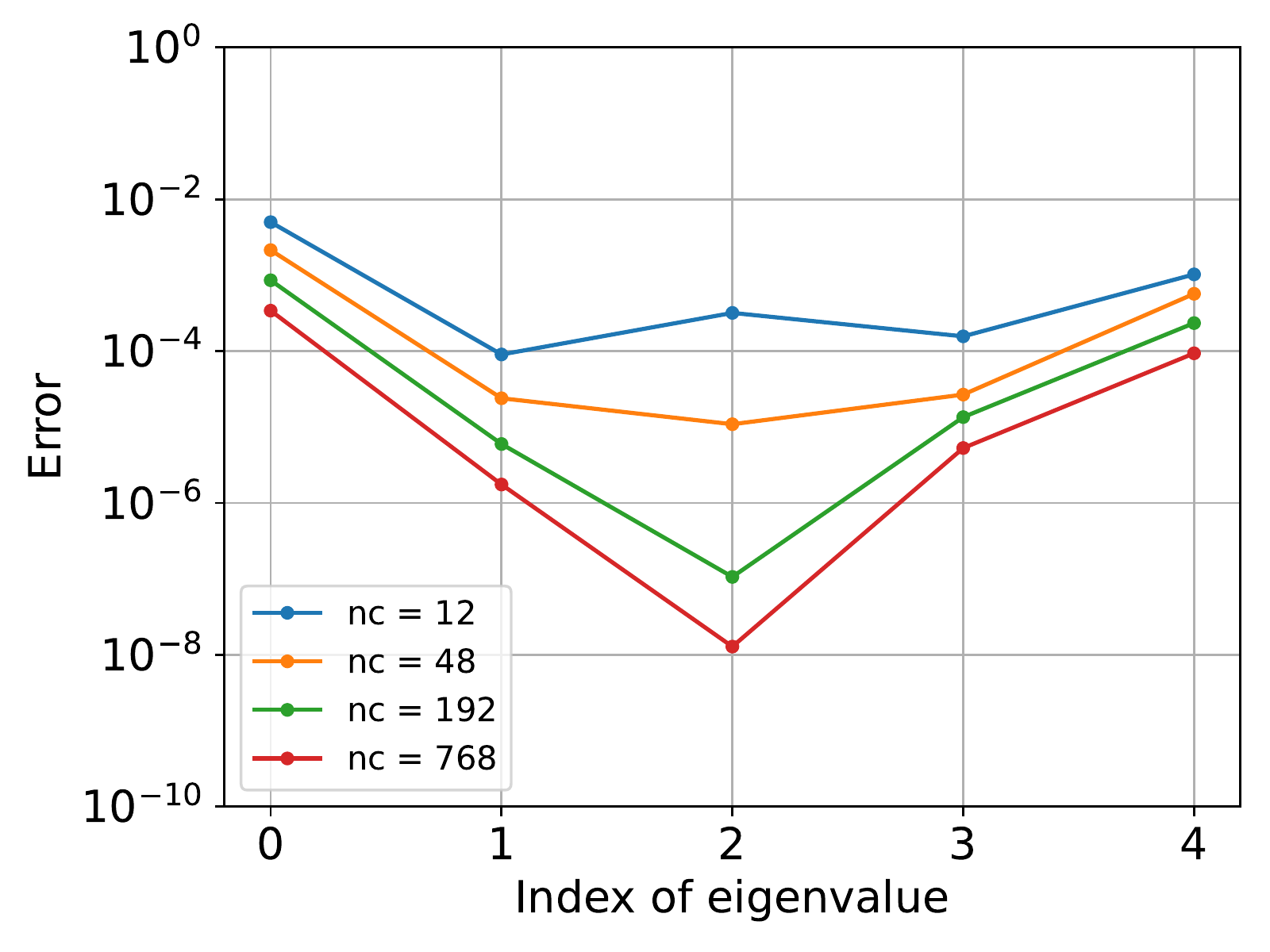}
  }
  \hspace{15pt}
  \subfloat[pretzel domain, $p=3$]{%
    \includegraphics[width=0.4\textwidth]{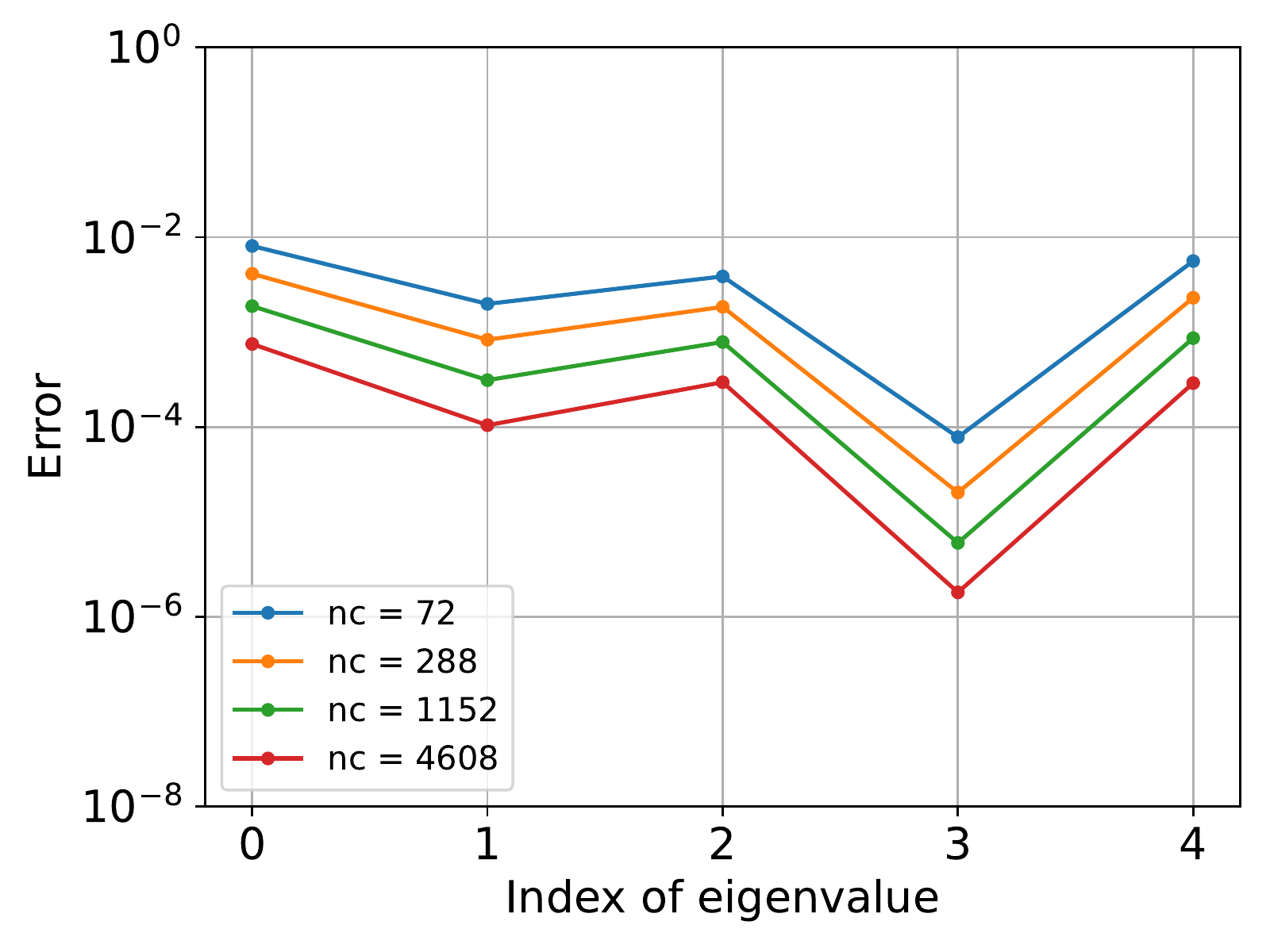}
  }
  \\
  \subfloat[curved L-shaped domain, $p=5$]{%
    \includegraphics[width=0.4\textwidth]{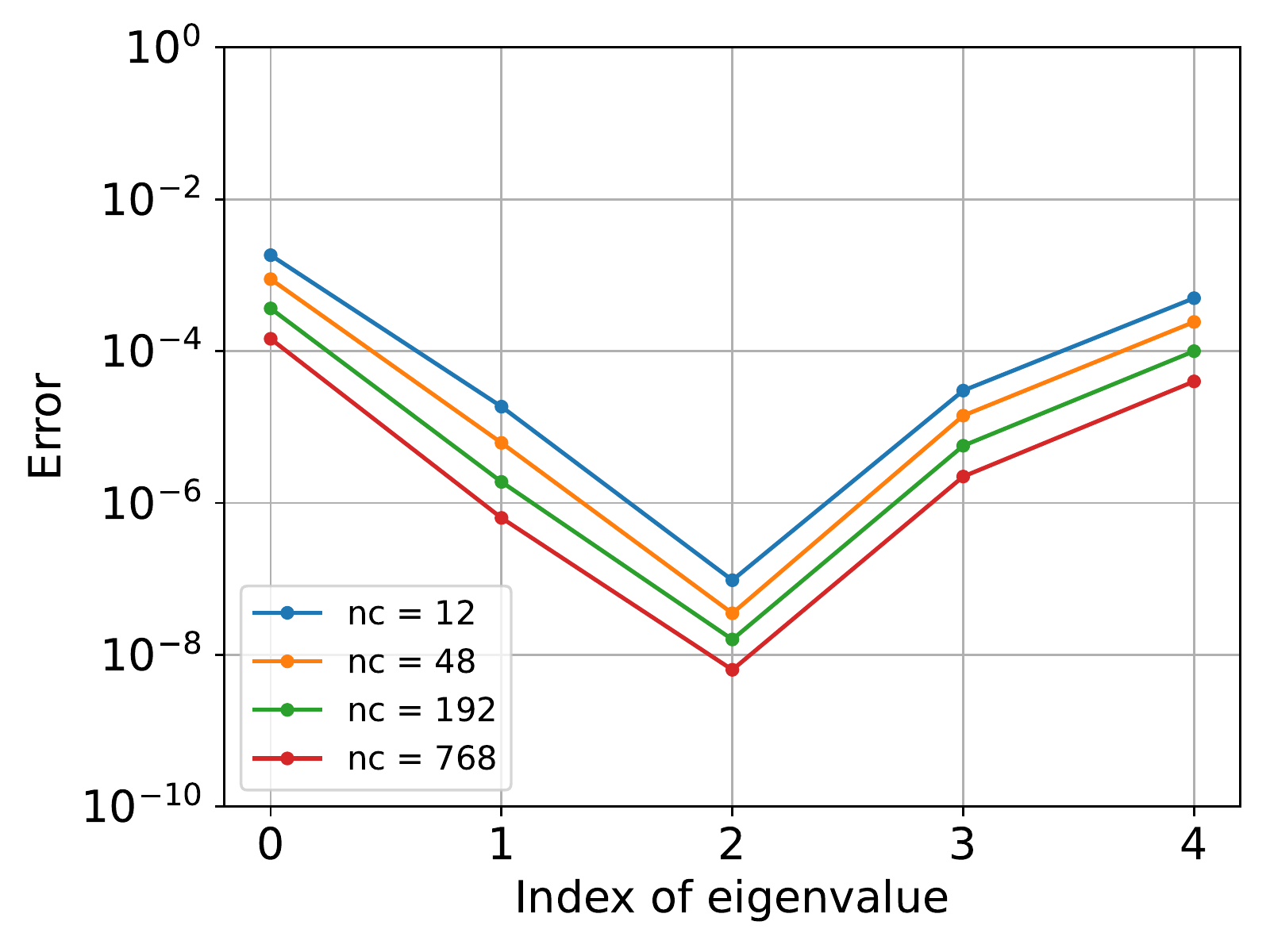}
  }
  \hspace{15pt}
  \subfloat[pretzel domain, $p=5$]{%
    \includegraphics[width=0.4\textwidth]{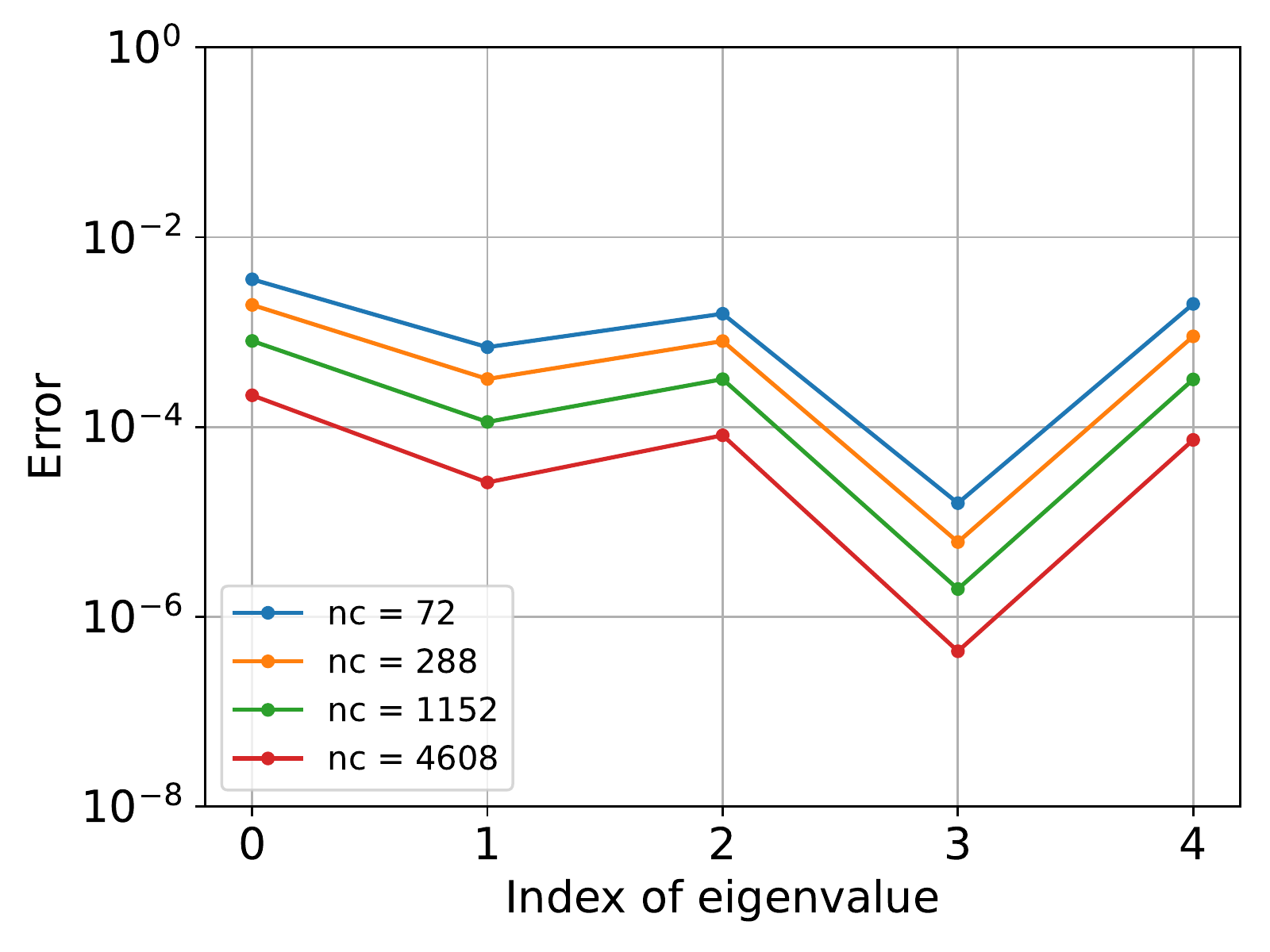}
  }
\end{center}
\caption{
Relative eigenvalue errors \eqref{err_lambda} for the curl-curl problem on the 
curved-L-shaped domain (left) 
and on the pretzel domain (right), using 
elements of degree $p \times p$ as indicated, and 
$N \times N$ cells per patch with $N = 2, 4, 8$ and $16$. The resulting 
total numbers of cells are $\texttt{nc} = 3N^2$ and $18N^2$ for the respective
multipatch domains.
}
\label{fig:ev_pbm_cc}
\end{figure}

\subsection{Magnetostatic test-cases}

We next study the CONGA discretizations of the magnetostatic problems
presented in Section~\ref{sec:MS}: either the one 
for the problem with pseudo-vacuum boundary conditions \eqref{magneto_w}
or the one with metallic boundary conditions  \eqref{magneto_mw}. 

To this end we consider a scalar dipole current source,
\begin{equation} \label{tc_ms_Jz}
  J_z = \psi_0 - \psi_1
  \qquad \text{ where } \psi_m = \exp\Big(-\frac{\big((x-x_m)^2 + (y-y_m)^2\big)^2}{2 \sigma^2} \Big)
\end{equation}
in the pretzel-shaped domain, with $\sigma = 0.02$.
We set the positive current pole at $x_0 = y_0 = 1$ and the negative one at $x_1 = y_1 = 2$.
We then consider the discrete solvers described in Section~\ref{sec:MS_vbc} and \ref{sec:MS_mbc}
for the problems with pseudo-vacuum and metallic boundary conditions, respectively.

In Figure~\ref{fig:MS_source} we plot the scalar source $J_z$ together with the
vector-valued $\bcurl J_z$ field, and for each of the boundary conditions, 
we plot in Figure~\ref{fig:MS_ref_sols} fine solutions computed using the CONGA scheme 
on a mesh with $20 \times 20$ cells per patch (i.e. 7200 cells in total)
and elements of degree $6 \times 6$ in each patch. 

In Table~\ref{tab:MS_L2_error} we then show the relative $L^2$ errors 
corresponding to coarser grids and lower spline degrees for both boundary 
conditions, using as reference the fine solutions shown in Figure~\ref{fig:MS_ref_sols}. 
We also plot in Figure~\ref{fig:MS_sols} the solutions corresponding 
to spline elements of degree $3 \times 3$ on each patch.
Again these results indicate that our CONGA solutions converge nicely as the grids are refined, 
with smaller errors associated with higher polynomial degrees.

\begin{figure}[!htbp]
\def \plotdir {magnetostatic_metal/pretzel_f_dipole_J_nc=20_deg=6}
\def \plotfnJ {j2h}
\def \plotfnCurlJ {f1h_P_L2_wcurl_J}
\begin{center}
  \subfloat[$J_z$ source]{%
    \includegraphics[height=0.35\textwidth]{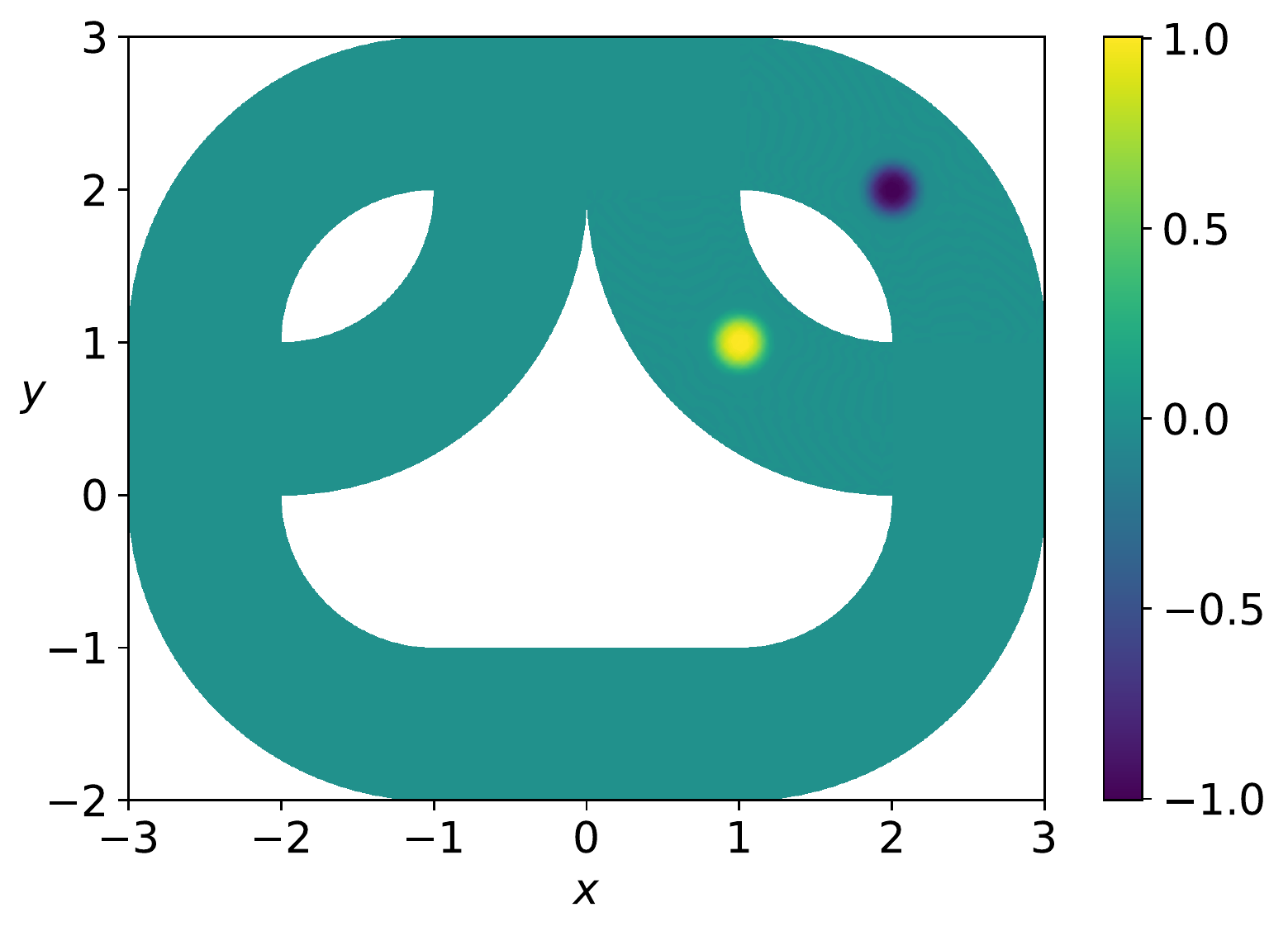}
  }
  \hspace{5pt}
  \subfloat[$\bcurl J_z$]{%
    \includegraphics[height=0.35\textwidth]{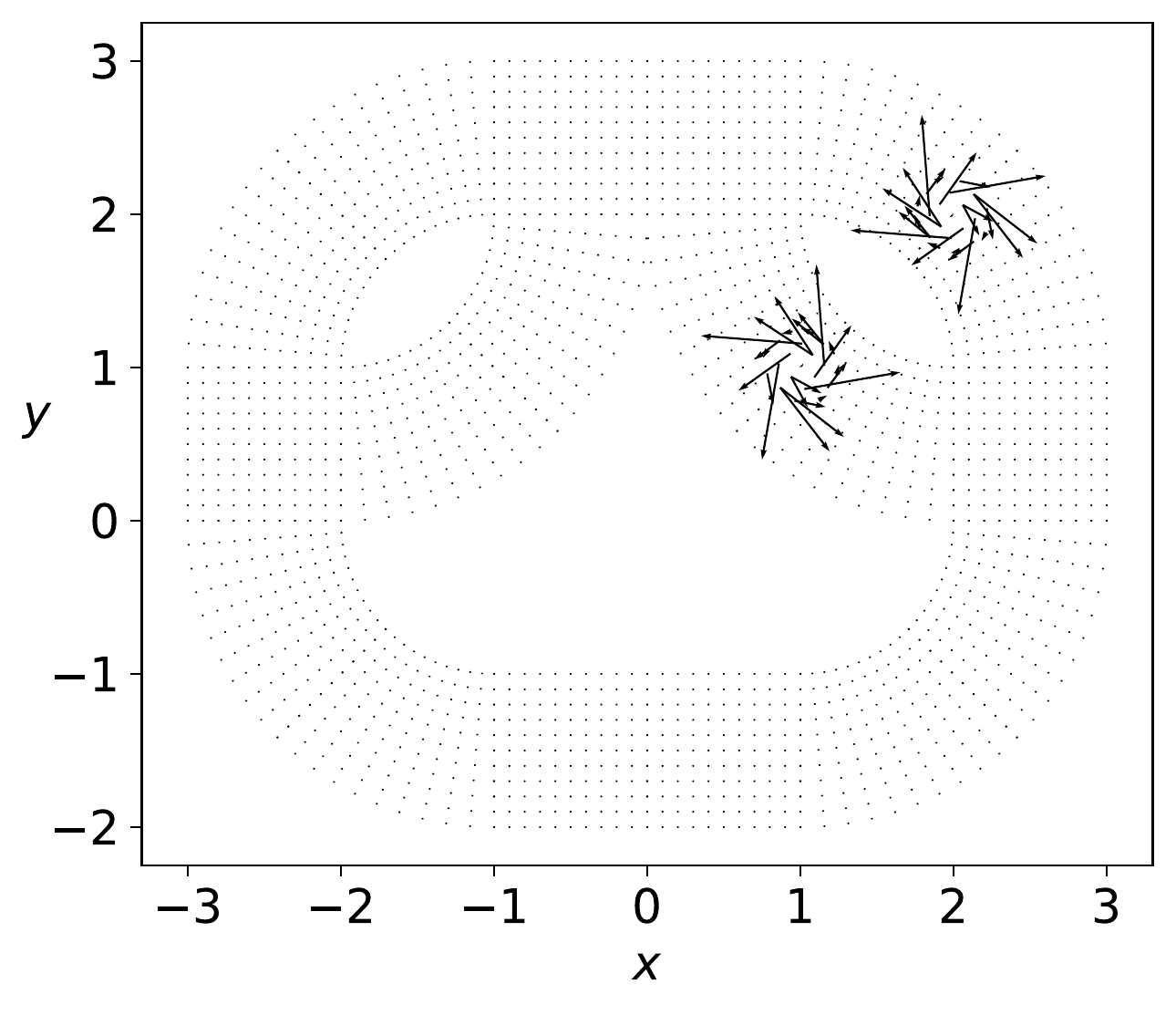}
  }
\end{center}
\caption{
Source for the magnetostatic test-cases:
the scalar current density~\eqref{tc_ms_Jz} is plotted on the left panel
and its vector-valued curl is shown on the right.}
\label{fig:MS_source}
\end{figure}

\begin{figure}[!htbp]
  \def \plotdirMMS {magnetostatic_metal/pretzel_f_dipole_J_nc=20_deg=6}
  \def \plotdirVMS {magnetostatic_vacuum/pretzel_f_dipole_J_nc=20_deg=6}
  \def \plotfnBvf {gamma0_h=10.0_gamma1_h=10.0_uh_vf}
  \def \plotfnB   {gamma0_h=10.0_gamma1_h=10.0_uh}

\begin{center}
  \subfloat[vector field $\bB$ (VBC)]{%
    \includegraphics[height=0.35\textwidth]{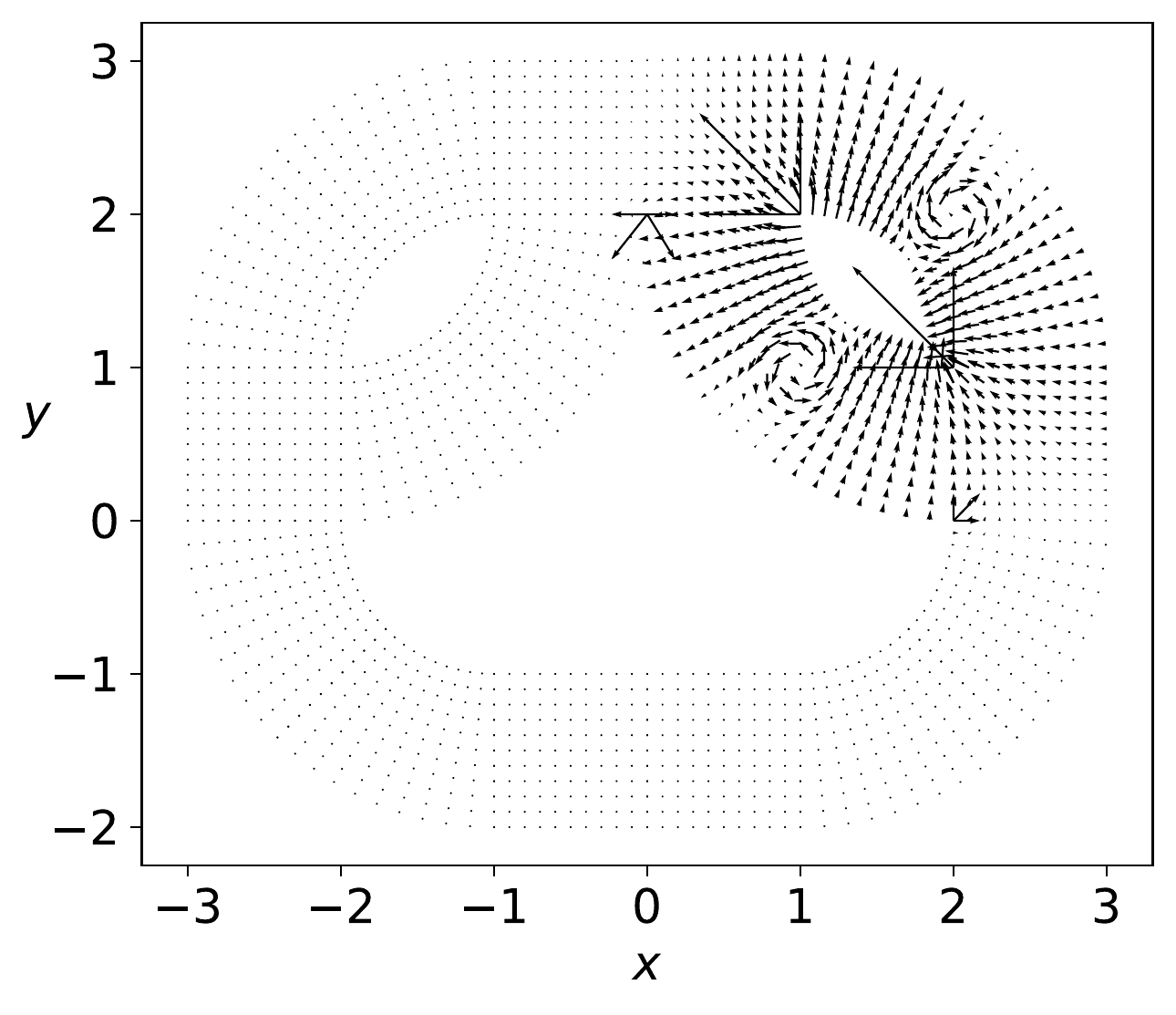}
  }
  \hspace{5pt}
  \subfloat[amplitude $\abs{\bB}$ (VBC)]{%
    \includegraphics[height=0.35\textwidth]{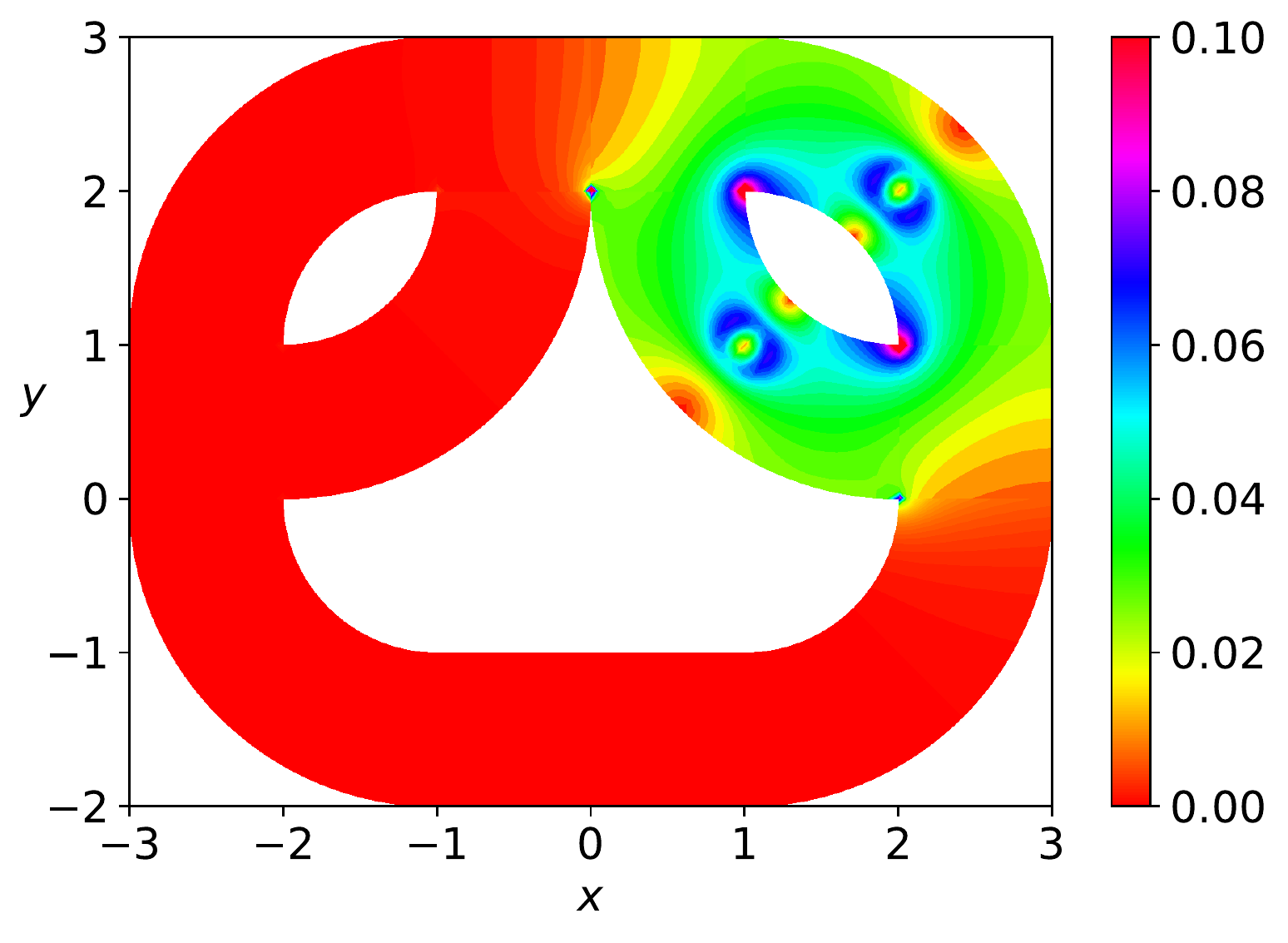}
  }
  \\
  \subfloat[vector field $\bB$ (MBC)]{%
    \includegraphics[height=0.35\textwidth]{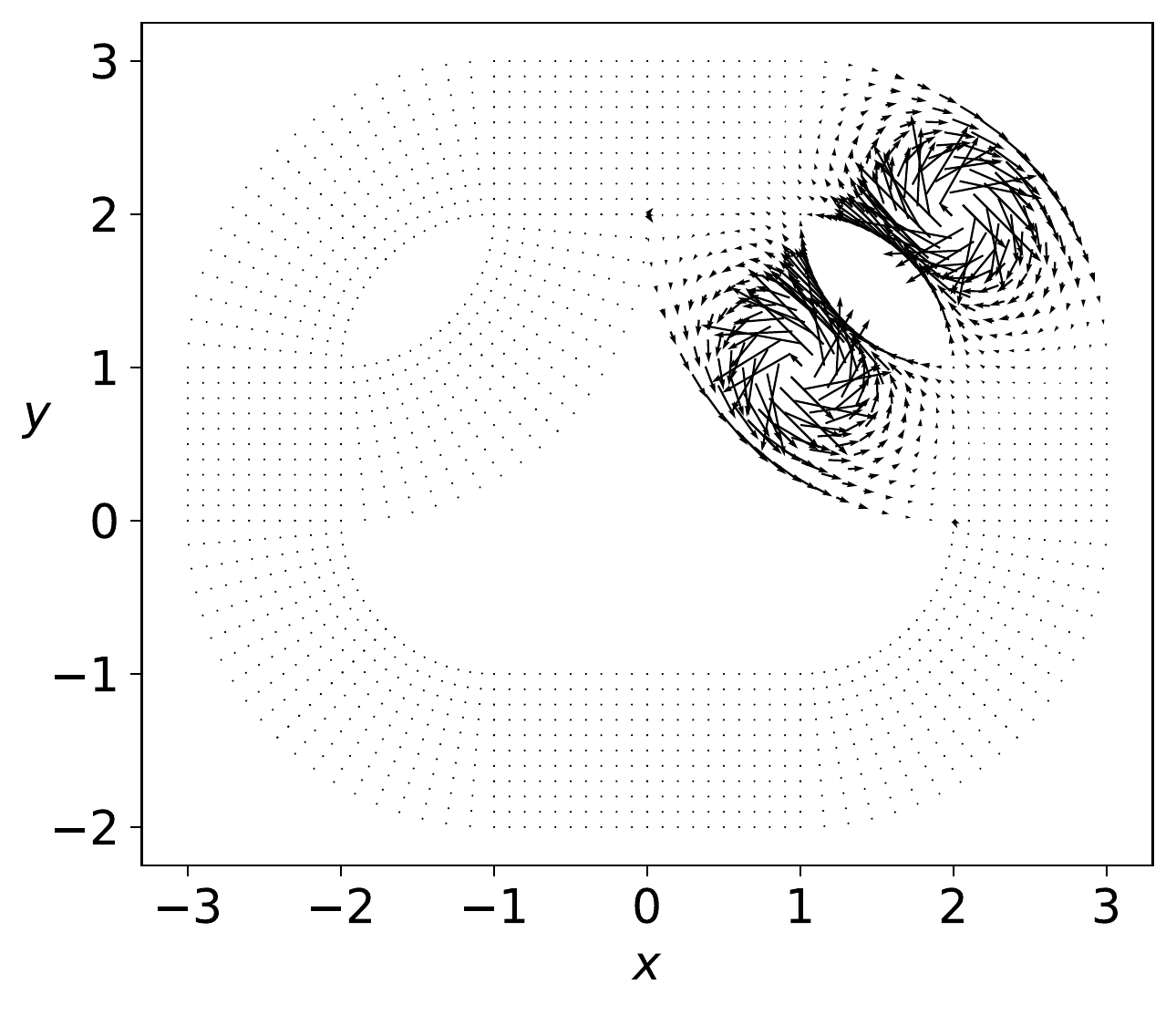}
  }
  \hspace{5pt}
  \subfloat[amplitude $\abs{\bB}$ (MBC)]{%
    \includegraphics[height=0.35\textwidth]{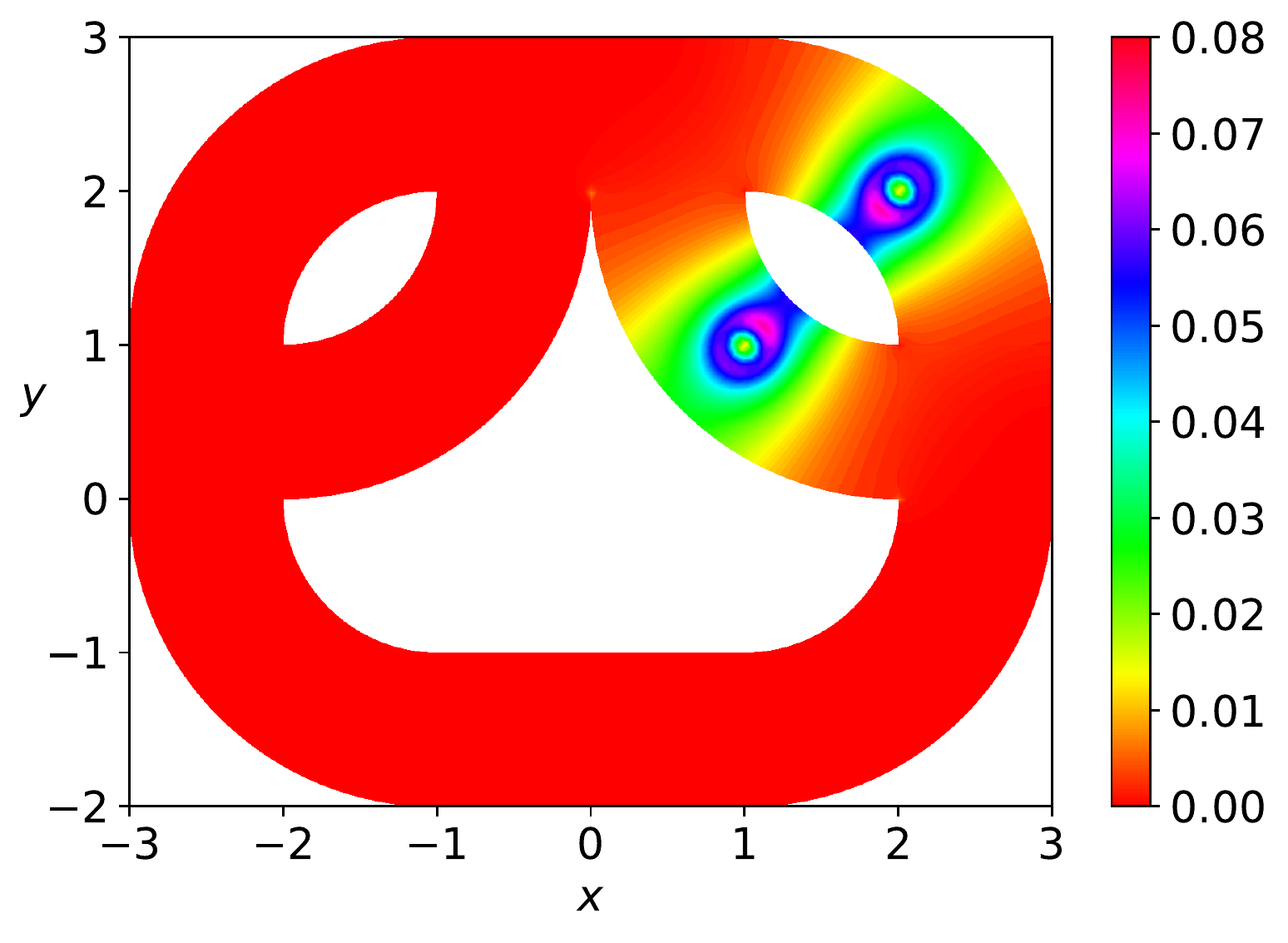}
  }
\end{center}
\caption{
Reference numerical solutions obtained with $20 \times 20$ cells per patch and 
elements of degree $6 \times 6$, for the magnetostatic test-cases
with pseudo-vacuum (top) and metallic boundary conditions (bottom).
The vector fields are shown on the left while the amplitudes are shown on the right.
}
\label{fig:MS_ref_sols}
\end{figure}

\begin{figure}[!htbp]
\begin{center}
  \def \plotdircoarse {magnetostatic_vacuum/pretzel_f_dipole_J_nc=2_deg=3}
  \def \plotdirmed {magnetostatic_vacuum/pretzel_f_dipole_J_nc=4_deg=3}
  \def \plotdirfine {magnetostatic_vacuum/pretzel_f_dipole_J_nc=8_deg=3}
  \def \plotfn {gamma0_h=10.0_gamma1_h=10.0_uh}
  \subfloat[$2 \times 2$ cells p.p.]{%
    \includegraphics[width=0.3\textwidth]{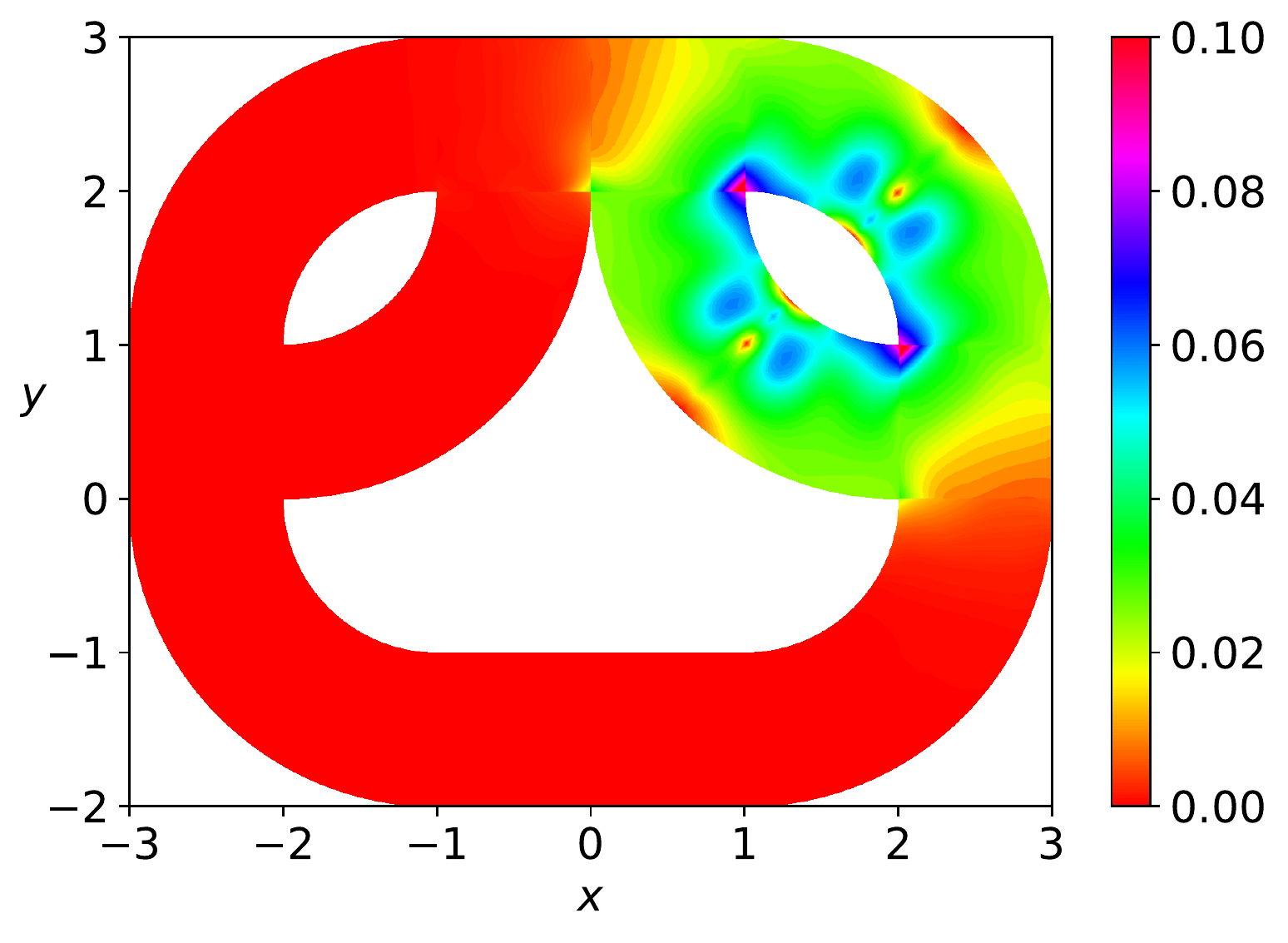}
  }
  \hspace{10pt}
  \subfloat[$4 \times 4$ cells p.p.]{%
    \includegraphics[width=0.3\textwidth]{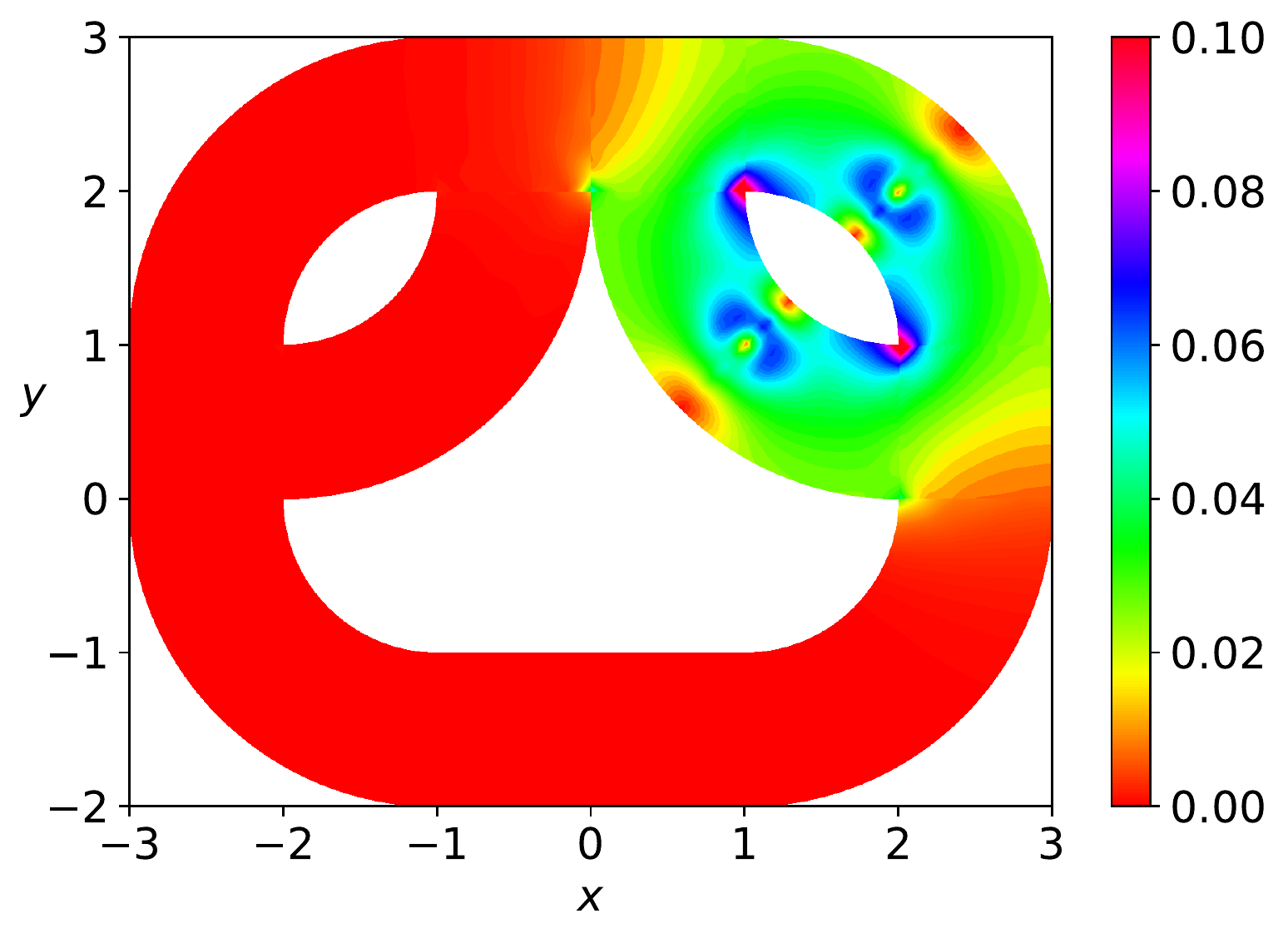}
  }
  \hspace{10pt}
  \subfloat[$8 \times 8$ cells p.p.]{%
    \includegraphics[width=0.3\textwidth]{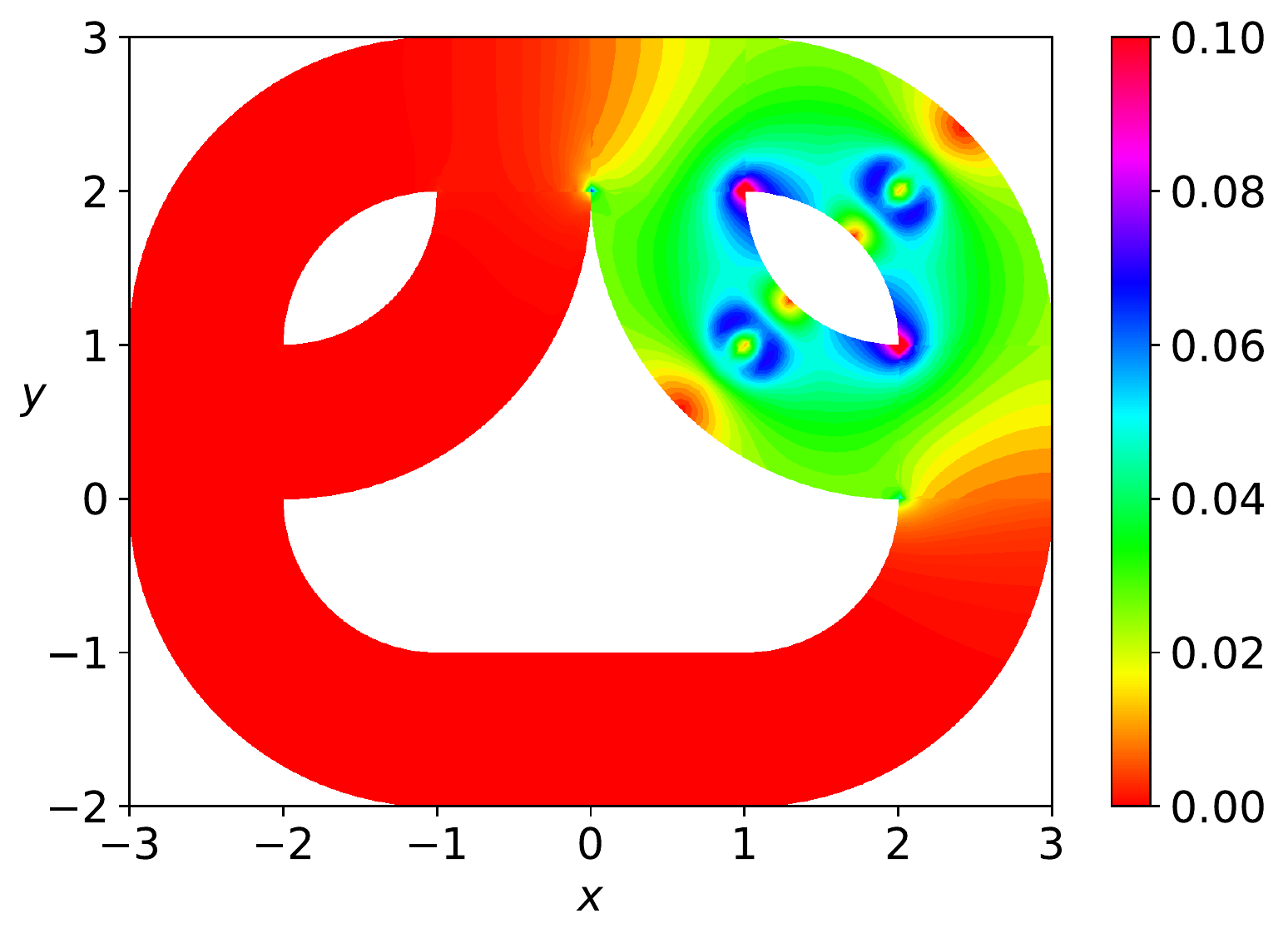}
  }
  \\
  \def \plotdircoarse {magnetostatic_metal/pretzel_f_dipole_J_nc=2_deg=3}
  \def \plotdirmed {magnetostatic_metal/pretzel_f_dipole_J_nc=4_deg=3}
  \def \plotdirfine {magnetostatic_metal/pretzel_f_dipole_J_nc=8_deg=3}
  \def \plotfn {gamma0_h=10.0_gamma1_h=10.0_uh}
  \subfloat[$2 \times 2$ cells p.p.]{%
    \includegraphics[width=0.3\textwidth]{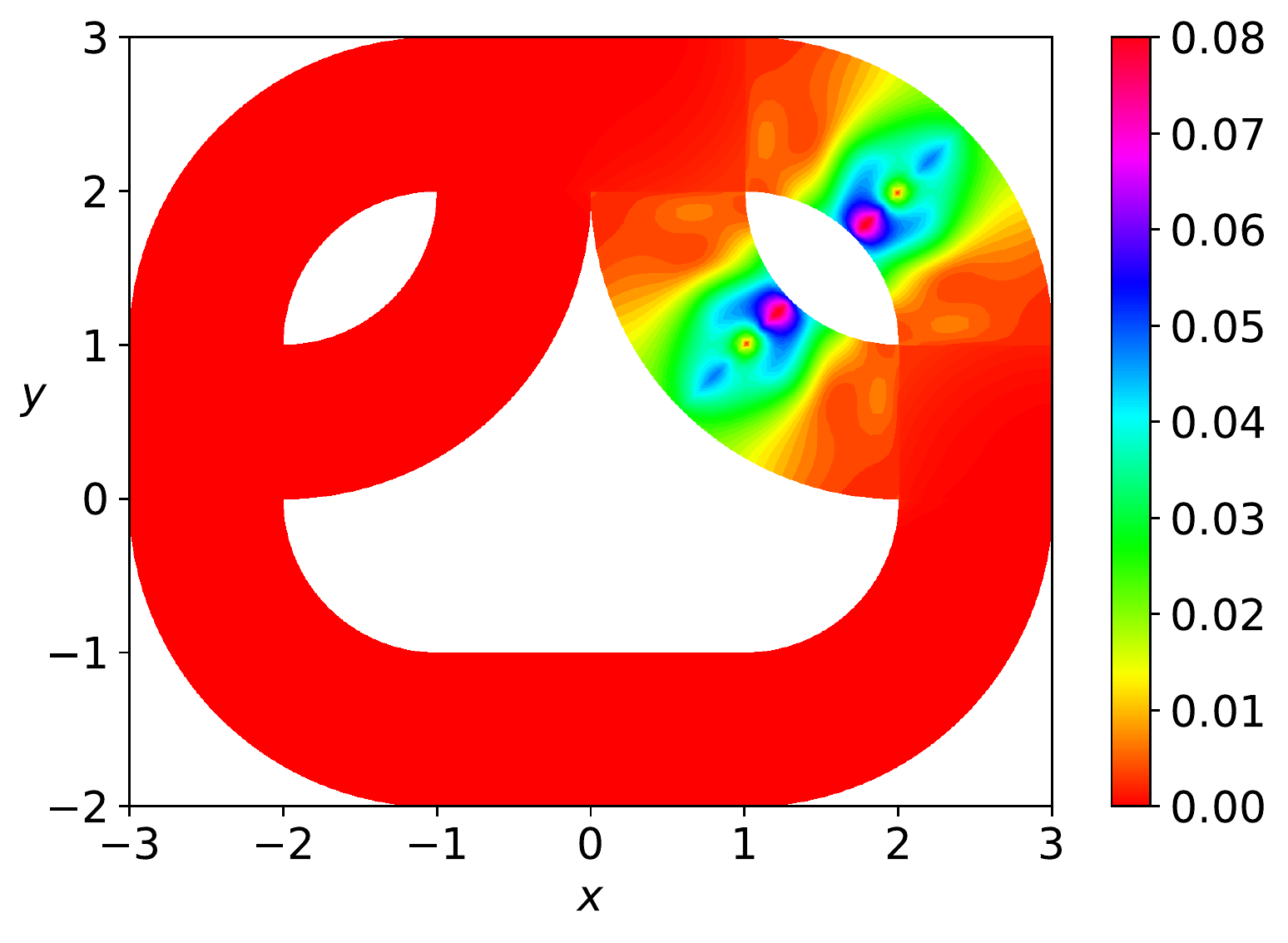}
  }
  \hspace{10pt}
  \subfloat[$4 \times 4$ cells p.p.]{%
    \includegraphics[width=0.3\textwidth]{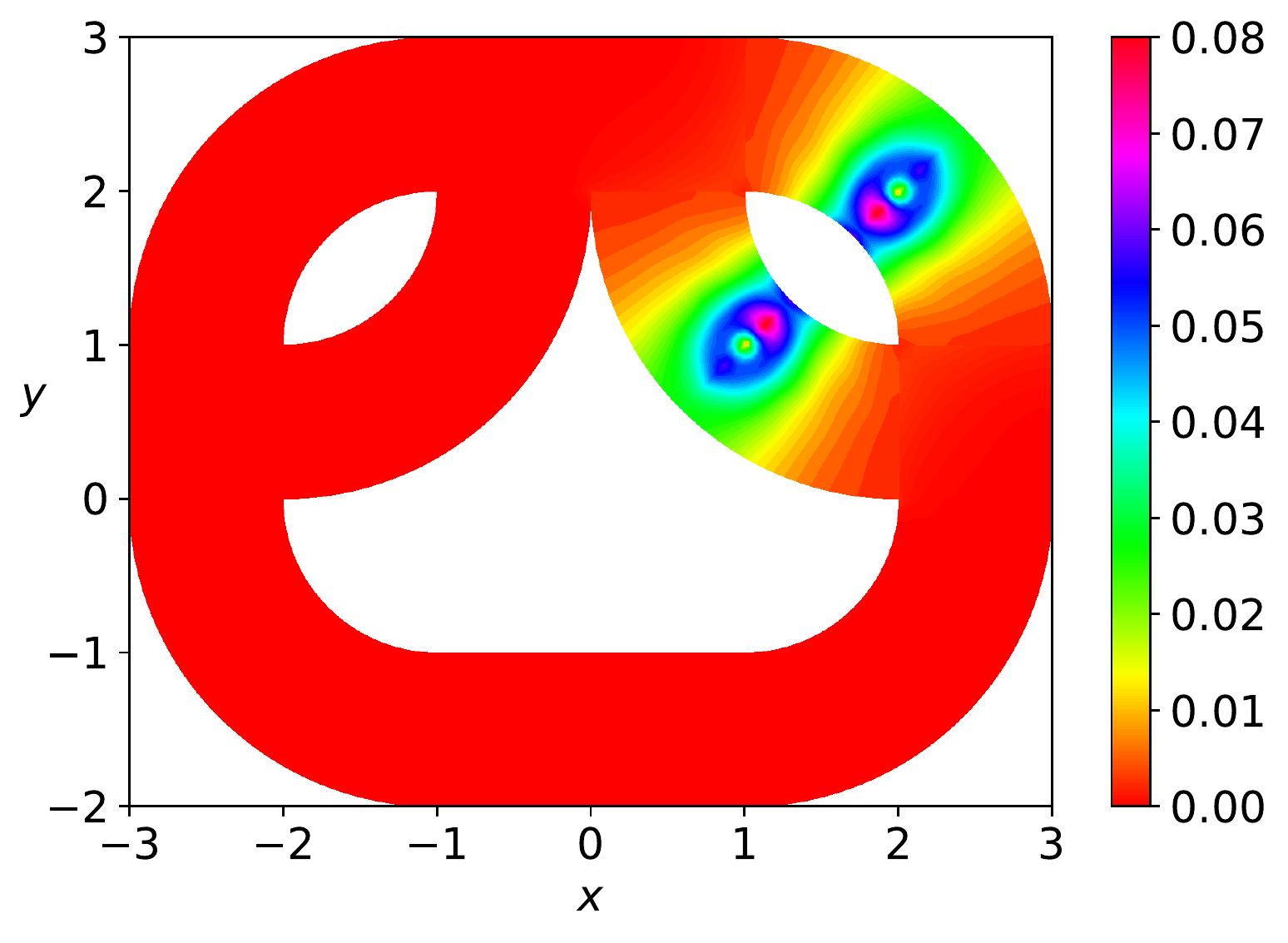}
  }
  \hspace{10pt}
  \subfloat[$8 \times 8$ cells p.p.]{%
    \includegraphics[width=0.3\textwidth]{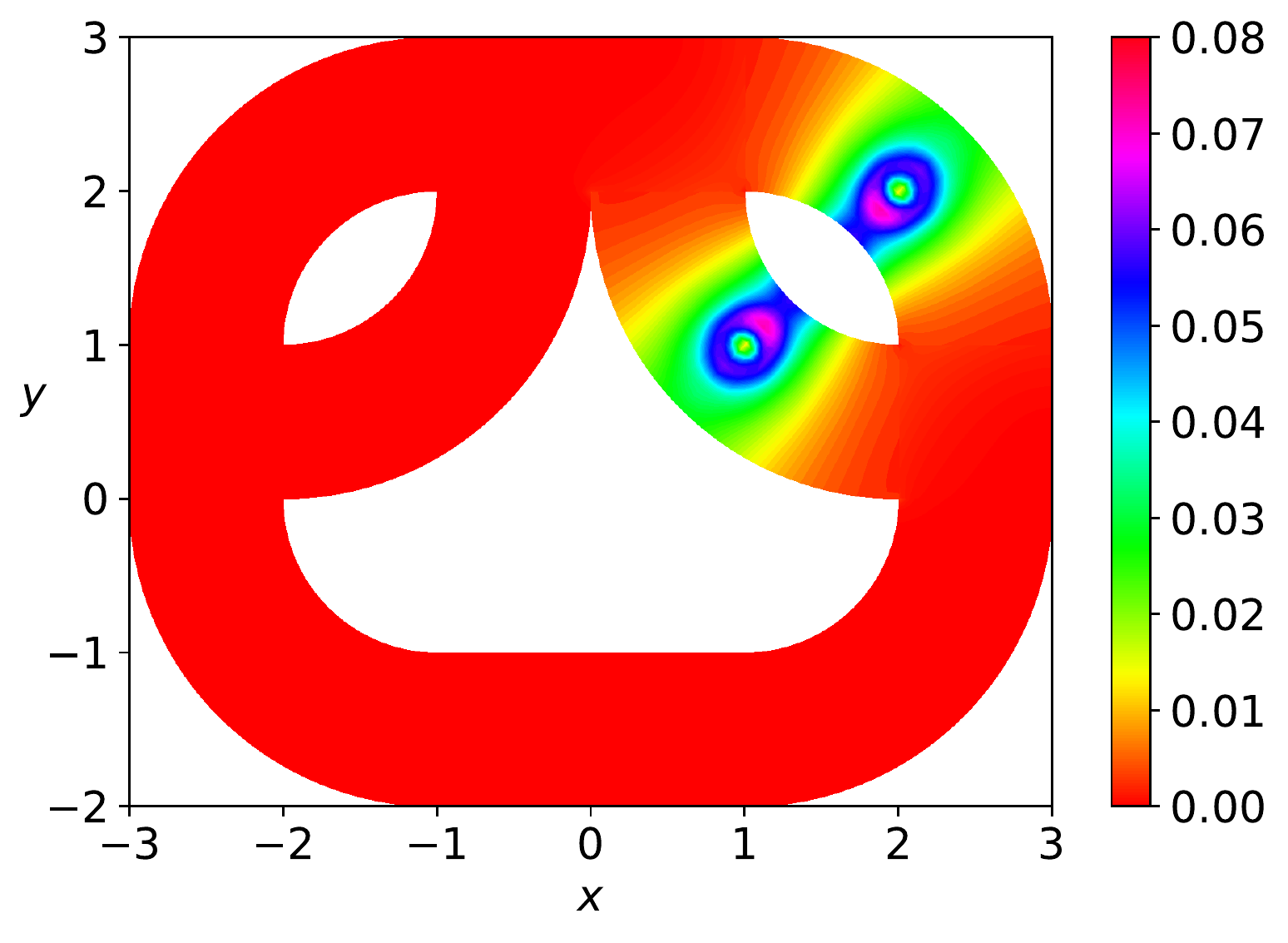}
  }
\end{center}
\caption{
Numerical solutions for the magnetostatic test-cases with pseudo-vacuum (top) 
and metallic boundary conditions (bottom), obtained 
with elements of degree $3 \times 3$ and different patch-wise grids, as indicated.
}
\label{fig:MS_sols}
\end{figure}

\begin{table}[!htbp]
\centering
\begin{tabular}{c}
\hline
$\phantom{\Big(} \mspace{50mu}$ 
Errors for pseudo-vacuum boundary conditions
$\phantom{\Big(} \mspace{50mu}$  \\
\hline   
\end{tabular}
\begin{tabular}{|l|l|l|l|l|}
\hline
\diagbox[width=10em]{cells p.p.}{degree} & \multicolumn{1}{p{12mm}|}{$2\times 2$} & \multicolumn{1}{p{12mm}|}{$3 \times 3$} & \multicolumn{1}{p{12mm}|}{$4 \times 4$} & \multicolumn{1}{p{12mm}|}{$5 \times 5$}
\\ 
\hline   
$2  \times 2$   &  0.23443 & 0.16028 & 0.10290 & 0.06425 \\ \hline
$4  \times 4$   &  0.10514 & 0.05855 & 0.05773 & 0.03029 \\ \hline
$8  \times 8$   &  0.03909 & 0.02414 & 0.01893 & 0.01542 \\ \hline
$16 \times 16$  & 0.018733 & 0.01193 & 0.008375 & 0.00567 \\ \hline
\end{tabular}
\begin{tabular}{c}
\hline
$\phantom{\Big(} \mspace{50mu}$ 
Errors for metallic boundary conditions
$\phantom{\Big(} \mspace{50mu}$  \\
\hline   
\end{tabular}
\begin{tabular}{|l|l|l|l|l|}
\hline
\diagbox[width=10em]{cells p.p.}{degree} & \multicolumn{1}{p{12mm}|}{$2\times 2$} & \multicolumn{1}{p{12mm}|}{$3 \times 3$} & \multicolumn{1}{p{12mm}|}{$4 \times 4$} & \multicolumn{1}{p{12mm}|}{$5 \times 5$}
\\ 
\hline   
$2  \times 2$   & 0.37023 & 0.24518 & 0.15902 & 0.09266 \\ \hline
$4  \times 4$   & 0.15490 & 0.07747 & 0.08660 & 0.03508 \\ \hline
$8  \times 8$   & 0.03776 & 0.01840 & 0.01736 & 0.01550 \\ \hline
$16 \times 16$  & 0.00544 & 0.00187 & 0.00153 & 0.00085 \\ \hline
\end{tabular}
\caption{$L^2$ errors for the $\bB_h$ solutions of the magnetostatic test-cases
with pseudo-vacuum and metallic boundary conditions, 
corresponding to the amplitude plots shown in Figures~\ref{fig:MS_sols}. Here the errors are computed using numerical reference solutions
obtained with $20 \times 20$ cells per patch and a degree $6 \times 6$.}
\label{tab:MS_L2_error}
\end{table}

\subsection{Time-dependent Maxwell equation}
\label{sec:td_max_num}

We finally assess the qualitative and quantitative properties of our mapped spline-based 
CONGA scheme for the time-dependent Maxwell system, using a leap-frog time stepping 
\eqref{AF_lf}--\eqref{AF_mat}. 
We will consider two test-cases.

Our first test-case consists of an initial electric pulse 
\begin{equation} \label{td_tc1_E0}  
  \bE(t=0) = \bcurl \psi, 
  \qquad 
  \psi(\bx) = \exp\Big(-\frac{\big((x-x_0)^2 + (y-y_0)^2\big)^2}{2 \sigma^2} \Big)  
\end{equation}
in the pretzel-shaped domain, with $x_0 = y_0 = 1$ and $\sigma = 0.02$.
The initial magnetic field and the source is 
\begin{equation} \label{td_tc1_B0_J}
B(t=0) = 0, \qquad \text{ and } \qquad \bJ(t,\bx) = 0.
\end{equation}
In Figure~\ref{fig:td_max_sols} we compare successive snapshots of two solutions 
corresponding to different numbers $N \times N$ of cells per patch and degrees $p \times p$: 
on the left plots the solution is computed with $N=8$ and $p=3$, while that on the 
right plots uses $N=20$ and $p=6$.
We observe that these results display a qualitatively correct behaviour for 
propagating waves with reflecting boundary conditions, and that the profile
for both resolutions are similar up to small scale features, which can be seen as a 
practical indicator of convergence.

Since this test-case is without source the discrete energy should be preserved up to bounded oscillations,
as recalled in Section~\ref{sec:td_max}.
In Figure~\ref{fig:td_max_diags} we plot the time evolution of the discrete energy
$\cH_h(t^n) = \frac 12 \big(\norm{\bE^n_h}^2 + \norm{B^{n}_h}^2\big)$, together with 
that of the electric and magnetic fields, and we observe that the total energy is very stable.
We also plot the amplitude of the discrete divergence as time evolves:
assuming all computations exact it should be zero according to \eqref{Gauss_h}, indeed
we have $\rho = \Div \bE(t=0) = 0$ in this test case.
On Figure~\ref{fig:td_max_diags} we verify that it is zero up to machine accuracy and small quadrature errors.

\def \plotdirC {td_maxwell_E0_pulse_P_L2/pretzel_f_elliptic_J_P=GSP_nc=8_deg=3}
\def \plotdirF {td_maxwell_E0_pulse_P_L2/pretzel_f_elliptic_J_P=GSP_nc=20_deg=6}

\begin{figure}[!htbp]

\begin{center}
    \includegraphics[width=0.4\textwidth]{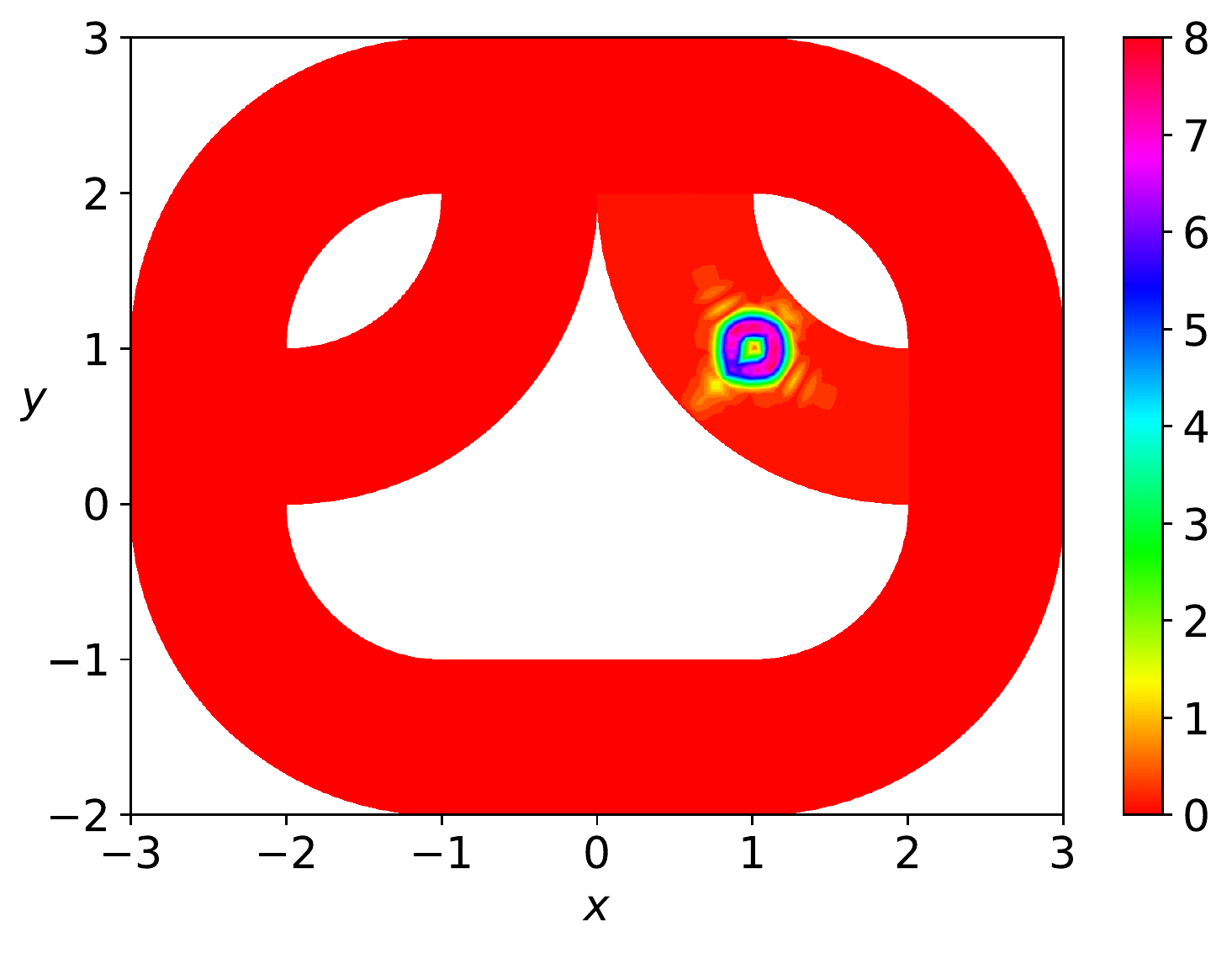}
  \hspace{15pt}
    \includegraphics[width=0.4\textwidth]{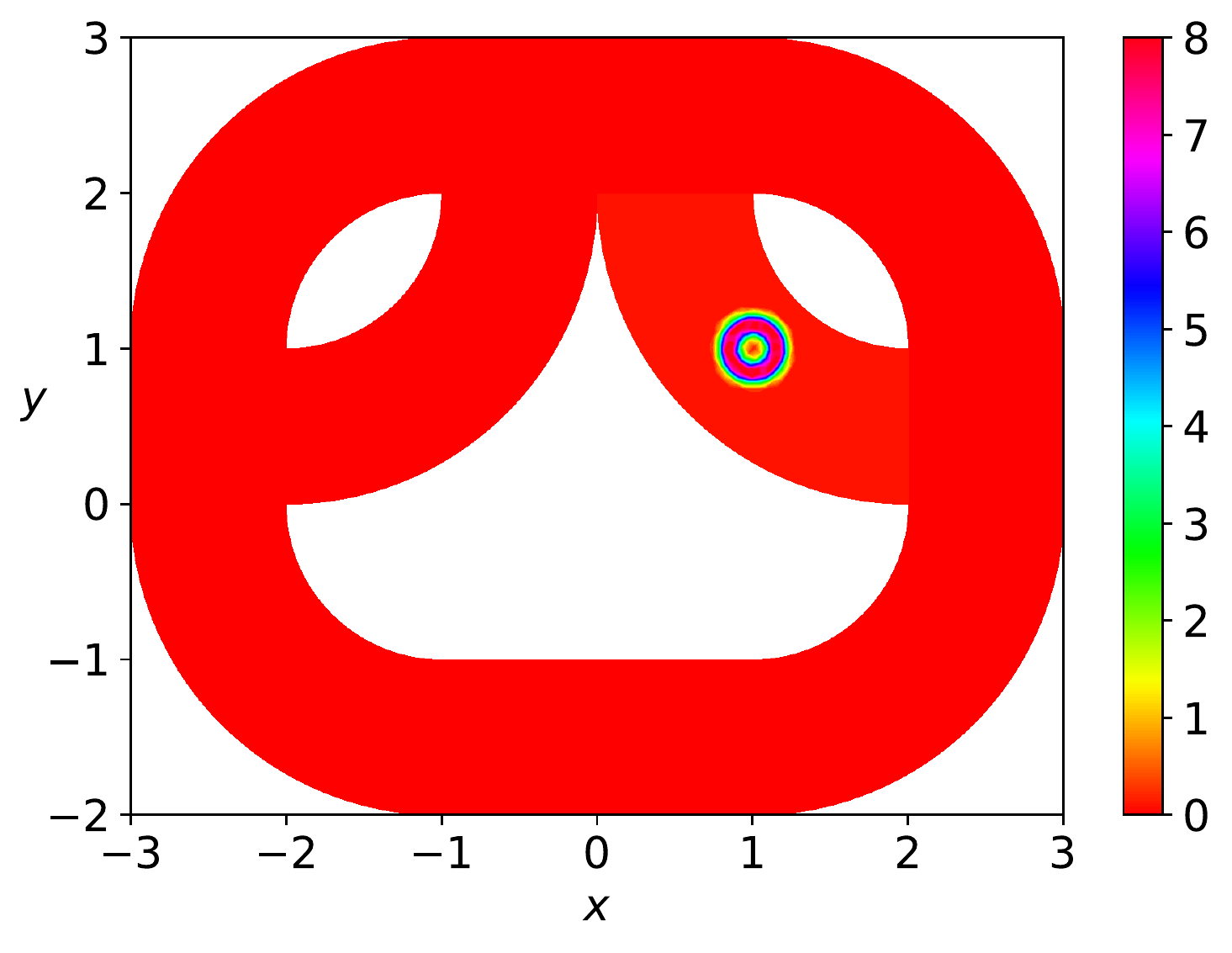}
  \\
    \includegraphics[width=0.4\textwidth]{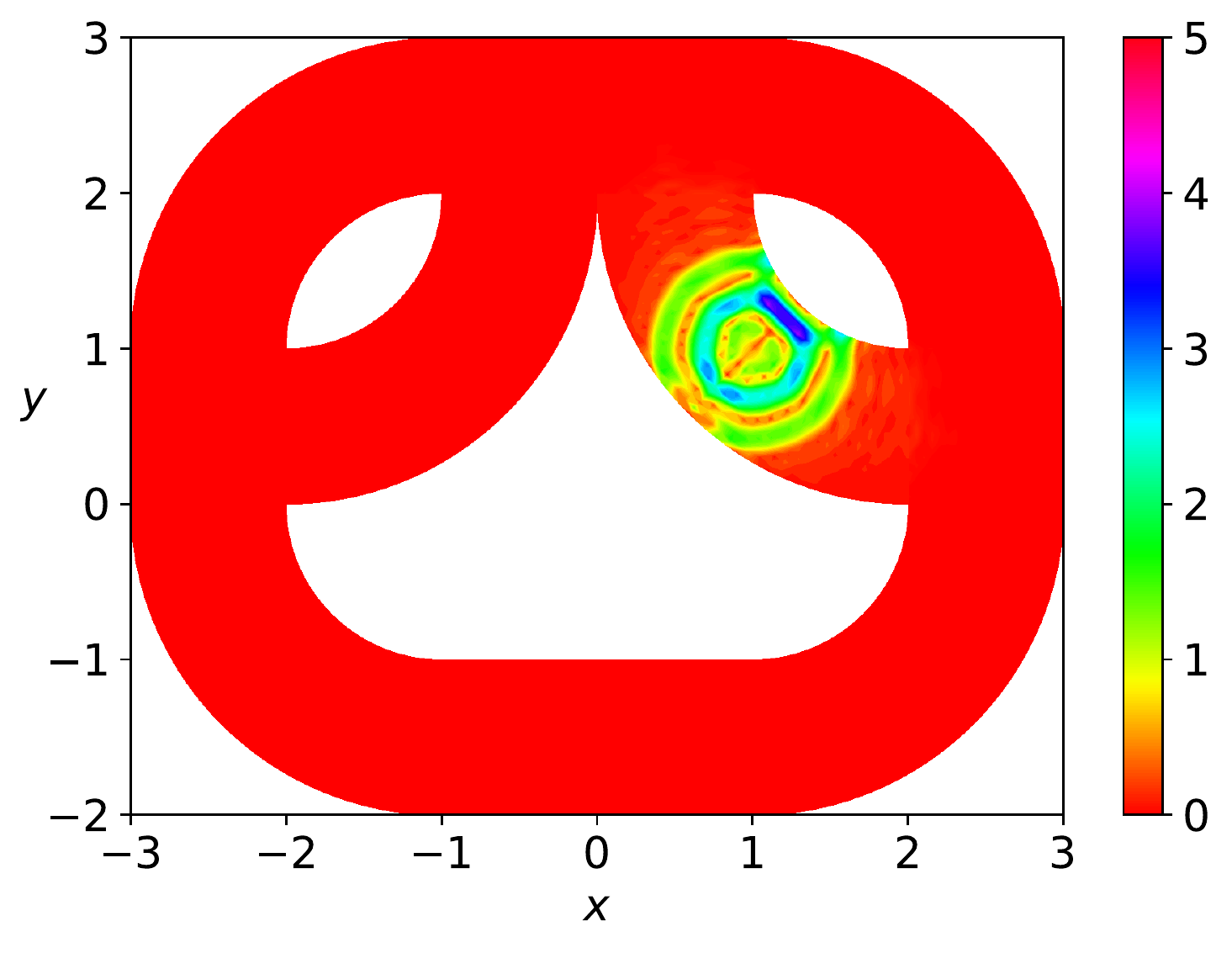}
  \hspace{15pt}
    \includegraphics[width=0.4\textwidth]{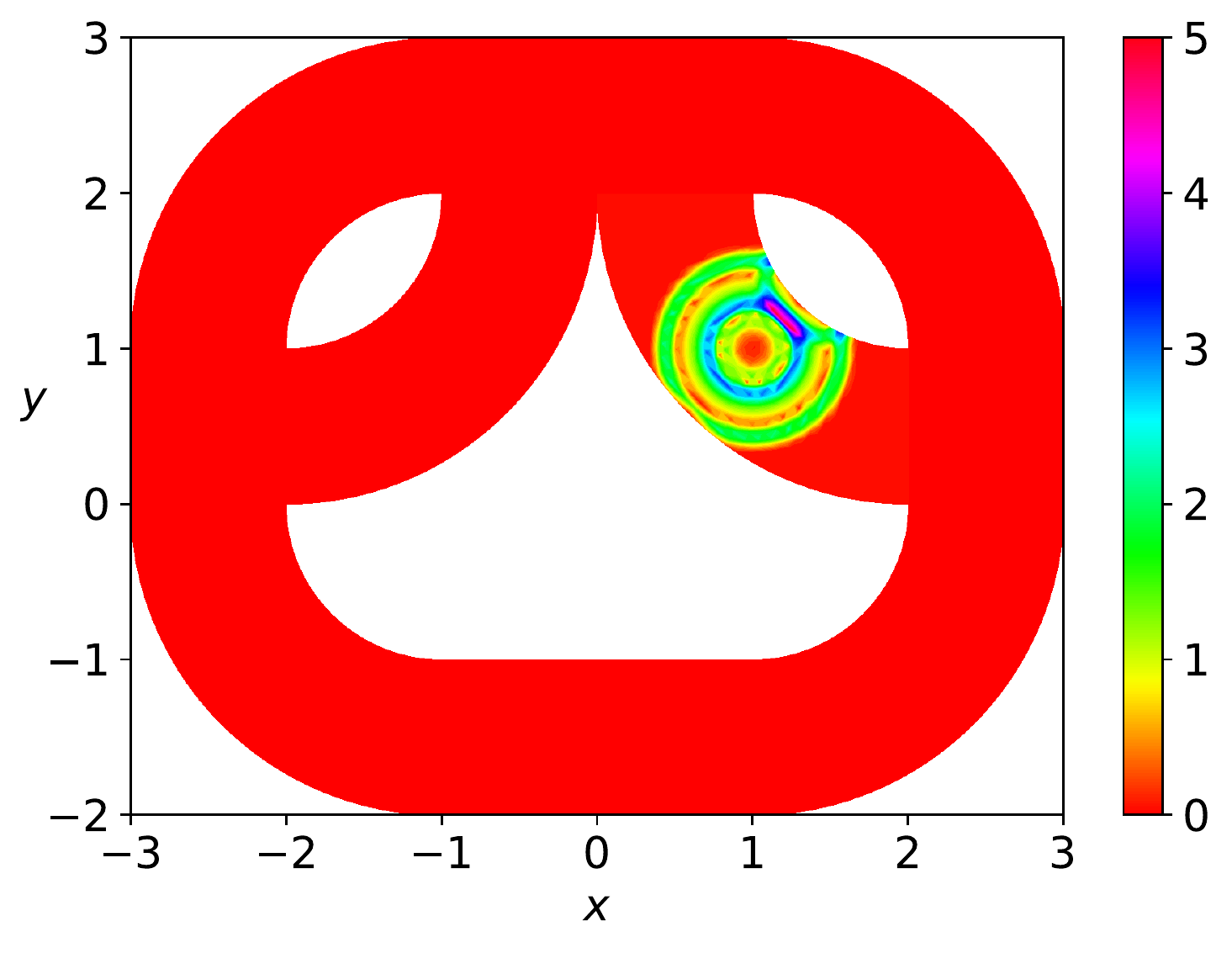}
  \\
    \includegraphics[width=0.4\textwidth]{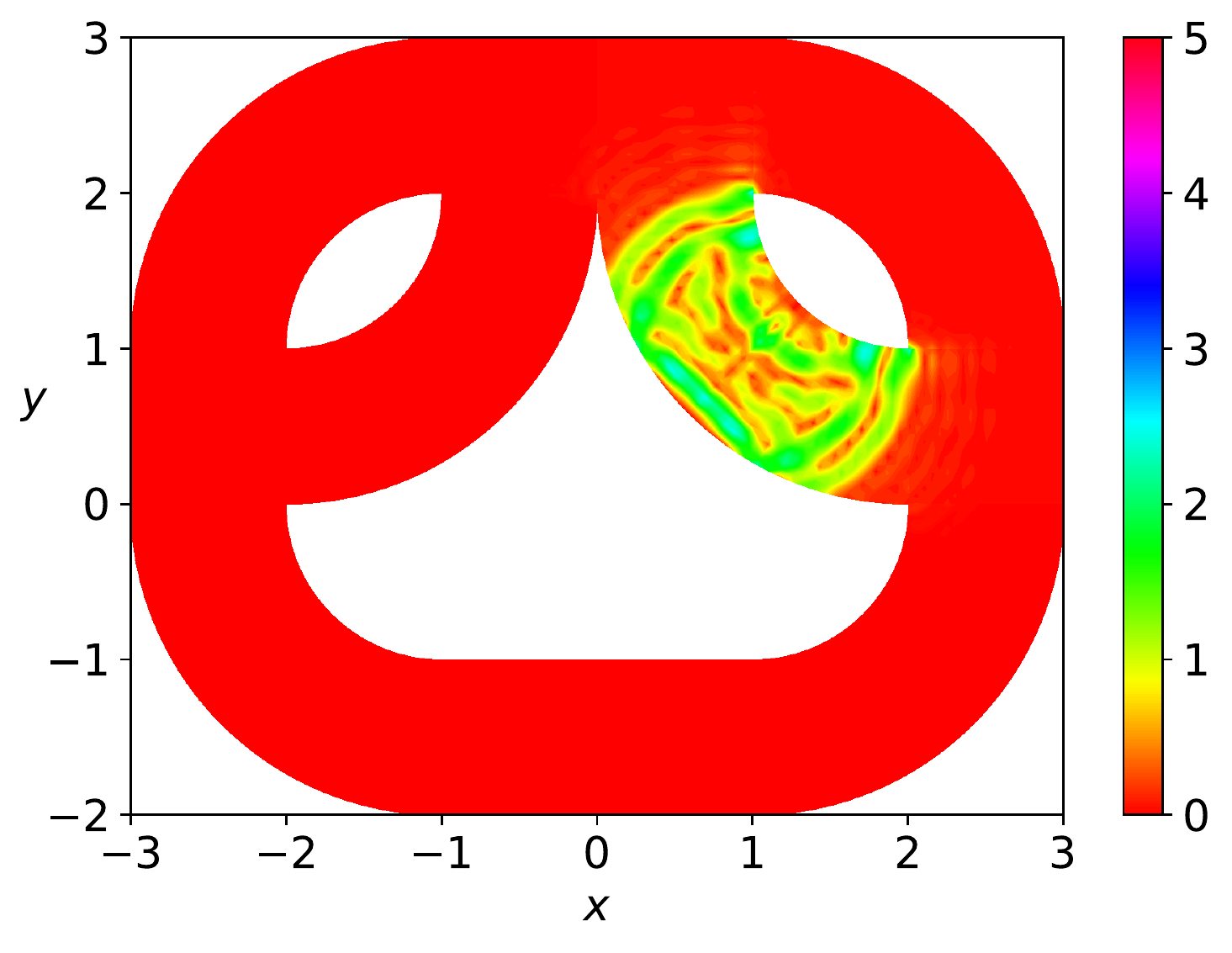}
  \hspace{15pt}
    \includegraphics[width=0.4\textwidth]{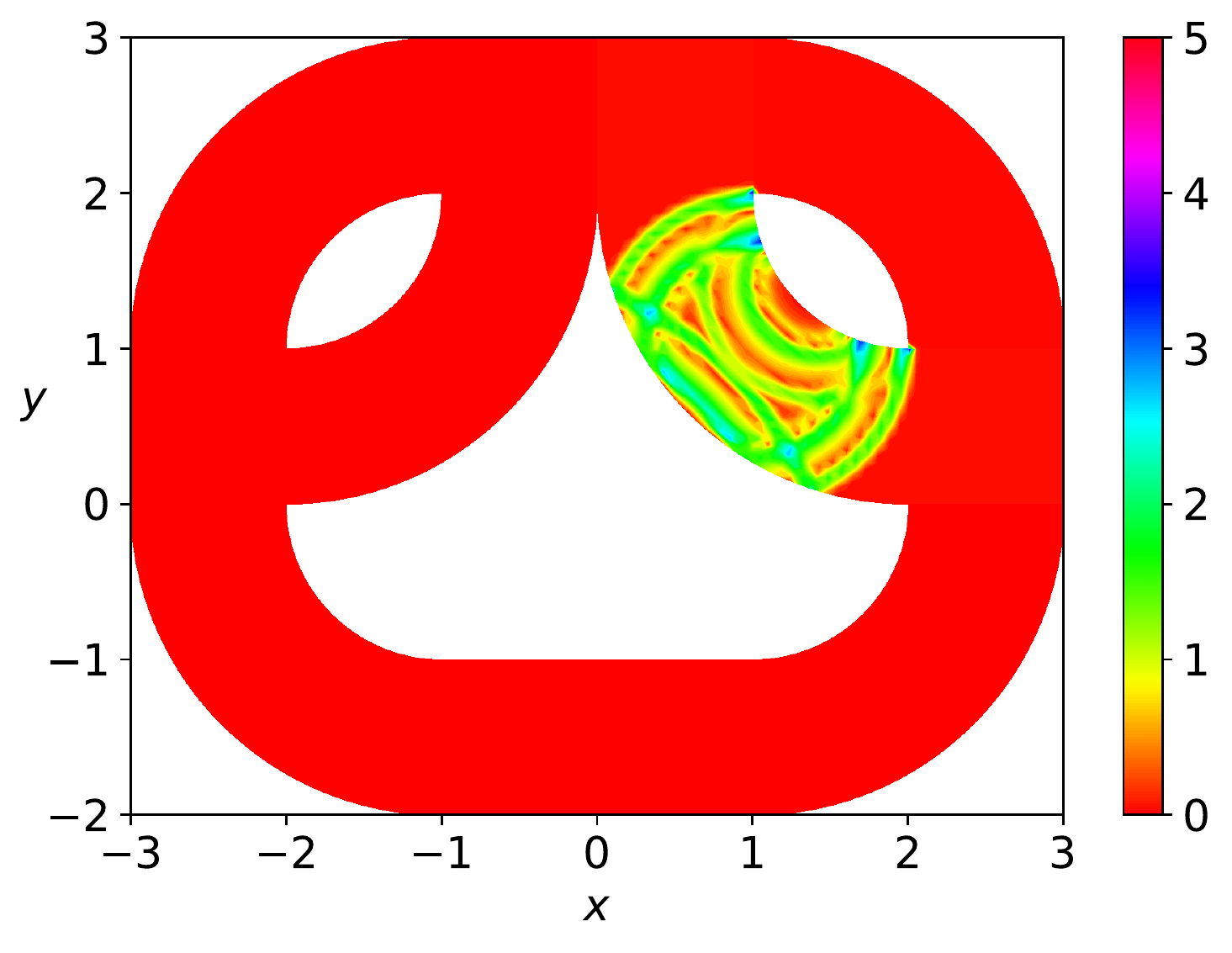}
  \\
    \includegraphics[width=0.4\textwidth]{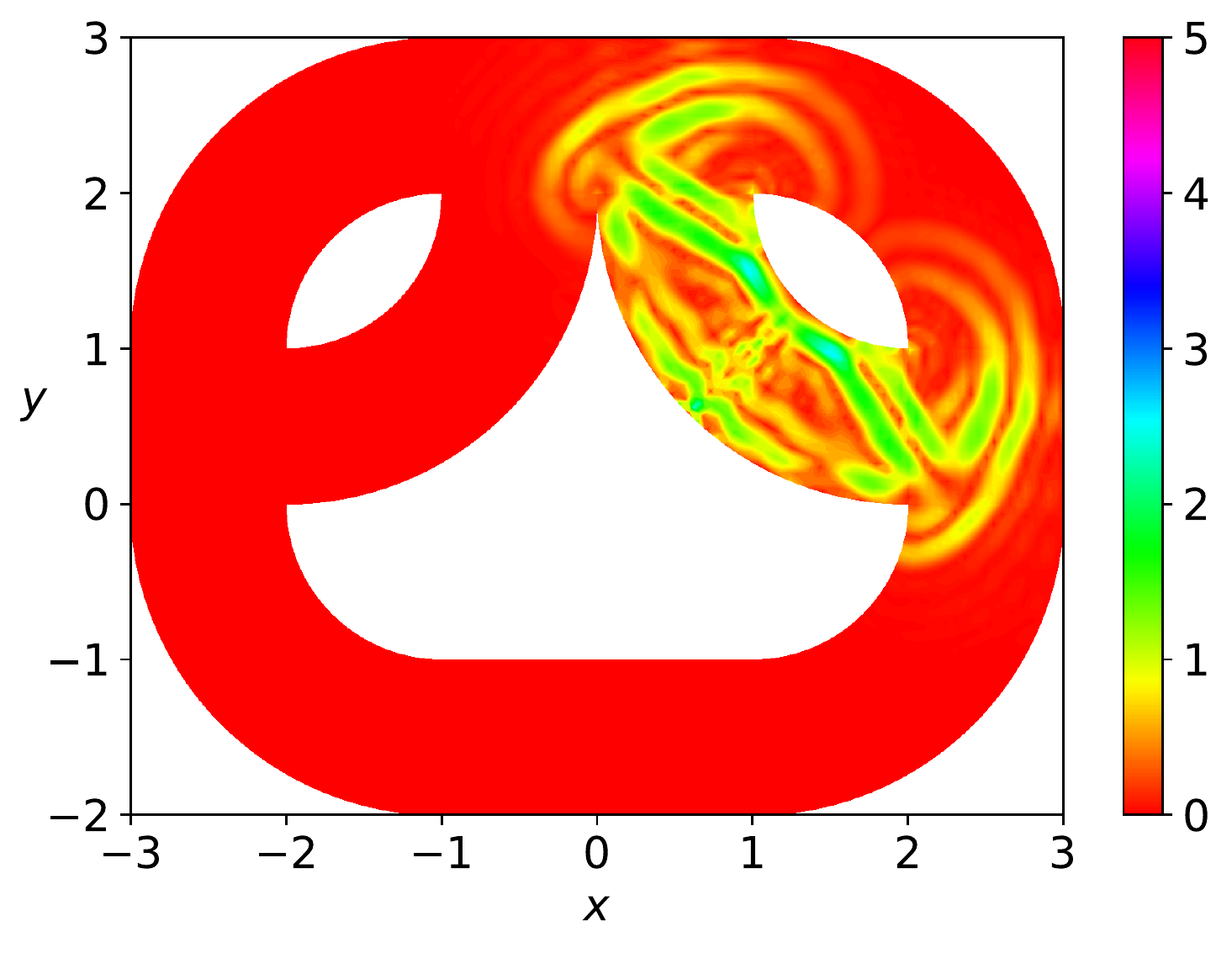}
  \hspace{15pt}
    \includegraphics[width=0.4\textwidth]{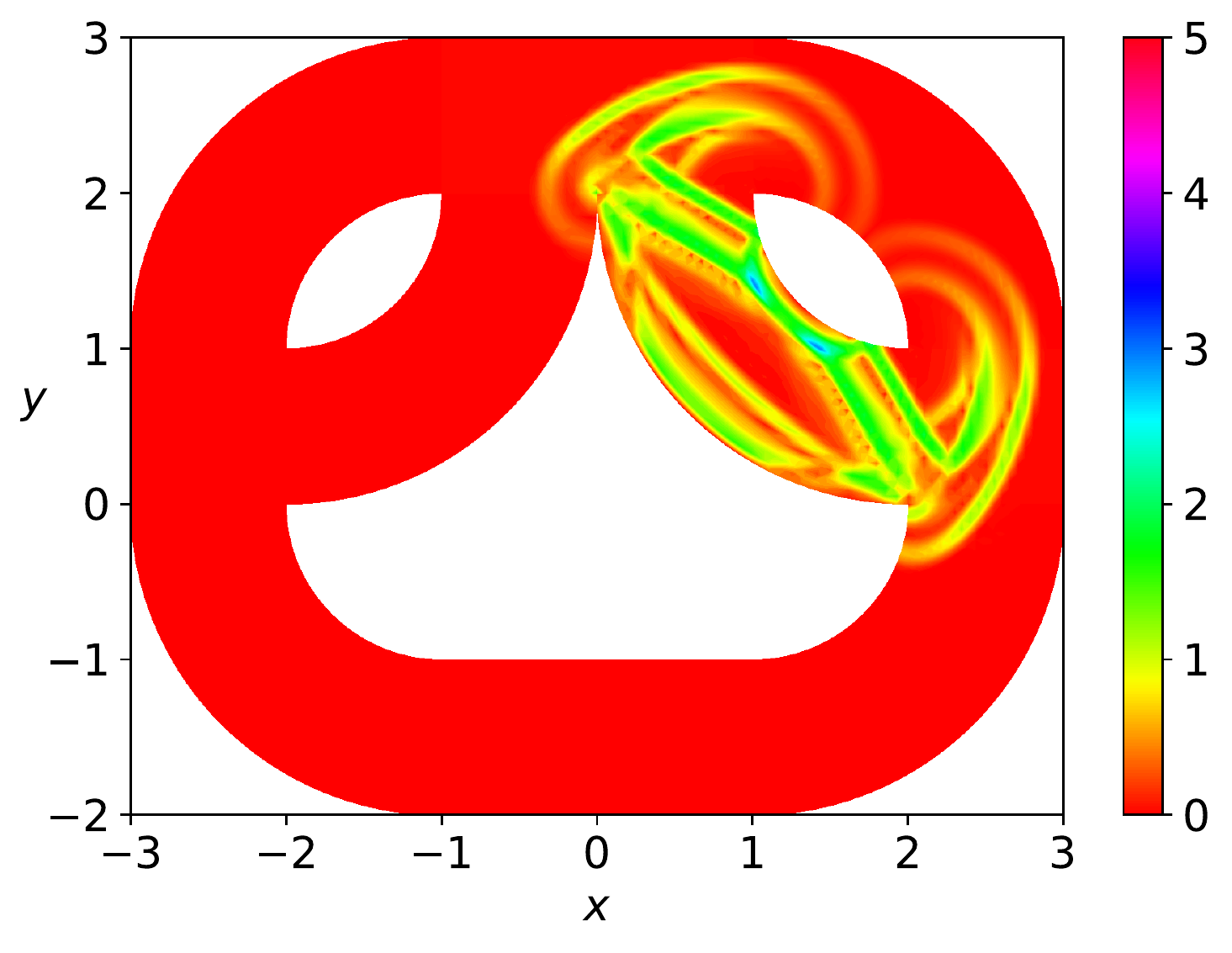}
  \\
    \includegraphics[width=0.4\textwidth]{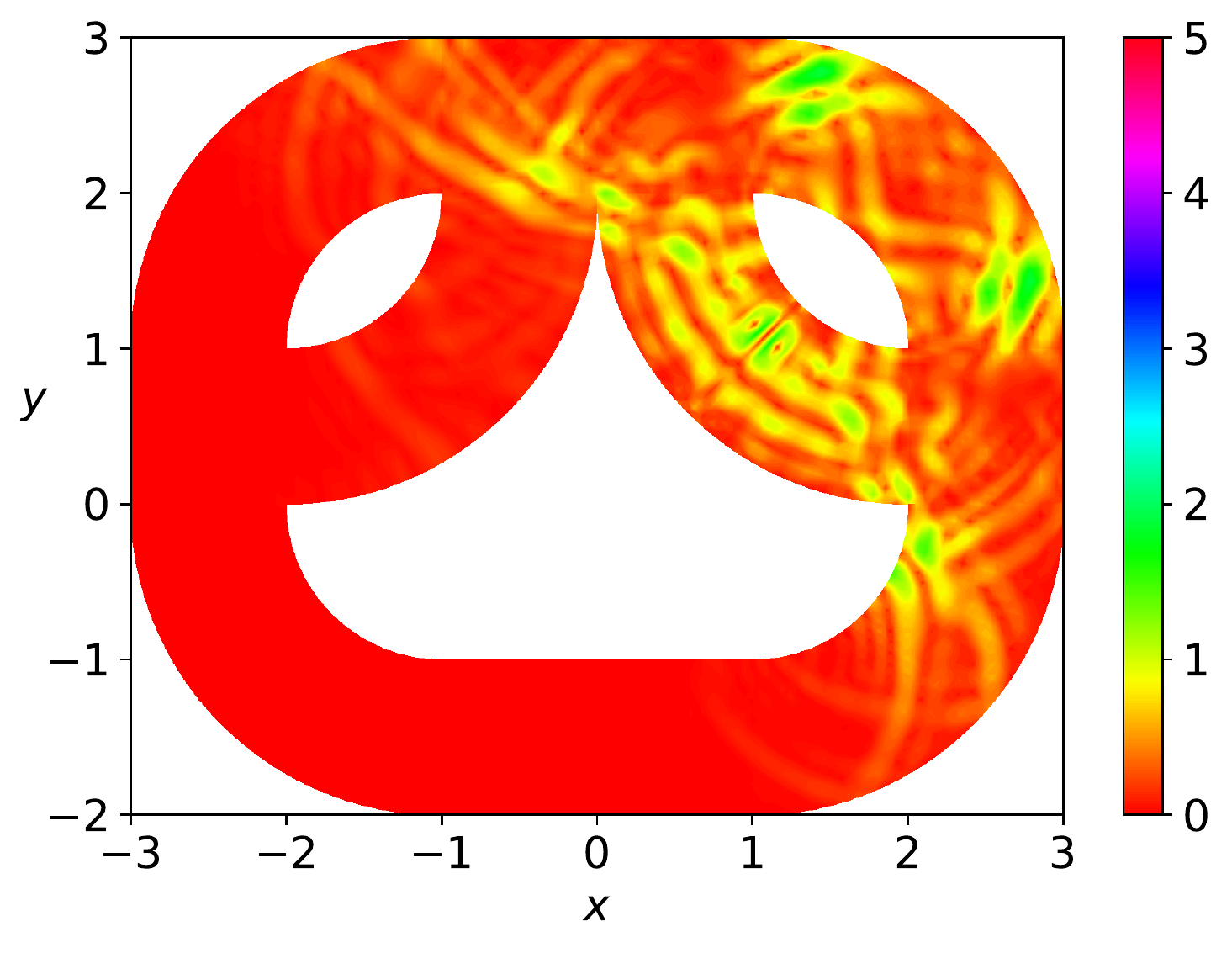}
  \hspace{15pt}
    \includegraphics[width=0.4\textwidth]{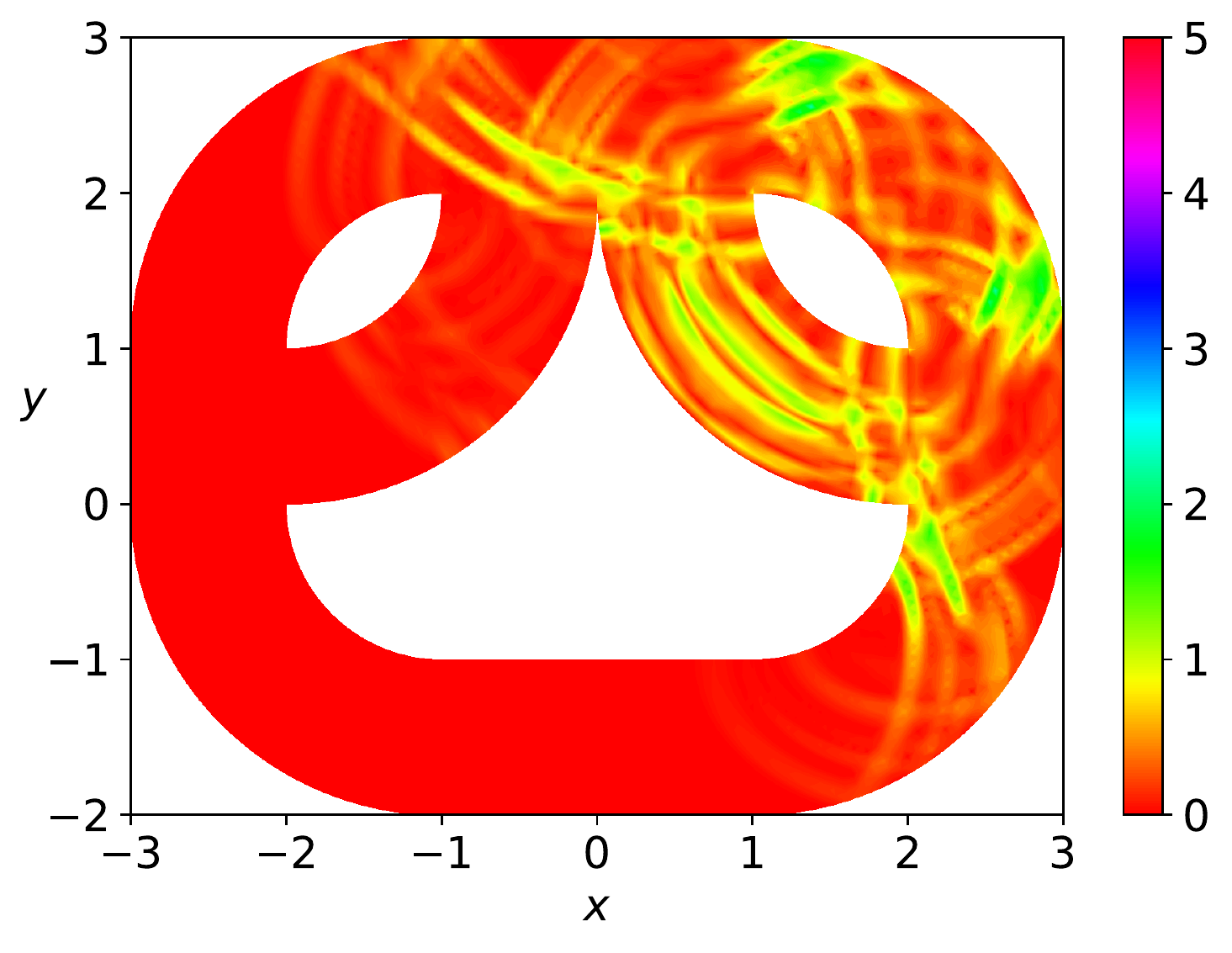}
\end{center}
\caption{
Snapshots of numerical solutions to the time-dependent Maxwell test-case
\eqref{td_tc1_E0}--\eqref{td_tc1_B0_J}
at $t = 0$, $0.4$, $0.8$, $1.6$ and $3.2$ (from top to bottom) 
discretized with the CONGA method~\eqref{AF_h} using spline elements of degree 
$p \times p$ and $N \times N$ cells per patch, and a leap-frog time stepping. 
The solution on the left corresponds to $N = 8$ and $p=3$, while 
the one on the right has been obtained $N = 20$ and $p=6$.
}
\label{fig:td_max_sols}
\end{figure}

\begin{figure}[!htbp]
\begin{center}
  \subfloat[Energy $\cH_h(t^n)$ for $N = 8$, $p=3$]{%
    \includegraphics[width=0.4\textwidth]{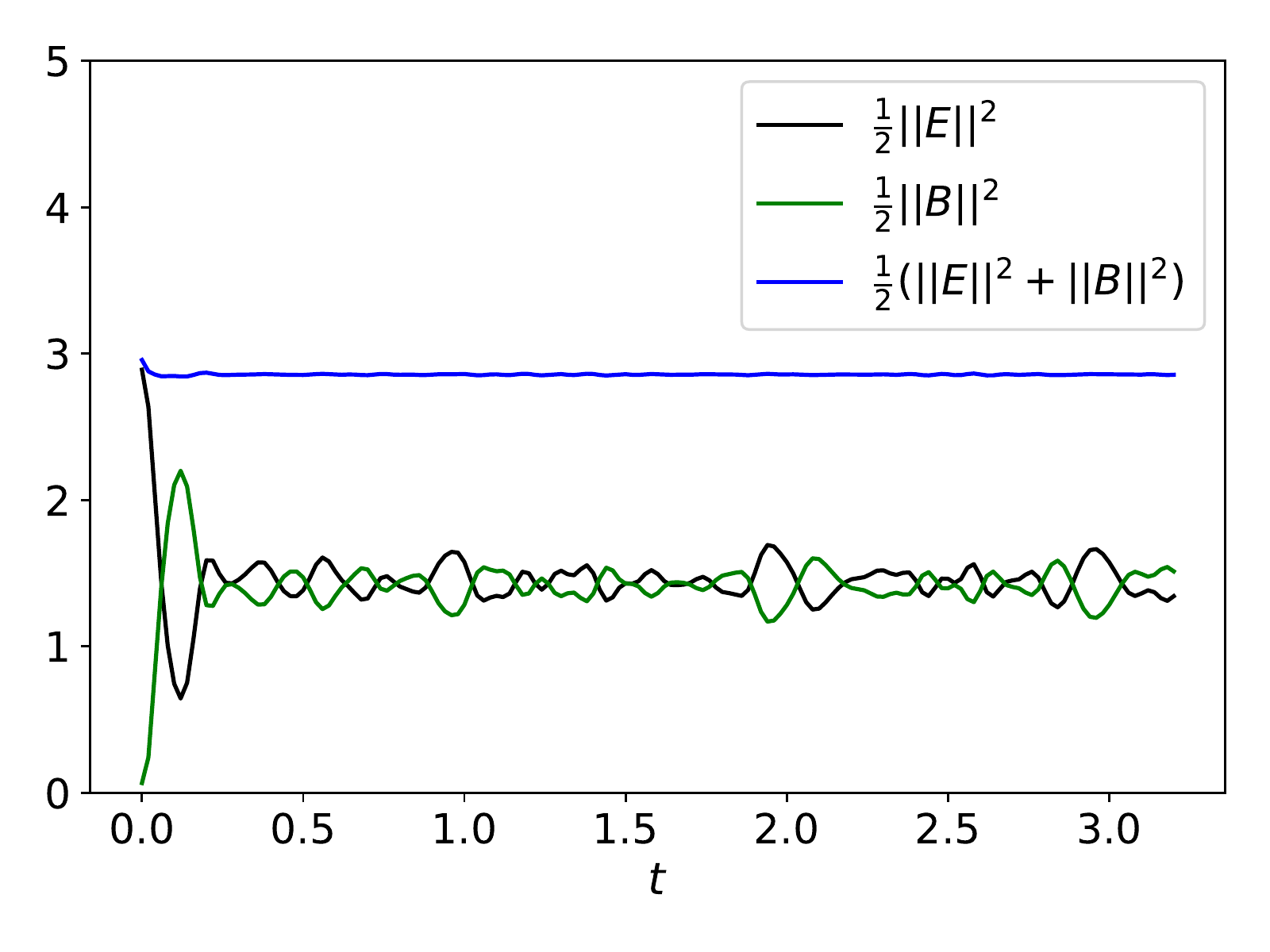}
  }
  \hspace{15pt}
  \subfloat[Energy $\cH_h(t^n)$ for $N = 20$, $p=6$]{%
    \includegraphics[width=0.4\textwidth]{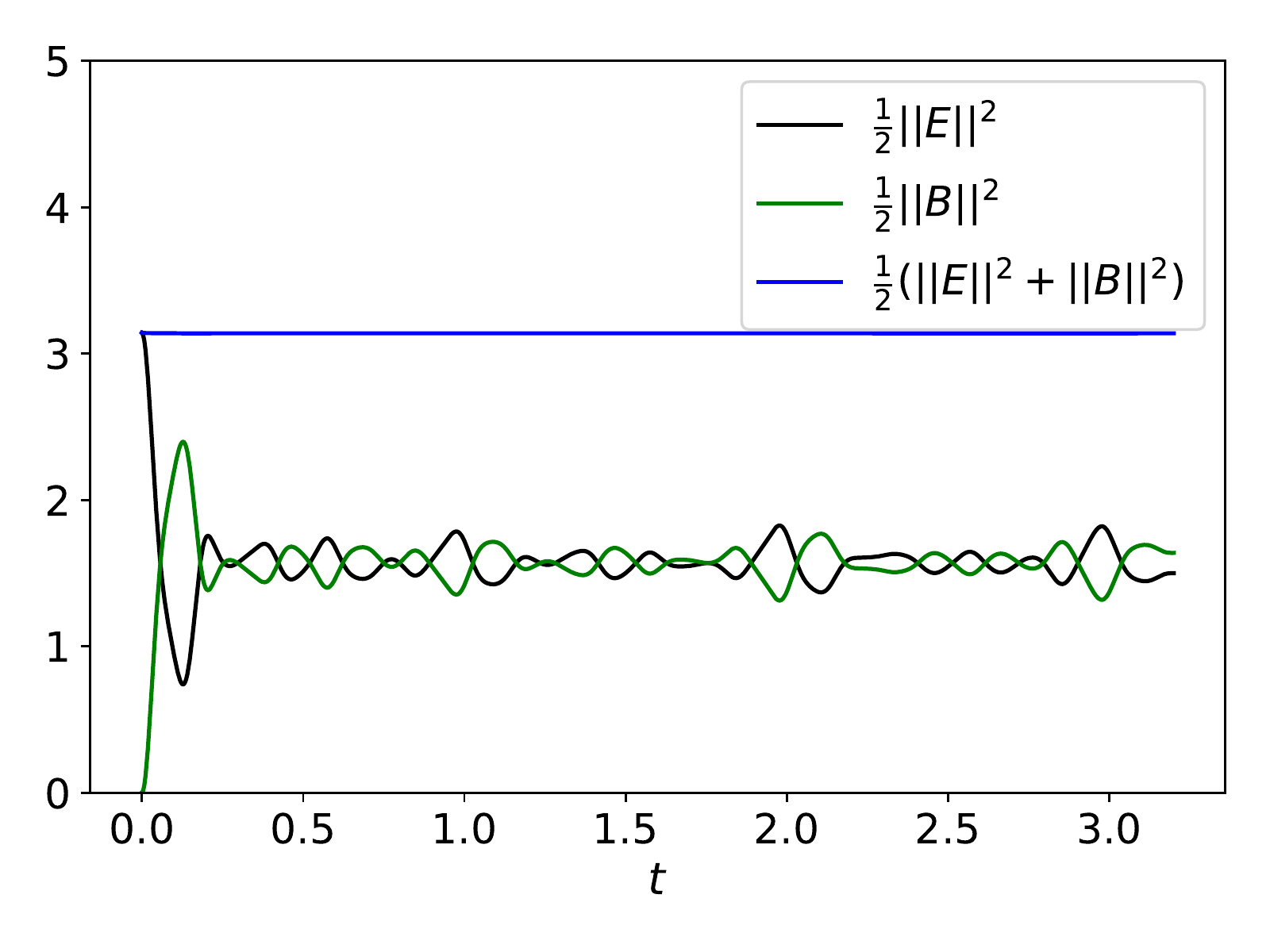}
  }
  \\
  \subfloat[$\norm{\wt \Div_h \bE_h(t^n)}$ for $N = 8$, $p=3$]{%
    \includegraphics[width=0.4\textwidth]{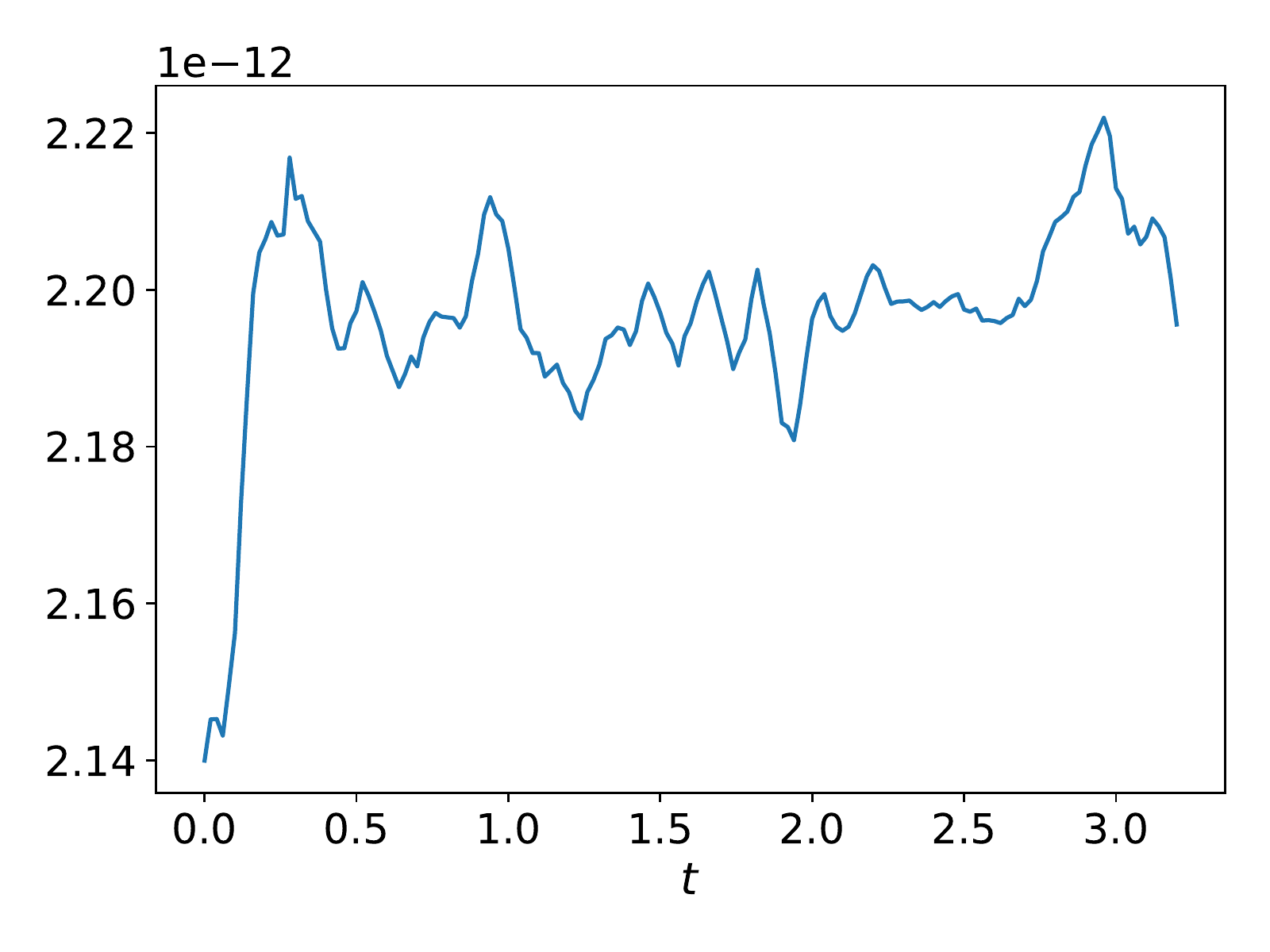}
  }
  \hspace{15pt}
  \subfloat[$\norm{\wt \Div_h \bE_h(t^n)}$ for $N = 20$, $p=6$]{%
    \includegraphics[width=0.4\textwidth]{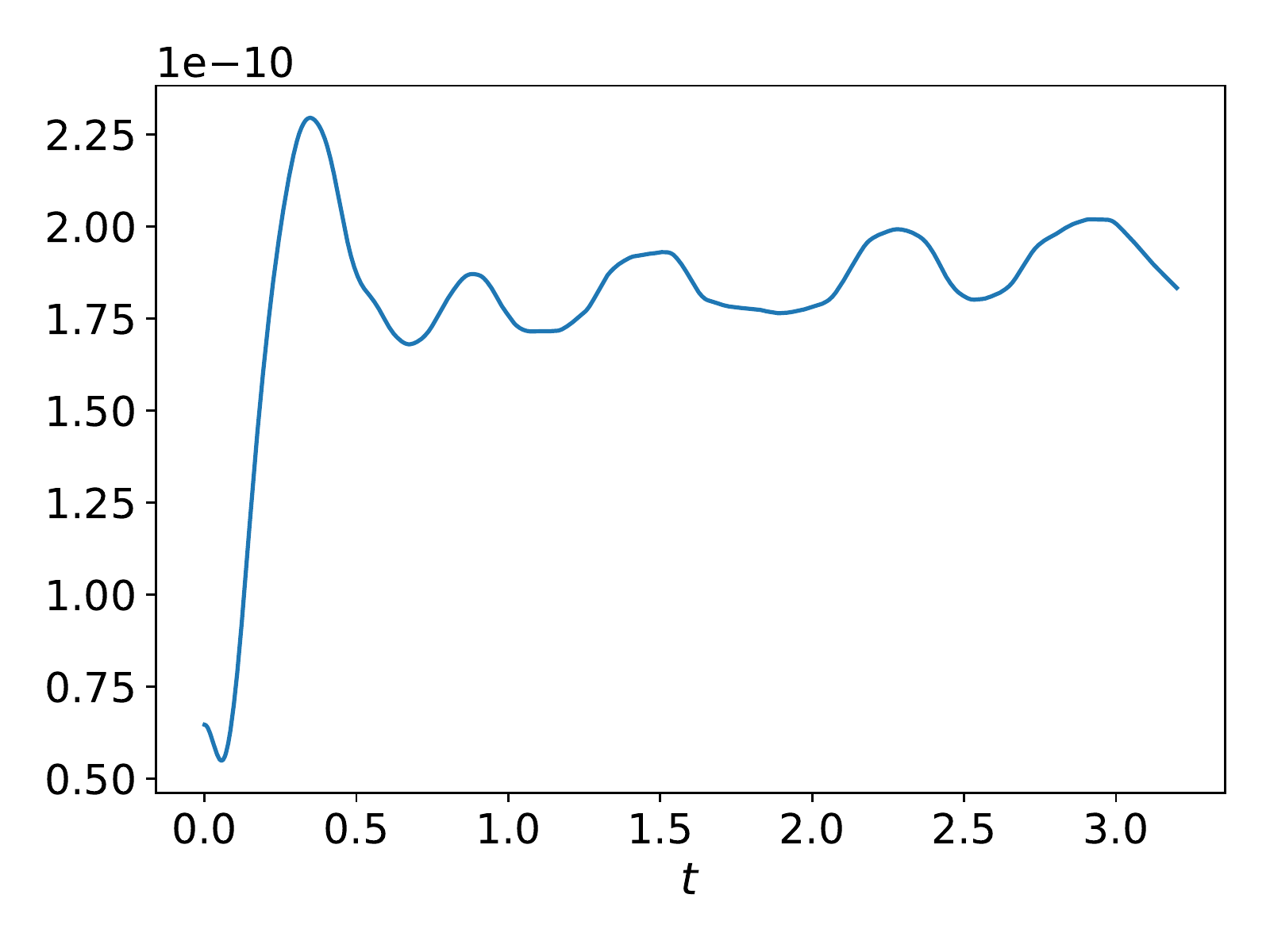}
  }
\end{center}
\caption{
Energy (top) and discrete divergence amplitude (bottom) of the coarse (left) and fine (right) numerical solutions
of the time-dependent Maxwell test-case \eqref{td_tc1_E0}--\eqref{td_tc1_B0_J} shown in Figure\ref{fig:td_max_sols}.
}
\label{fig:td_max_diags}
\end{figure}

We next study the quality of the source approximation operator 
$\bJ \to \bJ_h$ which is involved in the discrete Ampere equation \eqref{AF_h}.
For this we use a second time-dependent test-case with a zero initial condition 
\begin{equation} \label{td_tc2_EB0}
  \bE(t=0) = 0, \qquad \text{ and } \qquad
  B(t=0) = 0
\end{equation}
and a source of the form
\begin{equation} \label{td_tc2_J}
\bJ(t) = \bcurl \psi - \cos(\omega t) \grad \psi
\end{equation}
with $\psi$ as in \eqref{td_tc1_E0}.
The associated charge density is then
$$
\rho(t) = -\int_0^t \Div \bJ(t') \rmd t' = \omega^{1}\sin(\omega t) \Delta \psi.
$$
In Figures~\ref{fig:td_max_sols_J}--\ref{fig:td_max_diags_J} 
we compare three different approximation operators for the current source, namely:
\begin{itemize}
  \item[(i)] the primal finite element projection: $\bJ_h = \Pi^1_h \bJ$, 
  \item[(ii)] the $L^2$ projection on the broken space: $\bJ_h = P_{V^1_h} \bJ$, 
  \item[(iii)] the dual projection: $\bJ_h = \t \Pi^1_h \bJ$.
\end{itemize}
We note that each of these projection operators are local, in the sense that none 
requires solving a global problem on the computational domain $\Omega$.
We also remind that the primal projection $\Pi^1_h$ interpolates the geometric (edge) 
degrees of freedom and satisfies a commuting diagram property with the primal (strong) 
differential operators, but not with the dual ones. Hence it does not allow to preserve the 
discrete electric Gauss law in \eqref{Gauss_h}. In contrast, both the $L^2$ 
projection on the broken space $V^1_h$ and the dual projection $\t \Pi^1_h$
guarantee the preservation of the discrete Gauss laws,
however we expect an increased stability for the latter one
as discussed in Remark~\ref{rem:Gauss_h}.

In Figure~\ref{fig:td_max_sols_J} we first show some snapshots of the 
three numerical solutions on a time range $t \in [0,T]$ with $T=20$: 
There we see that the primal projection yields a very strong and steadily growing field 
in the region of the source (visible from the changing color scale)
which points towards a large error.
In contrast, the $L^2$ and the dual projections produce solutions with moderate amplitude
with some visible differences, namely an electric field that also builds up in the region where the source is located.


To better analyse the quality of these simulations we next show in Figure~\ref{fig:td_max_diags_J}
two error indicators for each one of the numerical strategies described above, namely
the Gauss law errors associated with the broken solution $\cG^n_h(\bE^n_h)$, and that of its conforming
projection $\cG^n_h(P^1_h \bE^n_h)$, where we have denoted
\begin{equation} \label{Gnh}
  \cG^n_h(\bF_h) = \norm{\wt \Div_h \bF_h - \t\Pi^0 \rho(t^n)}
  \qquad \text{ for } ~ \bF_h \in V^1_h.  
\end{equation}
Here the former errors are expected to be zero for both the $L^2$ and dual projections,
see Remark~\ref{rem:Gauss_h}: The numerical results confirm this value, up to machine accuracy,
whereas they show significant errors for the primal projection shown in the left plots.
As for the second error, it has no reason to be strictly zero but should remain small
for accurate and stable solutions, hence it is also an interesting error indicator. 
Here the curves show very large values (around $80$ at $t=T$) for the primal projection operator $\Pi^1_h$,
with a linear time growth (which is somehow consistent with the strong growth of the former 
error indicator $\cG^n_h(\bE^n_h)$). For the $L^2$ projection the error is smaller but it is far from being 
negligible (close to $3$), and also grows linearly in time. In contrast, the error is much smaller 
(on the order of $0.01$) for the dual projection $\t \Pi^1_h$, and it oscillates but does not seem to grow.
These results tend to indicate that the growing field visible in 
Figure~\ref{fig:td_max_sols_J} corresponds to a numerical error, and they
highlight the improved stability of the CONGA scheme with a proper source approximation.

\def \plotdirGeom  {td_maxwell_Issautier_like_source_J_proj=P_geom_qp4_nb_tau=100}
\def \plotdirPnofil{td_maxwell_Issautier_like_source_J_proj=P_L2_qp4_nb_tau=100}
\def \plotdirPfil  {td_maxwell_Issautier_like_source_J_proj=tilde_Pi_qp4_nb_tau=100}

\begin{figure}[!htbp]

\begin{center}
    \includegraphics[width=0.3\textwidth]{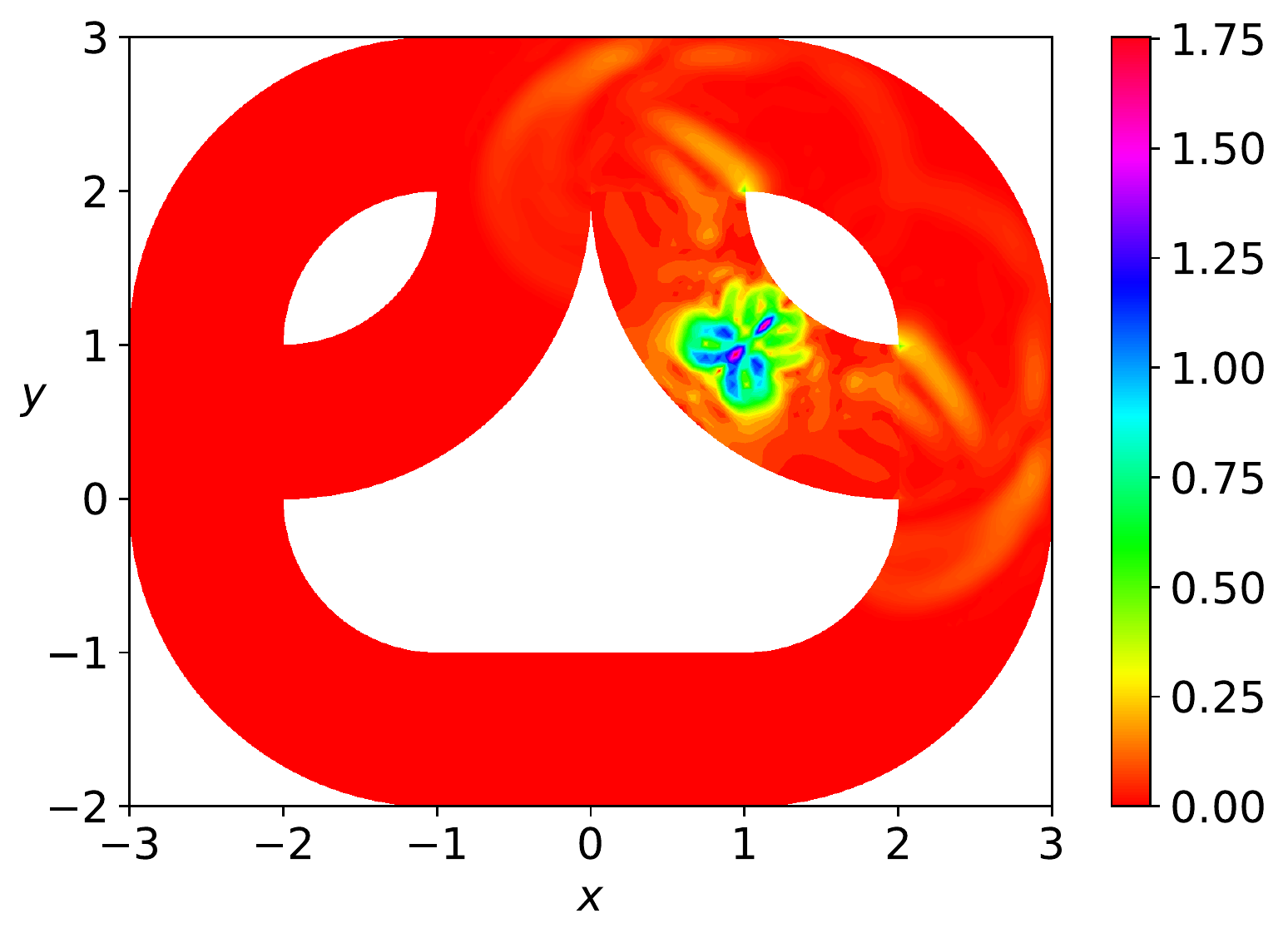}
  \hspace{10pt}
    \includegraphics[width=0.3\textwidth]{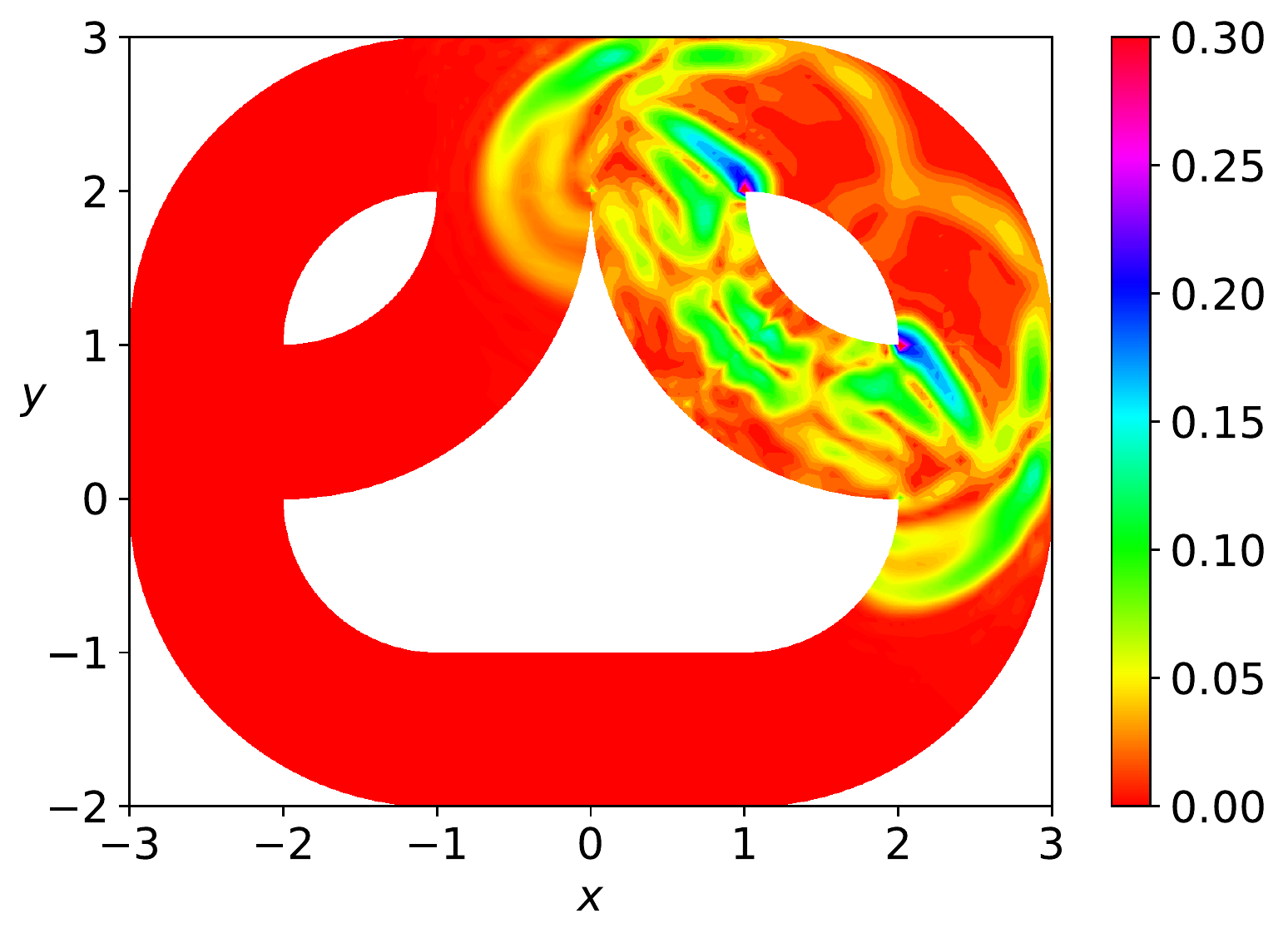}
  \hspace{10pt}
    \includegraphics[width=0.3\textwidth]{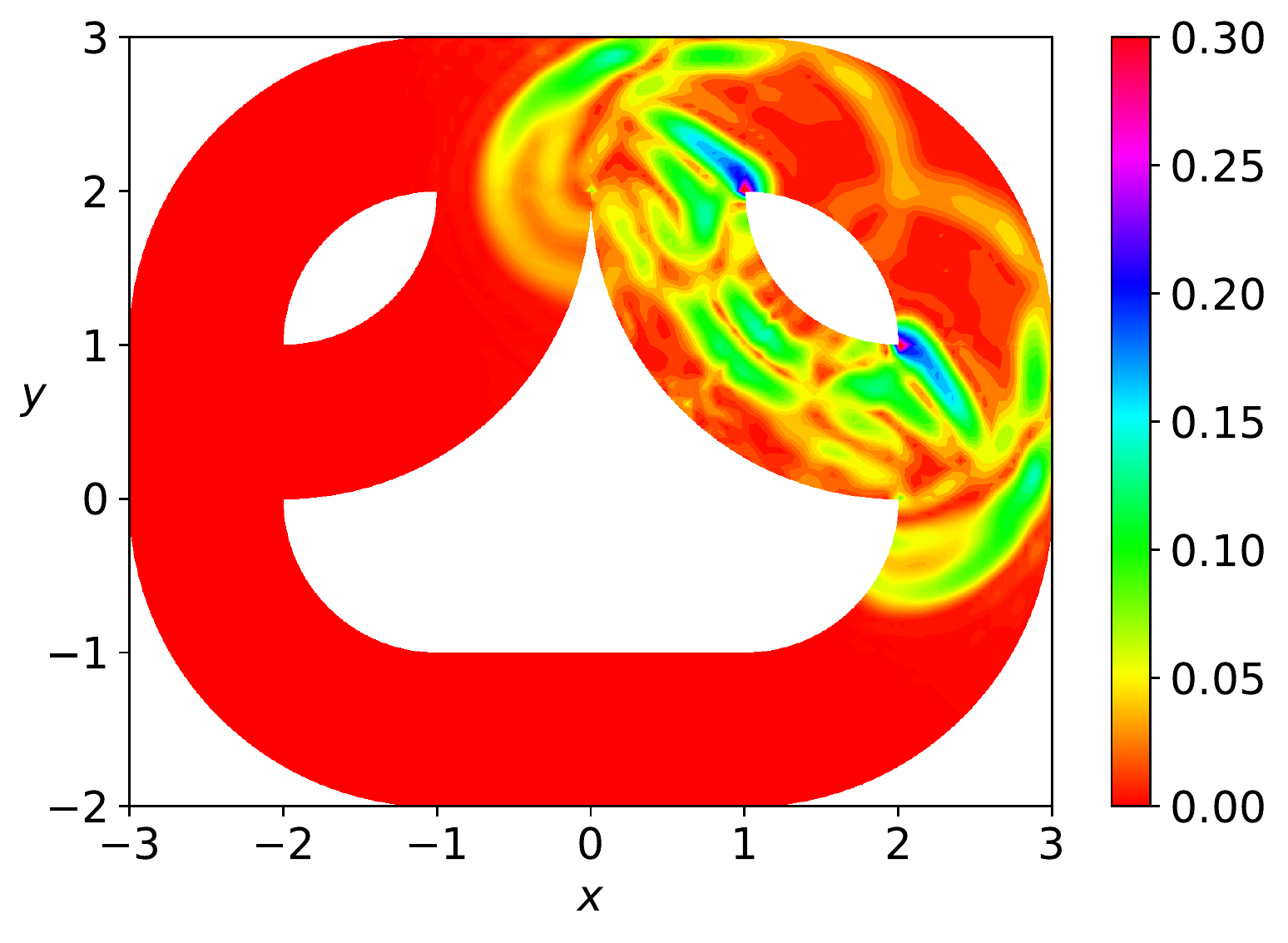}
  \\
    \includegraphics[width=0.3\textwidth]{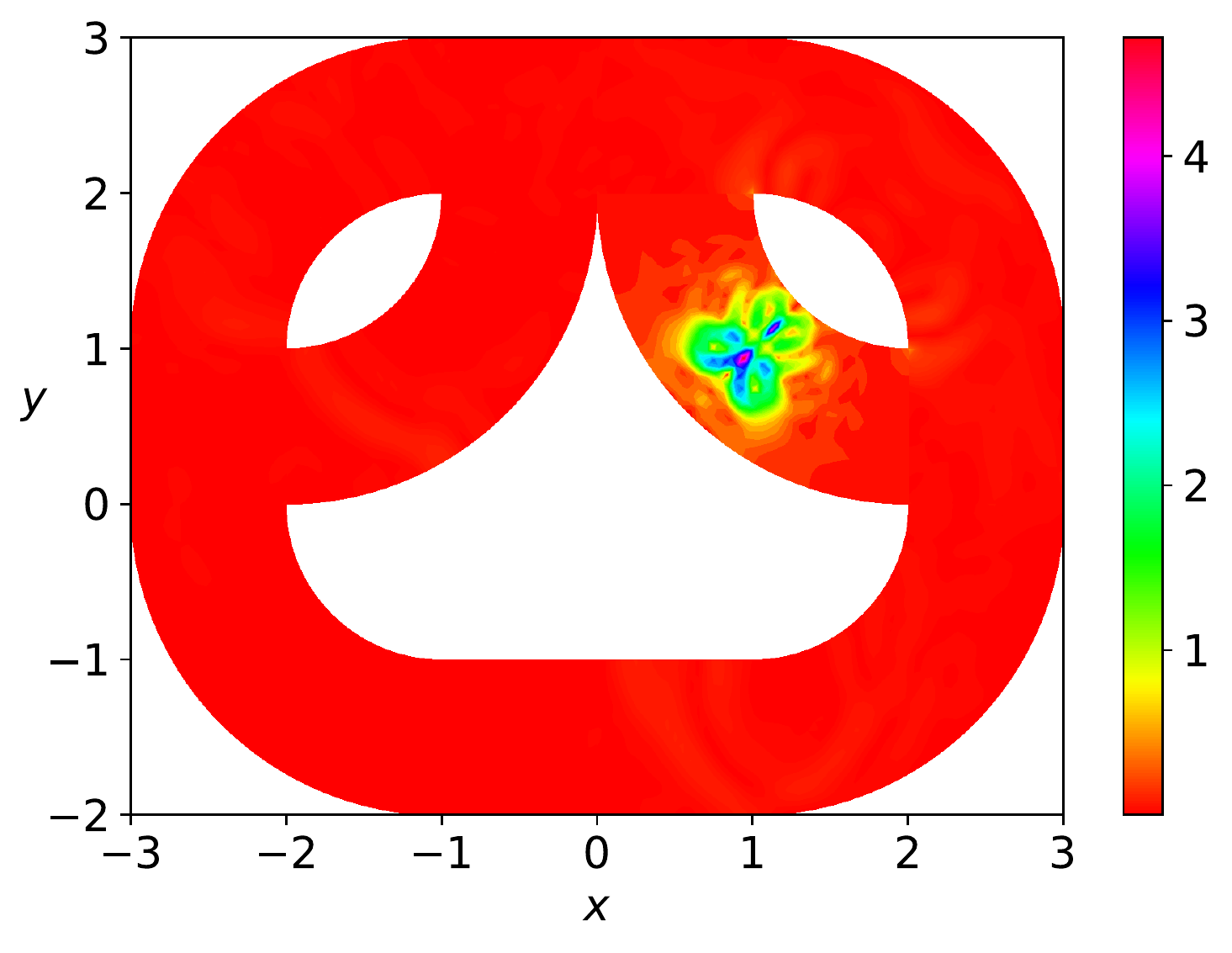}
  \hspace{10pt}
    \includegraphics[width=0.3\textwidth]{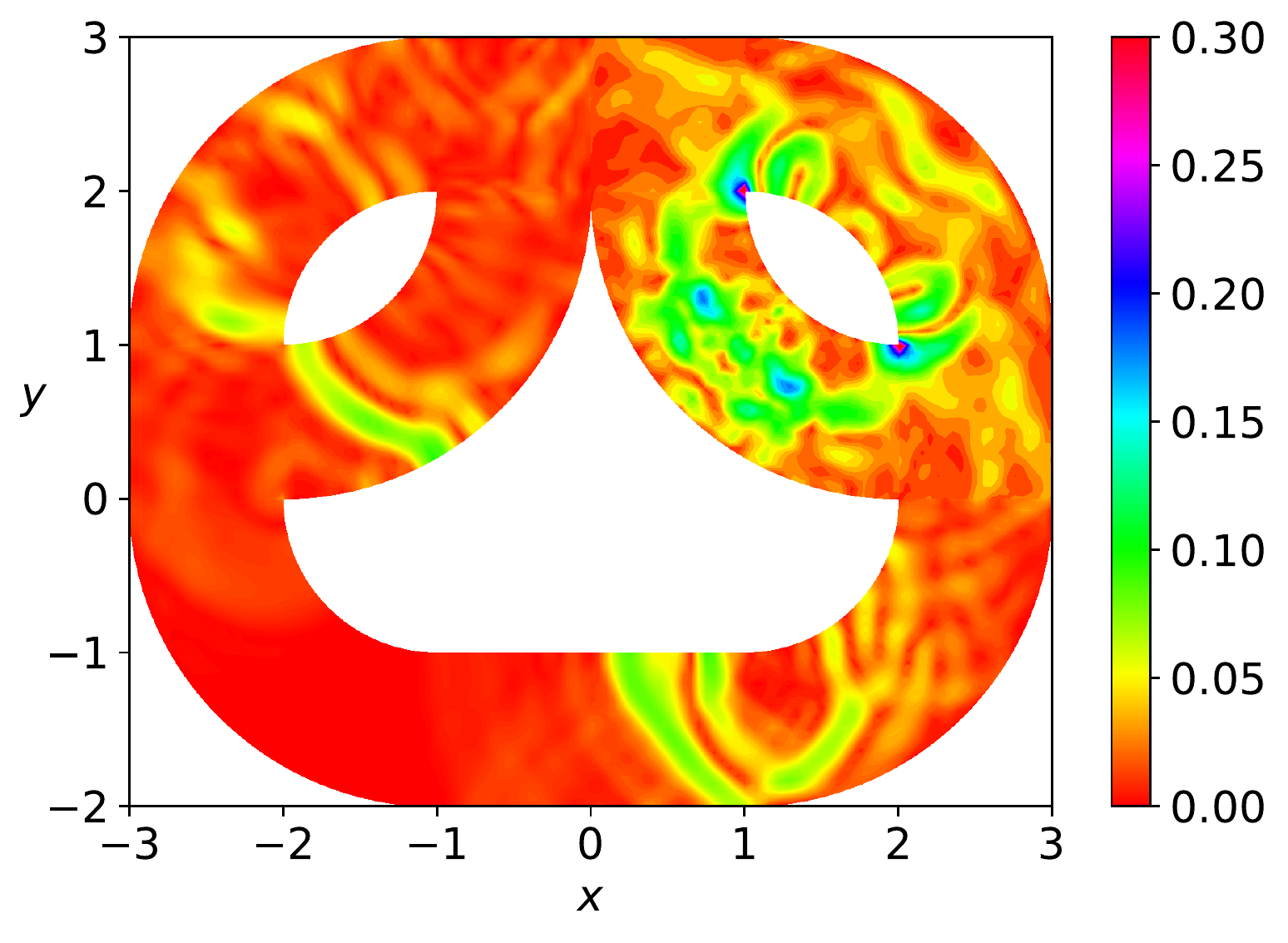}
  \hspace{10pt}
    \includegraphics[width=0.3\textwidth]{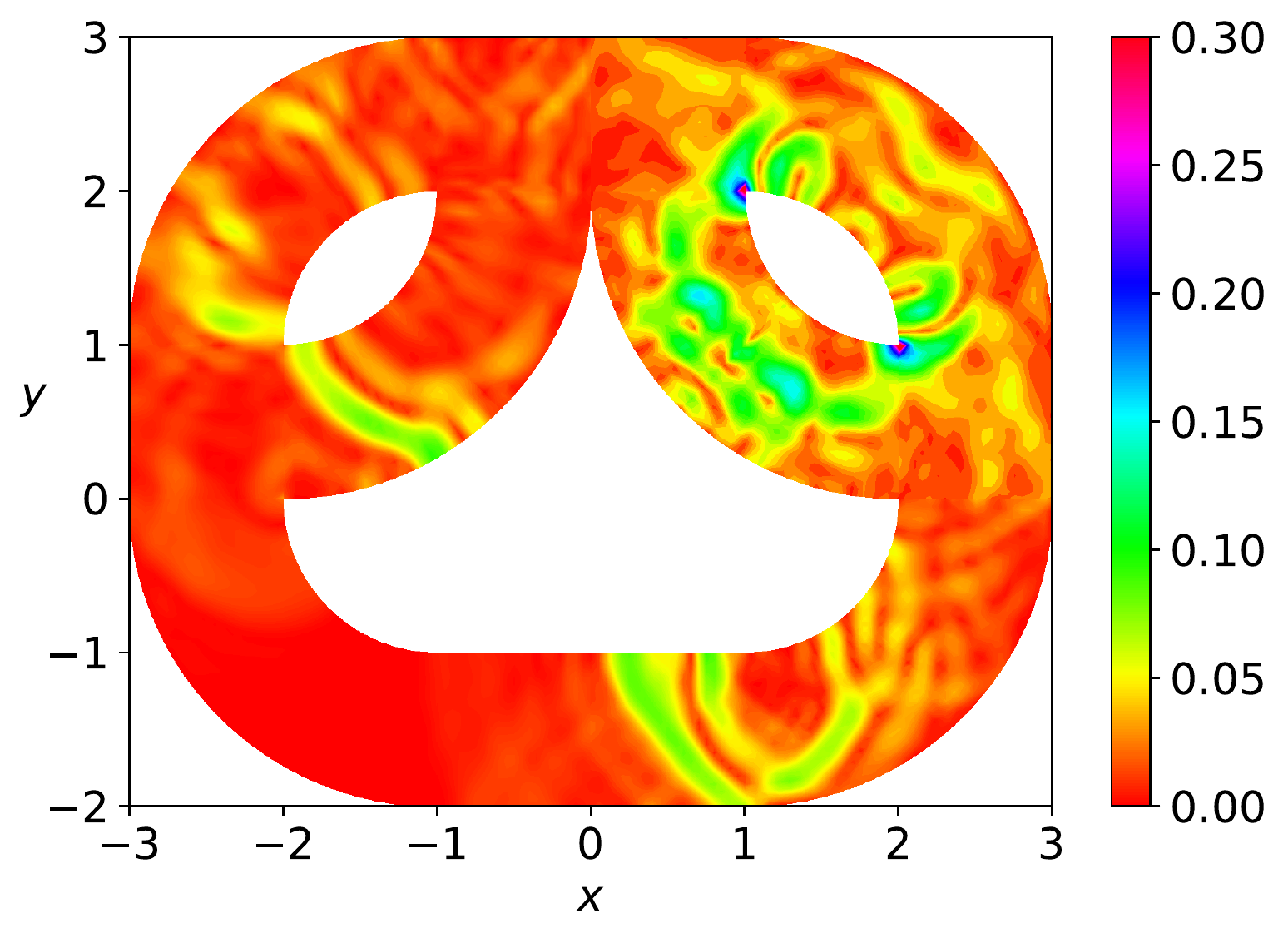}
  \\
    \includegraphics[width=0.3\textwidth]{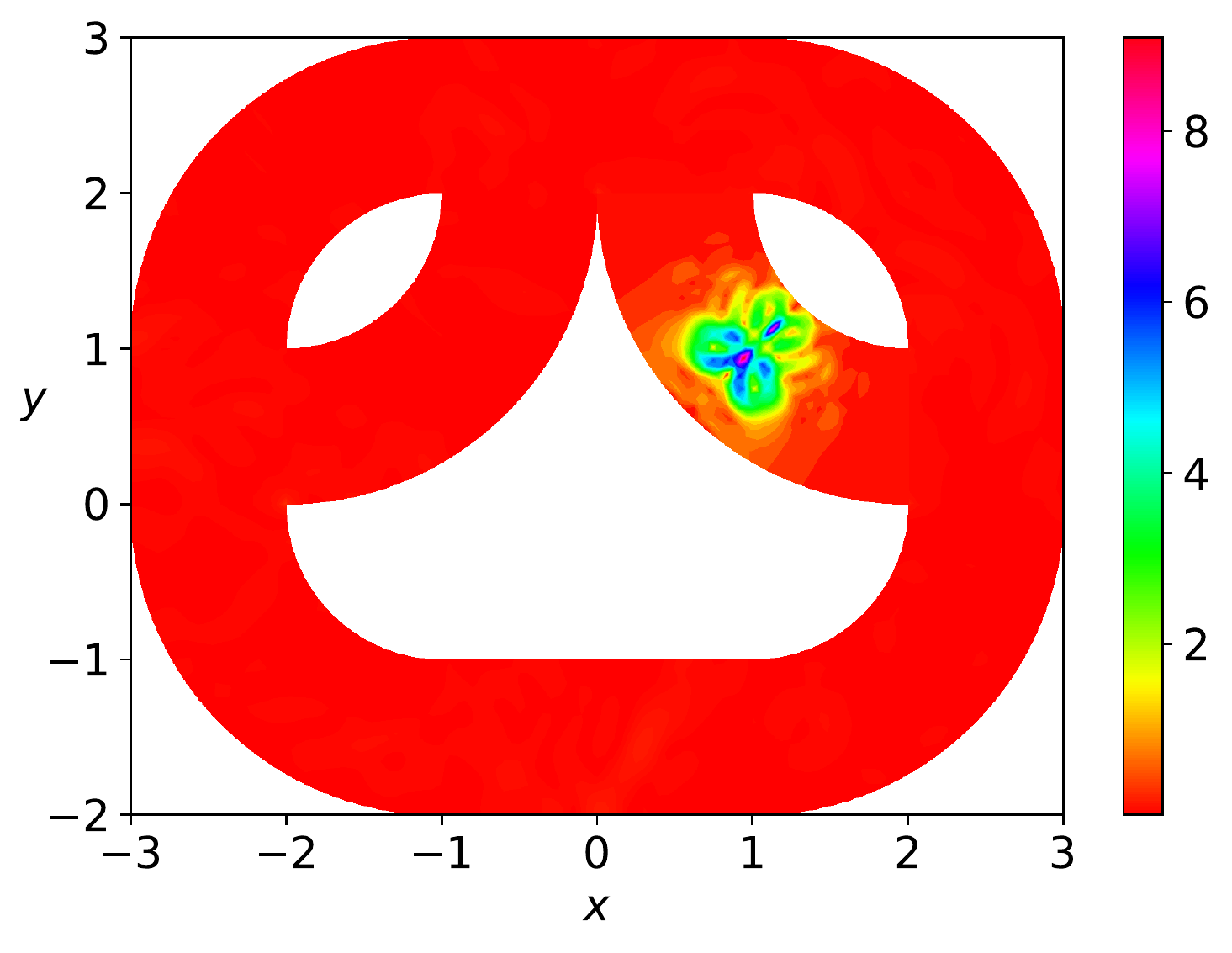}
  \hspace{10pt}
    \includegraphics[width=0.3\textwidth]{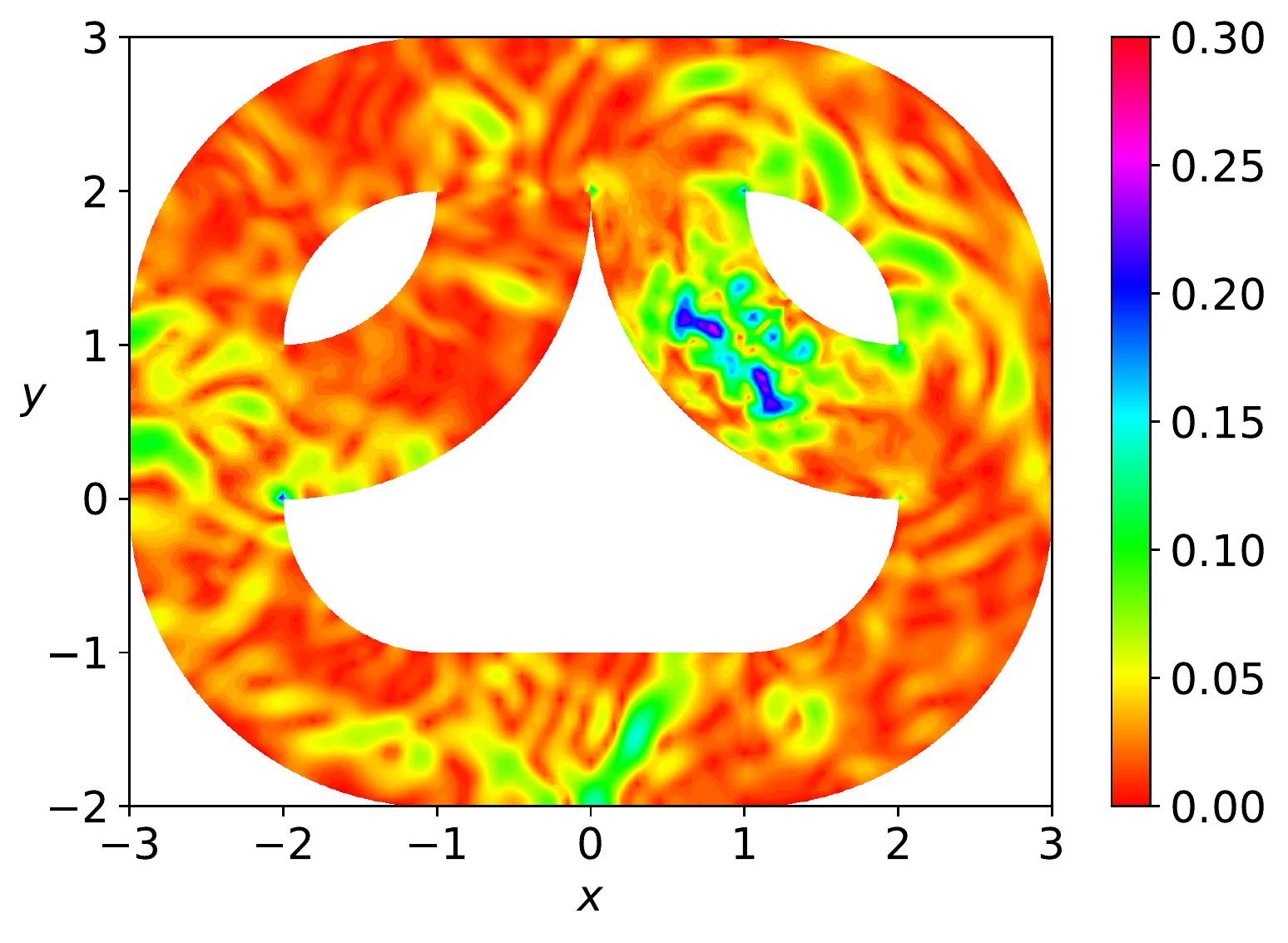}
  \hspace{10pt}
    \includegraphics[width=0.3\textwidth]{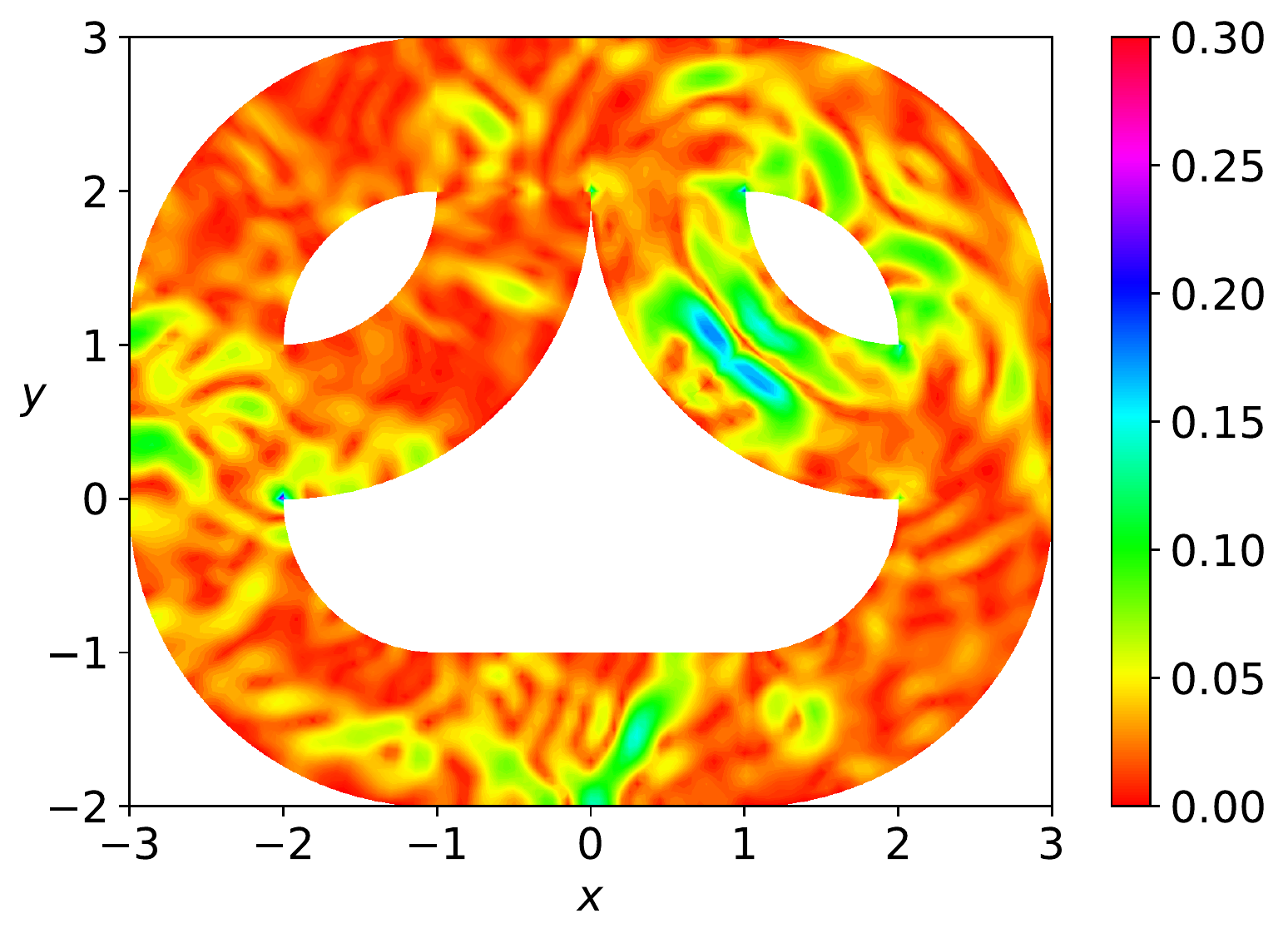}
  \\
    \includegraphics[width=0.3\textwidth]{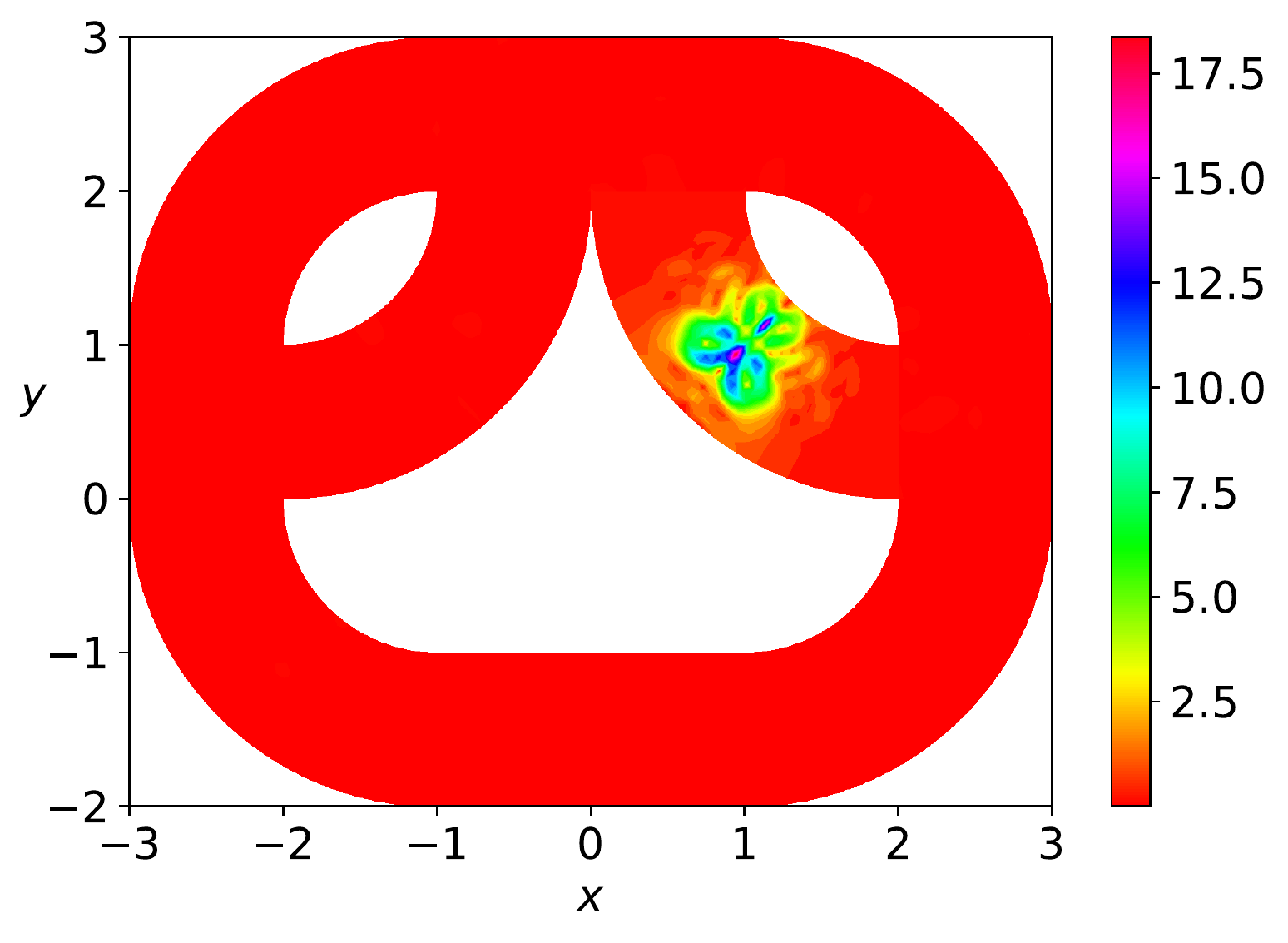}
  \hspace{10pt}
    \includegraphics[width=0.3\textwidth]{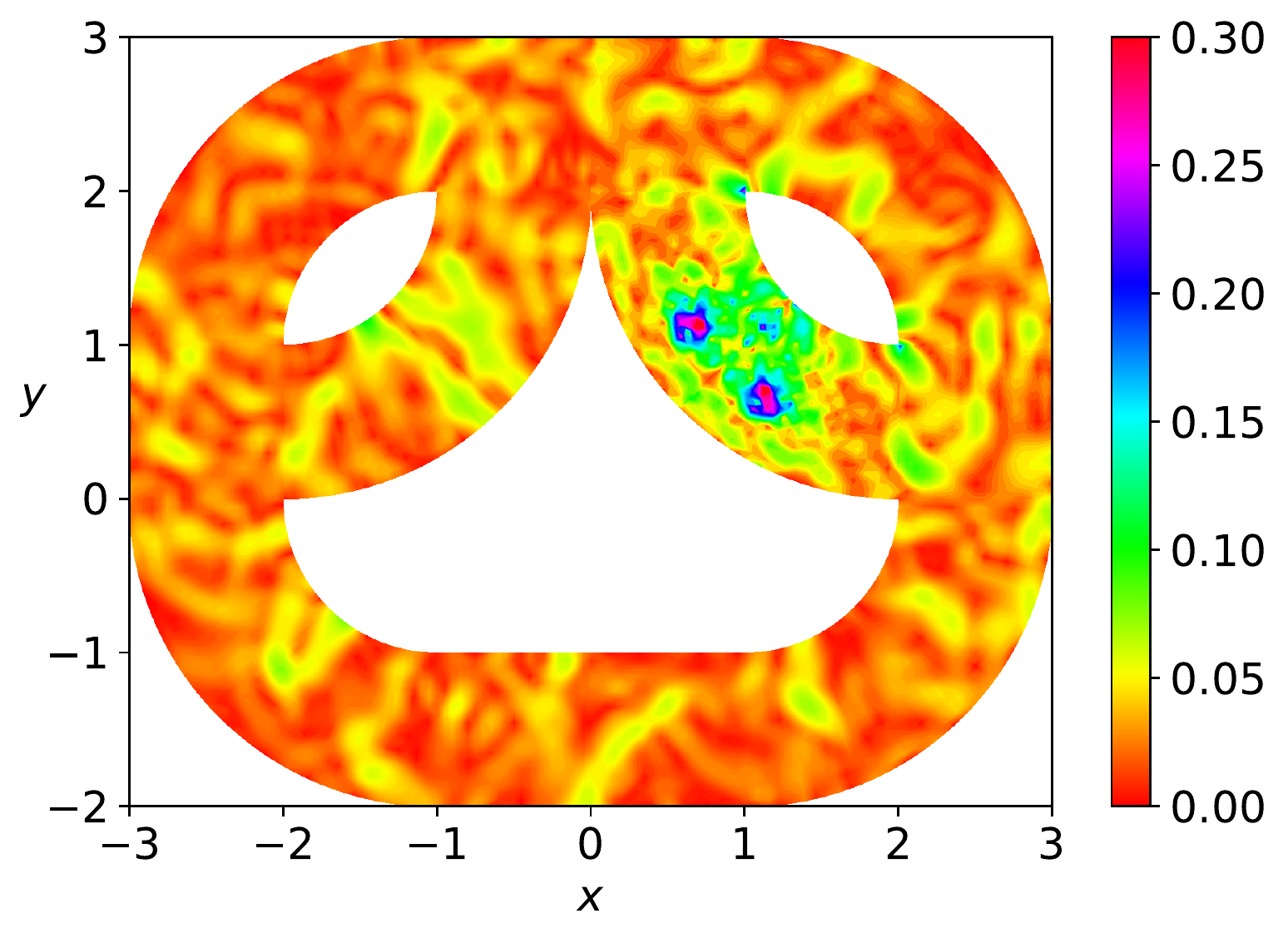}
  \hspace{10pt}
    \includegraphics[width=0.3\textwidth]{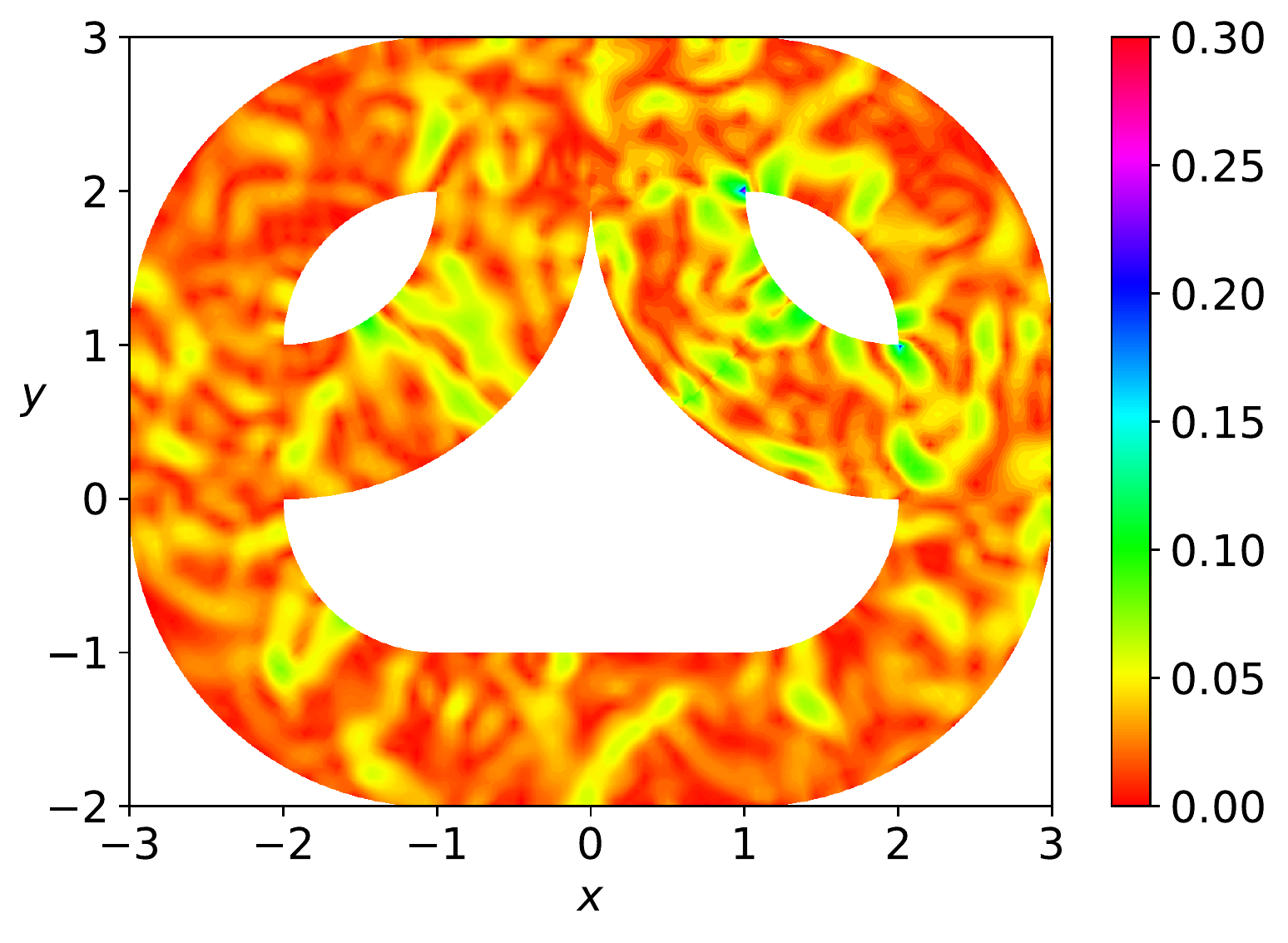}
\end{center}
\caption{
Snapshots of numerical solutions to the time-dependent Maxwell problem with
a time-varying source \eqref{td_tc2_J} discretized with the CONGA 
scheme~\eqref{AF_h} using spline elements of degree $3 \times 3$ and 
$8 \times 8$ cells per patch. 
The amplitudes $\abs{P^1_h\bE^n_h}$ are shown 
at $t = 2$, $5$, $10$ and $20$ (from top to bottom).
In the left panels the source projection $\bJ \to \bJ_h$ 
is a primal projection operator $\Pi^1_h$ which does not commute with the dual differential 
operators (note the time-varying color scale). 
The middle panels use an orthogonal projection $P_{V^1_h}$,
and the ones on the right a dual commuting projection $\t \Pi^1_h$. See the 
text for more details.
}
\label{fig:td_max_sols_J}
\end{figure}

\begin{figure}[!htbp]
\begin{center}
  \subfloat[$\cG^n_h(\bE^n_h)$ for $P_J = \Pi^1_h$]{%
    \includegraphics[width=0.3\textwidth]{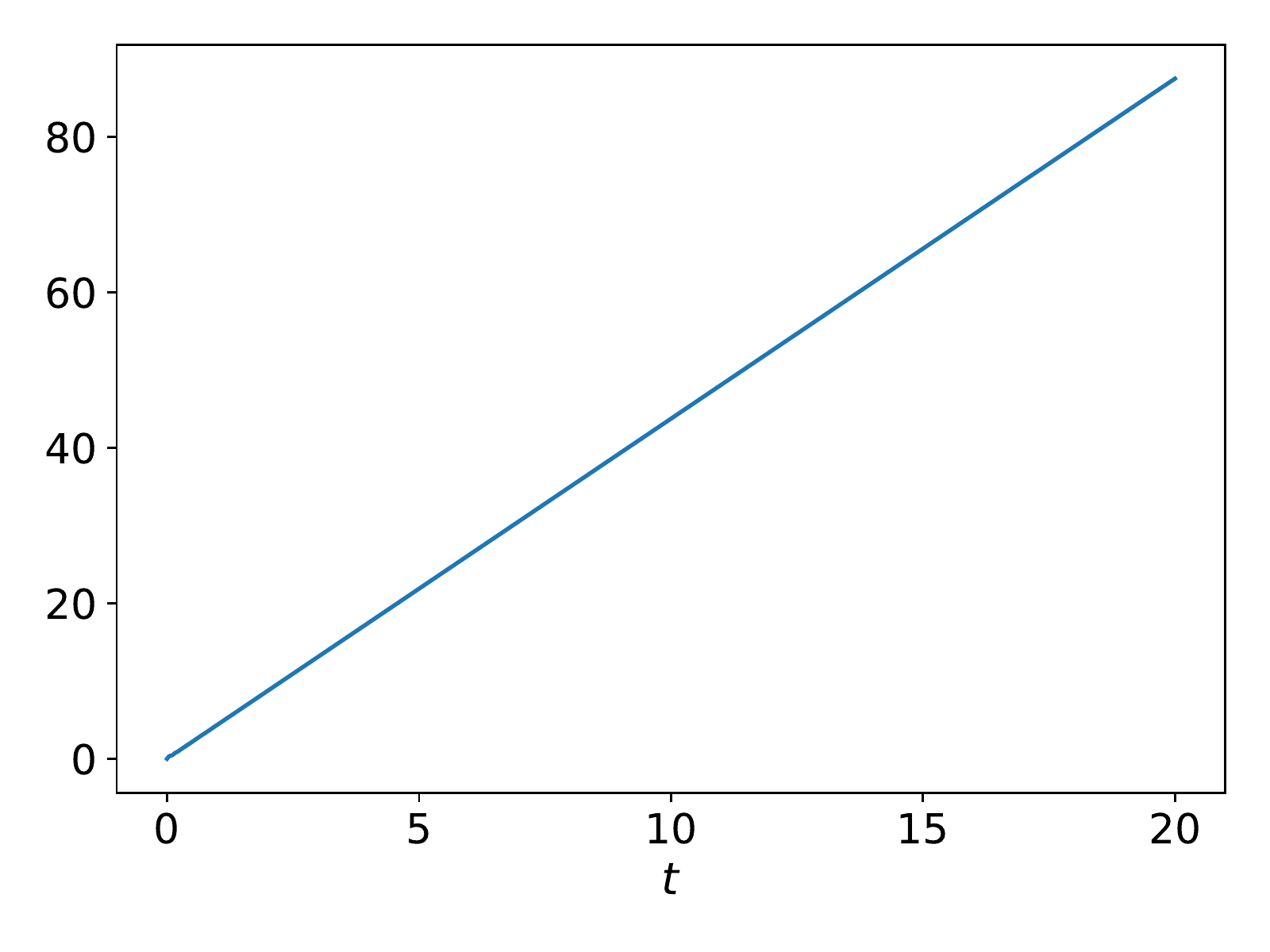}
  }
  \hspace{10pt}
  \subfloat[$\cG^n_h(\bE^n_h)$ for $P_J = P_{V^1_h}$]{%
    \includegraphics[width=0.3\textwidth]{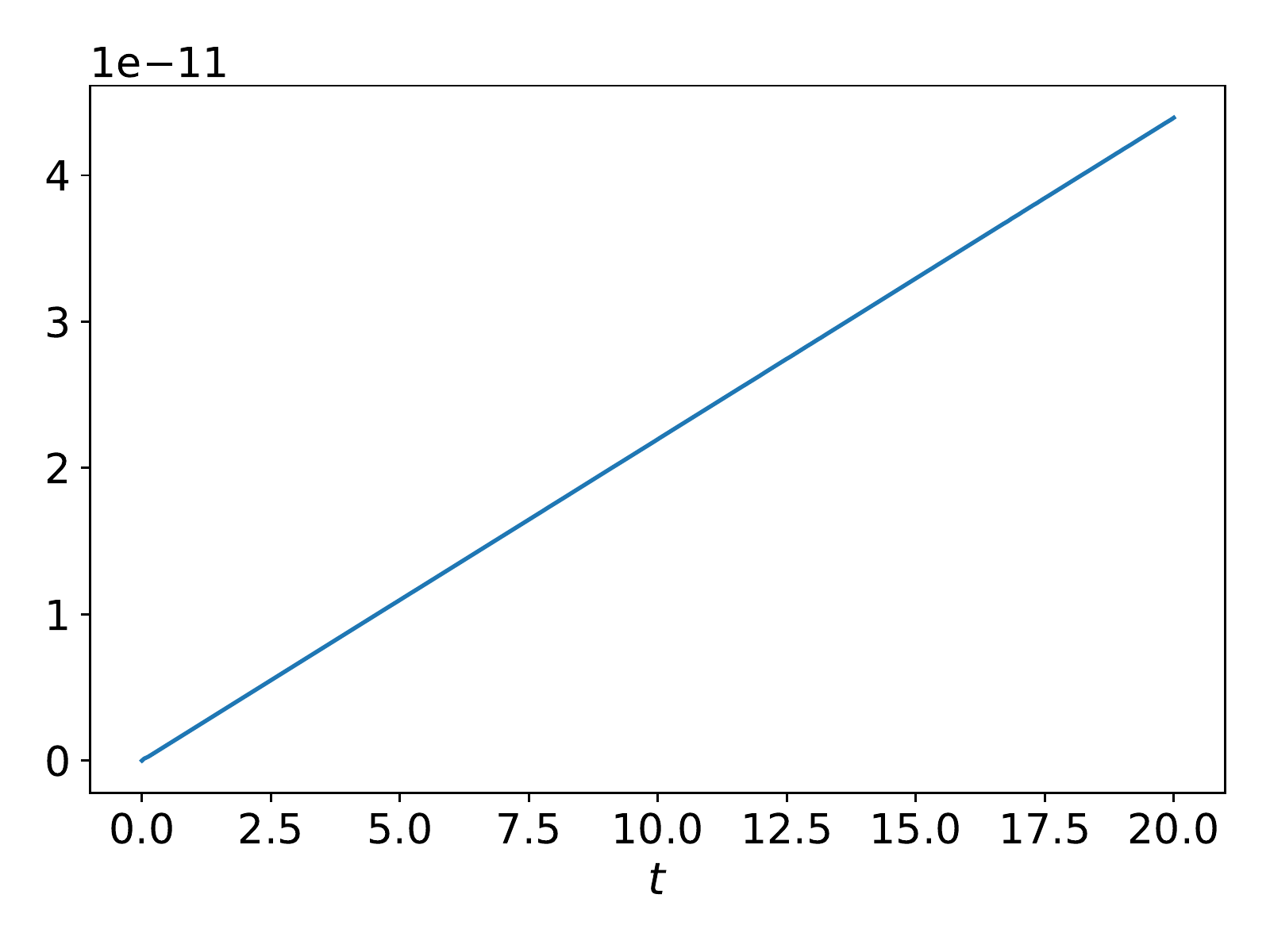}
  }
  \hspace{10pt}
  \subfloat[$\cG^n_h(\bE^n_h)$ for $P_J = \t\Pi^1_h$]{%
    \includegraphics[width=0.3\textwidth]{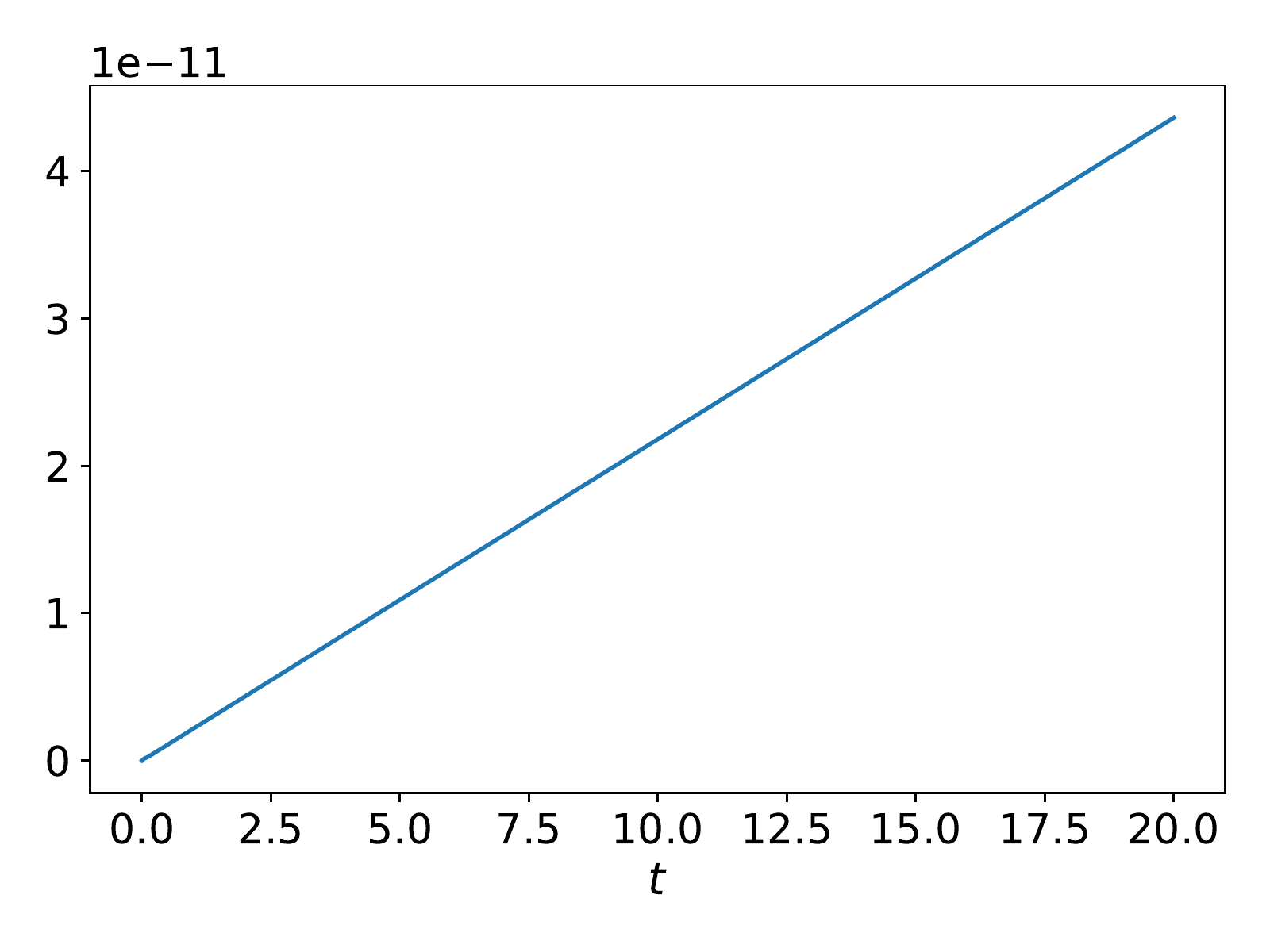}
  }
  \\
  \subfloat[$\cG^n_h(P^1_h \bE^n_h)$ for $P_J = \Pi^1_h$]{%
    \includegraphics[width=0.3\textwidth]{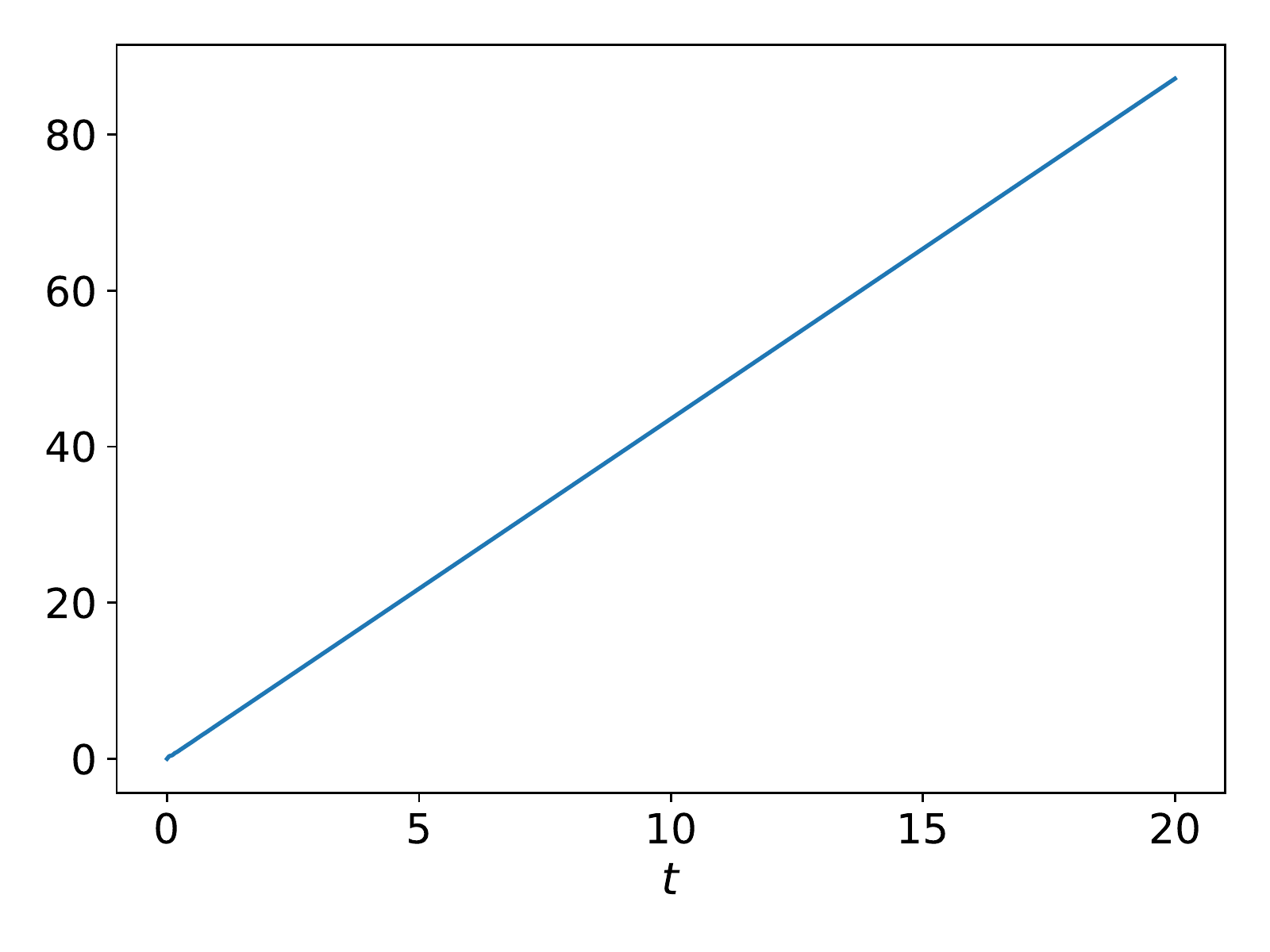}
  }
  \hspace{10pt}
  \subfloat[$\cG^n_h(P^1_h \bE^n_h)$ for $P_J = P_{V^1_h}$]{%
    \includegraphics[width=0.3\textwidth]{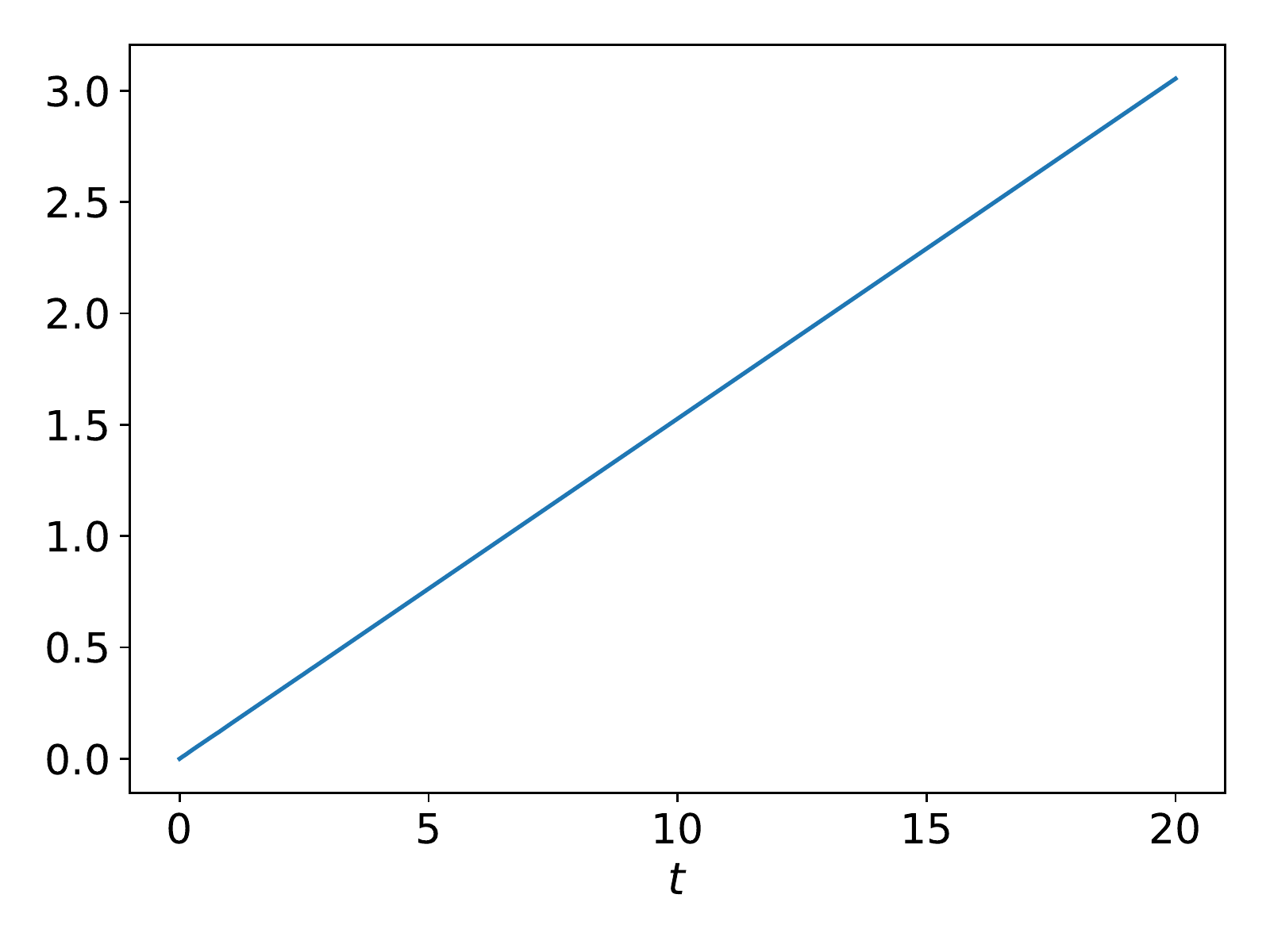}
  }
  \hspace{10pt}
  \subfloat[$\cG^n_h(P^1_h \bE^n_h)$ for $P_J = \t \Pi^1_h$]{%
    \includegraphics[width=0.3\textwidth]{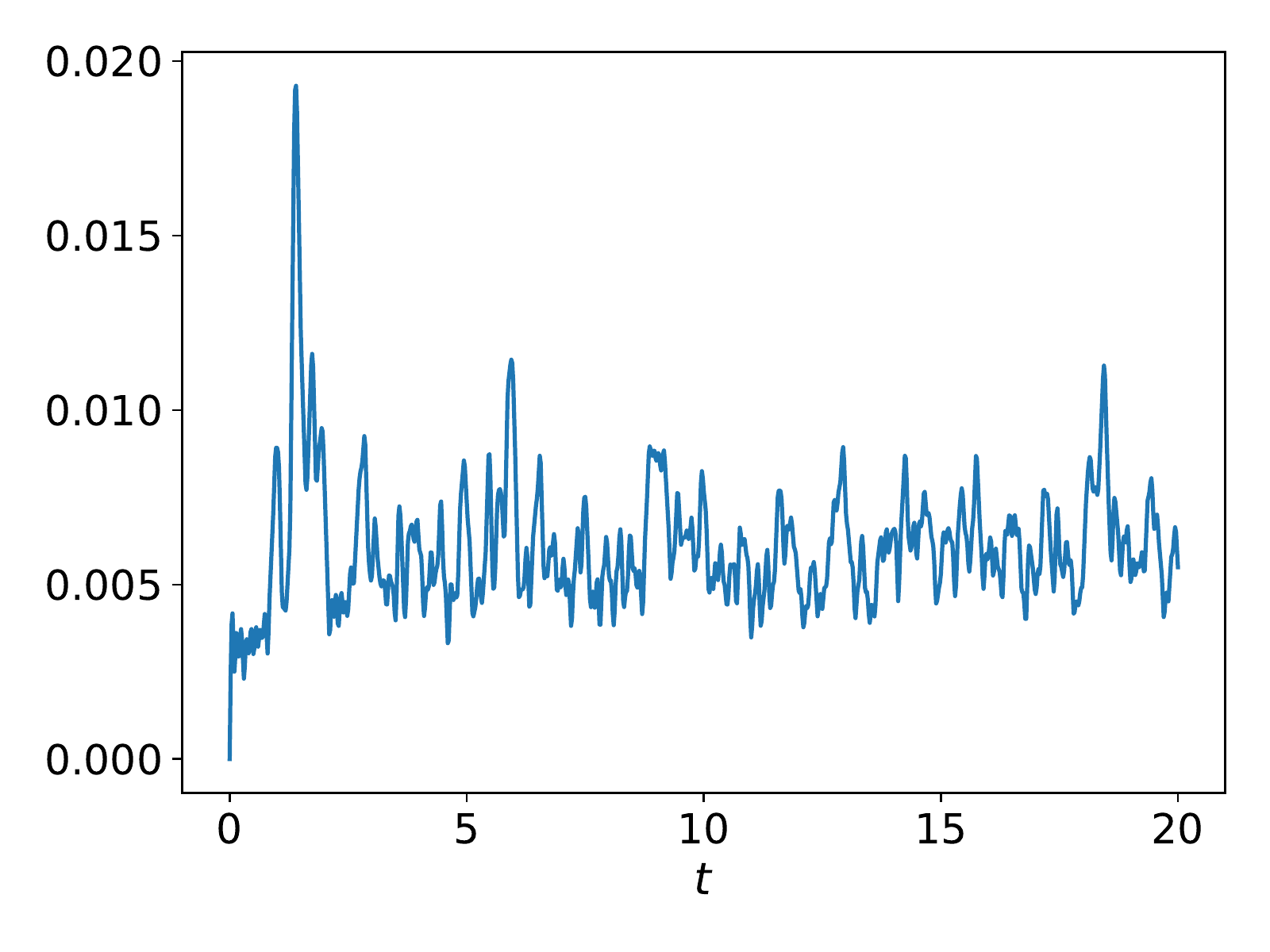}
  }
\end{center}
\caption{
Discrete Gauss law errors \eqref{Gnh} as a function of time for the test-case \eqref{td_tc2_EB0}--\eqref{td_tc2_J},
using different approximation operators for the discrete source $\bJ_h = P_J \bJ$. 
The top panels show the errors relative to the broken field $\bE_h$, whereas 
those on the bottom panels show the error relative to its conforming projection $P^1_h \bE_h$.
For each indicator, the left, middle and right plots correspond to the different 
source approximation operators shown in Figure\ref{fig:td_max_sols_J}.
}
\label{fig:td_max_diags_J}
\end{figure}

\section{Conclusions}
\label{sec:con}

In this work we have extended the classical theory of finite element exterior calculus (FEEC)
to mapped multipatch domains, using finite element spaces that are fully discontinuous across
the patch interfaces.
We refer to this approach as the ``broken FEEC'' or ``CONGA'' (COnforming/Non conforming GAlerkin) method.
While the foundational theory relative to the solution of the Hodge-Laplace equation was presented
in recent work~\cite{conga_hodge}, here we have presented stable broken-FEEC formulations
for many problems arising in electromagnetic applications, including Poisson's equation,
time-dependent and time-harmonic (source and eigenvalue) Maxwell's problems, and
magnetostatic problems with pseudo-vacuum and metallic boundary conditions.
Further, we have detailed a numerical framework based on tensor-product splines on each patch,
under the assumption of geometric conformity across the patch interfaces.

For all the electromagnetic problems presented, we have verified our broken FEEC framework
through extensive numerical testing in L-shaped and pretzel-like two-dimensional domains.
The latter geometry is particularly challenging because of its three holes and sharp reentrant corner.
The nominal order of accuracy was achieved in all cases, and the structure-preserving properties (such as divergence of harmonic constraints) were respected to floating-point accuracy.
For the time-dependent Maxwell problem with a current source, we could also observe long-term stability of the method, and presented alternative formulations which lack this property.

Given its solid theoretical bases and convincing numerical results, we are confident that the broken FEEC framework will find practical use in the computational physics community.
To this end we plan to relax the grid conformity constraints at the interfaces,
allowing for independent refinement of the patches,
and to investigate the efficiency of parallel implementations for high-performance computing applications.
This will allow us to tackle large problems in three dimensions, including MHD and kinetic models for plasma physics.


\section*{Acknowledgments}

The authors would like to thank Eric Sonnendrücker for inspiring discussions throughout this work.
The work of Yaman Güçlü was partially supported by the European Council under the Horizon 2020 
Project Energy oriented Centre of Excellence for computing applications - EoCoE, Project ID 676629.

\bibliography{bibfile_ppc}   

\end{document}